\documentclass[11pt]{article}

\usepackage{amsthm,amsfonts,amssymb,amsmath,oldgerm}
\usepackage{fullpage}
\usepackage[pdftex]{graphicx}
\usepackage{epstopdf}
\usepackage{mathrsfs}
\usepackage{mathtools}
\usepackage{xcolor}
\usepackage[hypertexnames=false,colorlinks=true,urlcolor=blue,linkcolor=blue,citecolor=blue]{hyperref}
\usepackage[text={7in,9.5in},centering,letterpaper]{geometry}
\usepackage[colorinlistoftodos]{todonotes}            
\usepackage[font=footnotesize]{caption}  
\numberwithin{equation}{section}
\usepackage[numbers,sort&compress]{natbib}

\newcommand\smallO{
  \mathchoice
    {{\scriptstyle\mathcal{O}}}
    {{\scriptstyle\mathcal{O}}}
    {{\scriptscriptstyle\mathcal{O}}}
    {\scalebox{.7}{$\scriptscriptstyle\mathcal{O}$}}
  }
\newcommand{\Z}{\mathbb Z}
\newcommand{\R}{\mathbb R}
\newcommand{\C}{\mathbb C}
\newcommand{\ri}{\mathrm{i}}
\newcommand{\re}{\mathrm{e}}
\newcommand{\de}{\mathrm{d}}
\newcommand{\El}{\mathcal{L}}

\newcommand{\N}{\mathbb{N}}
\newcommand{\T}{\mathcal{T}}
\newcommand{\f}{\mathrm{f}}
\newcommand{\bb}{\mathrm{b}}
\newcommand{\rr}{\mathrm{r}}
\newcommand{\lr}{\mathrm{l}}
\newcommand{\lf}{\mathrm{lf}}
\newcommand{\uf}{\mathrm{uf}}
\newcommand{\uu}{\mathrm{u}}
\newcommand{\su}{\mathrm{s}}
\newcommand{\cc}{\mathrm{c}}
\newcommand{\F}{\mathcal{F}}
\newcommand{\eps}{\varepsilon}
\let\epsilon\varepsilon

\newtheorem{theorem}{Theorem}[section]
\newtheorem{proposition}[theorem]{Proposition}

\newtheorem{lemma}[theorem]{Lemma}

\theoremstyle{definition}
\newtheorem{definition}[theorem]{Definition}

\allowdisplaybreaks[3]

\title{Diffusive synchronization of phase waves in the FitzHugh--Nagumo system
}

\author{Montie Avery\thanks{Department of Mathematics, Emory University, 400 Dowman Drive, Atlanta, Georgia 30322, USA, \texttt{msavery@emory.edu}}\and Paul Carter\thanks{Department of Mathematics, University of California, Irvine, 340 Rowland Hall, Irvine, CA 92697, USA, \texttt{pacarter@uci.edu}}\and  Bj\"orn de Rijk\thanks{Department of Mathematics, Karlsruhe Institute of Technology, Englerstra{\ss}e 2, 76131 Karlsruhe, Germany, \texttt{bjoern.rijk@kit.edu}}\and  Arnd Scheel\thanks{School of Mathematics, University of Minnesota,  206 Church St. S.E., Minneapolis, MN 55455, USA, \texttt{scheel@umn.edu}}}

\date{\today}

\begin{document}

\maketitle

 \begin{abstract}
   We analyze synchronization of relaxation oscillations in multiple-timescale reaction-diffusion systems. Interpreting synchronization as convergence to frequency-synchronized wave-train solutions, we resolve for the first time the case of \emph{phase waves}. These waves are nearly phase-synchronized relaxation oscillations, featuring quasistationary plateaus of length $\eps^{-1}$ separated by fast transition layers, where $\eps\ll1$ is the timescale separation parameter. Tracking the decay of modulations via a Bloch-wave eigenfunction analysis, we find a remarkably weak interaction strength of order $\eps^{8/3}$. This weak layer interaction and many of the technical difficulties arise from repeated scattering of eigenfunctions through fold points at the ends of the quasistationary plateaus. We capture this by combining a novel geometric desingularization approach with Lin's method, exponential trichotomies, and the Riccati transform. While our spectral stability analysis yields diffusive synchronization of all phase waves in the FitzHugh--Nagumo system, it also identifies potential finite-wavelength instabilities, which we realize in a system variant.
\end{abstract}



\tableofcontents

\section{Introduction}\label{sec:introduction}

The collective behavior of oscillators has been at the center of many questions in nonlinear dynamics, with applications including synchronization in mechanical systems~\cite{strogatzhuygens}, in neuropathology~\cite{Hammond2007,epilepsy}, in the social sciences~\cite{millenium}, in electrical grids~\cite{electricgrid}, or in chemical experiments~\cite{bzdroplets}. Synchronization here means first of all that the temporal behavior  of any oscillator in the collection is periodic with one common period for all oscillators,  that is, the collective behavior is \emph{frequency synchronized}. In an even stronger form of synchronization, all individual oscillators go through  a particular fixed phase of the oscillation at the same time, that is, they are \emph{phase synchronized}. Predicting synchronization, and failure thereof, can then explain observations, guide treatments, and more theoretically elucidate our understanding of collective behavior between order and chaos. Our results are concerned with frequency synchronization and possible desynchronization in relaxation oscillators. We shall start in \S\ref{s:mos} with background on questions of  synchronization relevant to our results. We then set up the FitzHugh--Nagumo system with its wave-train solutions in \S\ref{s:setup} and state our main stability result and its implications in \S\ref{s:mr} and \S\ref{s:instability}. 

\subsection{Phenomenology of synchronization}\label{s:mos}
\paragraph{Models of synchrony.} The simplest descriptions of oscillators rely on modeling the phase $\Phi$ of the oscillation as through a differential equation $\Phi_t=g(\Phi)$ on the circle $\Phi\in S^1=\R/\Z$,  leading to phase models such as the celebrated Kuramoto model~\cite{kuramoto1975international,RevModPhys.77.137}, where the transition from synchronization to unlocked states is now well-understood in kinetic limits~\cite{chiba}. More realistic models include inertia~\cite{kuramotoinertia} or amplitudes~\cite{pikovsky2001synchronization}. The Landau oscillator, a universal normal form near Hopf bifurcations, incorporates amplitudes in the simplest form. In spatially extended oscillatory systems, it leads to the complex Ginzburg--Landau equation as a modulation equation~\cite{worldofcgl,mielkecgl}. We are concerned here with what are often more representative and relevant models of oscillators that describe relaxation oscillations in multiple timescale systems. Specifically, we are interested in the FitzHugh--Nagumo system, or the closely related van-der-Pol equations. In contrast to phase or Landau oscillators, these models offer more realistic phase-resetting curves, that is, the response of oscillations 
to external stimuli depends quite sensitively on the phase of the oscillation. The multiple timescale structure -- with fast relaxation on $\mathcal{O}(1)$ time scales between slow adaptation on $\mathcal{O}(1/\eps)$ time scales, where $\eps > 0$ is a small parameter -- also appears intrinsically in applications including ecosystem dynamics~\cite{klausmeier}, neuroscience~\cite{RevModPhys.78.1213}, heart arrhythmias~\cite{tyson1987spiral}, chemical reactions~\cite{fieldnoyes}, and even fluid flows and turbulence~\cite{barkleypipe}. From this broad phenomenological perspective, our work here addresses two questions:

\begin{enumerate}
\item What are time scales of synchronization?
\item When and how does synchronization fail?
\end{enumerate}

We explore these questions in the context of a continuous distribution of oscillators, modeled through reaction-diffusion systems on the line, with a specific focus on the FitzHugh--Nagumo system in the oscillatory regime; see \S\ref{s:setup}, below. The spatially continuous setting avoids additional complicating effects due to spatial discreteness. Moreover, the planar dynamics of the ODE, describing spatially independent solutions of the FitzHugh--Nagumo system, isolate key features of multiple timescale relaxation, thus making this equation and its variants a model of choice in many of the applications mentioned above. Indeed, much is known about the spatio-temporal dynamics of the FitzHugh--Nagumo equation, including existence, stability, and bifurcation results for traveling fronts~\cite{deng1991,liu2006}, pulses~\cite{carpenter1977geometric,CSosc, CSbanana, hastings1982,flores1991stability,jones1991construction}, and periodic wave trains~\cite{eszter1999evans,hastings1974,hastings1976,li2025nonlinearstabilitylargeperiodtraveling, soto2001geometric}. Our focus here is on these wave trains which are periodic in both space and time and propagate as traveling waves in this spatially continuous setting. The collective dynamics of wave trains are frequency synchronized -- that is, all points in space oscillate with the same frequency -- while the  phase of the oscillation has a constant spatial gradient; see Figure~\ref{fig:oscillations_spacetime}. Synchronization then corresponds to the convergence to one of those frequency-synchronized wave trains, from initial conditions that are close to the wave train -- that is, already almost synchronized -- or more globally. We shall focus on the former, local question, which translates (i) and (ii) into questions of stability of oscillations: can wave trains be unstable (ii), and what is the rate of convergence to wave trains when they are stable (i)? Our work is motivated by somewhat dramatic differences in frequency synchronization that are illustrated in Figure~\ref{f:random}, which shows how synchronization time scales differ within the same equation by only changing the wavelength, and by the possible failure of synchronization in related models, illustrated in Figure~\ref{fig:toy_wavetrain}. 

\begin{figure}
\centering
\includegraphics[width=0.45\linewidth]{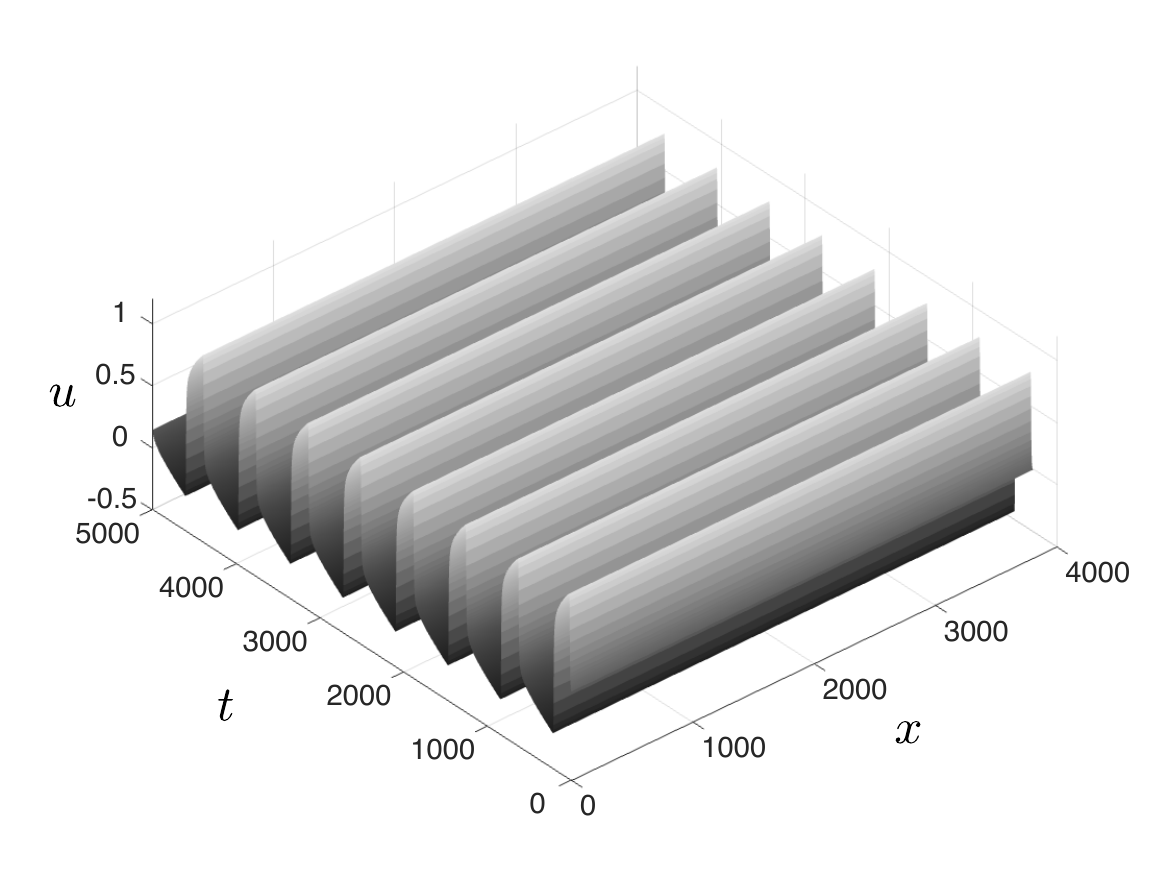}\hspace{0.05\textwidth}
\includegraphics[width=0.45\linewidth]{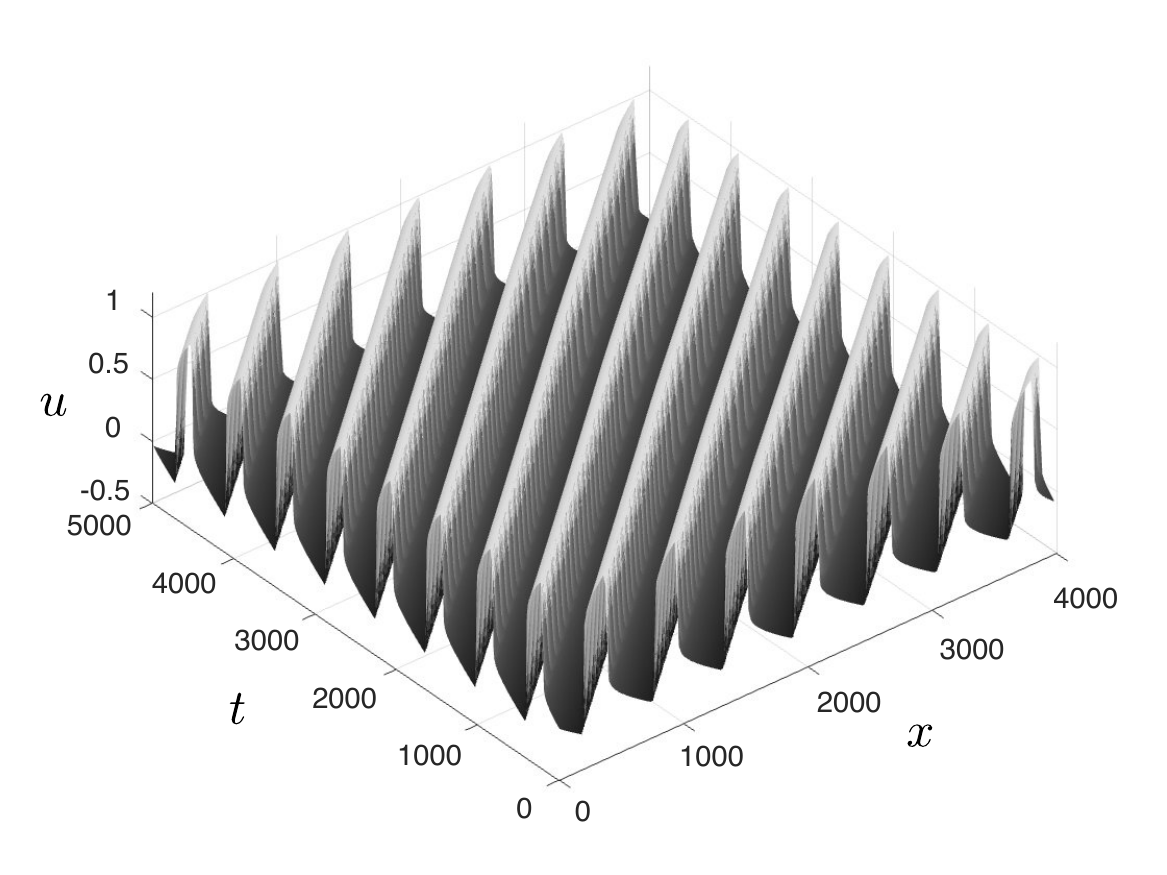}
\caption{Spacetime plot of profiles $u(x,t)$ corresponding to spatially homogeneous oscillations (left) and traveling wave trains (right) obtained numerically in~\eqref{eq:FHN_pde} for $a=0.2, \gamma=1, \eps = 0.001$. } 
\label{fig:oscillations_spacetime}
\end{figure}

\paragraph{Effective diffusivity and stability.} For single oscillators, or small groups of synchronized oscillators, stability and rates of convergence are determined by the real part of Floquet exponents $\lambda$ of the linearized equation: positive real parts yield instability and desynchronization, negative real parts $\Re(\lambda)=-\eta$ give exponential rates $\eta$ of synchronization. Perturbing a synchronized state will generally lead to a phase shift of the collective oscillation due to the trivial Floquet exponent $\lambda=0$ associated with time translations. In large or even moderately sized systems, Floquet exponents closest to the imaginary axis arise from modulating this phase response across oscillators. In other words, the phase of oscillators may be perturbed in a non-uniform fashion across the system and temporal dynamics only slowly heal the resulting  phase mismatch. The dynamics of phase modulations of oscillators in large, spatially extended systems can be described using long-wavelength modulation theory~\cite{DSSS}. In this regime, the phase dynamics are well-approximated by a viscous eikonal equation. Modulating the phase of a wave-train solution $u(x,t) = u_{\mathrm{wt}}(l x - \omega t)$ with wavenumber $l \in \R \setminus \{0\}$ and frequency $\omega \in \R$ via  
\[
u_\mathrm{wt}(\ell x-\omega t +\Phi(x,t))\sim u_\mathrm{wt}(\ell x-\omega t )+\Phi(x,t)\cdot u_\mathrm{wt}'(\ell x-\omega t),
\]
one finds in a long-wavelength limit
\[
\Phi_t+d_2 \Phi_x^2=d_\mathrm{eff}\Phi_{xx},
\]
where $d_\mathrm{eff} > 0$ is an effective diffusivity and the coefficient $d_2 \in \R$ encodes nonlinear dispersion. Wave-train solutions are now of the form $\Phi(x,t)=\ell_* x-\omega_* t$, with $\omega_*=d_2\ell_*^2$.  Small-amplitude perturbations of wave trains can be shown to behave according to solutions of the linearized equation, a convection-diffusion equation
\[
\Phi_t+2d_2 \ell_* \Phi_x =d_\mathrm{eff}\Phi_{xx}.
\]
That is, for $t\to\infty$ solutions are well-approximated by solutions to the heat equation, advected with the group velocity $2d_2\ell_*$. The time scales of synchronization correspond to the rate of decay in the diffusion equation, which in large systems is 
\[
\|\Phi(\cdot,t)\|_{L^\infty}\sim (d_\mathrm{eff}t)^{-1/2}\|\Phi(\cdot,0)\|_{L^1}, \qquad t > 0.
\]
We therefore refer to $T_\mathrm{eff}=1/d_\mathrm{eff}$ as the \emph{time scale of synchronization}, thus encoding how we shall answer question (i) in what follows. The effective diffusive behavior is illustrated in direct simulations in Figure~\ref{f:deff}, showing in particular the appearance of diffusive profiles $\Phi(x,t)\cdot u_\mathrm{wt}'(\ell_* x-\omega t)$.

The second question, from our point of view, asks for spectral instabilities of the synchronized state. In a simple dichotomy, this potential instability is either caused by  (A) a sign change in the effective diffusivity $d_\mathrm{eff}$, that is, ill-posedness of the phase approximation, or (B) by a finite-wavelength mode that is invisible in the phase approximation. 

Roughly speaking,  we answer the questions (i) and (ii) above for phase-wave trains in terms of the slow timescale parameter $\eps$ as follows:
\begin{enumerate}
\item in FitzHugh--Nagumo and variants, $d_\mathrm{eff}\sim \eps^{2/3}$ and $T_\mathrm{sync}\sim \eps^{-2/3}$, see Theorem~\ref{thm:spectral_stability};
\item in FitzHugh--Nagumo, no instabilities, see Theorem~\ref{thm:spectral_stability}; in variants of the FitzHugh--Nagumo system, $d_\mathrm{eff}>0$, but oscillatory Turing instabilities can occur at modulation wavenumber 
$k\sim \eps^{1/6}$, see \S\ref{s:instability}.
\end{enumerate}
In terms of the dichotomy mentioned above, the instability in (ii) is always of type (B), i.e.~not due to a change of sign of $d_\mathrm{eff}$. Therefore, it is not visible in a phase-reduction that fixes the small parameter $\eps$ in the relaxation oscillation, but rather caused by an instability of the neutral mode at a finite (for fixed $\eps$) wavenumber. We found the rigidity in (ii) in regard to the consistent stable effective diffusivity, $d_\mathrm{eff}>0$,
remarkable. It is strikingly different from the well-understood example of the complex Ginzburg--Landau equation, where the prevalent instability mechanism is a change in sign of $d_\mathrm{eff}$, known there as a sideband or Benjamin--Feir instability.
The sideband instability plays an outsized role in the transition from coherent spatio-temporal dynamics with ultimate phase-synchrony, to spatio-temporal chaos in its various forms~\cite{worldofcgl,CrossHohenberg}.  For frequency-synchronized states with a phase gradient, that is, for wave-train solutions, there is however in certain parameter regimes a finite-wavelength instability in the complex Ginzburg--Landau equation preceding the sideband instability~\cite{vH95}, somewhat reminiscent of the potential instability we observe here.

\begin{figure}
\begin{minipage}{0.48\textwidth}
\begin{center}  trigger waves\\[0.1in]
\includegraphics[height=2.1in]{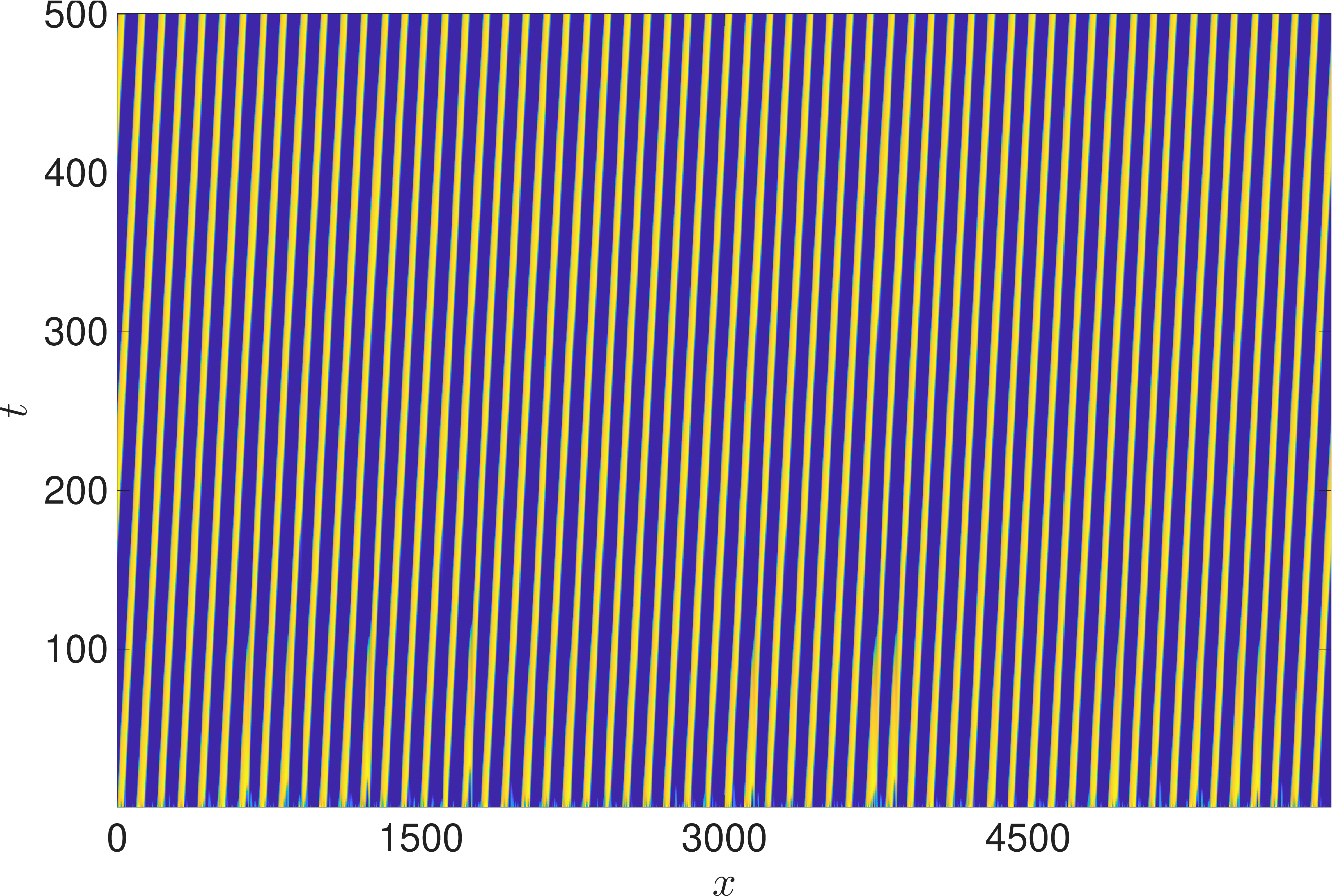}
    \end{center}
\end{minipage}\hfill
\begin{minipage}{0.48\textwidth}
 \begin{center}   phase waves\\[0.1in]
    \includegraphics[height=2.09in]{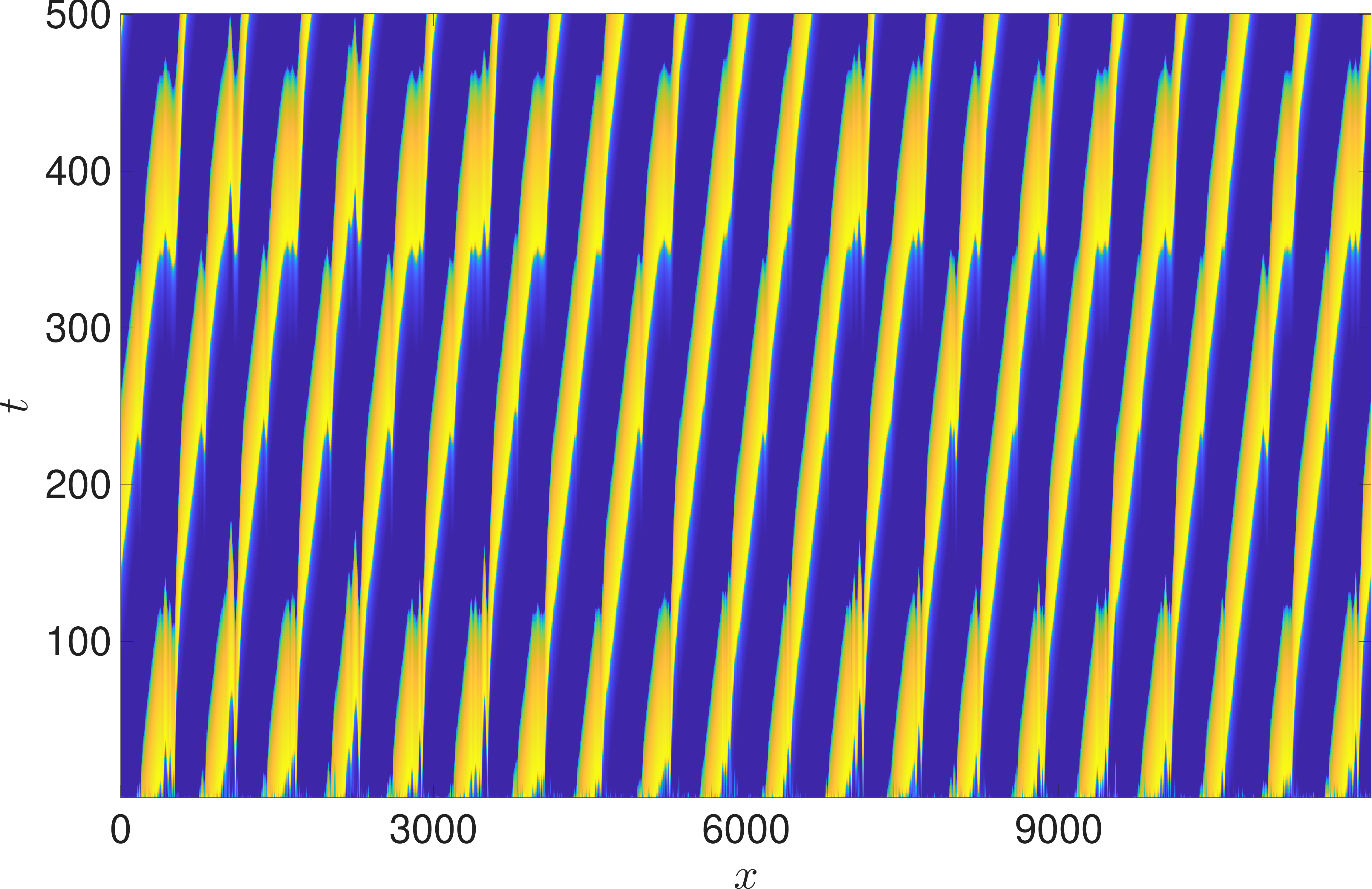}\end{center}
    \end{minipage}
        \caption{Illustration of the fast decay of perturbations of  trigger waves (left) compared to the weak relaxation of random perturbations of phase waves (right); see Appendix~\ref{s:dns} for details on implementation. }\label{f:random}
\end{figure}

\begin{figure}
\begin{minipage}{0.48\textwidth}
\begin{center}  trigger waves\\[0.1in]  \includegraphics[height=2.05in]{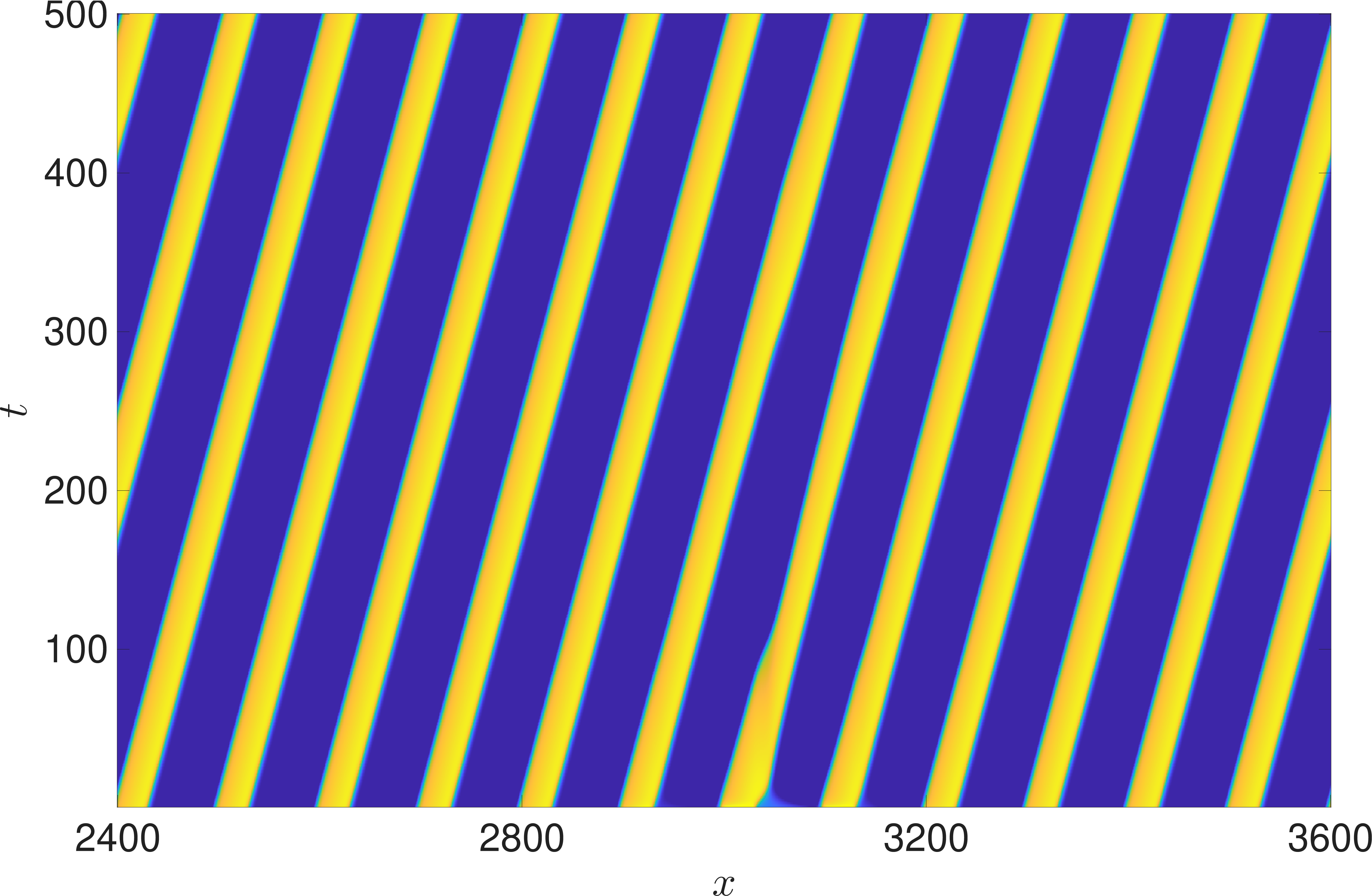}\\[0.2in] 
    \includegraphics[height=2in]{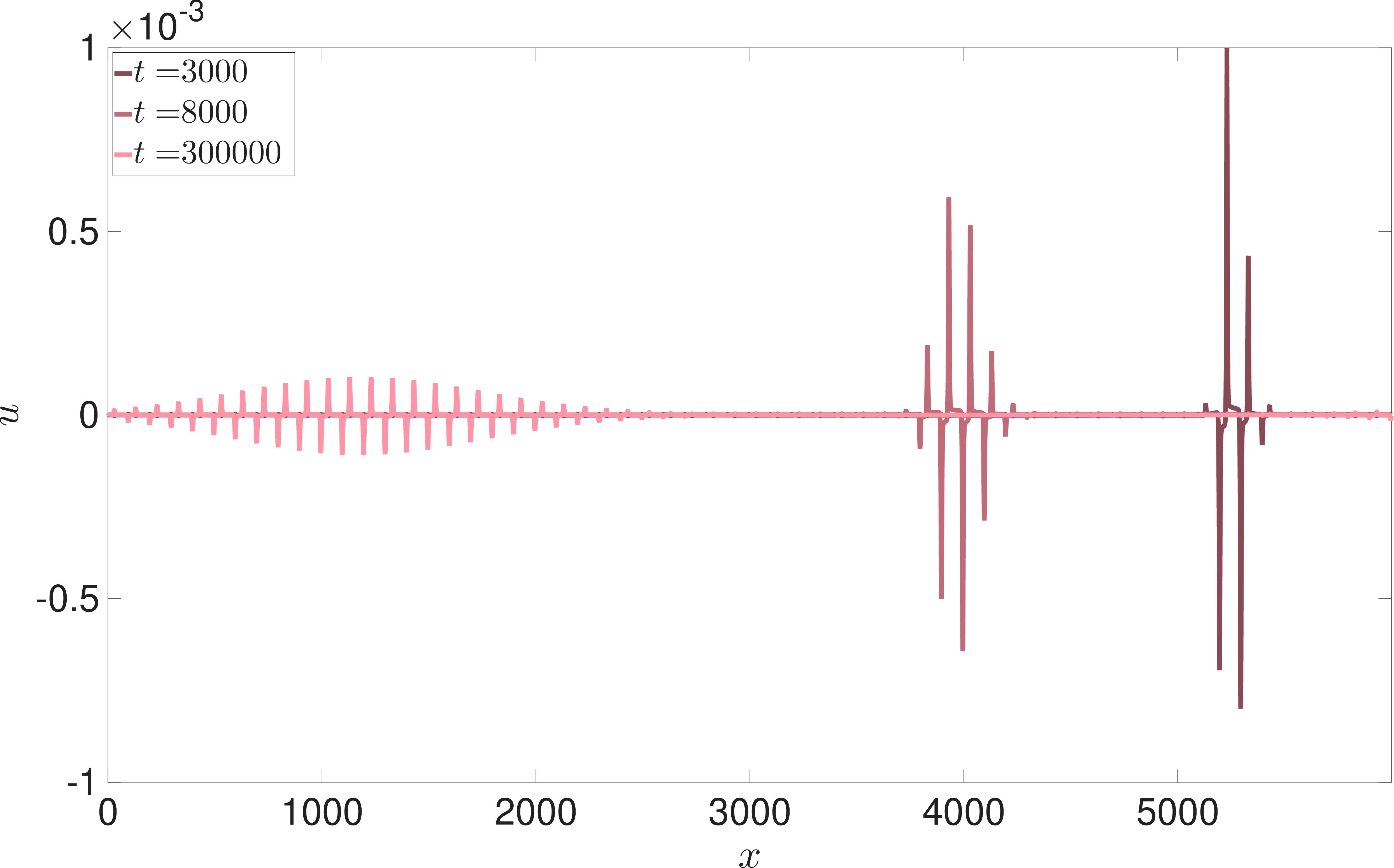}\\[0.2in] 
    \includegraphics[height=2.in]{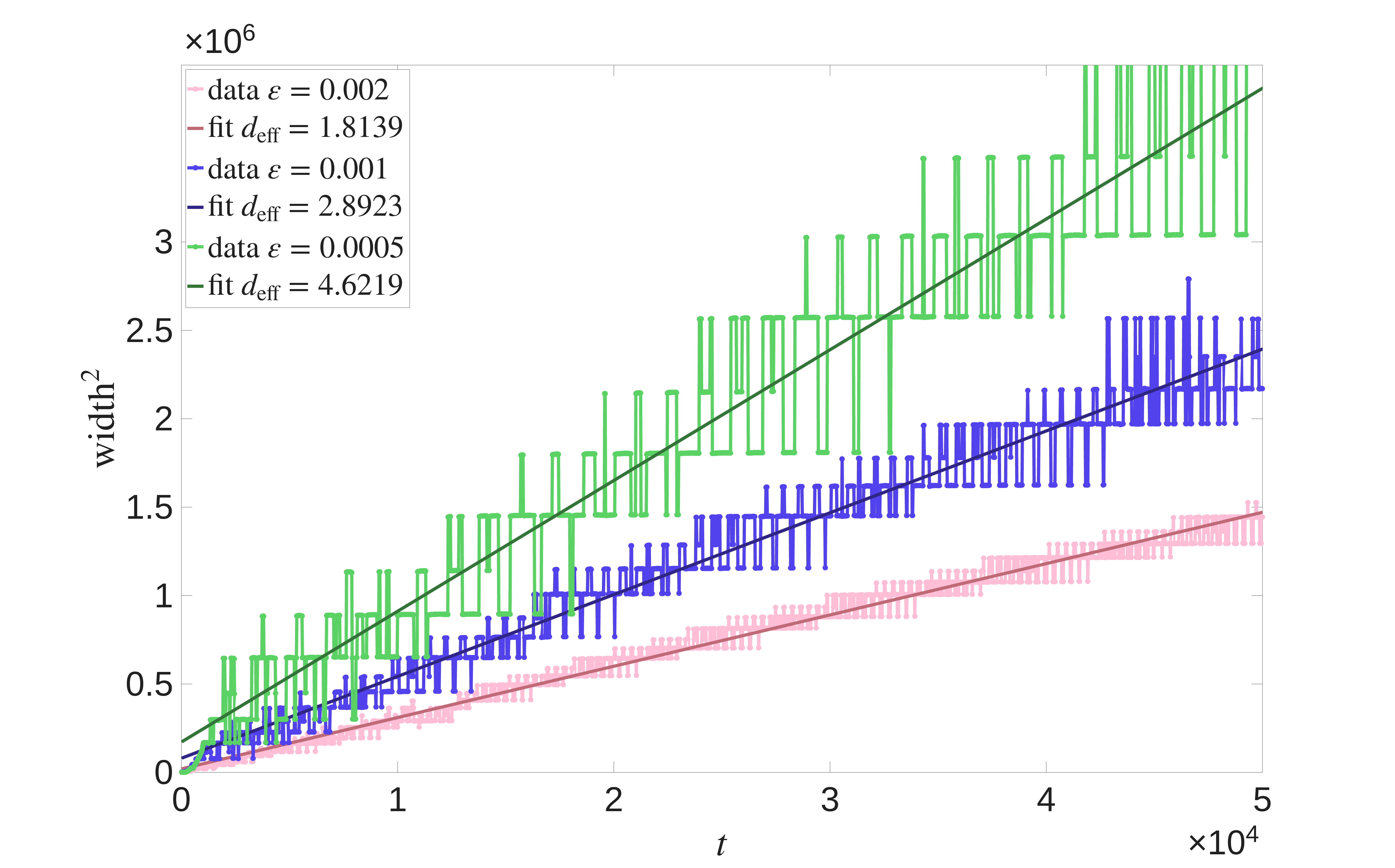}
    \end{center}
\end{minipage}\hfill
\begin{minipage}{0.48\textwidth}
 \begin{center}   phase waves\\[0.1in]
 \includegraphics[height=2.0in]{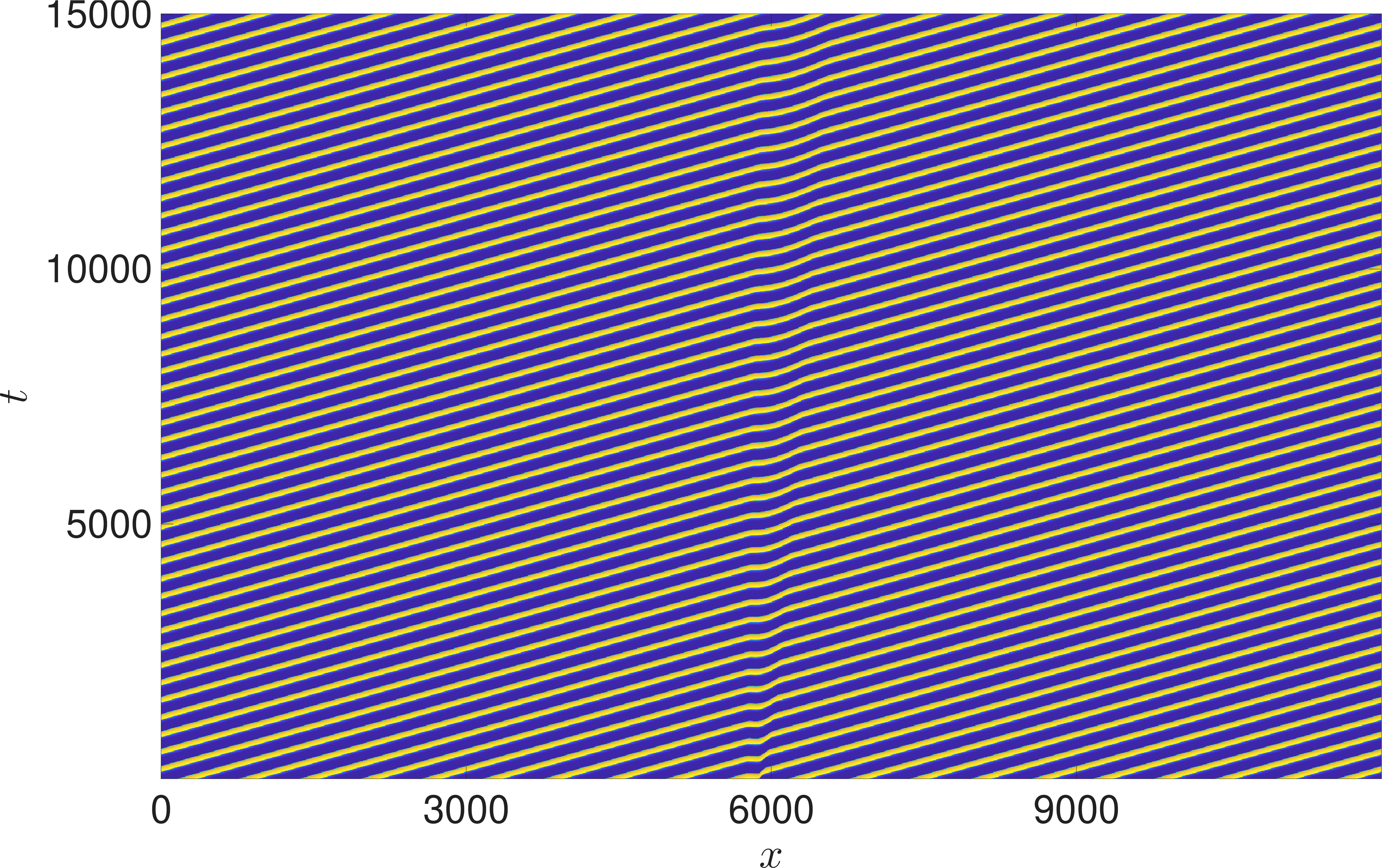}\\[0.2in] 
    \includegraphics[height=2in]{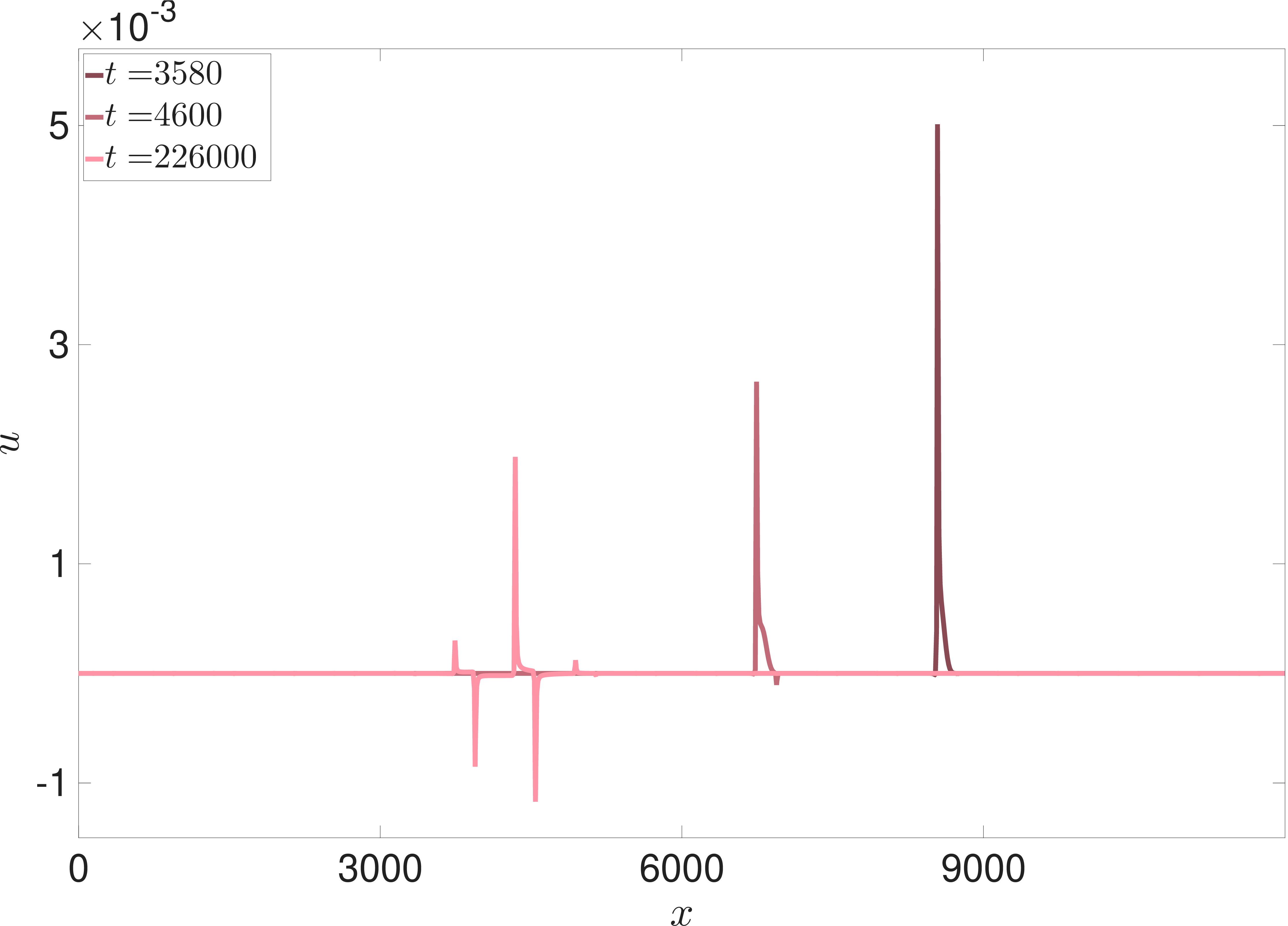}\\[0.2in] 
    \includegraphics[height=2in]{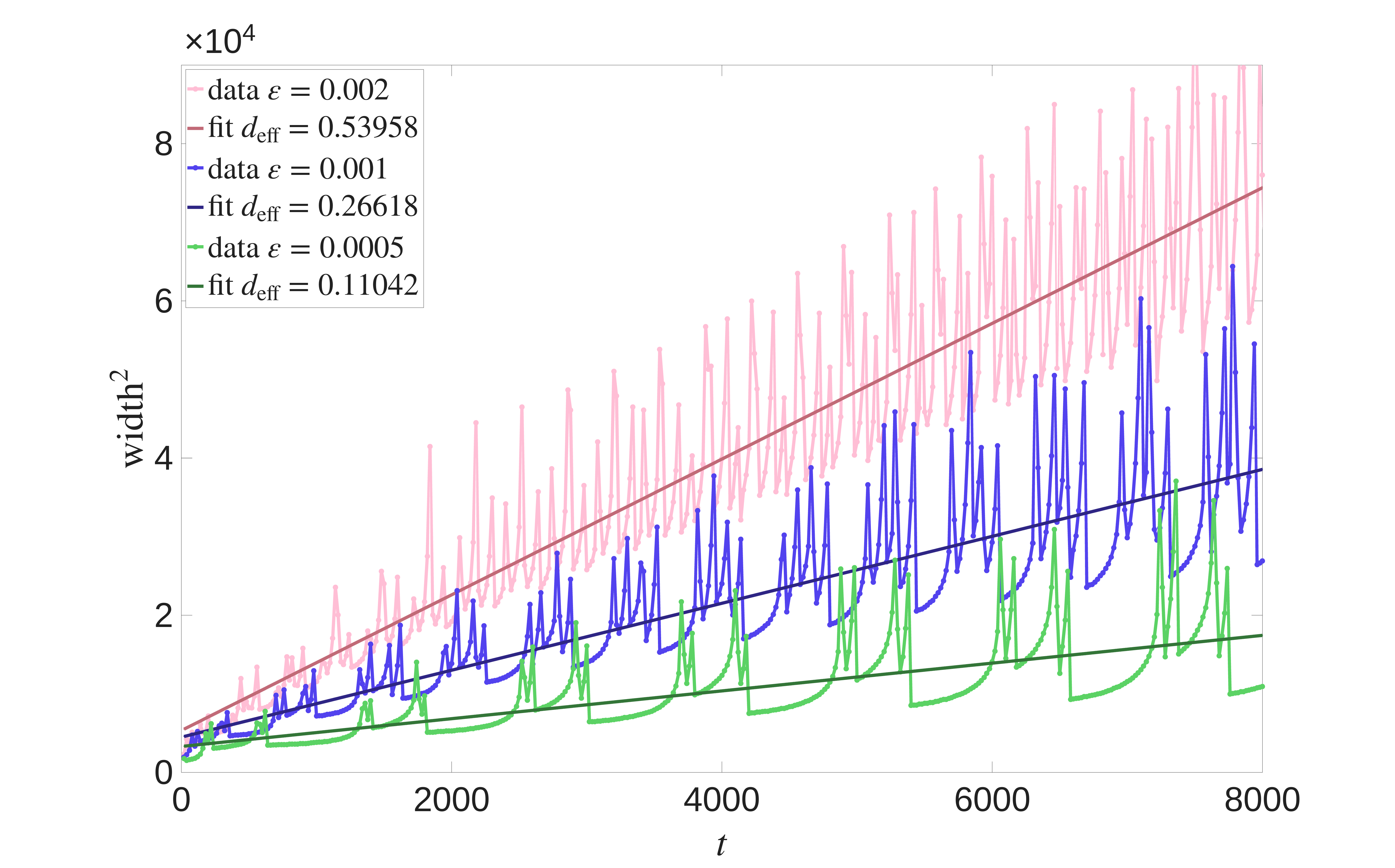}\end{center}
    \end{minipage}
\caption{Localized perturbation of trigger waves (left column) and phase waves (right column) in the first component by $10^{-2}\exp{(-x^2/100)}$. Top: space-time plot of $u$-component (only partial domain on left) shows quick relaxation in the trigger-wave case, with perturbation only visible up to $t=200$, and a persistent defect in the phase wave case visible up to $t=15000$ (note the different spatial and temporal plot ranges). Middle: Snapshots of perturbation profiles, subtracting a closest perfect periodic wave train from the solution. Perturbations are modulated and large near interfaces. They slowly travel to the left with the group velocity and decay in amplitude as their width grows (again fast in the trigger- and slow in the phase-wave case). Bottom: The decay is diffusive, which is illustrated here by plotting the square of the width of the region where the perturbation exceeds $10^{-5}$ versus time. Data and linear fit with slope $4d_\mathrm{eff}$ shows, as $\eps$ decreases,  the increase $d_\mathrm{eff}\sim 1/\eps$ in the trigger case, and the decrease  $d_\mathrm{eff}\sim \eps^{2/3}$ in the phase wave case; see Appendix~\ref{s:dns} for details on implementation.  }\label{f:deff}
\end{figure}

\paragraph{Synchronization via transition layer interaction.} 
Relaxation oscillations in FitzHugh--Nagumo- or van-der-Pol-type oscillators consist of two rapid switches between states of slowly varying amplitudes. 
In the spatially homogeneous oscillation, the rapid switches are initiated by the slow evolution hitting a fold point of a slow manifold. This slow passage through the fold contributes a characteristic $\mathcal{O}(\eps^{-1/3})$-correction to the leading-order $\mathcal{O}(\eps^{-1})$-part in the period of oscillations, which stems from the maximal time spent drifting along the slow branch; see Figure~\ref{fig:singular_slow}. 
Modulating homogeneous oscillations spatially, the time instances of the rapid switches vary in space: for wave trains, the switches occur along characteristics $\omega (\ell)t = \ell x$, with $\omega(0)$ the frequency of the spatially homogeneous oscillation; see Figure~\ref{fig:oscillations_spacetime}. Dispersion, that is, the dependence of $\omega$ on the spatial wavenumber $\ell$ encoded in the coefficient $d_2$ in the modulation equation above, can be thought of as reflecting the interaction of these rapid switches. 
When $\ell$ increases, i.e. for spatially more narrowly spaced transition layers, the drift along the slow manifold is cut short by an early transition triggered by the interaction of layers, rather than induced by the phase of the oscillation reaching its final state. The associated periodic wave trains for larger $\ell$ are thus called \emph{trigger waves}, while the waves for small $\ell$ are called \emph{phase waves}; see Figure~\ref{fig:continuation} for a depiction of these waves in the FitzHugh--Nagumo system.

Trigger waves are related to waves in the excitable regime~\cite{bordiougov2006trigger}. Here, in the absence of spatially homogeneous oscillations, oscillatory behavior is organized by excitation pulses, each consisting of a rapid jump at the front and a second relaxation at the back, and their interaction. Interaction of layers in this excitable regime, or more generally the regime of trigger waves, is fairly well-understood through a perturbative analysis; see~\cite{sslongwavelength} for large separation at fixed $\eps$ and~\cite{eszter1999evans,li2025nonlinearstabilitylargeperiodtraveling} for $\mathcal{O}(\eps^{-1})$-separation and layers away from the fold points of the slow manifold. In both cases, the position of transition layers is naturally associated with a zero eigenvalue in the linearization and therefore a ``soft mode'', leading to effective reduced descriptions. 

In contrast, the transition layers in phase waves do not possess such a natural zero eigenvalue, a fact that was noticed when analyzing the stability of pulses in a modified FitzHugh--Nagumo system~\cite{beckjonesschaefferwechselberger}. 
Therefore, while the stability analysis features a characteristic neutral mode associated with translations, it does not possess individual modes associated with translations of the two distinct transition layers, and there does not appear an obvious way to cast the dynamics as reduced weak-interaction dynamics.
Our main result, that characterizes in particular $d_\mathrm{eff}$, can be thought of in this language of transition layer interaction  as, for the first time, quantitatively characterizing the interaction of fast transition layers resulting from the passage through a fold. 

Quantitatively, suppose that modulations of layers at distance $\mathcal{O}(\eps^{-1})$ relax on a time scale $\eps^{-\beta}$. We can predict this time scale by substituting modulations on a spatial scale $\eps^{-1}$ into the effective diffusive eikonal approximation equation $\Phi_t+d_2 \Phi_x^2=d_\mathrm{eff}\Phi_{xx}$ and thereby predict layer dynamics on time scales much longer than the previously defined synchronization time scale, $(d_\mathrm{eff}\eps^2)^{-1}=\eps^{-2}T_\mathrm{sync}$. Equating $\eps^{-\beta}=(d_\mathrm{eff}\eps^2)^{-1}$ then also gives a comparative interpretation of our results and analogous results on  trigger waves in~\cite{eszter1999evans}:
\begin{itemize}
    \item \emph{trigger waves:} $d_\mathrm{eff}\sim \eps^{-1} \longrightarrow $ layer interaction strength $\eps^\beta$, relaxation time scale $\eps^{-\beta}$, $\beta=1$;
    \item \emph{phase waves:} $d_\mathrm{eff}\sim \eps^{2/3} \longrightarrow $ layer interaction strength $\eps^\beta$, relaxation time scale  $\eps^{-\beta}$, $\beta=8/3$. 
\end{itemize}
We refer to Figure~\ref{f:deff} for an illustration of the relaxation near both trigger and phase waves. We remark that the relaxation time scale of trigger waves also appears as a weak interaction eigenvalue of order $\eps$ in the stability of  pulses~\cite{PBR,HS, van2008pulse}.

In summary, our results exhibit an extraordinarily weak interaction of transition layers of phase waves, contrasting a wealth of results on layers in excitable media.  Direct simulations in Figure~\ref{f:random} illustrate how the predictions manifest themselves in the extremely slow synchronization of phase waves when subjected to random perturbation, as opposed to the rapid synchronization of trigger waves. On the other hand, we believe that the mathematical techniques developed here will be useful in singularly perturbed spectral stability problems exhibiting fold dynamics far beyond the specific setting that we focus on.

\paragraph{Defect-mediated frequency synchronization.}
In addition to these natural questions of stability and synchronization, our work is strongly motivated by defect-mediated  synchronization phenomena. A simple intuitive example of such synchronization is the presence of ``pacemakers'', such as localized regions that oscillate at a different frequency which in turn propagates through the medium~\cite{stichmikhailov, kollarscheel}. As a result, a coherent state spreads with finite speed, rather than  only diffusively, with synchronization in a system of size $L$ achieved after time $T\sim L$ rather than $T\sim L^2$. The coherent state is not phase synchronized but rather frequency synchronized: the phase exhibits a constant gradient while information propagates away from the pacemaker. The constant-gradient state thus corresponds to a wave train, periodic in time and space while rigidly propagating. In the absence of  external pacemakers, self-organized pacemakers such as spiral waves~\cite{ssspiral} and sources~\cite{SandstedeScheelDefects,dodsonlewis} can have a similar effect, establishing domains of frequency synchronization surrounding individual pacemakers in a glassy state. 

A similar mechanism of frequency synchronization is related to the growth of the region where oscillations are observed,  induced either by an external quench~\cite{gohscheel} or, again in a self-organized fashion, by the spreading of the oscillatory instability through an invasion front~\cite{vanSaarloos}. The latter was studied in the example of the FitzHugh--Nagumo system in~\cite{CASCH,FHNpulled}; see also Figure~\ref{fig:invasion_front}. 
Such invasion fronts generate wave trains, i.e.~frequency-synchronized states, in their wake in the sense that their group velocity is smaller than the propagation speed of the front interface~\cite{SandstedeScheelDefects}.

\subsection{Wave trains in the FitzHugh--Nagumo system}\label{s:setup}

We introduce the FitzHugh--Nagumo system in the oscillatory regime and present existence results on phase and trigger waves. 

\paragraph{The oscillatory regime.}
We consider the FitzHugh--Nagumo system 
\begin{align}\label{eq:FHN_pde}
\begin{split}
u_t &= u_{xx} + f(u) - w,\\
w_t &=\epsilon(u-\gamma w - a),
\end{split}
\qquad x \in\R, \, t \geq 0, \, (u,w)\in\R^2,
\end{align}
with cubic nonlinearity $f(u) = u(u-a)(1-u)$, parameters $0 < \varepsilon \ll 1$, $0 < a < \frac{1}{2}$ and 
\begin{align} 0<\gamma<\gamma_*(a) \coloneqq  {9}\left(1+2a-2a^2+(1-2a)\sqrt{a^2-a+1}\right)^{-1}.\label{e:gammacond}\end{align}
The planar system governing $x$-independent solutions,
\begin{align}\label{eq:FHN_kinetics}
\begin{split}
u_t &= f(u) - w,\\
w_t &=\epsilon(u-\gamma w - a),
\end{split}
\end{align}
behaves much like the classical van-der-Pol equation, with an unstable equilibrium $(u,w)=(a,0)$ and large-amplitude relaxation oscillations, which arise as periodic orbits in~\eqref{eq:FHN_kinetics} when $\eps>0$ is sufficiently small. The key feature in the geometric construction of the periodic orbits is the S-shaped critical slow manifold $\mathcal{M}_0 = \{(u,w) \in \R^2 : w=f(u)\}$. The cubic $w = f(u)$ attains a local minimum value at 
\begin{align} \label{defu1}
u_1 = \frac{1}{3} \left(1 + a - \sqrt{1 - a + a^2}\right) 
\end{align}
and a local maximum value at 
\begin{align} \label{defu1s}
\bar{u}_1 = \frac{1}{3} \left(1 + a + \sqrt{1 - a + a^2}\right),
\end{align}
splitting $\mathcal{M}_0$ into three normally hyperbolic branches $\mathcal{M}^{\mathrm{l}, \mathrm{m}, \mathrm{r}}_0$ and two fold points $(u_1,f(u_1))$, $(\bar{u}_1,f(\bar{u}_1))$; see Figure~\ref{fig:singular_slow}. The nullcline $u-\gamma w - a = 0$ intersects the cubic $w = f(u)$ in the point $(a,0)$ in the $(u,w)$-plane, and~\eqref{e:gammacond} ensures that there are no  intersections with the right and left branches  $\mathcal{M}^{\mathrm{l,r}}_0$ which would lead to stable equilibria; see again Figure~\ref{fig:singular_slow}. One can then construct relaxation oscillations by concatenating portions of the left and right branches $\mathcal{M}^{\mathrm{l},\mathrm{r}}_0$ of the critical manifold with fast orbits that originate at each fold point and jump to the opposing normally hyperbolic branch at points $(u_2,f(u_1))$ and $(\bar{u}_2,f(\bar{u}_1))$ with
\begin{align} \label{defu2}
    u_2 = \frac{1}{3} \left(1 + a + 2\sqrt{1 - a + a^2}\right), \qquad  \bar{u}_2 = \frac{1}{3} \left(1 + a - 2\sqrt{1 - a + a^2}\right),
\end{align}
see Figure~\ref{fig:singular_slow}. We hence refer to the parameter regime $0<a<\frac{1}{2}, 0<\gamma<\gamma_*(a)$ where~\eqref{eq:FHN_kinetics} exhibits this oscillatory behavior as the \emph{oscillatory regime}.

 \begin{figure}
\centering
\includegraphics[width=0.65\linewidth]{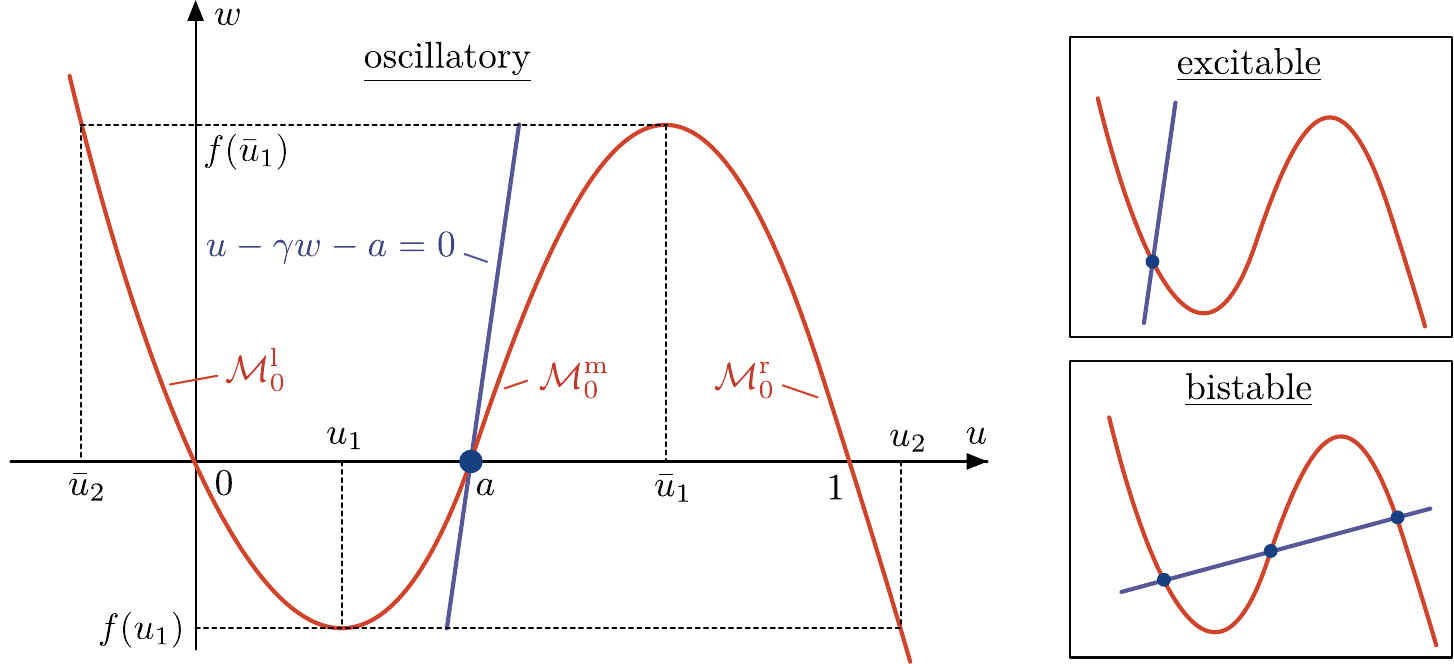}
\caption{The left schematic diagram depicts the left, middle, and right branches $\mathcal{M}^\mathrm{l,m,r}_0$ of the critical manifold $\mathcal{M}_0$, the locations of the fold points $(u,w)=(u_1,f(u_1))$ and $(u,w)=(\bar{u}_1,f(\bar{u}_1))$, and the nullcline $u-\gamma w-a=0$, under the conditions $0<a<1/2$ and $0<\gamma<\gamma_*(a)$. These conditions ensure that~\eqref{eq:FHN_kinetics} is in the oscillatory regime and exhibits relaxation oscillations. In particular, they exclude configurations in which equilibria lie on the outer branches $\mathcal{M}^\mathrm{l,r}_0$ of the critical manifold, such as the excitable and bistable regimes depicted in the top right and bottom right insets, respectively. }
\label{fig:singular_slow}
\end{figure}

\paragraph{Periodic wave trains.} 
In addition to the spatially homogeneous oscillations in~\eqref{eq:FHN_kinetics},  the PDE~\eqref{eq:FHN_pde} admits large-amplitude spatially periodic  wave trains, parameterized for instance by their wave speed. They arise as solutions of an associated traveling-wave equation as follows. Passing to a co-moving frame $(u,w)(\xi,t) = (u,w)(x-ct,t)$ with wave speed $c$, we rewrite~\eqref{eq:FHN_pde} as
\begin{align}\label{FHN}
\begin{split}
u_t &= u_{\xi\xi} + f(u) - w + cu_\xi,\\
w_t &=\epsilon(u-\gamma w - a) + cw_\xi,
\end{split}
\qquad \xi \in\R, \, t \geq 0, \, (u,w)\in\R^2,
\end{align}
where $\xi=x-ct$. Stationary solutions $(u,w)(x,t) = (u,w)(\xi)$ satisfy the traveling wave ODE
\begin{align}\label{eq:FHN_twode}
\begin{split}
0 &= u_{\xi\xi} + f(u) - w + cu_\xi,\\
0 &=\epsilon(u-\gamma w - a) + cw_\xi,
\end{split}
\end{align}
which, upon setting $v = u_\xi$, can be written as a singularly perturbed first-order system
\begin{align}\label{TW}
\begin{split}
u_\xi &= v,\\
v_\xi &= -cv - f(u) + w,\\
w_\xi &= -\frac{\epsilon}{c}(u-\gamma w - a).
\end{split}
\end{align}
Wave trains then correspond to periodic orbits in~\eqref{TW}. For each $0<a<\frac{1}{2}$, $0<\gamma<\gamma_*(a)$, and each sufficiently small $\eps>0$,~\eqref{TW} admits a family of wave trains parameterized by the speed $c$~\cite{CASCH,soto2001geometric}. This family naturally splits into two sub-families, namely the trigger waves for $c<c_*(a)$ and the phase waves for $c>c_*(a)$, where 
\begin{align*}
    c_*(a) \coloneqq  \sqrt{\frac{1-a+a^2}{2}} > 0.
\end{align*}
Trigger and phase waves are distinguished by the location of their fast transitions relative to the fold points at the local extrema of the cubic $w=f(u)$. For trigger waves, the fast jumps occur away from the fold points, whereas for phase waves they occur near the fold points; see Figure~\ref{fig:continuation}. In both cases, the period (or spatial wavelength) is an increasing function of the wave speed $c$. Moreover, the amplitude of phase waves remains nearly constant as $c$ increases, while the amplitude of trigger waves increases with $c$.

The existence of the family of trigger waves was established in~\cite{soto2001geometric}, while the family of phase waves was constructed more recently in~\cite{CASCH}. Technically, constructing  phase waves is significantly more subtle, due to the loss of normal hyperbolicity near the extrema of the cubic $w=f(u)$, requiring the use of geometric desingularization techniques to complete the construction. The following theorem summarizes existence results for both trigger and phase waves.
\begin{theorem}{\cite[Theorem 1.1]{CASCH}}\label{thm:existence}
Fix $0<a<\frac{1}{2}, 0<\gamma<\gamma_*(a)$ and $c>0$. Then, for all sufficiently small $\eps>0$, the system~\eqref{TW} admits a periodic orbit $\Gamma_\eps(c)$ with period $L_\eps(c)$. The function $L_\eps(c)$ is monotonically increasing in $c$, and satisfies $\lim_{\eps\to0}\eps L_\eps(c)=L_0(c)$ for a monotonically increasing function $L_0(c)$. For fixed $c<c_*(a)$ and $\eps>0$ sufficiently small, $\Gamma_\eps(c)$ is a trigger wave, while for fixed $c>c_*(a)$ and $\eps>0$ sufficiently small, $\Gamma_\eps(c)$ is a phase wave.
\end{theorem}

\begin{figure}
\centering
\includegraphics[width=0.315\linewidth]{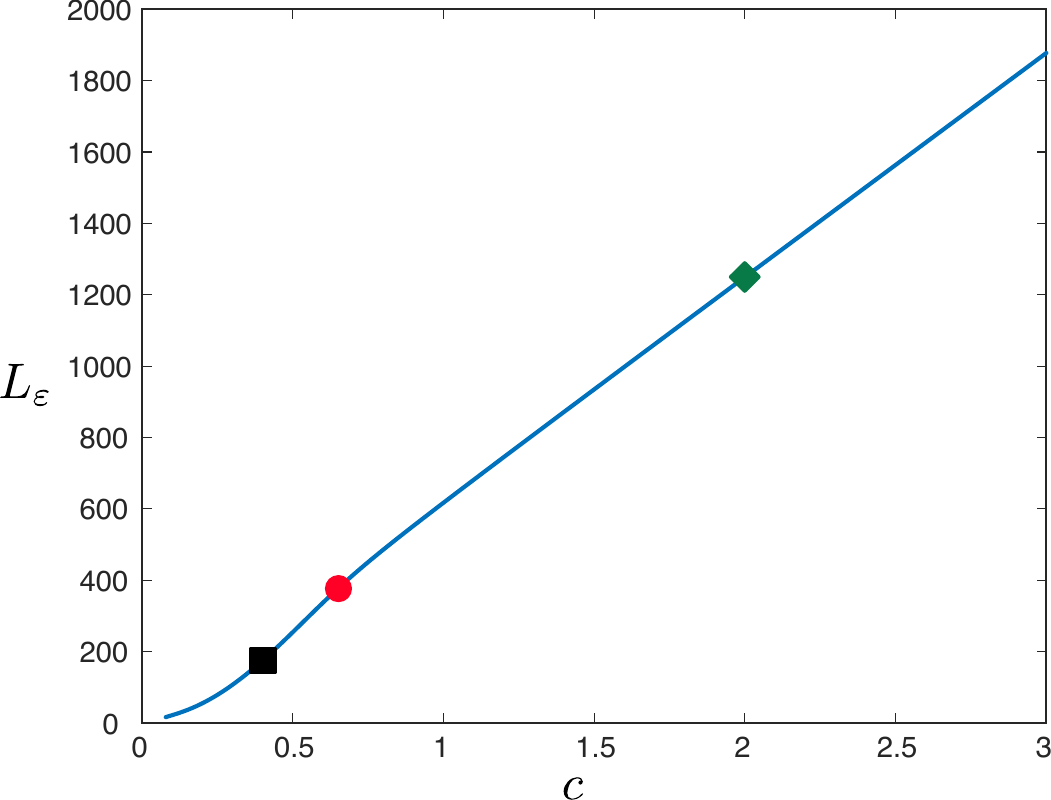}\hspace{0.035\textwidth}
\includegraphics[width=0.29\linewidth]{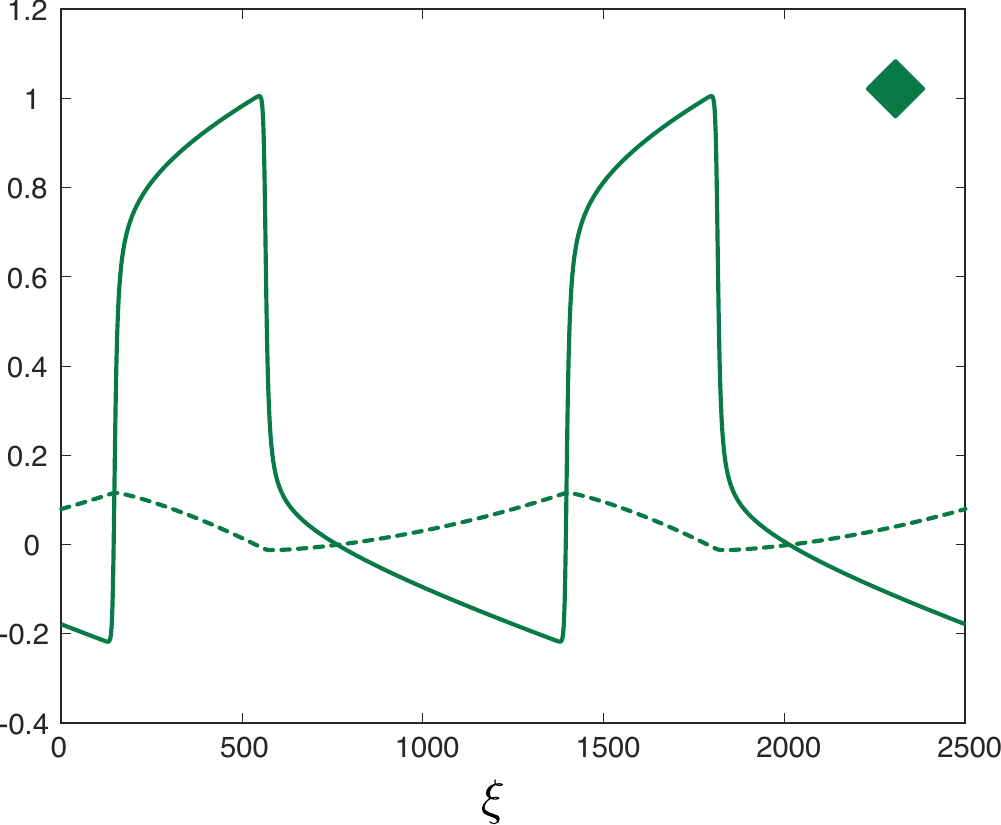}\hspace{0.025\textwidth}
\includegraphics[width=0.31\linewidth]{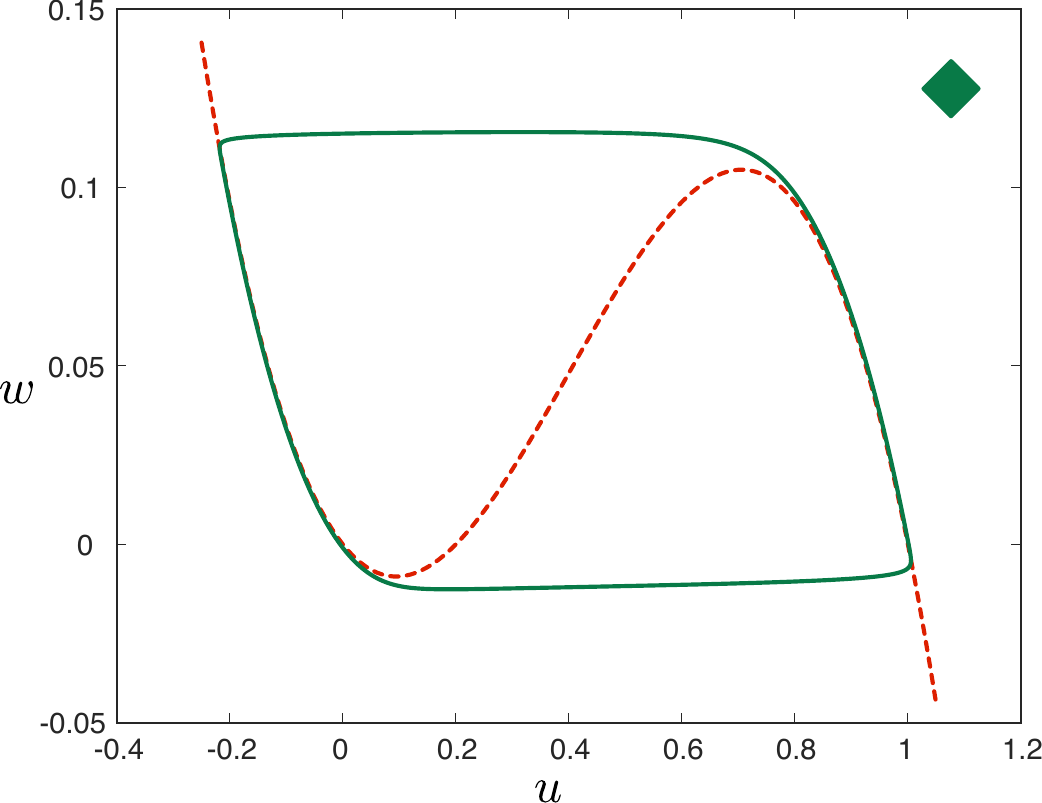}\\
\vspace{10pt}
\hspace{0.005\textwidth}
\includegraphics[width=0.315\linewidth]{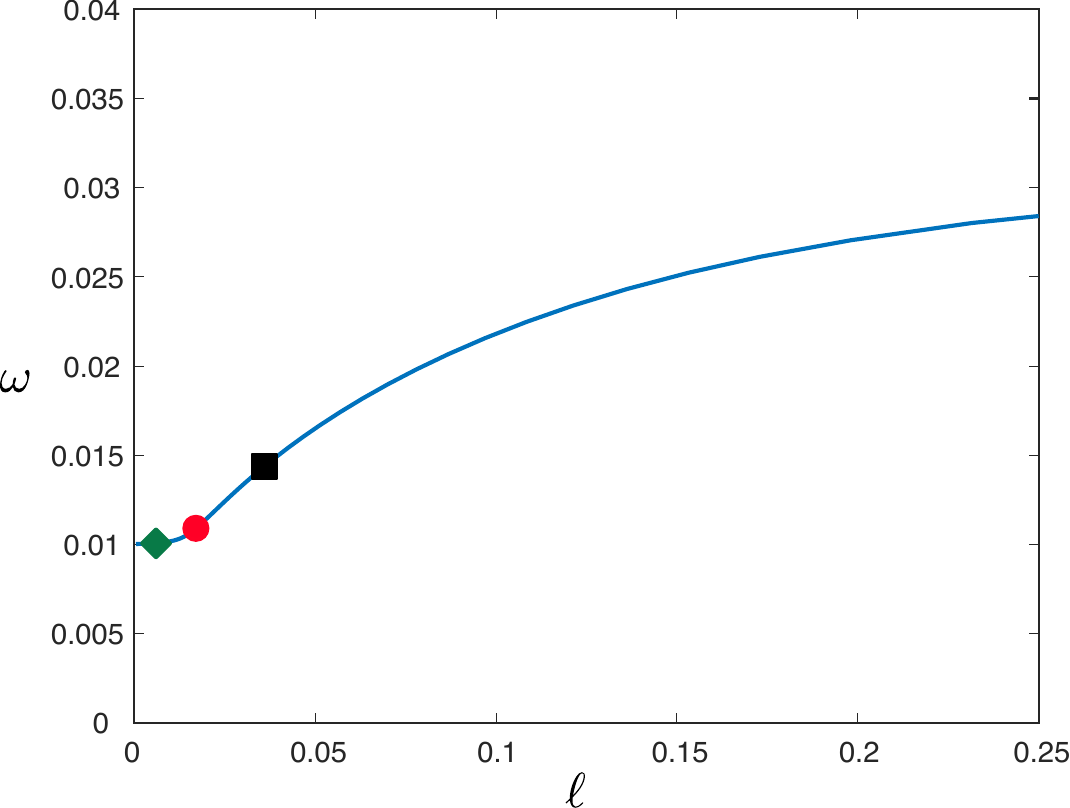}
\hspace{0.02\textwidth}
\includegraphics[width=0.29\linewidth]{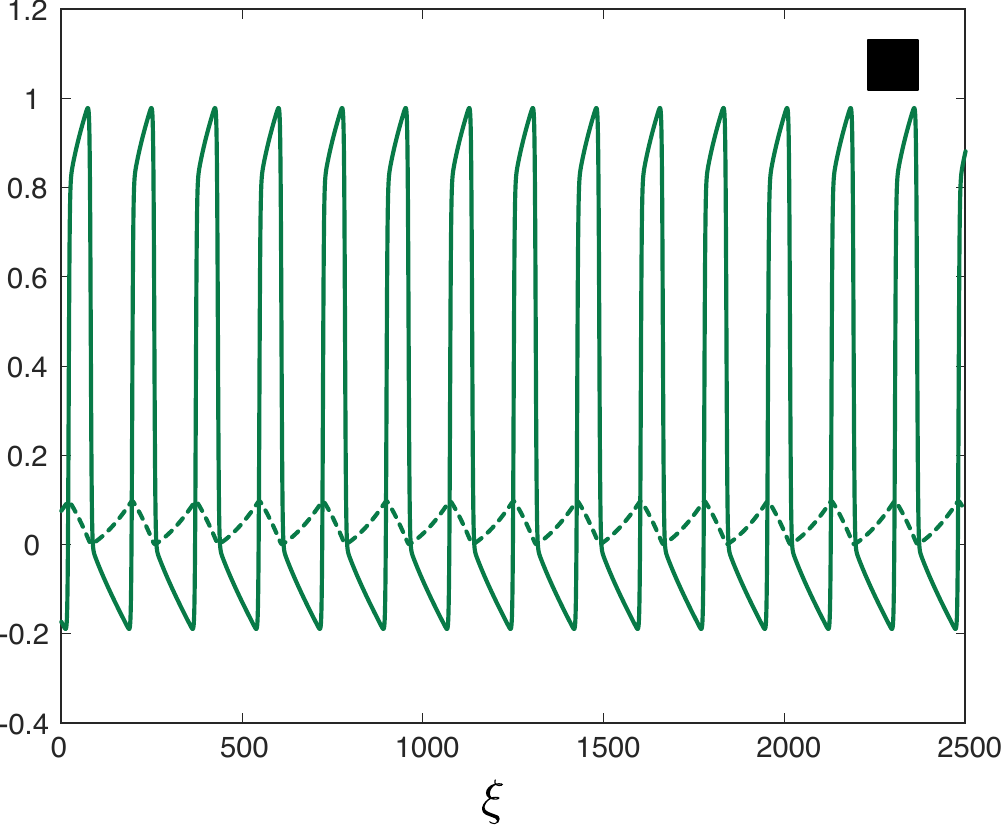}\hspace{0.025\textwidth} 
\includegraphics[width=0.31\linewidth]{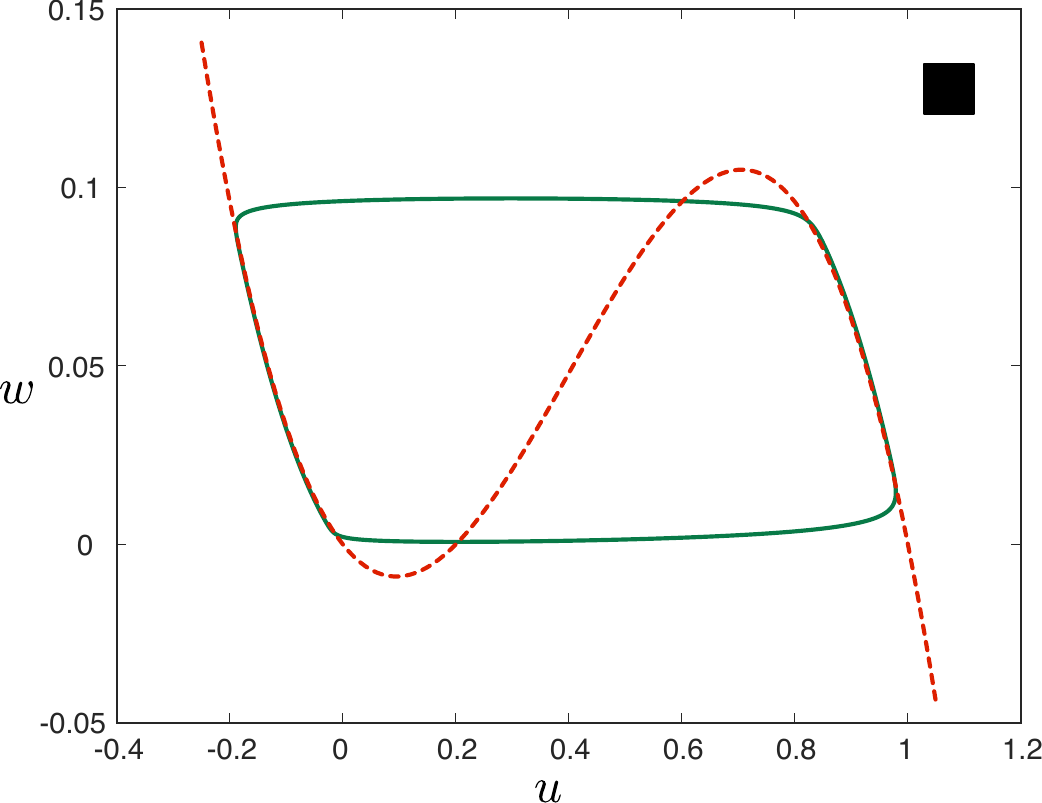}
\caption{Left column: (Top panel) Speed-period relation from numerical continuation of~\eqref{eq:FHN_twode} for fixed $a=0.2, \gamma=1, \eps=0.001$. (Bottom panel) Plot of the temporal frequency $\omega(\ell)\coloneqq \ell c(\ell)$, where $\ell\coloneqq \tfrac{2\pi}{L_\eps}$ is the spatial wavenumber. The value $c_*(a)$ denoting the transition from trigger to phase waves in the singular limit is marked by a red circle in both plots. Middle column: Profiles $u$ (solid) and $w$ (dashed) for a phase wave at $c=2$ (top panel) and a trigger wave at $c=0.4$ (bottom panel), marked by a green diamond and a black square, respectively, in the left panels. Note that the spatial period scales with the speed $c$, so that the slower trigger waves are much more narrowly spaced compared with the phase waves when plotted on the same spatial scale.  Right column: phase space plots of the phase- and trigger-wave trains (solid green) from middle panels along with the cubic nullcline  $w=f(u)$ (dashed red). } 
\label{fig:continuation}
\end{figure}

Figure~\ref{fig:continuation} depicts the results of numerical continuation of wave trains in the traveling-wave equation~\eqref{eq:FHN_twode} in the wave speed parameter $c$ and period $L_\eps$ for fixed $a=0.2, \gamma=1, \eps=0.001$. We see that the period $L_\eps$ increases monotonically in $c$, reflecting a negative group velocity $c_\mathrm{g}-c=-L\frac{dc}{dL}$ in the comoving frame propagating with speed $c$. In other words, the apparent phase velocity $\omega/\ell$ is always larger than the group velocity describing the speed of propagation of disturbances, $d\omega/d\ell$; see~\cite[Remark 1.6]{CASCH}.  Figure~\ref{fig:continuation} also depicts a phase-wave train profile obtained for $c=2>c_*(a) \approx 0.648$ and a trigger wave profile obtained for $c=0.4<c_*(a)$. We note the apparent change in concavity of $\omega(\ell)$ near the transition between trigger and phase waves (see Figure~\ref{fig:continuation}, bottom left panel). An examination of the relation between speed and period~\cite[\S4.4]{CASCH} suggests that a change in sign of $\omega''(\ell)$ occurs within the family of trigger waves at a speed somewhat close to -- but in fact $\mathcal{O}(1)$ in $\eps$ away from -- the trigger/phase wave transition.

\subsection{Main result and consequences: spectral stability of phase-wave trains}\label{s:mr}

Our main result establishes spectral properties of the linearization at phase-wave trains which translate into frequency synchronization near phase waves, time scales of synchronization, and the possibility of desynchronization. 

\paragraph{Stability and effective diffusivity near phase-wave trains --- main result.} A periodic orbit $\Gamma_\eps$ of Theorem~\ref{thm:existence} corresponds to a stationary, $L_\eps$-periodic wave-train solution $\phi_\eps(\xi) = (u_\eps(\xi),w_\eps(\xi))$ to~\eqref{FHN}. Linearizing~\eqref{FHN} about this solution, we obtain the $L_\eps$-periodic differential operator
\begin{align*}
\El_\eps \begin{pmatrix} u \\ w\end{pmatrix} = \begin{pmatrix} u_{\xi\xi} + f'(u_\epsilon)u - w + c u_\xi\\ \epsilon(u-\gamma w) + cw_\xi\end{pmatrix},
\end{align*}
acting on $L^2(\R,\C) \times L^2(\R,\C)$ with domain $H^2(\R,\C) \times H^1(\R,\C)$. The spectrum of $\El_\eps$ is characterized by
Floquet--Bloch theory~\cite{Gardner1993,ReedSimon}. We set $e_\rho(\xi) = \re^{-\ri \rho \xi}$ for $\rho \in \R$ and define the family of Bloch operators
\begin{align*}
    \El_{\rho,\eps} \begin{pmatrix} u \\ w\end{pmatrix} =  e_\rho^{-1} \El_\eps \left(e_\rho \begin{pmatrix} u \\ w\end{pmatrix}\right) = \begin{pmatrix} \big(\partial_\xi+\ri\rho\big)^2+ c\big(\partial_\xi+\ri\rho\big) + f'(u_\epsilon) & - 1 \\ \epsilon & c\big(\partial_\xi+\ri\rho\big) - \eps \gamma \end{pmatrix}\begin{pmatrix} u \\ w\end{pmatrix},
\end{align*}
acting on $L^2(\R / L_\eps \Z, \C) \times L^2 (\R / L_\eps \Z, \C)$ with domain $H^2(\R / L_\eps \Z, \C) \times H^1 (\R / L_\eps \Z, \C)$. Since $\El_{\rho,\eps}$ has compact resolvent by the Rellich--Kondrachov theorem, its spectrum consists of isolated eigenvalues of finite algebraic multiplicity only. Floquet--Bloch theory then asserts that the spectrum of $\El_\eps$ is given by the union
\begin{align*}
    \Sigma \left( \El_\eps \right) = \bigcup_{\rho \in \left[-\tfrac{\pi}{L_\eps},\tfrac{\pi}{L_\eps} \right)} \Sigma(\El_{\rho,\eps}).
\end{align*}
Consequently, $\lambda \in \C$ lies in the spectrum of $\El_\eps$ if and only if there exists $\rho \in \R$ such that the eigenvalue problem 
\begin{align}\label{eq:Floquet_eigenvalue_problem}
    \El_{\rho,\eps} \begin{pmatrix} u\\ w\end{pmatrix} = \lambda \begin{pmatrix} u\\ w\end{pmatrix}
\end{align} 
admits a nontrivial solution $(u,w)^\top \in H^2(\R / L_\eps \Z, \C) \times H^1 (\R / L_\eps \Z, \C)$. 

Our main result shows that the spectrum of $\El_\eps$ is confined to the left-half plane, except for the simple translational eigenvalue of $\El_{0,\eps}$ at the origin, and is uniformly bounded away from the imaginary axis outside a small neighborhood of the origin. Moreover, it establishes that the critical spectrum near the origin  is determined through an implicit transcendental equation -- referred to as the \emph{main formula} -- which relates $\lambda$ to the Floquet--Bloch frequency variable $\rho$. 
The leading-order coefficients of this equation are fully explicit in terms of $a$, $c$, and $\gamma$. In a neighborhood of $\lambda=0$, the solution to the main formula is given by a smooth curve $\lambda_\eps(\rho)$, called the \emph{critical spectral curve} or \emph{linear dispersion relation}, which touches the origin in a quadratic tangency at $\rho = 0$. Following~\cite[\S4]{DSSS}, we find that the first derivative $\ri \lambda_\eps'(0) = c_g - c$ yields the group velocity in the co-moving frame with speed $c$, while the second derivative $\lambda_\eps''(0) = -d_{\mathrm{eff}}$ provides the effective diffusivity.

\begin{theorem}[Main result]\label{thm:spectral_stability} Let $0<a<\frac{1}{2}$, $0<\gamma<\gamma_*(a)$, and $c>c_*(a)$. Fix $\delta > 0$ arbitrarily small. Then, there exist constants $C,\mu>0$ such that, for all $\eps>0$ sufficiently small, the linearization $\El_\eps$ of~\eqref{FHN} about the $L_\eps$-periodic wave train $\phi_\eps(\xi) = (u_\eps(\xi),w_\eps(\xi))$, established in Theorem~\ref{thm:existence}, satisfies the following properties:
\begin{enumerate}
    \item  \label{thm:spectral_stability_i} \emph{Spectral stability:} We have $\Sigma(\El_\eps) \subset \{ \lambda \in \C : \Re(\lambda) < 0 \} \cup \{ 0 \}$. Furthermore, there exists $\eta(\eps)>0$ such that any $\lambda\in \Sigma(\El_\eps) $ with $|\lambda|\geq  \mu\eps^{1/6}$ satisfies $\Re(\lambda)\leq -\eta(\eps)$.
    \item \label{thm:spectral_stability_ii} \emph{Main formula:} A point $\lambda \in \C$ with $|\Im(\lambda)| \leq \mu$ and $|\Re(\lambda)| \leq \mu \eps^{1/6}$ lies in the spectrum $\Sigma(\El_\eps)$ if and only if it obeys the formula
    \begin{align} \label{eq:mainformapprox}
    \re^{\left(\frac{\lambda}{c} -\ri\rho\right)L_\eps} = \left(1 +  \Upsilon_\mathrm{lf}\left(\frac{\lambda}{\eps^{\frac16}}\right) \frac{u_1-u_2  + \mathcal{E}_{\mathrm{lf},\eps}(\lambda)}{u_2 - \gamma f(u_1) - a}\right)\!\left(1 +  \Upsilon_\mathrm{uf}\left(\frac{\lambda}{\eps^{\frac16}}\right) \frac{\bar{u}_1-\bar{u}_2  + \mathcal{E}_{\mathrm{uf},\eps}(\lambda)}{\bar{u}_2 - \gamma f(\bar{u}_1) - a}\right) + \mathcal{E}_{\eps}(\lambda) 
    \end{align}
    for some $\rho \in \R$, where the residual terms satisfy
    \begin{align*}
    \left|\mathcal{E}_{\mathrm{lf},\eps}(\lambda)\right|, \left|\mathcal{E}_{\mathrm{uf},\eps}(\lambda)\right| \leq \delta, \qquad \left|\mathcal{E}_{\eps}(\lambda)\right| \leq C \left(\eps^{\frac13} + \left|\lambda \log|\lambda|\right|\right)
    \end{align*}
    and where $\Upsilon_{\mathrm{lf}},\Upsilon_{\mathrm{uf}} \colon \C \to \C$ are explicit entire functions, defined by~\eqref{eq:deflfuf} below. In this case, $\lambda$ is an eigenvalue of the Bloch operator $\El_{\rho,\eps}$. In particular, $0$ is an algebraically simple eigenvalue of $\El_{0,\eps}$.
    \item \label{thm:spectral_stability_iii} \emph{Critical spectral curve:} Locally near $(0,0)$, the set of $(\lambda,\rho) \in \C \times \R$ solving~\eqref{eq:mainformapprox} is given by a smooth curve $\lambda_\eps \colon I_\eps \to \C$ with $\lambda_\eps(0) = 0$, where $I_\eps \subset \R$ is a neighborhood of $0$.  The group velocity $c_{\mathrm{g}} = \ri\lambda_\eps'(0) + c$ and the effective diffusivity $d_\mathrm{eff} = -\lambda_\eps''(0) > 0$ obey the estimates    \begin{equation*}
    \left|c_{\mathrm{g}}\right| \leq \delta, \qquad \left|d_{\mathrm{eff}} - \frac{2\kappa c^3}{L_0} \, \eps^{\frac23}\right| \leq \delta \eps^{\frac23},
    \end{equation*}
    where the coefficient $\kappa >0$ is given explicitly by~\eqref{eq:quad_coeff}, and
    $L_0 =\lim_{\eps\to 0}\eps L_\eps=L_{\rr}+L_{\lr}$ is given by~\eqref{periodexpr1} and~\eqref{periodexpr2}, below.
\end{enumerate}
\end{theorem}

The full setup and strategy for our analysis of the eigenvalue problem~\eqref{eq:Floquet_eigenvalue_problem}, leading to the proof of Theorem~\ref{thm:spectral_stability}, are presented in~\S\ref{S:setup}. Figure~\ref{fig:lambda_coeff_cont} illustrates the result by displaying the typical shape of the critical spectral curve $\lambda_\eps(\rho)$ for fixed small $\eps > 0$, together with the dependence of $\lambda_\eps''(0)$ on $\eps$; see Appendix~\ref{s:num} for details on computations. We now outline several implications of Theorem~\ref{thm:spectral_stability}. 

\begin{figure}
\centering
\includegraphics[width=0.3\linewidth]{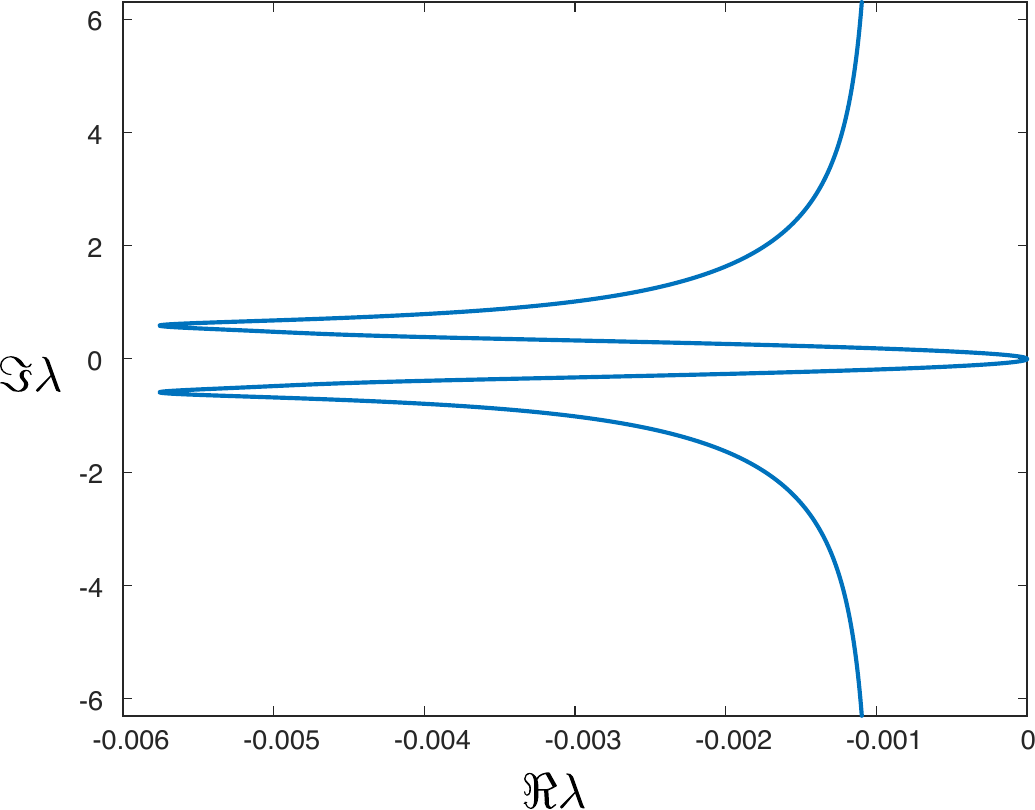}
\includegraphics[width=0.32\linewidth]{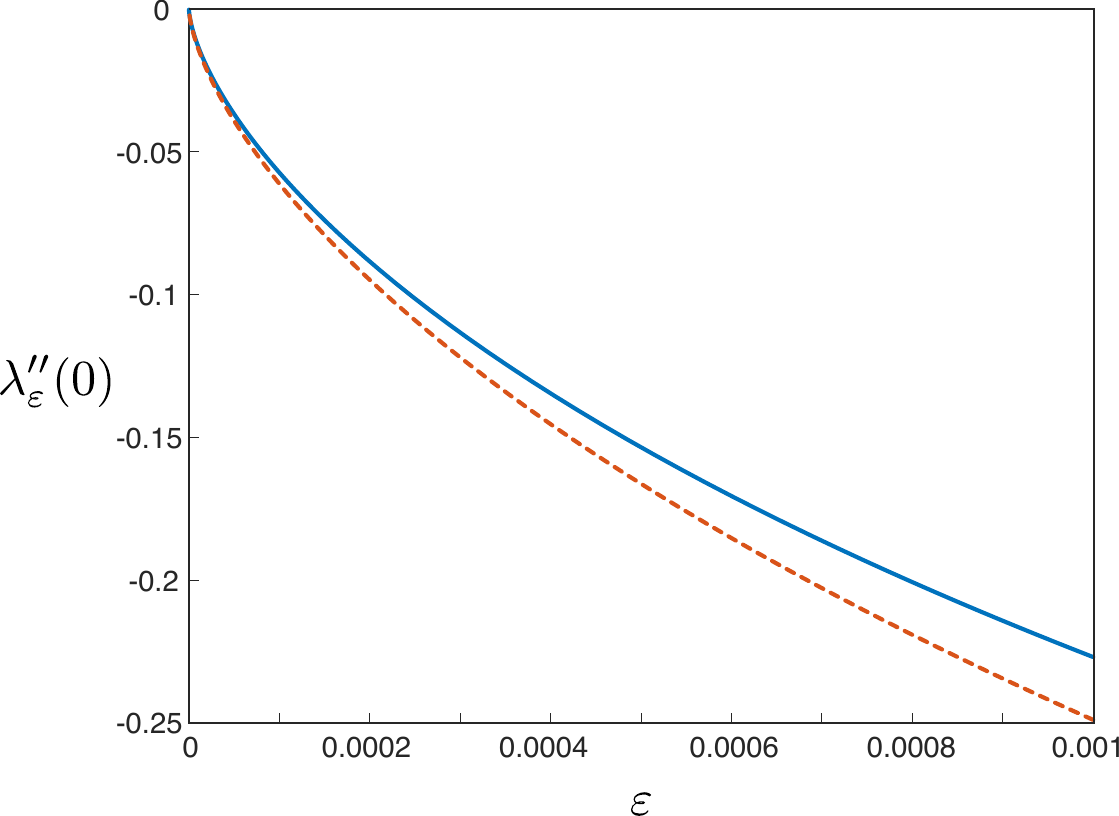}
\includegraphics[width=0.31\linewidth]{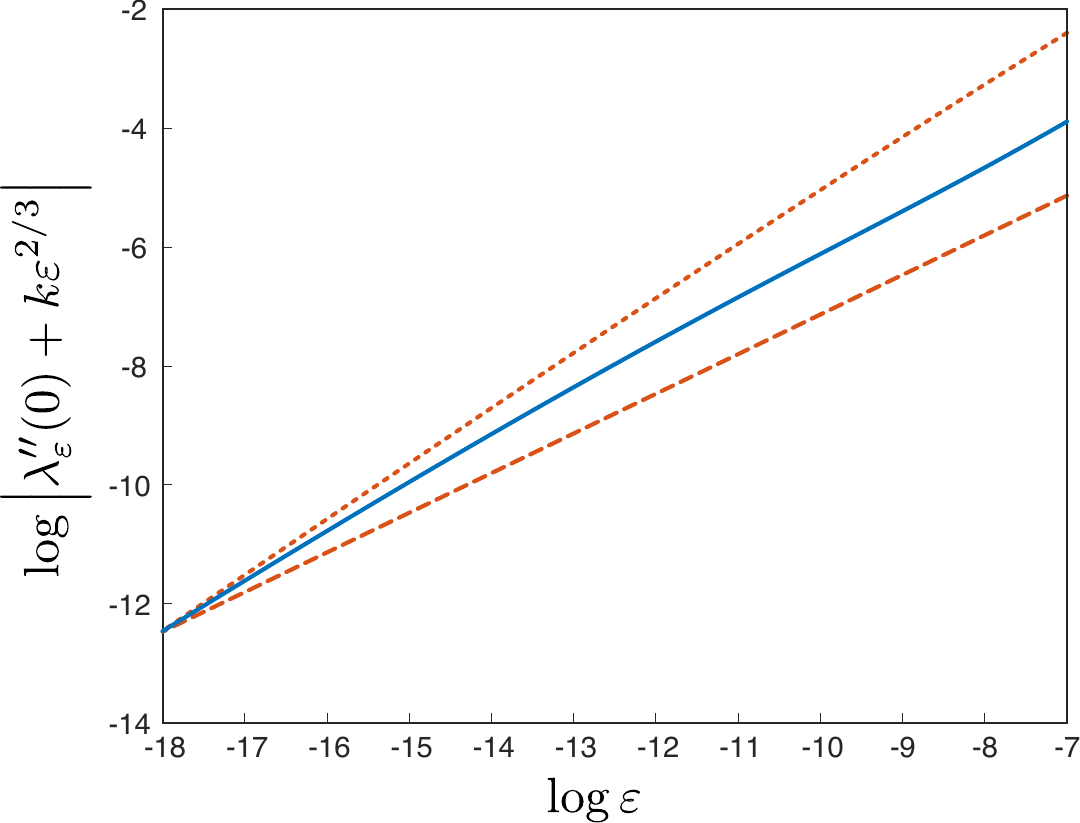}
\caption{Left: Critical spectral curve $\lambda_\eps(\rho)$ of Theorem~\ref{thm:spectral_stability} for fixed $a=0.2, \gamma=1, \eps=0.001, c=2$ associated with the phase-wave train depicted in Figure~\ref{fig:continuation}. Plotted is $\Im(\lambda_\eps(\rho))$ versus $\Re(\lambda_\eps(\rho))$. Center and right: Numerical continuation of the coefficient $\lambda''_\eps(0) = -d_{\mathrm{eff}}$ versus $\eps$ for $a=0.2, \gamma=1, c=2$: The center plot depicts the numerically computed expression for $\lambda_\eps''(0)$ for values of $\eps \in (10^{-8},10^{-3})$ (blue) alongside the leading-order analytical expression $-k\eps^{2/3}$ (dashed red), where $k = 2\kappa c^3/L_0$. The right plot depicts a log-log plot of the difference between these two expressions (blue). Also shown is a line of slope $2/3$ (dashed red), as well as a curve of the form $\log \eps+\log|\log\eps|$ (dotted red), which suggests that the error between the two expressions is of $\mathcal{O}(\eps|\log \eps|)$. We refer to Appendix~\ref{s:num} for details on how these computations were performed. }
\label{fig:lambda_coeff_cont}
\label{fig:critical_curve}

\end{figure}

\paragraph{Comparison: trigger waves.} The results for trigger waves in~\cite{eszter1999evans} are stated somewhat differently but are in many ways equivalent to Theorem~\ref{thm:spectral_stability}. As mentioned previously, the analysis in the case of trigger waves is significantly simpler since these waves avoid the fold points on the critical manifold. Calculations in~\cite{eszter1999evans} are further simplified by setting $\gamma=0$, though we expect the argument presented there to hold for nonzero $\gamma$ after  appropriate modifications. Translated into the formulation~\eqref{eq:FHN_pde} of the FitzHugh--Nagumo system, in~\cite{eszter1999evans} it is shown for fixed $0<a<\tfrac{1}{2}$, $0<\gamma<\gamma_*(a)$, $\delta > 0$, and $0<c<c_*(a)$ (rather than $c>c_*(a)$ for phase waves) that, provided $\eps > 0$ is sufficiently small, the operator $\mathcal{L}_{\eps}$ satisfies
\begin{enumerate}
    \item \emph{Spectral stability:} We have $\Sigma(\El_\eps) \subset \{ \lambda \in \C : \Re(\lambda) < 0 \} \cup \{ 0 \}$; 
    \item \emph{Critical spectrum:} There exist an open interval $I_\eps \subset \R$ containing $0$ and a smooth curve $\lambda_{\eps} \colon I_\eps \to \C$ such that $\lambda_{\eps}(\rho)$ is an eigenvalue of $\El_{\rho,\eps}$ for all $\rho \in I_\eps$. In particular, $\lambda_\eps(0) = 0$ is an algebraically simple eigenvalue of $\El_{0,\eps}$. Moreover, the group velocity $c_{\mathrm{g}} = \ri\lambda_\eps'(0) + c$ and the effective diffusivity $d_\mathrm{eff} = -\lambda_\eps''(0) > 0$ obey the estimates    \begin{equation*}
    \left|c_{\mathrm{g}} - c\right| \leq \delta, \qquad \left|d_{\mathrm{eff}} - k \,\eps^{-1}\right| \leq \delta \eps^{-1},
    \end{equation*}
    where the coefficient $k>0$ depends only on $a,\gamma$, and $c$.
\end{enumerate}
In particular, effective diffusivites $\eps^{-1}$ are much larger compared to the $\eps^{2/3}$-diffusivities for phase waves in Theorem~\ref{thm:spectral_stability}. In addition, group velocities in the steady frame agree with phase velocities at leading order.

\paragraph{Comparison: spatially homogeneous relaxation oscillations.} The limiting case $c\to\infty$, which is not covered by Theorem~\ref{thm:spectral_stability}, corresponds to waves near spatially homogeneous oscillations; see Figure~\ref{fig:oscillations_spacetime}. Frequency-synchronization properties similar to those described in Theorem~\ref{thm:spectral_stability} and examples of instability similar to our example in~\S\ref{s:instability} below, can be analyzed using an approach based on the techniques developed here. We refer to~\cite{homoscACRS} for details of the adapted analysis and associated phenomena, as well as to~\cite{AEKV} for related observations in the case of just two coupled oscillators. 

\paragraph{Consequences: nonlinear stability of phase-wave trains.} The diffusive spectral stability established in Theorem~\ref{thm:spectral_stability} is sufficient to guarantee nonlinear stability of wave trains for reaction-diffusion systems which are strictly parabolic~\cite{schn96,GS1,JZ,JNRZ,SSSU}. Systems such as~\eqref{FHN} with degenerate diffusion introduce additional challenges as one must control high-frequency modes to obtain a spectral mapping estimate, and nonlinear stability of wave trains in such systems is thus not immediately guaranteed by existing results for general reaction-diffusion systems. Nevertheless, we show in the companion paper~\cite[\S2.6]{FHNpulled}, that the spectral stability result, Theorem~\ref{thm:spectral_stability}, is sufficient to obtain nonlinear stability of the phase-wave trains of Theorem~\ref{thm:existence} against spatially localized perturbations, with asymptotically diffusive decay on the time scale $\smash{d_\mathrm{eff}^{-1/2}}$. This result was later extended to fully nonlocalized perturbations in~\cite{alexopoulos2024nonlinear}. We note however that the estimates in~\cite{FHNpulled,alexopoulos2024nonlinear} apply for fixed $\eps > 0$ only, that is, constants are not uniform in $\eps$. 

\paragraph{Consequences: frequency synchronization through fronts.}  In the PDE~\eqref{eq:FHN_pde}, in the oscillatory regime $0 < a < \frac12$, $0 < \gamma < \gamma_*(a)$, perturbations of the unstable homogeneous rest state $(u,w)=(a,0)$ can lead to large-amplitude spatial patterns which spread into this unstable state; see Figure~\ref{fig:invasion_front}. The invading spatial patterns are selected from the family of wave trains of Theorem~\ref{thm:existence}, parameterized by the  wave speed $c>0$. The selected speed is determined by an invasion front in the leading edge of the spreading process. The marginal stability conjecture~\cite{deelanger,vanSaarloos,AveryScheelSelection},  which has not yet been rigorously verified for any pattern-forming front, states that the selected front should be marginally spectrally stable in an optimal weighted space. The front is categorized as pushed or pulled depending on whether the marginally stable spectrum lies in the point spectrum or essential spectrum. In~\cite{CASCH}, both pushed and pulled pattern-forming invasion fronts were constructed in~\eqref{eq:FHN_pde} as traveling waves which arise as connections between the unstable rest state $(u,w)=(a,0)$ and a periodic wave train of Theorem~\ref{thm:existence} in the wake. In the pulled case, the corresponding wave train can be a trigger wave or a phase wave, depending on the choice of the parameter $a$, while in the pushed case, all selected wave trains are phase waves. Diffusive spectral stability of the associated wave train, as guaranteed by Theorem~\ref{thm:spectral_stability}, is an essential ingredient in the nonlinear stability argument for both pulled and pushed pattern-forming fronts, as detailed in~\cite{FHNPushed,FHNpulled}, and hence represents an important step towards resolving the marginal stability conjecture for pattern-forming fronts in~\eqref{eq:FHN_pde}. 
\begin{figure}
\centering
\includegraphics[width=0.31\linewidth]{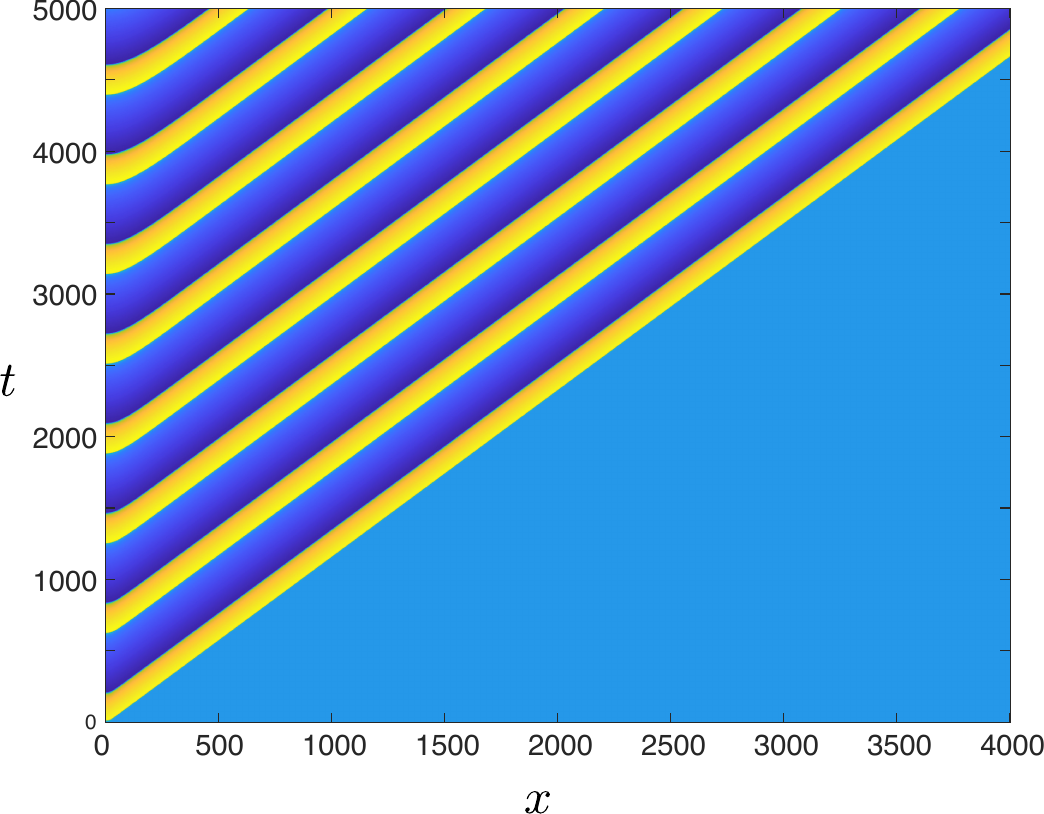}\hspace{0.05\linewidth}\includegraphics[width=0.3\linewidth]{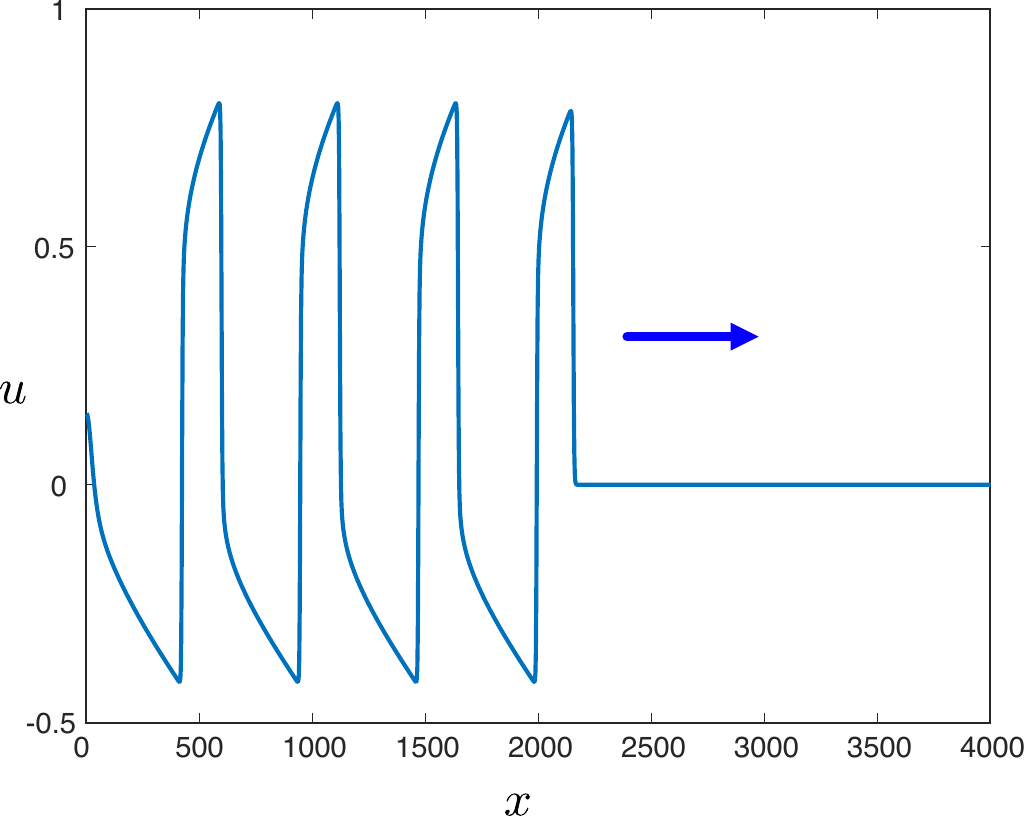}
\caption{(Left) Spacetime plot of the $u$-profile of a pushed pattern-forming invasion front in~\eqref{eq:FHN_pde} which selects a phase-wave train, obtained by direct numerical simulation for the parameter values $a=0.2, \gamma=1, \eps = 0.001$. (Right) Snapshot of the $u$-profile at time $t=2500$. }
\label{fig:invasion_front}
\end{figure}

\paragraph{Technical approach.}

In the context of singularly perturbed PDEs, a common approach to analyzing spectra utilizes the Evans function~\cite{evans1972nerve}, which can admit a decomposition, or factorization, into slow and fast components based on the geometric singular perturbation structure of the underlying solution~\cite{Alexander_1990,BDR}. This approach has been used to prove the (in)stability of traveling pulses and fronts in singularly perturbed reaction-diffusion systems such as the FitzHugh--Nagumo, Gray--Scott, and Gierer--Meinhardt systems, among others~\cite{doelman1998stability,jones1984stability, veerman2013pulses}. In those cases, one is typically most concerned with the location of one or more isolated eigenvalues in the point spectrum, which can be studied using, for instance, winding number arguments.
In the case of wave trains, Evans-function techniques were used to locate spectral curves parameterized by the Floquet--Bloch frequency variable  in~\cite{Gardner1993} and then later in the singularly perturbed setting in~\cite{eszter1999evans,vdPloeg_2005,BDR,BDR2} for the FitzHugh--Nagumo and Gierer--Meinhardt systems, as well as for more general reaction-diffusion models. As mentioned previously, these latter works focus on trigger waves or periodic pulse/spike patterns which avoid challenges associated with loss of hyperbolicity at fold points. 

The case of phase waves involves unfolding the critical Floquet--Bloch spectral curves in the presence of a loss of hyperbolicity, a combination that poses unique challenges. Our contribution here is, to our knowledge, the first analysis which treats this case. To prove Theorem~\ref{thm:spectral_stability}, we adopt an approach based on Lin’s method~\cite{lin1990}, in which we directly construct piecewise continuous potential eigenfunctions. A matching procedure using  Melnikov theory then leads to a reduced algebraic equation relating the spectral parameter $\lambda$ to the Floquet--Bloch frequency $\rho$. This strategy has proven effective in the spectral stability analysis of wave trains in slow-fast reaction-diffusion systems~\cite{BDR2}, where exponential trichotomies are used to separate slow from fast behavior and to transfer Fredholm properties from reduced eigenvalue problems to the full system. In the present setting, however, new challenges arise due to the passage through nonhyperbolic fold points, which rules out uniform exponential trichotomies. To overcome these difficulties, we employ geometric desingularization, or blow-up, techniques~\cite{krupaszmolyan2001}. Inspired by previous work concerning stability of traveling-pulse solutions in the FitzHugh--Nagumo system~\cite{PBR}, we develop a novel application of these techniques by blowing up the eigenvalue problem alongside the existence problem, which involves including $\lambda$ in the blow-up transformation and analyzing the resulting eigenvalue problem in several scaling regimes. By interweaving this construction with Lin's method, while employing the Riccati transformation~\cite{BDR} to separate slow from fast dynamics near the wave train along hyperbolic portions of the critical manifolds, we are able to precisely characterize the critical spectrum in Theorem~\ref{thm:spectral_stability}. In particular, in contrast to prior analyses of traveling pulses~\cite{PBR, holzer2013existence}, we must keep track of the dynamics in the center space along the wave train and solve a series of boundary value problems to derive the reduced algebraic equation~\eqref{eq:mainformapprox} which fully describes the nature of the critical spectrum near the origin. This construction, and its interplay with the characteristic scalings associated with slow passage through fold bifurcations, leads to the somewhat unexpected $\eps^{1/6}$-scaling for the critical spectral region in Theorem~\ref{thm:spectral_stability}\ref{thm:spectral_stability_ii} and the $\eps^{2/3}$-scaling for the effective diffusivity in Theorem~\ref{thm:spectral_stability}\ref{thm:spectral_stability_iii}.  

We expect these tools to be broadly applicable in singularly perturbed eigenvalue problems, in particular those arising in the stability analysis of traveling waves in systems exhibiting loss of hyperbolicity through folds and other nonhyperbolic singularities. Such singularities are common in the study of traveling waves in reaction-diffusion systems, where folded singularities are associated with relaxation oscillations as well as complex spatio-temporal oscillatory phenomena such as canards and mixed mode oscillations~\cite{avitabile2020local,carter2022wiggly,kaper2021new, guckenheimer2010homoclinic, zhu2024existence}, and diffusion-induced instabilities~\cite{buchholtz1995diffusion}. Moreover, recent studies~\cite{jencks2025stable,vo2025canards} highlight the role of folded singularities in organizing spatially periodic waves emerging at a singular Turing bifurcation. To our knowledge, the spectral and nonlinear stability of these classes of solutions has not been explored analytically, and we believe the approach developed here presents a framework to address these and related problems.

\subsection{Rigidity at large scales and instability at intermediate scales}\label{s:instability}

Our spectral analysis identifies the regime $\lambda \sim \eps^{1/6}$ as a potential source for instabilities. Theorem~\ref{thm:spectral_stability}\ref{thm:spectral_stability_iii} shows that the critical spectrum touching $0$ is diffusive, which follows from the leading-order description of the critical spectral curve $\lambda_\eps(\rho)$ as a solution to the main formula~\eqref{eq:mainformapprox}. Based on the analysis of this curve near $\rho = 0$ in~\S\ref{sec:mainformula}, we expect that $d_{\mathrm{eff}} = -\lambda_\eps''(0) >0$ holds much more generally for reaction-diffusion systems with relaxation oscillations. The negative sign $\lambda_\eps''(0) < 0$, together with the $\eps^{1/6}$-scaling  in~\eqref{eq:mainformapprox} stemming from the slow passage through the fold points, indicates that no instability occurs for $|\lambda| \ll \eps^{1/6}$. On the other hand, the analysis of the main formula~\eqref{eq:mainformapprox} for $|\lambda| \gg \eps^{1/6}$, in combination with standard Sturm--Liouville arguments, precludes unstable spectrum for $|\lambda| \gg \eps^{1/6}$; see Proposition~\ref{prop:region_lambda3},~\S\ref{sec:R2}, and Appendix~\ref{app:r3}. However, we find that for $\lambda/\eps^{1/6}$ of intermediate size, the left- and right-hand sides of~\eqref{eq:mainformapprox} can be balanced to produce potential instabilities. 

To investigate this possibility, we note that the entire functions $\Upsilon_\mathrm{lf}, \Upsilon_\mathrm{uf} \colon \C \to \C$ in~\eqref{eq:mainformapprox} are given by
\begin{align} \label{eq:deflfuf}
\begin{split}
     \Upsilon_\mathrm{lf}(z)\coloneqq \frac{z^2}{\theta_\mathrm{lf}c^3\mathrm{Ai}'(-\Omega_0)^2 } \int_{-\Omega_0}^{\infty} \re^{\frac{z^2}{\theta_\mathrm{lf}c^3 }\left(s+\Omega_0\right) }\left(s\mathrm{Ai}(s)^2-\mathrm{Ai}'(s)^2\right)\mathrm{d}s,\\
     \Upsilon_\mathrm{uf}(z)\coloneqq \frac{z^2}{\theta_\mathrm{uf}c^3\mathrm{Ai}'(-\Omega_0)^2 } \int_{-\Omega_0}^{\infty} \re^{\frac{z^2}{\theta_\mathrm{uf}c^3 }\left(s+\Omega_0\right) }\left(s\mathrm{Ai}(s)^2-\mathrm{Ai}'(s)^2\right)\mathrm{d}s,
\end{split}
\end{align}
where $-\Omega_0 < 0$ denotes the largest zero of the Airy function $\mathrm{Ai}(z)$ (see Appendix~\ref{app:airy}), and
\begin{align*} 
\begin{split}
    \theta_\mathrm{lf} &\coloneqq  -\frac{(a^2-a+1)^{1/6}(u_1-\gamma f(u_1)-a)^{1/3}}{c}>0,\\
    \theta_\mathrm{uf} &\coloneqq  \frac{(a^2-a+1)^{1/6}(\bar{u}_1-\gamma f(\bar{u}_1)-a)^{1/3}}{c}>0.
    \end{split}
\end{align*}
In the oscillatory parameter regime $0 < a < \frac12$, $0 < \gamma < \gamma_*(a)$ in~\eqref{eq:FHN_pde}, we have $\theta_{\mathrm{uf}} > \theta_{\mathrm{lf}}$, which rules out the existence of unstable spectrum in the intermediate regime $|\lambda| \sim \eps^{1/6}$; see the proof of Proposition~\ref{prop:region_lambda2}.

These observations then motivate considering the following FitzHugh--Nagumo-type system with modified nonlinearities
\begin{align}\label{eq:FHN_mod_pde}
\begin{split}
u_t &= u_{\xi\xi} + F(u,w),\\
w_t &=\epsilon(u-\gamma w - a),
\end{split} \qquad\qquad     F(u,w)=\frac{u(u-a)\left(\tfrac{3}{2}-u\right)}{\tfrac{2}{5}+u-a}-w\left(\tfrac{5}{4}-r(u-a)\right),
\end{align}
and parameters $a=0.25, \gamma=0.01, \eps=0.002$, and $r=0.998$; see Figure~\ref{fig:toy_wavetrain} for nullclines. Analogous to the FitzHugh--Nagumo system,~\eqref{eq:FHN_mod_pde} admits a family of wave-train solutions passing near nonhyperbolic fold points, resembling relaxation oscillations. The nonlinearities are chosen in such a way that the corresponding quantities $\theta_\mathrm{uf},\theta_\mathrm{lf}$ introduced above  
satisfy $\theta_\mathrm{uf} < \theta_\mathrm{lf}$, allowing for the possibility of an instability of the wave-train solution to~\eqref{eq:FHN_mod_pde}. Examining~\eqref{eq:mainformapprox}, we expect the regime $|\rho|\sim |\lambda| \sim \eps^{1/6}$ to be relevant for instabilities. Figure~\ref{fig:toy_wavetrain} depicts a numerically computed wave-train solution of~\eqref{eq:FHN_mod_pde} as well as the associated unstable critical spectral curve, which crosses into the right half plane for a range of intermediate values of the Floquet--Bloch frequency parameter $\rho$.

\begin{figure}
\hspace{0.05\textwidth}
\includegraphics[width=0.35\textwidth]{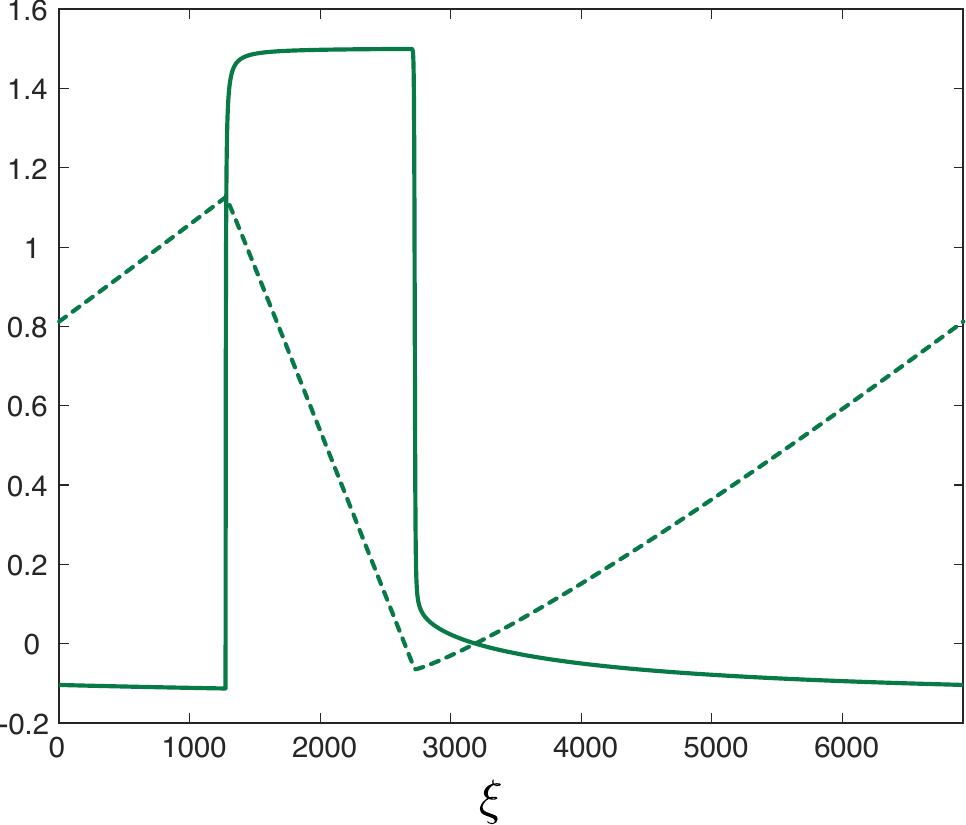}\hspace{0.09\textwidth}
\includegraphics[width=0.37\textwidth]{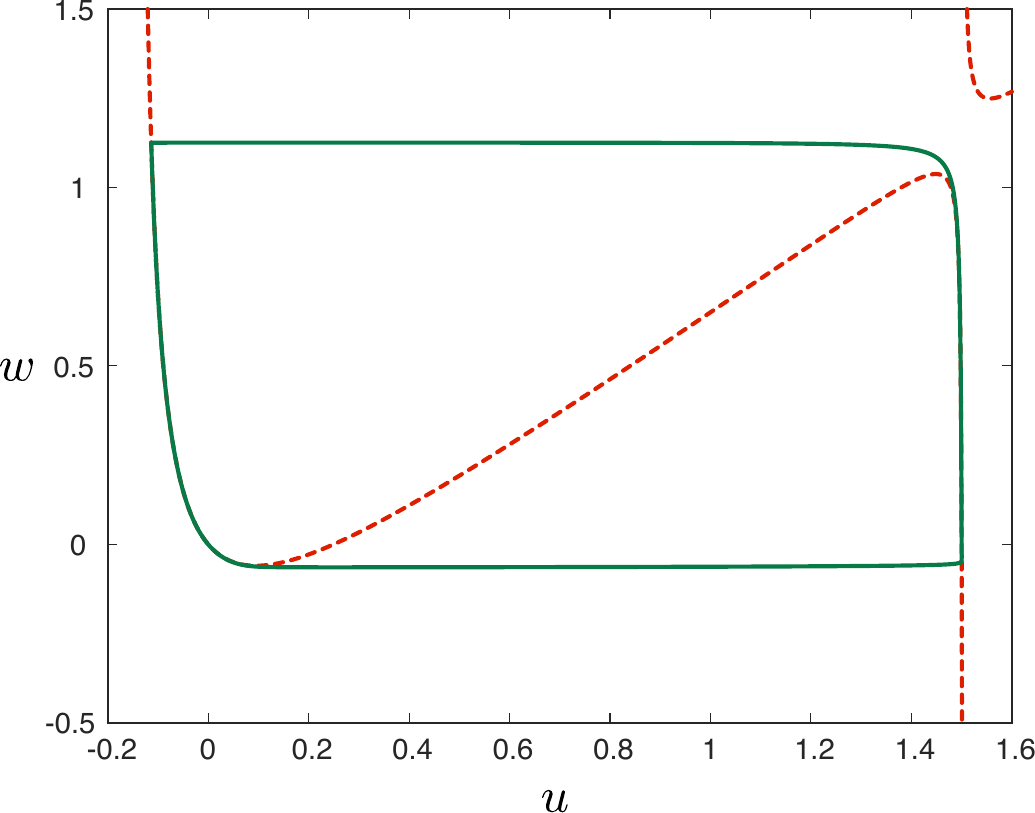}\\

\hspace{0.025\textwidth}
\includegraphics[width=0.39\textwidth]{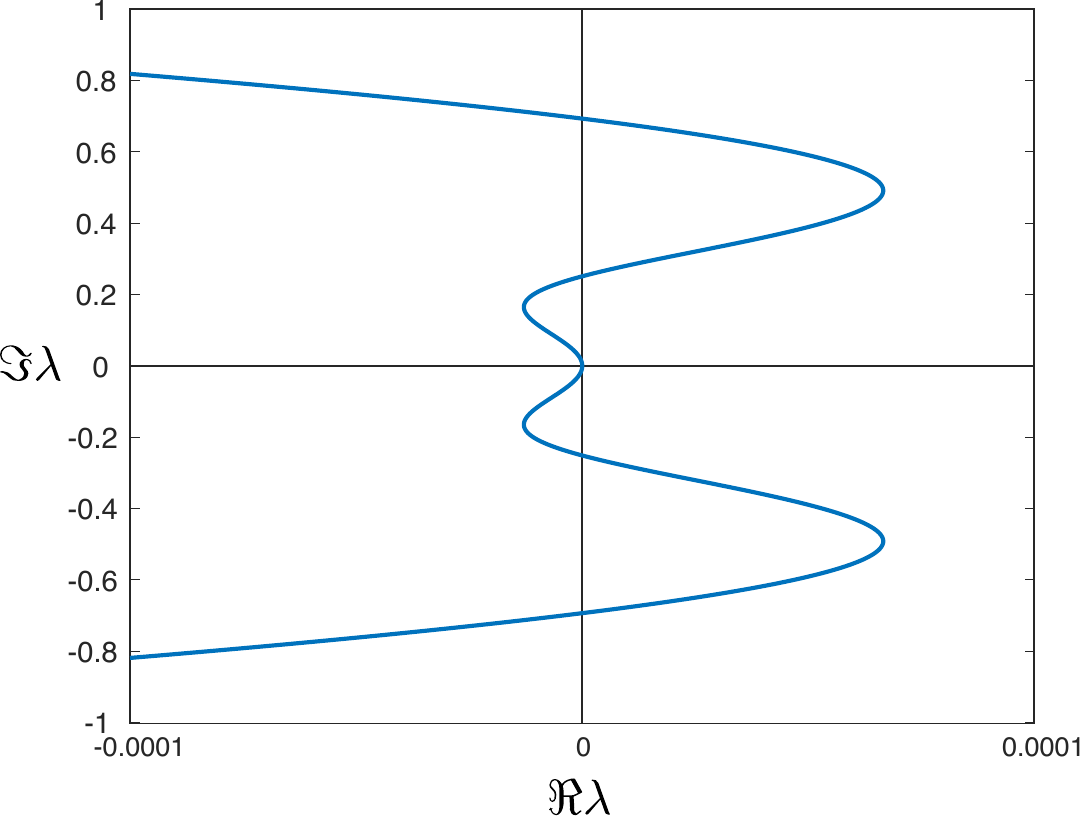}
\hspace{0.075\textwidth}
\includegraphics[width=0.36\textwidth]{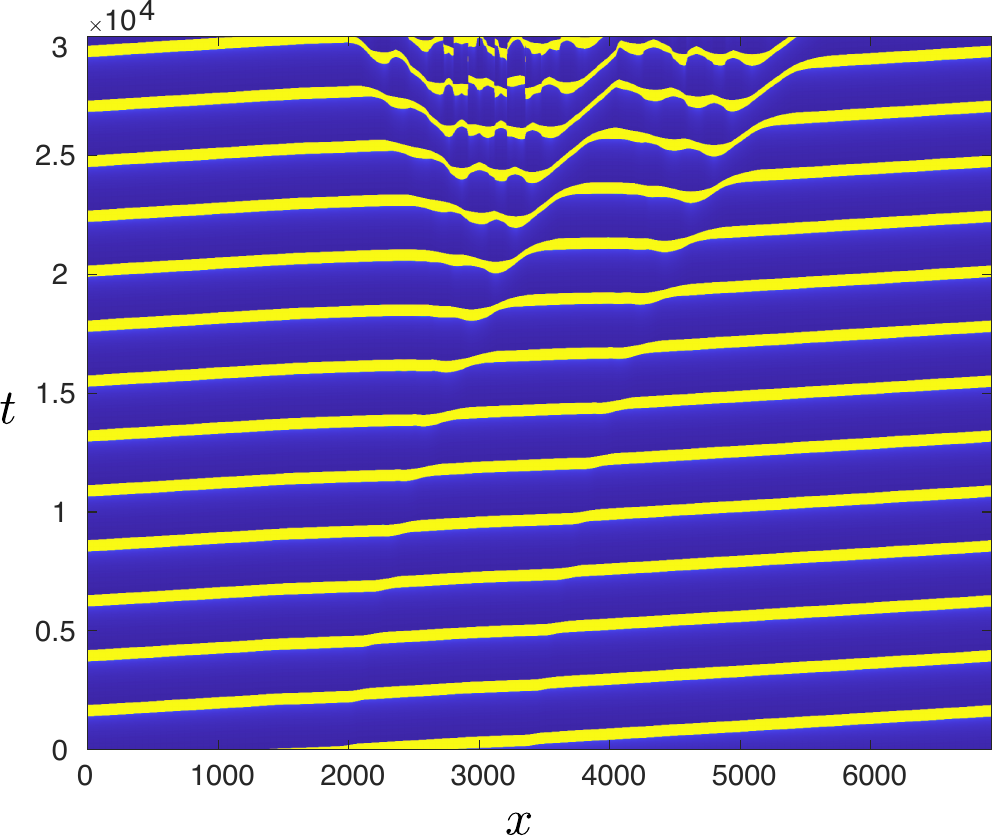}
\caption{(Top left) Shown are $u$ (solid) and $w$ (dashed) profiles of a wave-train solution of~\eqref{eq:FHN_mod_pde} with wave speed $c=3$ and parameters $a=0.25, \gamma=0.01, \eps=0.002$, and $r=0.998$. (Top right) Plot of the wave train from the left panel in $(u,w)$-space (solid green) along with the nullcline $F(u,w)=0$ (dashed red). 
(Bottom left) Results of numerical continuation of the critical spectral curve $\lambda_\eps(\rho)$ associated with the traveling wave train depicted in the top-left panel (compare with Figure~\ref{fig:critical_curve}). Plotted is $\Im (\lambda_\eps(\rho))$ versus $\Re(\lambda_\eps(\rho))$.
(Bottom right) Space-time plot of the $u$-component of the solution to~\eqref{eq:FHN_mod_pde} initiated with the (unstable) wave train at $c=3$ plus small noise. }
\label{fig:toy_wavetrain}
\end{figure}

\subsection{Outline of the paper} The remainder of this paper is devoted to the proof of Theorem~\ref{thm:spectral_stability}. In~\S\ref{S:setup}, we introduce our setup for analyzing the eigenvalue problem~\eqref{eq:Floquet_eigenvalue_problem} and we present the proof of Theorem~\ref{thm:spectral_stability}, which relies on combining three technical propositions concerning the behavior of~\eqref{eq:Floquet_eigenvalue_problem} for $\lambda$ in different regions of the complex plane. The region of small $\lambda$, which is the most delicate, is treated in~\S\ref{sec:R1}. This analysis hinges on the geometric desingularization of the system obtained by coupling the existence and eigenvalue problems near the fold points, which is carried out in~\S\ref{sec:folds}. The regime of intermediate $|\lambda|$ is discussed in~\S\ref{sec:R2}, while the regime of large $|\lambda|$ is addressed in Appendix~\ref{app:r3}. Appendices~\ref{appexpdi} and~\ref{app:airy} introduce exponential di- and trichotomies and the Airy function, respectively. Finally, Appendices~\ref{s:num} and~\ref{s:dns} provide details of the implementation of the numerical continuation and the direct numerical simulations presented in the introduction.

\subsection{Acknowledgments} \label{ssec:acknlowedgments}
The authors gratefully acknowledge support from the National Science Foundation through grants DMS-2510541 \& DMS-2202714 (M.~Avery), DMS-2238127 (P.~Carter), and DMS-2506837 \& DMS-2205663 (A.~Scheel). The work of B.~de Rijk is funded by the Deutsche Forschungsgemeinschaft (DFG, German Research Foundation) - Project-ID 491897824 and Project-ID 258734477 - SFB 1173.

The authors would like to thank the Collaborative Research Center ``Wave Phenomena'' at the Karlsruhe Institute of Technology for hosting M.~Avery and P.~Carter while working on this project in August 2022, as well as the Institute for Computational and Experimental Research in Mathematics in Providence, RI, USA, for hosting all authors in January 2023 through the Collaborate@ICERM program, and the Mathematical Institute at Leiden University for hosting M.~Avery, P.~Carter, and B.~de Rijk while working on this project during July 2024.

\section{Setup and strategy of proof}\label{S:setup}
Our analysis of the eigenvalue problem~\eqref{eq:Floquet_eigenvalue_problem} relies on the fast-slow structure of the traveling-wave equation~\eqref{TW}, which was used in the proof of Theorem~\ref{thm:existence} in~\cite{CASCH}, where phase-wave trains were constructed using geometric singular perturbation theory. We briefly outline this construction in~\S\ref{sec:existence_overview}, and we derive pointwise estimates for the proximity of the wave train to its singular limit. In~\S\ref{sec:spectral_setup}, we describe the setup and strategy for the analysis of the spectral problem~\eqref{eq:Floquet_eigenvalue_problem}. Based on this strategy, we present in~\S\ref{sec:proofmaintheorem} a break down of the proof of Theorem~\ref{thm:spectral_stability} into three propositions, which will be proved in the subsequent sections.
\subsection{Overview of existence analysis}\label{sec:existence_overview}
The construction of phase-wave trains is based on a fast-slow analysis of the traveling-wave equation~\eqref{TW}, which we repeat here for convenience
\begin{align}\label{eq:fast}
\begin{split}
u_\xi &= v,\\
v_\xi &= -cv - f(u) + w,\\
w_\xi &= -\frac{\epsilon}{c}(u-\gamma w - a).
\end{split}
\end{align}
We refer to~\eqref{eq:fast} as the fast system. Rescaling $y=\eps \xi$, we obtain the equivalent slow system 
\begin{align}\label{eq:slow}
\begin{split}
\eps u_y &= v,\\
\eps v_y &= -cv - f(u) + w,\\
w_y &= -\frac{1}{c}(u-\gamma w - a).
\end{split}
\end{align}
To construct phase-wave trains, we separately analyze~\eqref{eq:fast} and~\eqref{eq:slow} in the limit $\eps=0$; by concatenating orbits from the limiting systems, we obtain a singular periodic orbit, which can be shown to perturb to an actual periodic solution of the full system for all sufficiently small $\eps>0$. 

\subsubsection{Slow subsystem}
Setting $\eps=0$ in~\eqref{eq:slow}, we obtain 
\begin{align*}
\begin{split}
0 &= v,\\
0 &= -cv - f(u) + w,\\
w_y &= -\frac{1}{c}(u-\gamma w - a),
\end{split}
\end{align*}
for which the flow is restricted to the set $\mathcal{M}_0 = \{(u,v,w) \in \R^3 : v = 0, w = f(u)\}$, called the critical manifold. Away from points where $f'(u)=0$, the reduced flow on $\mathcal{M}_0$ is given by
\begin{align}\label{eq:reduced}
\begin{split}
cf'(u) u_y &= \gamma f(u) + a - u.
\end{split}
\end{align}
We recall from~\S\ref{sec:introduction} that $u_1$ and $\bar{u}_2$, given by~\eqref{defu1} and~\eqref{defu1s}, denote the local minimum and maximum of the cubic $w=f(u)$, respectively, at which $f'(u)=0$. The critical manifold therefore decomposes into three normally hyperbolic branches
\begin{align*}
    \mathcal{M}_0^\lr&\coloneqq  \{(u,v,w) \in \R^3 : v = 0, w = f(u), u\in(-\infty, u_1)\}\\
    \mathcal{M}_0^\mathrm{m}&\coloneqq  \{(u,v,w) \in \R^3 : v = 0, w = f(u), u\in(u_1,\bar{u}_1)\}\\
    \mathcal{M}_0^\rr&\coloneqq  \{(u,v,w) \in \R^3 : v = 0, w = f(u), u\in(\bar{u}_1,\infty)\}
\end{align*}
and two fold points $(u_1,0,f(u_1))$ and $(\bar{u}_1,0,f(\bar{u}_1))$; see Figure~\ref{fig:singular_slow}. We refer to these as the lower and upper fold points, respectively. In the parameter regime $0<\gamma<\gamma_*(a)$, the flow of~\eqref{eq:reduced} points upward on the left branch $\mathcal{M}^\lr_0$ and downward on the right branch $\mathcal{M}^\rr_0$. On the left branch $\mathcal{M}^\lr_0$ the dynamics in the $w$-variable is given by
\begin{align} \label{eq:slowl}
w_y = -\frac{1}{c}\left(f_{\lr}^{-1}(w) - \gamma w - a\right),
\end{align}
where $u = f_{\lr}^{-1}(w)$ denotes the smallest root of the cubic equation $f(u) = w$. Similarly, the dynamics in the $w$-variable on the right branch $\mathcal{M}^\rr_0$ is given by
\begin{align*} 
w_y = -\frac{1}{c}\left(f_{\rr}^{-1}(w) - \gamma w - a\right),
\end{align*}
where $u = f_{\rr}^{-1}(w)$ denotes the largest root of $f(u) = w$. 

\subsubsection{Layer problem}
The layer problem is obtained by setting $\eps=0$ in~\eqref{eq:fast}
\begin{align}\label{eq:layer}
\begin{split}
u_\xi &= v,\\
v_\xi &= -cv - f(u) + w.
\end{split}
\end{align}
Here $w \in \R$ acts as a parameter. This system admits a family of equilibria, given by the critical manifold $\mathcal{M}_0$. We focus on the behavior of the system for $w=f(u_1)$, that is, the layer containing the lower left fold point. The layer problem~\eqref{eq:layer} admits two fixed points: $(u_1,0)$, corresponding to the nonhyperbolic fold point, and a second saddle fixed point at $(u_2,0)$, where $u_2$ is given by~\eqref{defu2}; see also Figure~\ref{fig:singular_slow}. For $w=f(u_1)$,~\eqref{eq:layer} is a Fisher--KPP-type equation which, for values of $c\geq c_*(a)$, admits a traveling front solution $\phi_\f(\xi)=(u_{\f}(\xi),v_{\f}(\xi)) $ arising as a heteroclinic connection between the hyperbolic saddle $(u_2,0)$ and the nonhyperbolic fold $(u_1,0)$ where $u_{\f}(\xi)$ is monotonically decreasing~\cite{AronsonWeinberger,CASCH}. The linearization of~\eqref{eq:layer} at $w = f(u_1)$ about the fold point $(u_1,0)$ admits the eigenvalues $0$ and $-c$ with corresponding eigenvectors $(1,0)^\top$ and $(-1,c)^\top$, respectively. For $c>c_*(a)$, the front $\phi_\f$ approaches $(u_1,0)$ with algebraic decay along a center manifold associated with the zero eigenvalue, while in the critical case, $c=c_*(a)$, this fixed point is approached along its strong stable manifold with exponential decay. 

By symmetry, analogous results hold for $w=f(\bar{u}_1)$. In particular,~\eqref{eq:layer} again admits two fixed points: $(\bar{u}_1,0)$, corresponding to the nonhyperbolic fold point, and a saddle fixed point at $(\bar{u}_2,0)$, where $\bar{u}_2$ is given by~\eqref{defu2}. We note that
\begin{align} \label{eq:differenceu1u2}
\bar{u}_1 - \bar{u}_2 = u_2 - u_1 = \sqrt{1-a+a^2}.
\end{align}
For $w=f(\bar{u}_1)$, there exists a monotone traveling back solution $\phi_\bb(\xi)=(u_{\bb}(\xi),v_{\bb}(\xi))$ arising as a heteroclinic connection between the equilibria $\bar{u}_2$ and $\bar{u}_1$ where $u_{\bb}(\xi)$ is monotonically increasing. Again, for $c>c_*(a)$, $\phi_\bb$ approaches $(\bar{u}_1,0)$ with algebraic decay along a center manifold.

In the case $c>c_*(a)$, the above front and back solution of~\eqref{eq:layer} for $w=f(u_1), f(\bar{u}_1)$, respectively, satisfy the following pointwise estimates.

\begin{proposition}  \label{prop:frontback} 
Let $0 < a < \frac12$ and take $c > c^*(a)$. Then, there exist constants $C, \upsilon, \omega > 0$ such that 
\begin{align*}
\begin{split}
\left|u_{\f}(\xi) - u_2\right|, \left|v_{\f}(\xi)\right|, \left|u_{\bb}(\xi) - \bar{u}_2\right|, \left|v_{\bb}(\xi)\right| &\leq C\re^{\upsilon \xi}, \qquad \xi \leq 0,\\
\left|u_{\f}(\xi) - u_1 - \frac{\omega}{1+\xi}\right|, \left|v_{\f}(\xi)\right|, \left|u_{\bb}(\xi) - \bar{u}_1 + \frac{\omega}{1+\xi}\right|, \left|v_{\bb}(\xi)\right| &\leq \frac{C}{1+\xi^2} \qquad \xi \geq 0.
\end{split}
\end{align*}
In addition, $u_{\f}$ is monotonically decreasing and $u_{\bb}$ is monotonically increasing. 
\end{proposition}

\subsubsection{Construction of singular periodic orbit} It is readily seen that equation~\eqref{eq:slowl} can be solved by separation of variables, yielding a solution $w_{\lr}(y)$ with initial value $w_{\lr}(0) = f(u_1)$, which satisfies $w_{\lr}(L_\lr) = f(\bar{u}_2)$ for 
\begin{align} \label{periodexpr1}
L_\lr = \int_{f(u_1)}^{f(\bar{u}_2)} \frac{-c}{f_{\lr}^{-1}(w) - \gamma w - a} \de w = \int_{u_1}^{\bar{u}_2} \frac{-cf'(u)}{u - \gamma f(u) - a} \de u.\end{align}
Consequently, $u_{\lr}(y) = f_{\lr}^{-1}(w_{\lr}(y))$ solves~\eqref{eq:reduced} for $y > 0$ and satisfies $u_{\lr}(0) = u_1$ and $u_{\lr}(L_\lr) = \bar{u}_2$. Hence, the orbit $(u_{\lr}(y),0,f(u_{\lr}(y))$ lies on the left branch of the critical manifold $\mathcal{M}_0$ and connects the fold point $(u_1,0,f(u_1))$ to the point $(\bar{u}_2,0,f(\bar{u}_2))$. Similarly, there exists a solution $u_{\rr}(y)$ to~\eqref{eq:reduced} with boundary values $u_{\rr}(0) = \bar{u}_1$ and $u_{\rr}(L_\rr) = u_2$, where we have
\begin{align} \label{periodexpr2}
L_\rr = \int_{f(\bar{u}_1)}^{f(u_2)} \frac{-c}{f_{\rr}^{-1}(w) - \gamma w - a} \de w = \int_{\bar{u}_1}^{u_2} \frac{-cf'(u)}{u - \gamma f(u) - a} \de u.\end{align}
The orbit $(u_{\rr}(y),0,f(u_{\rr}(y))$ connects the fold point $(\bar{u}_1,0,f(\bar{u}_1))$ to the point $(u_2,0,f(u_2))$ on the right branch of $\mathcal{M}_0$. We define a singular periodic orbit $\Gamma_0(c)$ by concatenating these orbit segments on the left and right branches of $\mathcal{M}_0$ with the front and back solutions $(u_{\f}(\xi),v_{\f}(\xi),f(u_1))$ and $(u_{\bb}(\xi),v_{\bb}(\xi),f(\bar{u}_1))$; see Figure~\ref{fig:PO_nonhyp}. 

\begin{figure}
\centering
\includegraphics[width=0.6\linewidth]{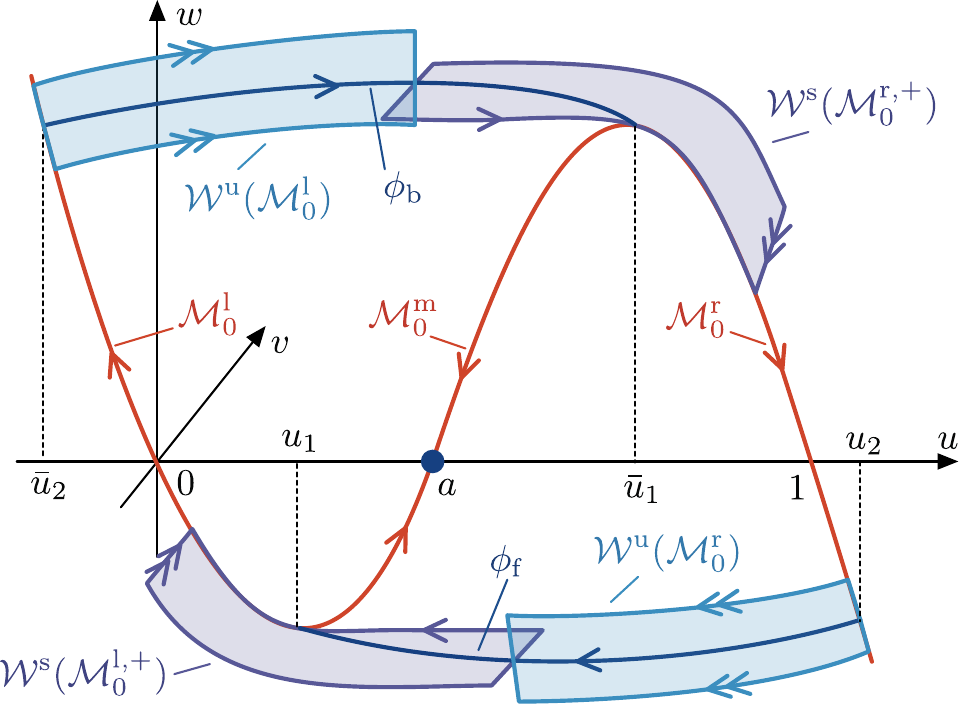}
\caption{Shown is the singular periodic orbit $\Gamma_0(c)$ formed by concatenating segments of the branches $\mathcal{M}^\mathrm{l,r}_0$ of the critical manifold $\mathcal{M}_0$ with the fast front and back solutions $\phi_\mathrm{f,b}$. }
\label{fig:PO_nonhyp}
\end{figure}

\subsubsection{Existence of nearby periodic orbit and pointwise estimates}\label{sec:existence}
In~\cite{CASCH} it is shown that the singular periodic orbit $\Gamma_0(c)$ perturbs to a nearby periodic solution $\Gamma_\eps(c)$ of~\eqref{eq:fast} for sufficiently small $\eps>0$. The proof is based on a fixed-point argument, using estimates which follow from results of geometric singular perturbation theory. In the forthcoming spectral stability analysis, we require more detailed pointwise estimates on the solution, which in part rely on the nature of the passage of the periodic orbit near the nonhyperbolic fold points. Hence we briefly review this aspect of the construction here.

Away from the fold points, the manifolds $\mathcal{M}^{\lr,\rr}_0$ are normally hyperbolic saddle-type critical manifolds. Therefore, by standard results of geometric singular perturbation theory~\cite{fenichel1979geometric}, for any $k>0$, (compact portions of) these critical manifolds, as well as their stable and unstable manifolds $\mathcal{W}^{\mathrm{s},\mathrm{u}}(\mathcal{M}^{\lr,\rr}_0)$, perturb to locally invariant manifolds $\mathcal{M}^{\lr,\rr}_\eps$ and $\mathcal{W}^{\mathrm{s},\mathrm{u}}(\mathcal{M}^{\lr,\rr}_\eps)$ for sufficiently small $\eps>0$, which are $\mathcal{O}(\eps)$-close in the $C^k$ sense to their singular counterparts. In~\cite{CASCH}, the persistence of the singular periodic orbit $\Gamma_0(c)$ relies on the transverse intersections of the manifolds $\mathcal{W}^{\mathrm{u}}(\mathcal{M}^{\rr}_\eps)$ and $\mathcal{W}^{\mathrm{s}}(\mathcal{M}^{\lr}_\eps)$ along the front $\phi_\f$ and of $\mathcal{W}^{\mathrm{u}}(\mathcal{M}^{\lr}_\eps)$ and $\mathcal{W}^{\mathrm{s}}(\mathcal{M}^{\rr}_\eps)$ along the back $\phi_\bb$. However, this transversality is not guaranteed by standard geometric singular perturbation theory, as the intersections between these manifolds along the orbits $\phi_\f, \phi_\bb$ occur in the layers $w=f(u_1),f(\bar{u}_1)$, in which the manifold $\mathcal{M}^{\lr}_0$ loses hyperbolicity at the lower fold point, while $\mathcal{M}^{\rr}_0$ loses hyperbolicity at the upper fold point, and the perturbed stable manifolds $\mathcal{W}^{\mathrm{s}}(\mathcal{M}^{\rr,\lr}_\eps)$ are not well defined near $\phi_\f, \phi_\bb$. 

To complete the construction, it is necessary to control the perturbed flow near the fold points, which requires the use of geometric desingularization techniques. These methods will also play an important role in the forthcoming spectral stability analysis. We consider the flow near the lower fold point $(u,v,w)=(u_1,0,f(u_1))$; the upper fold is similar. Following~\cite[\S4]{CASCH}, for any $k>0$, there exist a neighborhood of the origin $\mathcal{V}\subset \mathbb{R}^3$ and a $C^k$-change of coordinates $\mathcal{N}_\eps:\mathcal{V}\rightarrow \mathbb{R}^3$ such that the map $U= (u_1,0,f(u_1))^\top+\mathcal{N}_\eps(V)$ transforms~\eqref{eq:fast} to the system
\begin{align}
\begin{split}\label{eq:fold_normalform}
x_\zeta &= g_1(x,y;\eps),\\
y_\zeta &= \eps g_2(x,y;\eps),\\
z_\zeta &= zg_3(x,y,z;\eps)
\end{split}
\end{align}
in a neighborhood of the fold point, where
\begin{align*}
    g_1(x,y;\eps) &= -x^2+y+\mathcal{O}(xy, y^2, x^3,\eps),\\
    g_2(x,y;\eps) &=1+\mathcal{O}(x,y,\eps),\\
    g_3(x,y,z;\eps)&= -\frac{c}{\theta_\mathrm{lf}}+\mathcal{O}(x,y,z,\eps).
\end{align*}
 Here the traveling-wave coordinate has been rescaled as $\zeta=\theta_\mathrm{lf} \xi$, where the positive constant $\theta_\mathrm{lf}$ is given by
\begin{align*}
    \theta_\mathrm{lf} \coloneqq  -\frac{(a^2-a+1)^{1/6}(u_1-\gamma f(u_1)-a)^{1/3}}{c}>0.
\end{align*}
This coordinate transformation  explicitly distinguishes between the flow on a two-dimensional center manifold near the fold and the strongly attracting dynamics in the hyperbolic $z$-direction; see~\S\ref{sec:folds} for further details on the coordinate transformation $\mathcal{N}_\eps$. We note that a similar coordinate transformation exists near the upper fold point with corresponding rescaling
\begin{align*}
    \theta_\mathrm{uf} \coloneqq  \frac{(a^2-a+1)^{1/6}(\bar{u}_1-\gamma f(\bar{u}_1)-a)^{1/3}}{c}>0.
\end{align*}

\begin{figure}
\centering
\includegraphics[width=0.45\linewidth]{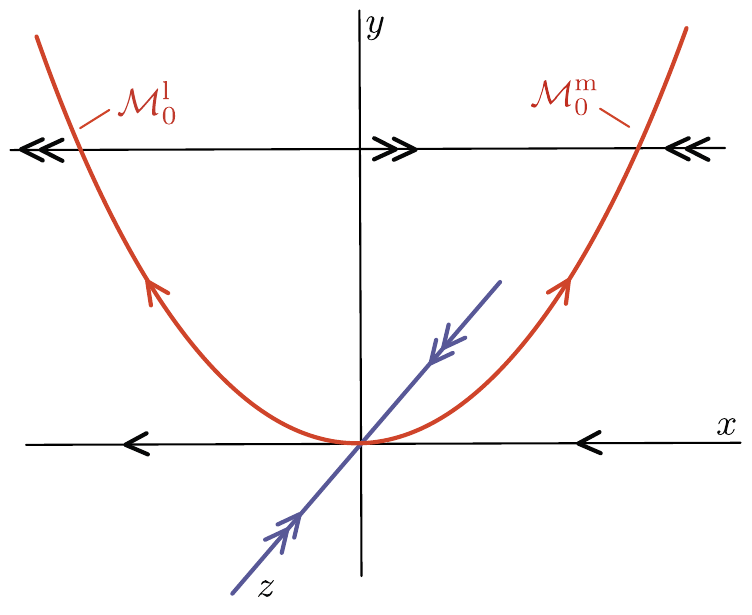}
\caption{Shown are the dynamics~\eqref{eq:fold_normalform} in the local $xyz$-coordinates near the lower fold point in the singular limit $\eps=0$.}
\label{fig:fold_singular}
\end{figure}
In the above transformed system~\eqref{eq:fold_normalform}, the critical manifold $\mathcal{M}_0$ is determined by the conditions $z=0$ and $g_1(x,y;0)=0$, taking the form of (approximately) an upward facing parabola centered at the origin $(x,y)=(0,0)$ in the subspace $z=0$, with the left and right branches of the parabola representing $\mathcal{M}^\lr_0$ and $\mathcal{M}^\mathrm{m}_0$, respectively; see Figure~\ref{fig:fold_singular}. Within the subspace $z=0$, we define the following continuation of the critical manifold $\mathcal{M}^{\lr}_0$ by
\begin{align}\label{eq:ml_plus}
    \mathcal{M}^{\lr,+}_0\coloneqq \mathcal{M}^{\lr}_0\cup\left\{ (x,y,z): y=z=0, x\geq0  \right\}
\end{align}
that is, we append the positive $x$-axis to $\mathcal{M}^{\lr}_0$. Away from the fold, using standard geometric singular perturbation theory as stated above, the manifold $\mathcal{M}^{\lr}_0$ perturbs to a locally invariant slow manifold $\mathcal{M}^\lr_\eps$ which is $C^k-\mathcal{O}(\eps)$-close to $\mathcal{M}^{\lr}_0$. Using blow-up desingularization techniques, in~\cite[\S4]{CSosc} it was shown that the \emph{extended} manifold $\mathcal{M}^{\lr,+}_0$ perturbs to a locally invariant manifold $\mathcal{M}^{\lr,+}_\eps$ which is $\mathcal{O}(\eps^{2/3})$-close in $C^0$ to $\mathcal{M}^{\lr,+}_0$ and $\mathcal{O}(\eps^{1/3})$-close in $C^1$ to $\mathcal{M}^{\lr,+}_0$. The family of strong stable fibers $\mathcal{W}^\mathrm{s}(\mathcal{M}^{\lr,+}_0)$ of $\mathcal{M}^{\lr,+}_0$ also perturbs to a two-dimensional locally invariant manifold $\mathcal{W}^\mathrm{s}(\mathcal{M}^{\lr,+}_\eps)$ which is similarly $\mathcal{O}(\eps^{2/3})$-close in $C^0$ and $\mathcal{O}(\eps^{1/3})$-close in $C^1$ to $\mathcal{W}^\mathrm{s}(\mathcal{M}^{\lr,+}_0)$. Following~\cite{CASCH}, one can then utilise the transverse intersection of $\mathcal{W}^{\mathrm{u}}(\mathcal{M}^{\rr}_\eps)$ and the extended manifold $\mathcal{W}^{\mathrm{s}}(\mathcal{M}^{\lr,+}_\eps)$ (and analogously for an extended manifold $\mathcal{M}^{\rr,+}_\eps$ near the upper fold) to complete the existence argument; see Figure~\ref{fig:fold_existence}.

\begin{figure}
\centering
\includegraphics[width=0.55\linewidth]{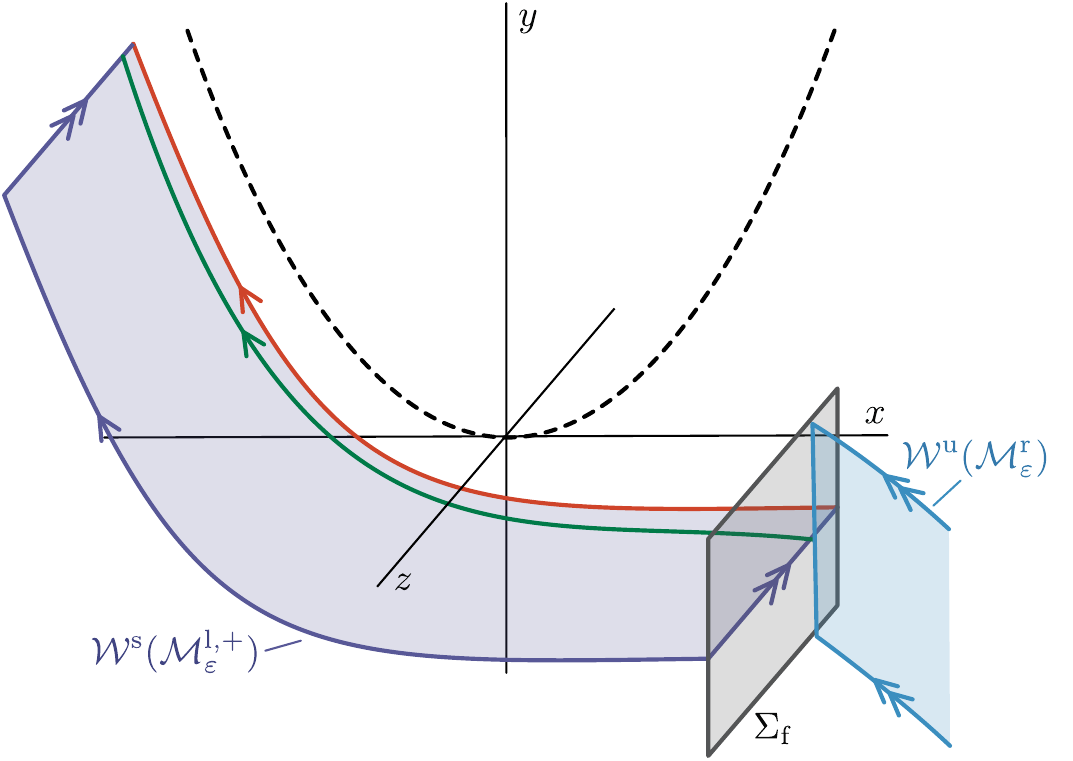}
\caption{Shown are the dynamics of~\eqref{eq:fold_normalform} near the lower fold point for $0<\eps\ll1$ as well as the manifolds $\mathcal{W}^{\mathrm{u}}(\mathcal{M}^{\rr}_\eps)$ and  $\mathcal{W}^{\mathrm{s}}(\mathcal{M}^{\lr,+}_\eps)$. The behavior near the upper fold point is similar. The existence of periodic orbits corresponding to traveling wave-train solutions is obtained in~\cite{CASCH} via a fixed point argument applied to the Poincar\'e map associated with a transverse section $\Sigma_\mathrm{f}$ near the fold. }
\label{fig:fold_existence}
\end{figure}

We have the following.

\begin{proposition}\label{prop:pointwise}
Let $0 < a < \frac12$. Fix $c > c_*(a)$ and $0 < \gamma < \gamma_*(a)$. Then, there exist a constant $C \geq 1$ such that, provided $0 < \epsilon \ll 1$,~\eqref{TW} admits a periodic orbit $\Gamma_\eps(c)$ corresponding to a  stationary solution $(u_{\epsilon},w_{\epsilon})(\xi)$~\eqref{FHN} of period $L_\eps = L_{\lr,\eps} + L_{\rr,\eps}$ with
\begin{align} \label{period}
 |\epsilon L_{\rr,\eps} - L_\rr|, |\epsilon L_{\lr,\eps} - L_\lr| \leq C\epsilon^{\frac{1}{3}}.
\end{align}
Furthermore, there exists a continuous map $\delta_0 \colon [0,\infty) \to [0,\infty)$ with $\delta_0(0) = 0$ such that for $0<\eps\ll\nu\ll1$, the following pointwise estimates hold
\begin{align} \label{pointwise:front/back}
\begin{split}
\left|u_{\epsilon}(\xi) - u_{\f}(\xi)\right|, \left|u_\eps'(\xi) - u_{\f}'(\xi)\right|, \left|w_\eps(\xi) - f(u_1)\right| &\leq C_\nu \epsilon^{\frac{2}{3}},  
\qquad \xi \in \left[\frac{\log(\epsilon)}{\nu},\frac{1}{\nu}\right],\\
\left|u_{\epsilon}(L_{\lr,\eps} + \xi) - u_{\bb}(\xi)\right|,\left|u_{\epsilon}'(L_{\lr,\eps} + \xi) - u_{\bb}'(\xi)\right|, \left|w_\eps(\xi) - f(\bar{u}_1)\right|  &\leq C_\nu \epsilon^{\frac{2}{3}},  
\qquad \xi \in \left[\frac{\log(\epsilon)}{\nu},\frac{1}{\nu}\right],\\
\left|u_\epsilon(\xi) - u_{\lr}(\epsilon \xi)\right|, \left|u_\epsilon'(\xi)\right| &\leq C_\nu \epsilon^{\frac{2}{3}}, \qquad \xi \in \left[\frac{\nu}{\epsilon}, L_{\lr,\eps} + \frac{\log(\epsilon)}{\nu}\right],\\
\left|u_\epsilon(L_{\lr,\eps} + \xi) - u_{\rr}(\epsilon \xi)\right|, \left|u_\epsilon'(L_{\lr,\eps} + \xi)\right| &\leq C_\nu \epsilon^{\frac{2}{3}}, \qquad \xi \in \left[\frac{\nu}{\epsilon}, L_{\rr,\eps} + \frac{\log(\epsilon)}{\nu}\right],\\
\left|u_\epsilon(\xi) - u_1\right|, \left|u_\eps'(\xi)\right|, \left|w_\eps(\xi) - f(u_1)\right| &\leq \delta_0(\nu), \qquad  \ \, \xi \in \left[\frac{1}{\nu},\frac{\nu}{\epsilon}\right],\\
\left|u_\epsilon(L_{\lr,\eps}+\xi) - \bar{u}_1\right|, \left|u_\eps'(L_{\lr,\eps}+\xi)\right|, \left|w_\eps(L_{\lr,\eps}+\xi) - f(\bar{u}_1)\right| &\leq \delta_0(\nu), \qquad  \ \, \xi \in \left[\frac{1}{\nu},\frac{\nu}{\epsilon}\right],\\
\left|u_\epsilon(\xi) - u_\lr(\eps \xi)\right|, \left|u_\epsilon(L_{\lr,\eps} + \xi) - u_\rr(\eps \xi)\right| &\leq \delta_0(\nu), \qquad \ \, \xi \in \left[\frac{1}{\nu},\frac{\nu}{\epsilon}\right],
\end{split}
\end{align} 
where $C_\nu \geq 1$ is an $\epsilon$-independent constant.
\end{proposition}
\begin{proof}
The existence result and the estimate~\eqref{period} follow directly from~\cite[Theorem 1.1]{CASCH}. It remains to obtain the pointwise estimates~\eqref{pointwise:front/back}. To do this, we return to  the existence construction~\cite[Proposition~4.3]{CASCH} for the phase-wave trains. The wave trains are constructed in the three-dimensional traveling-wave equation~\eqref{TW} using geometric singular perturbation theory and a fixed-point argument to obtain a periodic orbit near the singular orbit described above. The analysis in~\cite[Proposition~4.3]{CASCH} shows that the perturbed periodic orbit is $\mathcal{O}(\epsilon^{2/3})$-close to the singular orbit. However, to verify the pointwise estimates~\eqref{pointwise:front/back}, slightly more care is needed.

The periodic orbit is obtained as a fixed point of the return map to a section $\Sigma_\mathrm{f}$ transverse to the front $(u_{\f}(\xi),v_{\f}(\xi),f(u_1))$. Without loss of generality, we assume this intersection occurs at $\xi=0$, so that $(u_{\f}(0),v_{\f}(0),f(u_1))\in\Sigma_\mathrm{f}$. Within this section, the unstable manifold $\mathcal{W}^\mathrm{u}(M^\mathrm{r}_\epsilon)$ transversely intersects the stable manifold $\mathcal{W}^\mathrm{s}(M^{\mathrm{l},+}_\epsilon)$ of the trajectory $M^{\mathrm{l},+}_\epsilon$, which is the continuation of the slow manifold $M^\mathrm{l}_\epsilon$ through the fold. This intersection occurs at a point $(u,v,w)$ satisfying $\left|(u,v,w)-(u_{\f}(0),v_{\f}(0),f(u_1))\right| \leq C\epsilon^{2/3}$. The fixed point of this map, corresponding to the periodic orbit, is obtained exponentially close (in $\epsilon^{-1}$) to this intersection. Hence by a regular perturbation argument, we have that for any $\nu>0$, $\left|u_{\epsilon}(\xi) - u_{\f}(\xi)\right|\leq C_\nu \epsilon^{2/3}$ for $\xi \in \left[-\frac{1}{\nu},\frac{1}{\nu}\right]$. By taking $\nu>0$ sufficiently small, we can ensure that the periodic orbit is within a small neighborhood of the lower left fold point $(u,v,w)=(u_1,0,f(u_1))$ at $\xi=1/\nu$. To extend the left side of the interval to $\xi =\frac{\log(\epsilon)}{\nu}$, we apply standard corner estimates (see, e.g.~\cite[Theorem 4.5]{PBR}). A similar argument holds across the back, which completes the proof of the first two estimates.

Fixing $\nu > 0$ sufficiently small, we can ensure that at $\xi =\frac{\nu}{\epsilon}$, the wave train is exponentially close to a point on $\mathcal{M}^\mathrm{l}_\epsilon$ outside a small neighborhood of the lower left fold point, that is, in the region where $\mathcal{M}^\mathrm{l}_\epsilon$ is normally hyperbolic. Therefore, the third estimate follows from standard geometric singular perturbation theory, noting that the weaker $\epsilon^{2/3}$ estimate (as opposed to $\epsilon$) is due to the fact that the jump point from $\mathcal{M}^\mathrm{l}_\epsilon$ along the back at $\xi=L_{\lr,\eps} + \log(\epsilon)/\nu$ occurs at a location which is $\mathcal{O}(\epsilon^{2/3})$-close to the singular orbit by the first estimate. The fourth estimate concerning the passage near $\mathcal{M}^\mathrm{r}_\epsilon$ is obtained similarly.

Fixing $\nu>0$ sufficiently small in the first three estimates, by construction, the wave train is within a small neighborhood of the lower left fold point $(u,v) = (u_1,0)$ on the interval $\xi\in  \left[\frac{1}{\nu},\frac{\nu}{\epsilon}\right]$, and similarly for the upper right fold point $(u,v) = (\bar{u}_1,0)$, from which we obtain the fifth and sixth estimates. The final remaining estimate concerns the proximity of the wave train to the reduced slow solutions $u_\lr(\eps \xi), u_\rr(\eps \xi)$ near the folds. To obtain this, we show that
\begin{align*}
    |u_{\lr}(y) - u_1|, |u_{\rr}(y) - u_1|\leq \delta_0(\nu)
\end{align*}
for $y\in (0, \nu]$ and then use the fifth estimate. We focus on $u_\lr(y)$; the argument for $u_\rr(y)$ is similar. Recall that $u_\lr(y)$ is the solution of~\eqref{eq:reduced} satisfying $u_\lr(0)=u_1, u_\lr(L_\lr)=\bar{u}_2$. Solving~\eqref{eq:reduced} by separation of variables, we find that near $y=0$, 
\begin{align*}
    u_\lr(y) = u_1-\sqrt{\frac{2(a-u_1-\gamma f(u_1))y}{cf''(u_1)}}+\smallO(\sqrt{y})
\end{align*}
from which the result follows.
\end{proof}

\subsection{Spectral problem setup}\label{sec:spectral_setup}

Let $\rho \in \R$. Setting $(u,w) = \re^{-\ri\rho \xi}(\check{u}, \check{w})$, the Floquet--Bloch eigenvalue problem~\eqref{eq:Floquet_eigenvalue_problem} can be reformulated as a first-order boundary value problem in $\check{\Psi} = (\check{u},\check{u}_\xi,\check{w})$, which reads
\begin{align} \label{fulleigenvalueproblem_unscaled}
\check{\Psi}_\xi &= \check{A}(\xi;\epsilon,\lambda) \check{\Psi}, \qquad \check{A}(\xi;\epsilon,\lambda) = \begin{pmatrix} 0 & 1 & 0 \\ \lambda - f'(u_\epsilon(\xi)) & -c & 1 \\ -\frac{\epsilon}{c} & 0 & \frac{1}{c}\left(\epsilon\gamma + \lambda\right)\end{pmatrix}\\
\check{\Psi}(L_\eps) &= \re^{\ri \rho L_\eps} \check{\Psi}(0). \label{Floquet_BC}
\end{align}
To preserve the explicit fast-slow structure present in the existence problem~\eqref{TW}, we perform a rescaling to remove the $\eps$-independent $\lambda$ term from the equation for the third component. In the new coordinate $\Psi(\xi) = \smash{\re^{-\frac{\lambda}{c}\xi} \check{\Psi}(\xi)}$ the boundary value problem~\eqref{fulleigenvalueproblem_unscaled}-\eqref{Floquet_BC} transforms into the fast-slow eigenvalue problem
\begin{align} \label{eigenvalueproblem}
\begin{split}
\Psi_\xi &= A(\xi;\epsilon,\lambda) \Psi, \qquad A(\xi;\epsilon,\lambda) = \begin{pmatrix} A_{\f}(\xi;\epsilon,\lambda) & B_0 \\ \epsilon B_1 & \epsilon A_s\end{pmatrix},
\end{split}
\end{align}
supplemented with the Floquet boundary condition
\begin{align} \label{eigenvalueproblemBC}
\Psi(L_\eps) &= \re^{\left(\ri\rho - \frac{\lambda }{c}\right)L_\eps} \Psi(0),
\end{align}
where we denote
\begin{align*}
A_{\f}(\xi;\epsilon,\lambda) = \begin{pmatrix} -\frac{\lambda}{c} & 1 \\ \lambda - f'(u_\epsilon(\xi)) & -c-\frac{\lambda}{c} \end{pmatrix}, \qquad B_0 = \begin{pmatrix}0 \\ 1 \end{pmatrix}, \qquad B_1 = \begin{pmatrix} -\frac{1}{c} & 0 \end{pmatrix}, \qquad A_s = \frac{\gamma}{c}.\end{align*}
We note that in this formulation, the spectral parameter $\lambda$ appears both in the matrix $A(\xi;\eps, \lambda)$ as well as the boundary condition, while the Floquet parameter $\rho$ appears only in the boundary condition.

Our approach for analyzing the eigenvalue problem~\eqref{eigenvalueproblem}-\eqref{eigenvalueproblemBC} is multifaceted. In order to prove Theorem~\ref{thm:spectral_stability}, we must rule out the possibility of spectrum in the open right-half plane, aside from a simple eigenvalue of $\mathcal{L}_{0,\eps}$ at the origin due to translation invariance. To achieve this, inspired by~\cite{PBR}, we define three primary regions of the complex plane
\begin{align}
\begin{split}\label{eq:regions}
R_1(\mu) &= \{\lambda \in \C : |\lambda| < \mu\}\\
R_2(\mu,\varpi,\varrho) &= \{\lambda \in \C : \Re(\lambda) \geq -\varpi, \mu \leq |\lambda| \leq \varrho\}\\
R_{3,\eps}(\varrho) &= \{\lambda \in \C : \Re(\lambda) \geq -\tfrac{3}{4} \eps\gamma, |\Im(\lambda)| > \varrho\},
\end{split}
\end{align}
where $0 < \varpi \ll \mu \ll 1 \ll \varrho$ are $\eps$-independent constants, so that the union of these three sets covers the closed right half plane; see Figure~\ref{fig:regions}. Due to the presence of essential spectrum near the imaginary axis, each region presents unique challenges and requires different techniques to either preclude spectrum satisfying $\Re(\lambda)\geq0$ or, in the case of $R_1(\mu)$, to describe in detail the nature of the critical curve of spectrum containing the translation eigenvalue $\lambda=0$, which must be shown to satisfy the diffusive spectral stability condition of Theorem~\ref{thm:spectral_stability}~\ref{thm:spectral_stability_iii}. 

In~\S\ref{sec:proofmaintheorem} below, we present the proof of Theorem~\ref{thm:spectral_stability}, which is based on three technical propositions describing the behavior of the eigenvalue problem~\eqref{eigenvalueproblem}-\eqref{eigenvalueproblemBC} in each of the three regions~\eqref{eq:regions}. First, we preclude spectrum in the region $R_{3,\eps}(\varrho)$, which determines $\varrho > 0$. Then, we perform the spectral analysis in the ball $R_1(\mu)$. This then determines $\mu$. As highlighted above, the spectrum of the wave train necessarily passes through $\lambda=0$ when $\rho=0$, due to translation invariance of the wave. Hence the region $R_1(\mu)$ requires careful estimates as the critical spectral curve $\lambda_\eps(\rho)$ satisfying $\lambda_\eps(0)=0$ must be carefully expanded at $\rho = 0$ to rule out the possibility of instability for small $|\lambda|$; see Figure~\ref{fig:critical_curve} for a numerically computed example. Finally, given $\mu,\varrho > 0$, we show that all spectrum in $R_2(\mu,\varpi,\varrho)$ must lie in the open left-half plane. The remainder of the paper is then concerned with the proof of these propositions.

\begin{figure}
\centering
\includegraphics[width=0.4\linewidth]{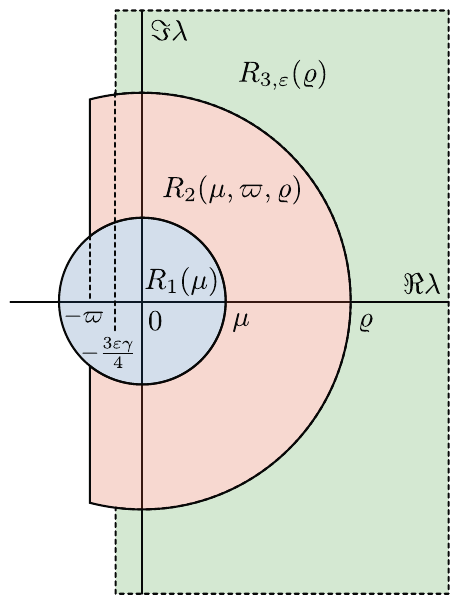}
\caption{Shown are the regions $R_1(\mu), R_2(\mu,\varpi,\varrho)$, and $R_{3,\eps}(\rho)$~\eqref{eq:regions}, of the complex plane in which analyze the eigenvalue problem~\eqref{fulleigenvalueproblem_unscaled}-\eqref{Floquet_BC}.}
\label{fig:regions}
\end{figure}

\subsection{Proof of Theorem~\ref{thm:spectral_stability}}\label{sec:proofmaintheorem}

Following the strategy outlined in~\S\ref{sec:spectral_setup}, we first preclude the existence of spectrum of the linearization $\El_{\eps}$ in the region $R_{3,\eps}(\varrho)$, where $\varrho > 0$ is a sufficiently large $\eps$-independent constant. We exploit that $\El_\eps$ can be written as the sum of a principal diagonal diffusion-advection operator, which is independent of $\eps$ and generates a strongly continuous semigroup, and an $\eps$-dependent  remainder operator, which obeys an $\eps$-independent bound. Consequently, standard resolvent bounds yield the invertibility of $\El_{\eps} - \lambda$ for $\Re(\lambda) > \varrho_1$, where $\varrho_1 > 0$ is a sufficiently large $\eps$-independent constant. The half-plane $\{\lambda \in \C : \Re(\lambda) > \varrho_1\}$ covers a large part of the region $R_{3,\eps}(\varrho)$. In the remaining part of $R_{3,\eps}(\varrho)$, which is characterized by large imaginary part and bounded real part, we proceed as in~\cite{FHNpulled} and rescale the eigenvalue problem. We observe that the rescaled linear operator $\El_\eps - \lambda$ can be inverted using a Neumann series expansion for $|\Im(\lambda)| > \varrho_2$ and $\Re(\lambda) \in [-\frac34 \eps\gamma,\varrho_1]$, where $\varrho_2 > 0$ is a sufficiently large $\eps$-independent constant. All in all, we arrive at the following result, which is proved in Appendix~\ref{app:r3}.

\begin{proposition} \label{prop: regionR3}
Let $0 < a < \frac12$. Fix $0 < \gamma < \gamma_*(a)$ and $c > c_*(a)$. There exists a constant $\varrho > 0$ such that, provided $0 < \eps \ll 1$, the linearization $\El_\eps$ of~\eqref{FHN} about $\phi_\eps(\xi)$ possesses no spectrum in the region $R_{3,\eps}(\varrho)$.
\end{proposition}
Next, for suitably small $\mu>0$, we consider the spectrum in the small ball $R_1(\mu)$, which is described in the following proposition. 

\begin{proposition}\label{prop:region_r1} Let $0<a<\frac{1}{2}$, $0<\gamma<\gamma_*(a)$, and $c>c_*(a)$. Fix $\delta > 0$. There exists $\mu>0$ such that, provided $0<\eps\ll1$, the linearization $\El_\eps$ of~\eqref{FHN} about $\phi_\eps(\xi)$ possesses no spectrum of nonnegative real part in $R_1(\mu) \setminus \{0\}$. Furthermore, a point $\lambda \in R_1(\mu)$ with $|\Re(\lambda)| \leq \mu \eps^{1/6}$ lies in the spectrum $\Sigma(\El_{\eps})$ if and only if it obeys the main formula~\eqref{eq:mainformapprox} for some $\rho \in \R$. In this case, $\lambda = \lambda_\eps(\rho)$ is an eigenvalue of the Bloch operator $\El_{\rho,\eps}$. Finally, locally near $(0,0)$ the set of $(\lambda,\rho) \in \C \times \R$ solving~\eqref{eq:mainformapprox} is given by a  smooth curve $\lambda_\eps \colon I_\eps \to \C$, where $I_\eps \subset \R$ is an interval containing $0$. It holds
\begin{align}\label{eq:r1_critical_curve_estimates}
\lambda_\eps(0) = \Re(\lambda_\eps'(0)) = 0, \qquad \lambda_\eps''(0) \in \R, \qquad       \left|\lambda_{\eps}'(0)-\ri c\right|\leq\delta , \qquad \left|\lambda_{\eps}''(0)+ \frac{2\kappa c^3}{L_0} \, \eps^{\frac23}\right| \leq \delta \eps^{\frac23},
\end{align}
where $\kappa > 0$ is given by~\eqref{eq:quad_coeff} and $L_0 = L_\rr + L_\lr > 0$ is defined by~\eqref{periodexpr1} and~\eqref{periodexpr2}.
\end{proposition}
The proof of Proposition~\ref{prop:region_r1} will be given in~\S\ref{sec:R1}. In the region $R_1(\mu)$ the spectrum necessarily contains a curve which meets the origin in a quadratic tangency due to translation invariance (see Figure~\ref{fig:critical_curve}). The $\eps^{2/3}$-scaling present in the quadratic coefficient arises due to interaction with the nonhyperbolic fold points; this scaling is also corroborated numerically; see Figure~\ref{fig:lambda_coeff_cont}. To prove Proposition~\ref{prop:region_r1}, our strategy is to derive a formula for this critical spectral curve to preclude any spectrum in the region $R_1(\mu)$ of nonnegative real part, except for the translation eigenvalue at the origin. To achieve this, we solve the eigenvalue problem~\eqref{eigenvalueproblemBC} using exponential trichotomies and Lin's method to construct potential eigenfunctions along subintervals of $[0,L_\eps]$ and then match these solutions together. Blow-up desingularization methods are needed to solve~\eqref{eigenvalueproblemBC} in regions where the wave train passes near the nonhyperbolic fold points. This procedure results in an implicit transcendental equation -- the main formula~\eqref{eq:mainformapprox}, see also Proposition~\ref{prop:mainformula} -- relating $\lambda$ to the Floquet parameter $\rho$, as well as the system parameters $(c,a,\gamma,\eps)$. Several different scaling regimes are needed to capture the behavior of the main formula in the ball $R_1(\mu)$, which we further divide into four sub-regions containing the relevant spectral information. In particular, we emphasize that while Proposition~\ref{prop:region_r1} guarantees that the phase-wave trains of Theorem~\ref{thm:existence} admit no unstable spectrum in the region $R_1(\mu)$, our analysis suggests that instability mechanisms may manifest in other systems with slightly modified nonlinearities; see~\S\ref{s:instability}.

Finally, fixing $\varrho$ as in Proposition~\ref{prop: regionR3} and $\mu$ as in Proposition~\ref{prop:region_r1} we have the following concerning spectrum in the region $R_2(\mu,\varpi,\varrho)$.

\begin{proposition}\label{prop:region_r2}
Let $0 < a < \frac12$. Fix $0 < \gamma < \gamma_*(a)$ and $c > c_*(a)$. Fix $\varrho$ as in Proposition~\ref{prop: regionR3} and $\mu$ as in Proposition~\ref{prop:region_r1}. Then, provided $0 < \eps \ll \varpi \ll 1$, the linearization $\El_\eps$ of~\eqref{FHN} about $\phi_\eps(\xi)$ possesses no spectrum of nonnegative real part in the compact set $R_2(\mu,\varpi,\varrho)$. 
\end{proposition}

The proof of Proposition~\ref{prop:region_r2} will be presented in~\S\ref{sec:R2}. It relies on a further decomposition of the region $R_2(\mu,\varpi,\varrho)$ into two parts: one where $|\Re(\lambda)| \leq \varpi$ and its complement. In the first part, the Riccati transform is employed to achieve a separation between slow and fast dynamics in the eigenvalue problem, while in the complementary region, this separation is obtained via exponential dichotomies. The slow-fast decomposition reveals that eigenvalues with nonnegative real part cannot occur, since the fast reduced eigenvalue problems along the front $\phi_{\f}$ and the back $\phi_{\bb}$ are of Fisher--KPP type and therefore admit no eigenvalues within $R_2(\mu,\varpi,\varrho)$.

We are now able to complete the proof of Theorem~\ref{thm:spectral_stability}.

\begin{proof}[Proof of Theorem~\ref{thm:spectral_stability}]
    Statements~\ref{thm:spectral_stability_i}-\ref{thm:spectral_stability_iii} follow immediately from Propositions~\ref{prop: regionR3},~\ref{prop:region_r1}, and~\ref{prop:region_r2}. The statement in~\ref{thm:spectral_stability_ii} concerning the algebraic simplicity of the translation eigenvalue $\lambda=0$ follows from the hyperbolicity of the periodic orbit in~\eqref{TW} corresponding to the wave train; see the discussion in~\cite[\S4.4]{CASCH}.
\end{proof}

\section{The region \texorpdfstring{$R_1(\mu)$}{R1}}\label{sec:R1}

To prove Proposition~\ref{prop:region_r1}, we directly solve the eigenvalue problem~\eqref{eigenvalueproblem} by constructing solutions on four subintervals of the shifted interval $\mathcal{I} =\left[\xi^0_{\eps,\nu}, \xi^L_{\eps,\nu}\right] \coloneqq \left[\tfrac{\nu}{\eps}, L_\eps+\tfrac{\nu}{\eps}\right]$ with the modified Floquet boundary condition
\begin{align}\label{eigenvalueproblemBCshift}
\Psi\left(\xi^L_{\eps,\nu}\right) = \re^{\left(\ri\rho-\frac{\lambda }{c}\right)L_\eps}\Psi\left(\xi^0_{\eps,\nu}\right).
\end{align}
We note that solutions to the shifted boundary-value problem~\eqref{eigenvalueproblem}/\eqref{eigenvalueproblemBCshift} are in one-to-one correspondence to those of the original one~\eqref{eigenvalueproblem}-\eqref{eigenvalueproblemBC}. Indeed, if $\Psi(\xi)$ is a solution to~\eqref{eigenvalueproblem}/\eqref{eigenvalueproblemBCshift}, then setting $\Psi(\xi) = \re^{-(\ri\rho-\lambda/c) L_\eps} \Psi(L_\eps + \xi)$ for $\xi \in [0,\xi^0_{\eps,\nu}]$ yields a solution to~\eqref{eigenvalueproblem}-\eqref{eigenvalueproblemBC} on $[0,L_\eps]$. Similarly, any solution to the boundary-value problem~\eqref{eigenvalueproblem}-\eqref{eigenvalueproblemBC} yields a solution to~\eqref{eigenvalueproblem}/\eqref{eigenvalueproblemBCshift}.  The reason for introducing the shift is that it turns out to be most convenient to apply the Floquet boundary condition along the slow manifold $\mathcal{M}^\lr_\eps$ just after the wave train has passed the lower left fold point. 

\begin{figure}
\centering
\includegraphics[width=0.7\linewidth]{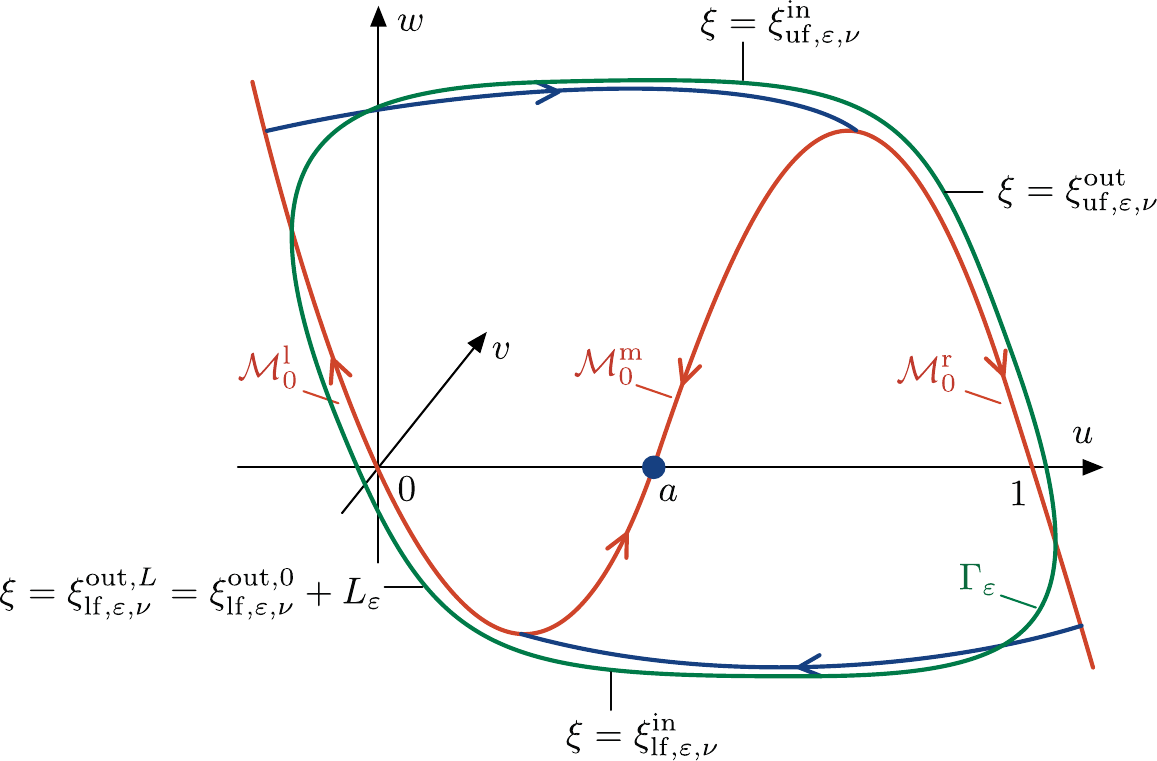}
\caption{Shown is the setup for Lin's method in the region $R_1(\mu)$: solutions are constructed along the four intervals $\mathcal{I}_\lr, \mathcal{I}_\uf, \mathcal{I}_\rr, \mathcal{I}_\lf$ and matched at the endpoints $\xi=\xi_{\mathrm{uf},\eps,\nu}^\mathrm{in}$ and $\xi=\xi_{\mathrm{uf},\eps,\nu}^\mathrm{out}$ near the upper fold point, and at $\xi=\xi_{\mathrm{lf},\eps,\nu}^\mathrm{in}$ near the lower fold point. Finally, the Floquet boundary condition~\eqref{eigenvalueproblemBCshift} is applied at $\xi=\xi_{\mathrm{lf},\eps,\nu}^{\mathrm{out},L}=\xi_{\mathrm{lf},\eps,\nu}^{\mathrm{out},0}+L_\eps$ near the lower point point.}
\label{fig:lin_setup}
\end{figure}

We recall that the period of the wave train is given by $L_\eps = L_{\lr,\eps}+L_{\rr,\eps}$, where $L_{\lr,\eps}$ measures the time spent along the left slow manifold $\mathcal{M}^\lr_\eps$, and $L_{\rr,\eps}$ denotes the time spent along the right slow manifold $\mathcal{M}^\rr_\eps$, so that $\xi = L_{\rr,\eps}$ occurs along the back $\phi_\mathrm{b}$, and $\xi = L_{\lr,\eps}+L_{\rr,\eps}$ is identified with $\xi=0$, occurring along the front $\phi_\f$. We split the interval $\mathcal{I}$ into four sub-intervals
\begin{align*}
\mathcal{I}_\lr& = \left[\xi_{\mathrm{lf},\eps,\nu}^{\mathrm{out},0},\xi_{\mathrm{uf},\eps,\nu}^\mathrm{in}\right]\coloneqq \left[\tfrac{\nu}{\eps},L_{\lr,\eps}+\tfrac{1}{\nu}\right],\\
\mathcal{I}_\uf& = \left[\xi_{\mathrm{uf},\eps,\nu}^\mathrm{in},\xi_{\mathrm{uf},\eps,\nu}^\mathrm{out}\right]\coloneqq \left[L_{\lr,\eps}+\tfrac{1}{\nu}, L_{\lr,\eps}+\tfrac{\nu}{\eps}\right],\\
\mathcal{I}_\rr& =\left[\xi_{\mathrm{uf},\eps,\nu}^\mathrm{out},\xi_{\mathrm{lf},\eps,\nu}^\mathrm{in}\right]\coloneqq   \left[L_{\lr,\eps}+\tfrac{\nu}{\eps}, L_{\lr,\eps}+L_{\rr,\eps}+\tfrac{1}{\nu}\right],\\
\mathcal{I}_\lf& =\left[\xi_{\mathrm{lf},\eps,\nu}^\mathrm{in},\xi_{\mathrm{lf},\eps,\nu}^{\mathrm{out},L}\right]\coloneqq  \left[L_{\lr,\eps}+L_{\rr,\eps}+\tfrac{1}{\nu}, L_{\lr,\eps}+L_{\rr,\eps}+\tfrac{\nu}{\eps}\right],
\end{align*}
for $0 < \eps \ll \nu \ll 1$, so that the intervals $\mathcal{I}_{\lr,\rr}$ describe the wave-train solution \emph{away} from the nonhyperbolic fold points, while the intervals $\mathcal{I}_\lf$ and $\mathcal{I}_\uf$ describe passage near the lower fold and upper fold, respectively; see Figure~\ref{fig:lin_setup}. We note that
\begin{align*}
   \xi_{\mathrm{lf},\eps,\nu}^{\mathrm{out},L} = \xi^L_{\eps,\nu}=  \xi^0_{\eps,\nu} +L_\eps = \xi_{\mathrm{lf},\eps,\nu}^{\mathrm{out},0} +L_\eps.
\end{align*}
Solving the eigenvalue problem~\eqref{eigenvalueproblem} then amounts to solving boundary value problems on each of these four intervals, matching the resulting solutions at the end points of each of the intervals, and applying the Floquet condition~\eqref{eigenvalueproblemBCshift}. Eliminating all free variables results in an implicit equation, which we call the ``main formula", which relates $\lambda, \eps$, and the Floquet parameter $\rho$ and whose solutions correspond to spectrum of the linearization $\El_\eps$ in the region $R_1(\mu)$. We then split the region $R_1(\mu)$ into several smaller subregions in order to obtain a leading-order expression for the critical spectral curve at the origin, and to rule out the possibility of spectrum in the part of $R_1(\mu)$ which lies in the right half plane.

We begin in~\S\ref{sec:reduced_variational}-\ref{sec:slow_manifolds} with some preliminary results concerning the existence of exponential trichotomies along the fast jumps and near the slow manifolds $\mathcal{M}^{\lr,\rr}_\eps$. This allows us to solve the resulting boundary value problems on the intervals $\mathcal{I}_{\lr,\rr}$ (``between" the fold points) in~\S\ref{sec:bvp_between_folds}. We then turn to the boundary value problems on the intervals $\mathcal{I}_{\lf,\uf}$ (``near" the fold points) in~\S\ref{sec:bvp_near_folds}; the analysis in these intervals requires the use of blow-up desingularization techniques to track the linearized problem alongside the existence problem, the technical details of which are presented in~\S\ref{sec:folds}. In~\S\ref{sec:criticalcurve}-\ref{sec:mainformula}, we solve the overall boundary value problem~\eqref{eigenvalueproblem} along with the Floquet condition~\eqref{eigenvalueproblemBCshift}, and we derive an expression for the critical spectral curve and analyze its behavior near the origin, allowing us to complete the proof of Proposition~\ref{prop:region_r1} in~\S\ref{sec:region_r1_proof}.

\subsection{The reduced variational problem}\label{sec:reduced_variational} 

The variational equation of the layer problem~\eqref{eq:layer} about the front and the back solution reads
\begin{align} \Psi_\xi = \widehat A_i(\xi) \Psi, \qquad \widehat A_i(\xi) = \begin{pmatrix} 0     & 1 \\ -f'(u_i(\xi)) &  -c \end{pmatrix} \label{varred}\end{align}
for $i = \f,\bb$. By Proposition~\ref{prop:frontback} the coefficient matrices $\widehat A_{\f}(\xi)$ and $\widehat A_{\bb}(\xi)$ converge exponentially to the asymptotic matrices 
$$\widehat A_{\f,-\infty} = \begin{pmatrix} 0 & 1 \\ -f'(u_2) & -c \end{pmatrix}, \qquad \widehat A_{\bb,-\infty} = \begin{pmatrix} 0 & 1 \\ -f'(\bar{u}_2) & -c \end{pmatrix},$$
respectively, as $\xi \to -\infty$. In addition, $\widehat A_{\f}(\xi)$ and $\widehat A_{\bb}(\xi)$ converge algebraically to the asymptotic matrix
\begin{align*}
\widehat A_{\infty} = \begin{pmatrix} 0 & 1 \\ 0 &  -c \end{pmatrix}
\end{align*}
as $\xi \to \infty$. The hyperbolic matrices $\widehat A_{\f,-\infty}$ and $\widehat A_{\bb,-\infty}$ possess one positive and one negative eigenvalue. Therefore,~\cite[Lemma~3.4]{PAL} yields the existence of an exponential dichotomy.

\begin{proposition} \label{prop:varred}
There exist $K, \alpha > 0$ such that for each $\nu > 0$ there exists a constant $C_\nu > 0$ such that system~\eqref{varred} has an exponential dichotomy on $(-\infty,\frac{1}{\nu}]$ with constants $C_\nu, \alpha > 0$. The associated projections $\smash{\widehat P_{i}(\xi)}$ have rank $1$ and satisfy
\begin{align*} \ker(\widehat P_{i}(\xi)) = \mathrm{Sp}\left\{\begin{pmatrix} u_i'(\xi) \\ v_i'(\xi)\end{pmatrix}\right\},\end{align*}
and
\begin{align*} 
\begin{split}
\left\|\widehat P_i(\xi) - \widehat{\mathcal P}_{i,-\infty}\right\| &\leq K\re^{\alpha \xi}, \qquad \xi \leq 0,\\
\left\|\widehat P_i(\xi) - \widehat{\mathcal P}_\infty\right\| &\leq \frac{K}{1+\xi}, \qquad \xi \in \left[0,\frac{1}{\nu}\right],
\end{split}
\end{align*}
where $\widehat{\mathcal P}_{i,-\infty}$ is the spectral projection onto the stable eigenspace of $\widehat{A}_{i,-\infty}$ for $i = \f,\bb$, and where $\widehat{\mathcal P}_\infty$ is the spectral projection onto the stable eigenspace of $\widehat{A}_\infty$. 
\end{proposition}
\begin{proof}
Let $i \in \{\f,\bb\}$. First, Proposition~\ref{prop:frontback} yields constants $C, \upsilon > 0$ such that
\begin{align} \label{asympmatrixest} 
\begin{split}
\left\|\widehat A_i(-\xi) - \widehat A_{i,-\infty}\right\| &\leq C\re^{-\upsilon \xi}, \qquad
\left\|\widehat A_i(\xi) - \widehat A_{\infty}\right\| \leq \frac{C}{1+\xi},\end{split}
\end{align}
for $\xi \geq 0$. Noting that $f'(u_2), f'(\bar{u}_2) < 0$ and $c > c_*(a) > 0$, we find that $\widehat A_{i,-\infty}$ is hyperbolic with one positive and one negative eigenvalue. Consequently, there exists a constant $\kappa_0 \in (0,c)$ such that for any $\kappa \in [0,\kappa_0]$ the constant-coefficient system
\begin{align} \label{asympsys} \Psi_\xi = \left(\widehat A_{i,-\infty} + \kappa I_2\right) \Psi, \end{align}
has an exponential dichotomy on $\R$ with rank-one projection $\widehat{\mathcal P}_{i,-\infty}$. On the other hand, since $\widehat A_{\infty}$ has the eigenvalues $0$ and $-c$, the constant coefficient system
\begin{align*} \Psi_\xi = \left(\widehat A_{\infty} + \kappa_0 I_2\right) \Psi,\end{align*}
possesses an exponential dichotomy on $\R$ with rank-one projection $\widehat{\mathcal P}_{\infty}$.

Hence, by estimate~\eqref{asympmatrixest} and~\cite[Lemma~3.4]{PAL} the weighted variational problem
\begin{align} \Psi_\xi = \left(\widehat A_i(\xi) + \kappa_0 I_2\right) \Psi, \label{varredweighted}
\end{align}
admits exponential dichotomies on both $(-\infty,0]$ and $[0,\infty)$ with associated rank-one projections $\widehat P_{i,\pm}(\pm \xi)$; see Appendix~\ref{appexpdi}. By Proposition~\ref{prop:frontback} the solution $(u_i'(\xi),v_i'(\xi))^\top \re^{\kappa_0 \xi}$ to~\eqref{varredweighted} decays exponentially as $\xi \to -\infty$ and increases exponentially as $\xi \to \infty$. Hence, $(u_i'(0),v_i'(0))$ must span the kernel of $\widehat P_{i,-}(0)$, and cannot lie in the range of $\widehat P_{i,+}(0)$. We conclude that $\ker(\widehat P_{i,-}(0)) \cap \widehat P_{i,+}(0)[\C^2] = \{0\}$, which implies by~\cite[Proposition~2.1]{PAL} that~\eqref{varredweighted} possesses an exponential dichotomy on $\R$ with projections $\widehat P_i(\xi)$. Combining estimate~\eqref{asympmatrixest} with~\cite[Lemma~3.4]{PAL} and its proof, we deduce that there exist constants $K,\alpha_0 > 0$ such that
\begin{align*} 
\begin{split}
\left\|\widehat P_i(-\xi) - \widehat{\mathcal P}_{i,-\infty}\right\| &\leq K\re^{-\alpha_0 \xi}, \qquad \left\|\widehat P_i(\xi) - \widehat{\mathcal P}_\infty\right\| \leq \frac{K}{1+\xi},
\end{split}
\end{align*}
for $\xi \geq 0$. 

On the other hand, by estimate~\eqref{asympmatrixest} and~\cite[Lemma 3.4]{PAL} the exponential dichotomy of~\eqref{asympsys} for $\kappa = 0$ carries over to an exponential dichotomy for the unweighted variational problem~\eqref{varred} on $(-\infty,0]$ with constants $C_1,\alpha_1 > 0$ and associated rank-one projections $\widehat Q_{i,-}(\xi)$. By Proposition~\ref{prop:frontback} the exponentially decaying solution $(u_i'(\xi),v_i'(\xi))^\top$ to~\eqref{varred} must span the kernel of $\widehat Q_{i,-}(\xi)$ for $\xi \leq 0$. On the other hand, there is freedom to choose the range of $\widehat Q_{i,-}(0)$ to be any subspace complementary to $\ker(\widehat Q_{i,-}(0))$, cf.~\cite[Lemma~1.2(ii)]{SAN1993}. Recalling $\ker(\widehat P_i(0)) = \ker(\widehat Q_{i,-}(0))$, we select $\widehat Q_{i,-}(0)[\C^2] = \widehat P_i(0)[\C^2]$. Since the unweighted problem~\eqref{varred} has the same action on subspaces as the weighted problem~\eqref{varredweighted}, we find $\widehat P_i(\xi) = \widehat Q_{i,-}(\xi)$ for all $\xi \leq 0$. By~\cite[p.~13]{COP} the exponential dichotomy of~\eqref{varred} on $(-\infty,0]$ can be extended to an exponential dichotomy on $(-\infty,\frac{1}{\nu}]$ with constants $C_\nu, \alpha_1 > 0$ and projections $\widehat P_i(\xi)$. Taking $\alpha = \min\{\alpha_0,\alpha_1\}$ the result follows.
\end{proof}

We consider the reduced variational problem
\begin{align} \label{varred2} \Psi_\xi = \widetilde A_i(\xi) \Psi, \qquad \widetilde A_i(\xi) = \begin{pmatrix} \widehat{A}_i(\xi) & B_0 \\ 0_{1 \times 2} & 0 \end{pmatrix}\end{align}
for $i = \f,\bb$, which arises by setting $\epsilon = \lambda = 0$ and replacing $u_\epsilon(\xi)$ by $u_i(\xi)$ in the eigenvalue problem~\eqref{eigenvalueproblem}. Denoting by $\widehat{\T}_i(\xi,y)$ the evolution of system~\eqref{varred}, the evolution $\widetilde{\T}_i(\xi,y)$ of the upper triangular system~\eqref{varred2} can be expressed as
\begin{align*}
\widetilde{\T}_i(\xi,y) = \begin{pmatrix} \widehat{\T}_i(\xi,y) & \displaystyle \int_y^\xi \widehat{\T}_i(\xi,z) B_0 \de z \\ 0_{1 \times 2} & 1\end{pmatrix}. 
\end{align*}
Thus, the exponential dichotomy of~\eqref{varred}, established in Proposition~\ref{prop:varred}, readily yields an exponential trichotomy for system~\eqref{varred2}.

\begin{proposition} \label{prop:varred2}
There exists $\widetilde{\alpha} > 0$ such that for each $\nu > 0$ there exists a constant $\widetilde{C}_\nu > 0$ such that system~\eqref{varred2} has an exponential trichotomy on $(-\infty,\frac{1}{\nu}]$ with constants $\smash{\widetilde{C}_\nu, \widetilde{\alpha}} > 0$. The associated projections are given by
\begin{align*}
\begin{split}
\widetilde P_{i,\nu}^{\su}(\xi) &= \begin{pmatrix} \widehat P_{i}(\xi) & \displaystyle -\int_{-\infty}^\xi \widehat{\T}_{i}^{\su}(\xi,y) B_0 \de y \\ 0_{1 \times 2} & 0 \end{pmatrix}, \qquad \widetilde P_{i,\nu}^{\uu}(\xi) = \begin{pmatrix} I_2 - \widehat P_{i}(\xi) & \displaystyle -\int_{\frac{1}{\nu}}^\xi \widehat{\T}_{i}^{\uu}(\xi,y) B_0 \de y \\ 0_{1 \times 2} & 0 \end{pmatrix},\\
\widetilde P_{i,\nu}^{\cc}(\xi) &= I_3 - \widetilde P_{i,\nu}^{\su}(\xi) - \widetilde P_{i,\nu}^{\uu}(\xi),
\end{split}
\end{align*}
where $\widehat{\T}_i^j(\xi,z)$, $j = u,s$, $i = \f,\bb$ denotes the (un)stable evolution of~\eqref{varred} under the exponential dichotomy established in proposition~\ref{prop:varred}. The projections have rank $1$ and satisfy
\begin{align} \label{projest}
\left\|\widetilde P_{i,\nu}^j(\xi) - \widetilde{\mathcal P}_{i,-\infty}^j\right\| \leq \widetilde{C}_\nu \re^{\widetilde{\alpha} \xi},
\end{align}
for $\xi \leq 0$, $i = \f,\bb$ and $j = \su,\uu,\cc$, where $\widetilde{\mathcal P}_{i,-\infty}^{\su},\widetilde{\mathcal P}_{i,-\infty}^{\uu}$ and $\widetilde{\mathcal P}_{i,-\infty}^{\cc}$ are the spectral projections of the asymptotic matrix
\begin{align*}
\widetilde A_{i,-\infty} = \begin{pmatrix} \widehat A_{i,-\infty} & B_0 \\ 0_{1 \times 2} & 0 \end{pmatrix},
\end{align*}
onto its stable, unstable and center eigenpace, respectively. Finally, we have
\begin{align} \label{projid}
\widetilde P_{i,\nu}^{\uu}(\xi)[\C^3] = \mathrm{Sp}\left\{\Phi_i(\xi)\right\}, \qquad \widetilde P_{i,\nu}^{\cc}(\xi)[\C^3] = \mathrm{Sp}\left\{\Psi_{i,\nu}(\xi)\right\},
\end{align}
where we denote
\begin{align*}
\Phi_i(\xi) = \begin{pmatrix} u_i'(\xi) \\ v_i'(\xi) \\ 0\end{pmatrix}, \qquad \Psi_{i,\nu}(\xi) = \begin{pmatrix} \displaystyle \int_{\frac{1}{\nu}}^\xi \widehat{\T}_{i}^{\uu}(\xi,y)B_0 \de y + \int_{-\infty}^\xi \widehat{\T}_{i}^{\su}(\xi,y)B_0 \de y  \\ 1\end{pmatrix},
\end{align*}
for $\xi \in (-\infty,\frac{1}{\nu}]$ and $i = \f,\bb$. 
\end{proposition}
\begin{proof}
Let $i \in \{\f,\bb\}$. Using Proposition~\ref{prop:varred}, one readily verifies that $\widetilde P_{i,\nu}^{\su}(\xi)$, $\widetilde P_{i,\nu}^{\uu}(\xi)$ and $\widetilde P_{i,\nu}^{\cc}(\xi)$ are projections satisfying~\eqref{projid}. We compute
\begin{align*} \widetilde P_{i,\nu}^{\su}(\xi)\widetilde{\T}_{i}(\xi,y) &= \begin{pmatrix} \widehat{\T}_i^{\su}(\xi,y) & \displaystyle-\int_{-\infty}^y \widehat{\T}_i^{\su}(\xi,z) B_0 \de z \\ 0_{1 \times 2} & 0 \end{pmatrix} = \widetilde{\T}_{i}(\xi,y)\widetilde P_{i,\nu}^{\su}(y),\\
\widetilde P_{i,\nu}^{\uu}(\xi)\widetilde{\T}_{i}(\xi,y) &= \begin{pmatrix} \widehat{\T}_i^{\uu}(\xi,y) & \displaystyle-\int_{\frac{1}{\nu}}^y \widehat{\T}_i^{\uu}(\xi,z) B_0 \de z \\ 0_{1 \times 2}  & 0 \end{pmatrix} = \widetilde{\T}_{i}(\xi,y)\widetilde P_{i,\nu}^{\uu}(y),\\
\widetilde P_{i,\nu}^{\cc}(\xi)\widetilde{\T}_{i}(\xi,y) &= \begin{pmatrix} 0_{2 \times 2} & \displaystyle\int_{-\infty}^\xi \widehat{\T}_i^{\su}(\xi,z) B_0 \de z + \int_{\frac{1}{\nu}}^\xi \widehat{\T}_i^{\uu}(\xi,z) B_0 \de z\\ 0_{1 \times 2} & 1\end{pmatrix} = \widetilde{\T}_{i}(\xi,y)\widetilde P_{i,\nu}^{\cc}(y).
\end{align*}
With the aid of Proposition~\ref{prop:varred} we estimate
\begin{align*}
\left\|\int_{-\infty}^y \widehat{\T}_i^{\su}(\xi,z) B_0 \de z\right\|
&\leq C_{\nu} \int_{-\infty}^y \re^{-\alpha(\xi - z)}\de z = \frac{C_{\nu}}{\alpha} \re^{-\alpha(\xi - y)}, & & y \leq \xi \leq \frac{1}{\nu},\\
\left\|\int_{\frac{1}{\nu}}^y \widehat{\T}_i^{\uu}(\xi,z) B_0 \de z\right\|
&\leq C_{\nu} \int_y^\infty \re^{\alpha(\xi - z)}\de z = \frac{C_{\nu}}{\alpha} \re^{\alpha(\xi - y)}, & & \xi \leq y \leq \frac{1}{\nu}.
\end{align*}
Hence, system~\eqref{varred2} has an exponential trichotomy on $(-\infty,\frac{1}{\nu}]$ with constants $\frac{2C_\nu}{\alpha}, \alpha > 0$ and projections $\smash{\widetilde P_{i,\nu}^j(\xi)}$, $j = \su,\uu,\cc$. 

Clearly, the weighted system
\begin{align*} \Psi_\xi = \left(\widetilde A_i(\xi) + \frac{\alpha}{2} I_3\right) \Psi,\end{align*}
possesses an exponential dichotomy on $(-\infty,\frac{1}{\nu}]$ with constants $\frac{2C_\nu}{\alpha}, \frac{\alpha}{2} > 0$ and projections $\widetilde P_{i,\nu}^{\su}(\xi)$. Similarly, the weighted system
\begin{align*} \Psi_\xi = \left(\widetilde A_i(\xi) - \frac{\alpha}{2} I_3\right) \Psi,\end{align*}
has an exponential dichotomy on $(-\infty,\frac{1}{\nu}]$ with constants $\frac{2C_\nu}{\alpha}, \frac{\alpha}{2} > 0$ and projections $I_3-\widetilde P_{i,\nu}^{\uu}(\xi)$. Hence, by combining estimate~\eqref{asympmatrixest} with~\cite[Lemma~3.4]{PAL} and its proof, we infer that there exist constants $\smash{\widetilde{M}_\nu}, \alpha_0 > 0$ such that
\begin{align*}
\left\|\widetilde P_{i,\nu}^j(\xi) - \widetilde{\mathcal P}_{i,-\infty}^j\right\| \leq \widetilde{M}_\nu \re^{\alpha_0 \xi},
\end{align*}
for $\xi \leq 0$ and $j = s,c$. So, taking $\widetilde{C}_\nu = \max\{\widetilde M_\nu, \frac{2C_\nu}{\alpha}\}$ and $\widetilde{\alpha} = \min\{\alpha,\alpha_0\}$ yields the desired result.
\end{proof}

\subsection{The eigenvalue problem along the left and right branch of the critical manifold} \label{sec:slow_manifolds} 
We establish exponential trichotomies for the eigenvalue problem~\eqref{eigenvalueproblem} along the left and the right branch of the critical manifold; see Appendix~\ref{appexpdi}. First of all, we show that the fast $(2 \times 2)$-subsystem
\begin{align}
\Psi_\xi = A_{\f}(\xi;\epsilon,\lambda) \Psi \label{fastsub}
\end{align}
of~\eqref{eigenvalueproblem} possesses an exponential dichotomy along these branches by using that its coefficient matrix is slowly varying and pointwise hyperbolic. Thus, we can diagonalize the full eigenvalue problem~\eqref{eigenvalueproblem} with the aid of the Riccati transformation~\cite{BDR,Chang_Riccati}, yielding the desired exponential trichotomy. The explicit diagonalization levaraged by the Riccati transformation allows us to determine the scalar dynamics in the one-dimensional center direction of the trichotomy to leading-order. All in all, we arrive at the following result.

\begin{proposition} \label{prop:slow}
Provided $0 < \epsilon,|\lambda| \ll \nu \ll 1$, system~\eqref{eigenvalueproblem} admits exponential trichotomies on $I_{\lr} = [\frac{\nu}{\epsilon},L_{\lr,\eps} + \frac{\log(\epsilon)}{\nu}]$ and on $I_{\rr} = [L_{\lr,\eps} + \frac{\nu}{\epsilon},L_{\lr,\eps}+L_{\rr,\eps}+\frac{\log(\epsilon)}{\nu}]$ with $\lambda$- and $\epsilon$-independent constants $C_\nu,\vartheta_\nu > 0$ and projections $P_{\lr,\epsilon,\lambda,\nu}^j(\xi)$ and $P_{\rr,\epsilon,\lambda,\nu}^j(\xi)$, $j = \su,\uu,\cc$, respectively. The projections have rank $1$ and satisfy
\begin{align} \label{projest2}
\left\|P_{i,\epsilon,\lambda,\nu}^j(\xi) - {\mathcal P}_{\epsilon,\lambda}^j(\xi)\right\| \leq C_\nu \epsilon^{\frac{2}{3}},
\end{align}
for $\xi \in I_i$, $i = \lr,\rr$ and $j = \su,\uu,\cc$, where $ {\mathcal P}_{\epsilon,\lambda}^{\su}(\xi), {\mathcal P}_{\epsilon,\lambda}^{\uu}(\xi)$ and $ {\mathcal P}_{\epsilon,\lambda}^{\cc}(\xi)$ are the spectral projections of the coefficient matrix $A(\xi;\epsilon,\lambda)$ of~\eqref{eigenvalueproblem} onto its stable, unstable and center eigenspace, respectively. Finally, the center evolutions $\T^\cc_{i,\epsilon,\lambda,\nu}(\xi,\zeta)$, $i = \lr,\rr$ under the exponential trichotomies satisfy
\begin{align} \label{approxid}
\begin{split}
\left| \left\langle \mathbf{e}_3, \T_{\lr,\epsilon,\lambda,\nu}^{\cc}\left(\frac{\nu}{\epsilon},L_{\lr,\eps} + \frac{\log(\epsilon)}{\nu}\right) \mathbf{e}_i\right\rangle\right| &\leq C_\nu \epsilon, \qquad i = 1,2,\\
\left| \left\langle \mathbf{e}_3, \T_{\lr,\epsilon,\lambda,\nu}^{\cc}\left(\frac{\nu}{\epsilon},L_{\lr,\eps} + \frac{\log(\epsilon)}{\nu}\right)\mathbf{e}_3\right\rangle - \frac{u_{\lr}(\nu) - \gamma f(u_\lr(\nu)) - a}{\bar{u}_2 - \gamma f(\bar{u}_2) - a}\right| &\leq C_\nu\left(|\lambda| + \epsilon^{\frac13}\right),
\end{split}
\end{align}
and
\begin{align*}
\begin{split}
\left| \left\langle \mathbf{e}_3, \T_{\rr,\epsilon,\lambda,\nu}^{\cc}\left(L_{\lr,\eps} + \frac{\nu}{\epsilon},L_{\lr,\eps}+L_{\rr,\eps}+\frac{\log(\epsilon)}{\nu}\right) \mathbf{e}_i\right\rangle\right| &\leq C_\nu \epsilon, \qquad i = 1,2\\
\left| \left\langle \mathbf{e}_3, \T_{\rr,\epsilon,\lambda,\nu}^{\cc}\left(L_{\lr,\eps} + \frac{\nu}{\epsilon},L_{\lr,\eps}+L_{\rr,\eps}+\frac{\log(\epsilon)}{\nu}\right) \mathbf{e}_3\right\rangle - \frac{u_{\rr}(\nu) - \gamma f(u_{\rr}(\nu)) - a}{u_2 - \gamma f(u_2) - a}\right| &\leq C_\nu\left(|\lambda| + \epsilon^{\frac13}\right),
\end{split}
\end{align*}
where $\{\mathbf{e}_1,\mathbf{e}_2,\mathbf{e}_3\}$ denotes the standard unit basis of $\C^3$.
\end{proposition}
\begin{proof}
In this proof $C_\nu \geq 1$ denotes an $\epsilon$-, $\lambda$- and $\xi$-independent constant, which will be taken larger if necessary.

We prove the result for $\xi \in I_{\lr}$, i.e.~along the left branch of the critical manifold, only. The argument for $\xi \in I_{\rr}$, i.e.~along the right branch of the critical manifold, is similar. We aim to diagonalize the block system~\eqref{eigenvalueproblem} via the Riccati transformation, cf.~\cite[Theorem~5.1]{BDR}. To do so, we establish an exponential dichotomy for the fast subsystem~\eqref{fastsub} using roughness techniques. By Proposition~\ref{prop:pointwise} we have
\begin{align} \left\|A_{\f}(\xi;\epsilon,\lambda) - A_{\lr}(\epsilon \xi)\right\| \leq C_\nu \left(\epsilon^{\frac{2}{3}} + |\lambda|\right), \label{matrixapp}\end{align}
for $\xi \in I_{\lr}$, where we denote
\begin{align*}
A_{\lr}(y) = \begin{pmatrix} 0 & 1 \\ -f'(u_{\lr}(y)) & -c \end{pmatrix}.\end{align*}
For each $y \in [\nu, L_\lr+1]$ it holds $f'(u_{\lr}(y)) \leq -1/C_\nu$. Therefore, there exists a constant $\vartheta_\nu > 0$ such that the matrix $A_{\lr}(y)$ is hyperbolic with one positive and one negative eigenvalue lying at distance $\geq 5\vartheta_\nu$ from the imaginary axis for each $y \in \left[\nu, L_\lr+1\right]$. Hence, provided $0 < \epsilon, |\lambda| \ll \nu \ll 1$, estimate~\eqref{matrixapp} implies that $A_{\f}(\xi;\epsilon,\lambda)$ is also hyperbolic with one positive and one negative eigenvalue lying at distance $\geq 4\vartheta_\nu$ from the imaginary axis for each $\xi \in I_{\lr}$. 

By~\cite[Proposition~6.1]{COP} the slowly varying system
\begin{align*}
\Psi_\xi = A_{\lr}(\epsilon \xi) \Psi,
\end{align*}
has an exponential dichotomy on $[\frac{\nu}{\epsilon}, \frac{L_\lr+1}{\epsilon}]$ with $\epsilon$-independent constants $C_\nu, 4\vartheta_\nu > 0$. Recalling estimates~\eqref{period} and~\eqref{matrixapp}, we observe that roughness results, cf.~\cite[Proposition~5.1]{COP}, yield an exponential dichotomy for system~\eqref{fastsub} on $I_{\lr}$ with $\lambda$- and $\epsilon$-independent constants $C_\nu, 3\vartheta_\nu > 0$, provided $0 < \epsilon,|\lambda| \ll \nu \ll 1$.

By~\cite[Theorem~5.1]{BDR} there exist continuous matrix functions $U_{\epsilon,\lambda,\nu} \colon I_{\lr} \to \C^{2 \times 1}$ and $S_{\epsilon,\lambda,\nu} \colon I_{\lr} \to \C^{1 \times 2}$ such that, if $\Psi(\xi)$ is a solution to~\eqref{eigenvalueproblem}, then 
$$\Phi(\xi) = H_{\epsilon,\lambda,\nu}(\xi)\Psi(\xi), \qquad H_{\epsilon,\lambda,\nu}(\xi) =  \begin{pmatrix} I_2-\epsilon U_{\epsilon,\lambda,\nu}(\xi) S_{\epsilon,\lambda,\nu}(\xi) & U_{\epsilon,\lambda,\nu}(\xi) \\ -\epsilon S_{\epsilon,\lambda,\nu}(\xi) & 1\end{pmatrix},$$  satisfies the block diagonal system
\begin{align*}
\begin{split}
\Phi_\xi = \begin{pmatrix} A_{\f}(\xi;\epsilon,\lambda) - \epsilon U_{\epsilon,\lambda,\nu}(\xi)B_1 & 0_{2 \times 1} \\
            0_{1 \times 2} & \epsilon A_s + \epsilon B_1 U_{\epsilon,\lambda,\nu}(\xi)
           \end{pmatrix} \Phi,
\end{split}
\end{align*}
for $\xi \in I_{\lr}$. Moreover, $U_{\epsilon,\lambda,\nu}$ and $S_{\epsilon,\lambda,\nu}$ are bounded by $\lambda$- and $\epsilon$-independent constants. Hence, the evolution 
\begin{align} \label{evolsform} \T_{\su,\epsilon,\lambda,\nu}(\xi,z) = \exp\left(\epsilon\left(\frac{\gamma}{c} (\xi-z) + \int_z^\xi B_1 U_{\epsilon,\lambda,\nu}(y) \de y\right)\right),\end{align}
of the scalar slow subsystem
\begin{align*} \Psi_\xi = \epsilon \left(A_s + B_1 U_{\epsilon,\lambda,\nu}(\xi)\right) \Psi,
\end{align*}
is bounded on $I_{\lr} \times I_{\lr}$ by an $\epsilon$- and $\lambda$-independent constant. Moreover, by roughness,~\cite[Theorem~5.1]{COP}, the exponential dichotomy of~\eqref{fastsub} carries over to an exponential dichotomy of the fast subsystem
\begin{align*}
\Psi_\xi = \left(A_{\f}(\xi;\epsilon,\lambda) - \epsilon U_{\epsilon,\lambda,\nu}(\xi)B_0\right) \Psi,
\end{align*}
on $I_{\lr}$ with $\lambda$- and $\epsilon$-independent constants $C_\nu, 2\vartheta_\nu > 0$ and projections $Q_{\lr,\epsilon,\lambda,\nu}(\xi)$. 

We conclude that
\begin{align*}
P_{\lr,\epsilon,\lambda,\nu}^{\su}(\xi) &= H_{\epsilon,\lambda,\nu}(\xi) \begin{pmatrix}  Q_{\lr,\epsilon,\lambda,\nu}(\xi) & 0_{2 \times 1} \\ 0_{1 \times 2}  & 0 \end{pmatrix} H_{\epsilon,\lambda,\nu}(\xi)^{-1}, \\
P_{\lr,\epsilon,\lambda,\nu}^{\uu}(\xi) &= H_{\epsilon,\lambda,\nu}(\xi) \begin{pmatrix} I_2 - Q_{\lr,\epsilon,\lambda,\nu}(\xi) & 0_{2 \times 1} \\ 0_{1 \times 2}  & 0 \end{pmatrix} H_{\epsilon,\lambda,\nu}(\xi)^{-1}, \\
P_{\lr,\epsilon,\lambda,\nu}^{\cc}(\xi) &= H_{\epsilon,\lambda,\nu}(\xi) \begin{pmatrix} 0_{2 \times 2} & 0_{2 \times 1} \\ 0_{1 \times 2} & 1 \end{pmatrix} H_{\epsilon,\lambda,\nu}(\xi)^{-1},
\end{align*}
are projections of an exponential trichotomy of the eigenvalue problem~\eqref{eigenvalueproblem} on $I_{\lr}$ with $\lambda$- and $\epsilon$-independent constants $C_{\nu}, 2\vartheta_\nu > 0$, where we use that the matrix function $H_{\epsilon,\lambda,\nu}(\xi)$ and its inverse 
$$H_{\epsilon,\lambda,\nu}(\xi)^{-1} = \begin{pmatrix} I_2 & -U_{\epsilon,\lambda,\nu}(\xi) \\ \epsilon S_{\epsilon,\lambda,\nu}(\xi) & 1-\epsilon S_{\epsilon,\lambda,\nu}(\xi)U_{\epsilon,\lambda,\nu}(\xi)\end{pmatrix},$$
are bounded by $\lambda$- and $\epsilon$-independent constants on $I_{\lr}$. 

To prove the estimate~\eqref{projest2}, we consider the positively weighted eigenvalue problem
\begin{align} \label{weightedev}
 \Psi_\xi = \left(A(\xi;\epsilon,\lambda) + \vartheta_\nu I_3\right) \Psi.
\end{align}
Clearly,~\eqref{weightedev} has an exponential dichotomy on $I_{\lr}$ with $\lambda$- and $\epsilon$-independent constants $C_{\nu}, \vartheta_\nu > 0$ and projections $P_{\lr,\epsilon,\lambda,\nu}^{\su}(\xi)$. First, we note that the coefficient matrix $A(\xi;\epsilon,\lambda) + \vartheta_\nu I_3$ is hyperbolic with two positive and one negative eigenvalue lying at distance $\geq \vartheta_\nu$ from the imaginary axis for each $\xi \in I_{\lr}$. Second, Proposition~\ref{prop:pointwise} yields 
\begin{align*}
 \left\|\partial_\xi A(\xi;\epsilon,\lambda)\right\| \leq C_\nu \epsilon^{\frac{2}{3}},
\end{align*}
for $\xi \in I_{\lr}$. Hence, following the proof of~\cite[Proposition~A.3]{BDR2} verbatim, one establishes estimate~\eqref{projest2} for $j = s$. Similarly, by considering the negatively weighted eigenvalue problem (with weight $-\vartheta_\nu$) one obtains estimate~\eqref{projest2} for $j = u$. Finally, estimate~\eqref{projest2} for $j = c$ follows readily from the ones for $j = s$ and $j = u$ after applying the triangle inequality.

All that remains is to establish~\eqref{approxid}. First, we observe that the spectral projection of 
\begin{align*} \begin{pmatrix} A_{\lr}(y) & B_0 \\ 0_{1 \times 2} & 0 \end{pmatrix} \end{align*}
onto its center eigenspace is given by
\begin{align*}
\mathcal Q^{\cc}(y) = \begin{pmatrix} 0 & 0 & \left(f'(u_{\lr}(y))\right)^{-1} \\ 0 & 0 & 0 \\ 0 & 0 & 1\end{pmatrix}.
\end{align*}
for $y \in [\nu,L_\lr + 1]$. Hence, by estimate~\eqref{matrixapp}, we have
\begin{align} \label{specprojest}
\left\|\mathcal P^{\cc}_{\epsilon,\lambda}(\xi) - \mathcal Q^{\cc}(\epsilon \xi)\right\| \leq C_\nu \left(\epsilon^{\frac{2}{3}} + |\lambda|\right),
\end{align}
for $\xi \in I_{\lr}$. Next, we compute
\begin{align} \label{approxid0}
\begin{split}
&\T_{\lr,\epsilon,\lambda,\nu}^{\cc}(\xi,z) = H_{\epsilon,\lambda,\nu}(\xi) \begin{pmatrix} 0_{2 \times 2} & 0_{2 \times 1} \\ 0_{1 \times 2} &\T_{\su,\epsilon,\lambda,\nu}(\xi,z) \end{pmatrix} H_{\epsilon,\lambda,\nu}(z)^{-1} \\
&= \begin{pmatrix} \epsilon U_{\epsilon,\lambda,\nu}(\xi) \T_{\su,\epsilon,\lambda,\nu}(\xi,z) S_{\epsilon,\lambda,\nu}(z) & U_{\epsilon,\lambda,\nu}(\xi) \T_{\su,\epsilon,\lambda,\nu}(\xi,z) \left(1-\epsilon S_{\epsilon,\lambda,\nu}(z)U_{\epsilon,\lambda,\nu}(z)\right)\\
\epsilon \T_{\su,\epsilon,\lambda,\nu}(\xi,z) S_{\epsilon,\lambda,\nu}(z) & \T_{\su,\epsilon,\lambda,\nu}(\xi,z)\left(1-\epsilon S_{\epsilon,\lambda,\nu}(z)U_{\epsilon,\lambda,\nu}(z)\right)\end{pmatrix}.
\end{split}
\end{align}
for $\xi, z \in I_{\lr}$. First, upon recalling that $S_{\epsilon,\lambda,\nu}$, $U_{\epsilon,\lambda,\nu}$ and $ \T_{\su,\epsilon,\lambda,\nu}$ are bounded on $I_\lr$, on $I_\lr$ and on $I_\lr \times I_\lr$, respectively, by $\epsilon$- and $\lambda$-independent constants, the first estimate in~\eqref{approxid} readily follows from~\eqref{approxid0}. Second, setting $z = \xi$ in~\eqref{approxid0}, noting $\T_{\lr,\epsilon,\lambda,\nu}^{\cc}(\xi,\xi) = P_{\lr,\epsilon,\lambda,\nu}^{\cc}(\xi)$ and applying estimates~\eqref{projest2} and~\eqref{specprojest}, we arrive at
\begin{align} \label{Uapprox}
\left\| U_{\epsilon,\lambda,\nu}(\xi) - \begin{pmatrix} \left(f'(u_{\lr}(\epsilon \xi))\right)^{-1} \\ 0 \end{pmatrix}\right\| \leq C_\nu \left(\epsilon^{\frac{2}{3}} + |\lambda|\right),
\end{align}
for $\xi \in I_{\lr}$, where we recall that $S_{\epsilon,\lambda,\nu}$ and $U_{\epsilon,\lambda,\nu}$ are bounded on $I_\lr$ by $\epsilon$- and $\lambda$-independent constants. So, with the aid of~\eqref{period},~\eqref{evolsform} and~\eqref{Uapprox} we approximate
\begin{align} \label{approxid1}
\left|\T_{\su,\epsilon,\lambda,\nu}\left(\frac{\nu}{\epsilon},L_{\lr,\eps} + \frac{\log(\epsilon)}{\nu}\right) - \exp\left(-\frac{\gamma}{c}\left(L_\lr - \nu\right) + \int_{\nu}^{L_\lr} \frac{1}{cf'(u_{\lr}(y))} \de y\right)\right| \leq C_\nu\left(|\lambda| + \epsilon^{\frac13}\right),
\end{align}
provided $0 < \epsilon, |\lambda| \ll \nu \ll 1$. Next, we use~\eqref{eq:reduced} and recall $u_{\lr}(L_\lr) = \bar{u}_2$ and $u_{\lr}(0) = u_1$ to compute
\begin{align} \label{approxid2}
\begin{split}
-\frac{\gamma}{c}\left(L_\lr - \nu\right) + \int_{\nu}^{L_\lr} \frac{1}{cf'(u_{\lr}(y))} \de y &= 
-\int_{\nu}^{L_\lr} \frac{\gamma f'(u_{\lr}(y)) u_{\lr}'(y)}{\gamma f(u_{\lr}(y)) + a - u_{\lr}(y)} \de y + \int_{\nu}^{L_\lr} \frac{u_{\lr}'(y)}{\gamma f(u_{\lr}(y)) + a - u_{\lr}(y)} \de y \\
&= -\int_{u_{\lr}(\nu)}^{u_{\lr}(L_\lr)} \frac{\gamma f'(u) - 1}{\gamma f(u) + a - u} \de u = \log\left(\frac{u_{\lr}(\nu) - \gamma f(u_\lr(\nu)) - a}{\bar{u}_2 - \gamma f(\bar{u}_2) - a}\right),
\end{split}
\end{align}
provided $0 < \nu \ll 1$, where we note that $u_1 - \gamma f(u_1) - a, u_2 - \gamma f(u_2) - a \neq 0$ as $0<\gamma<\gamma_*(a)$. Finally, combining~\eqref{approxid0},~\eqref{approxid1}, and~\eqref{approxid2}, we arrive at the second inequality in~\eqref{approxid}, where we recall that $S_{\epsilon,\lambda,\nu}$, $U_{\epsilon,\lambda,\nu}$ and $\T_{\su,\epsilon,\lambda,\nu}$ are bounded on $I_\lr$, on $I_\lr$ and on $I_\lr \times I_\lr$, respectively, by $\epsilon$- and $\lambda$-independent constants.
\end{proof}

\subsection{Solving the eigenvalue problem away from the fold points}\label{sec:bvp_between_folds} 

Here, we construct a solution to the eigenvalue problem~\eqref{eigenvalueproblem} on the intervals 
\begin{align*}
    \mathcal{I}_\mathrm{r}&=\left[ \xi_{\mathrm{uf},\eps,\nu}^\mathrm{out}, \xi_{\mathrm{lf},\eps,\nu}^\mathrm{in}\right]=\left[L_{\lr,\eps}+ \frac{\nu}{\epsilon},L_{\lr,\eps}+L_{\rr,\eps}+\frac{1}{\nu}\right] = \left[L_{\lr,\eps}+ \frac{\nu}{\epsilon},L_\eps+\frac{1}{\nu}\right],\\
\mathcal{I}_\mathrm{l}&=\left[\xi_{\mathrm{lf},\eps,\nu}^{\mathrm{out},0},\xi_{\mathrm{uf},\eps,\nu}^\mathrm{in}\right]=\left[ \frac{\nu}{\epsilon},L_{\lr,\eps}+\frac{1}{\nu}\right]
\end{align*}
away from the fold points with the aid of Lin's method. We employ the exponential trichotomies, established in Propositions~\ref{prop:varred2} and~\ref{prop:slow}, to represent  solutions to~\eqref{eigenvalueproblem} along the front or back and along the right or left branch of the critical manifold, respectively. The result concerning the eigenvalue problem along the front and the right branch of the critical manifold reads as follows. 

\begin{proposition}\label{prop:bvp_right}
Provided $0 < \epsilon, |\lambda| \ll \nu \ll 1$, there exists for each 
$$\gamma_{\rr} \in P_{\rr,\epsilon,\lambda,\nu}^{\su}\left( \xi_{\mathrm{uf},\eps,\nu}^\mathrm{out}\right)[\C^3], \quad \text{ and } \quad \alpha_\f,\beta_\f \in \C,$$
a unique solution $\psi \colon \mathcal{I}_\mathrm{r} \to \C^3$ to the eigenvalue problem~\eqref{eigenvalueproblem} subject to the boundary conditions
\begin{align}
P_{\rr,\epsilon,\lambda,\nu}^{\su}\left(\xi_{\mathrm{uf},\eps,\nu}^\mathrm{out}\right)\psi\left(\xi_{\mathrm{uf},\eps,\nu}^\mathrm{out}\right) &= \gamma_{\rr}, \label{initcond2}\\ 
\widetilde{P}_{\f,\nu}^{\cc}\left(\tfrac{1}{\nu}\right)\psi\left(\xi_{\mathrm{lf},\eps,\nu}^\mathrm{in}\right) = \beta_\f \Psi_{\f,\nu}\left(\tfrac{1}{\nu}\right), \quad & \quad \widetilde{P}_{\f,\nu}^{\uu}\left(\tfrac{1}{\nu} \right) \psi\left(\xi_{\mathrm{lf},\eps,\nu}^\mathrm{in}\right) = \alpha_{\f} \Phi_{\f}\left(\tfrac{1}{\nu}\right). \label{initcond1}
\end{align}
Moreover, there exist $\epsilon$- and $\lambda$-independent constants $C_\nu,\vartheta_\nu > 0$ such that the solution $\psi$ enjoys the estimates
\begin{align}
    \label{matchingbetween}
\begin{split}
\left\|\widetilde P_{\f,\nu}^{\su}\left(\tfrac{1}{\nu}\right) \psi\left(\xi_{\mathrm{lf},\eps,\nu}^\mathrm{in}\right)\right\| &\leq C_\nu \left(\left(\epsilon^{\frac{2}{3}}  + |\lambda|\right)\left(|\alpha_{\f}| + |\beta_{\f}|\right) + \re^{-\vartheta_\nu/\eps}\|\gamma_{\rr}\|\right), \\
\left\|P_{\rr,\epsilon,\lambda,\nu}^{\uu}\left(\xi_{\mathrm{uf},\eps,\nu}^\mathrm{out}\right)\psi\left(\xi_{\mathrm{uf},\eps,\nu}^\mathrm{out}\right)\right\| &\leq C_\nu \re^{-\vartheta_\nu/\eps} \left(|\alpha_{\f}| + |\beta_{\f}| + \|\gamma_{\rr}\|\right),\\
\left\|P_{\rr,\epsilon,\lambda,\nu}^{\cc}\left(\xi_{\mathrm{uf},\eps,\nu}^\mathrm{out}\right)\psi\left(\xi_{\mathrm{uf},\eps,\nu}^\mathrm{out}\right)\right\| &\leq C_\nu \left(|\beta_{\f}| + \left(\epsilon^{\frac{2}{3}}  + |\lambda|\right) |\alpha_{\f}| + \re^{-\vartheta_\nu/\eps}\|\gamma_{\rr}\|\right),
\end{split}
\end{align}
and
\begin{align} \label{matchingbetweenrefined}
\begin{split}
&\left|\left\langle \mathbf{e}_3, P_{\rr,\epsilon,\lambda,\nu}^{\cc}\left(\xi_{\mathrm{uf},\eps,\nu}^\mathrm{out}\right)\psi\left(\xi_{\mathrm{uf},\eps,\nu}^\mathrm{out}\right)\right\rangle - \frac{u_{\rr}(\nu) - \gamma f(u_{\rr}(\nu)) - a}{u_2 - \gamma f(u_2) - a}\left(\beta_{\f} + \frac{\epsilon}{c} \left(u_{\f}(\tfrac{1}{\nu}) - u_2\right)\alpha_{\f}\right) \right|\\ 
&\qquad \leq C_\nu \left(\left(\epsilon^{\frac{1}{3}} + |\lambda|\right)\left(|\beta_{\f}| + \epsilon |\alpha_{\f}|\right) + \re^{-\vartheta_\nu/\eps}\|\gamma_{\rr}\|\right),
\end{split}
\end{align}
where $\mathbf{e}_3 = (0,0,1)^\top$ is the third standard basis vector.
\end{proposition}
\begin{proof}
In this proof $C_\nu \geq 1$ denotes an $\epsilon$-, $\lambda$- and $\xi$-independent constant, which will be taken larger if necessary. Moreover, we introduce the short-cut notation $l_{\epsilon,\nu} = \smash{\frac{\log(\epsilon)}{\nu}}$.

We wish to express the solution $\psi$ on $\smash{[L_\eps+l_{\epsilon,\nu},L_\eps+\frac{1}{\nu}]=[L_\eps+l_{\epsilon,\nu},\xi_{\mathrm{lf},\eps,\nu}^\mathrm{in}]}$ using the variation of constants formula by regarding the eigenvalue problem~\eqref{eigenvalueproblem} as a perturbation of the reduced variational problem~\eqref{varred2}. Thus, we determine the perturbation matrix
\begin{align*}
A(\xi;\epsilon,\lambda) - \widetilde A_{\f}(\xi-L_\eps)= \epsilon \mathcal B_0 + B_{\f}(\xi;\epsilon,\lambda),
\end{align*}
where we denote
\begin{align*}
\mathcal B_0 = \frac{1}{c} \begin{pmatrix} 0 & 0 & 0 \\ 0 & 0 & 0 \\ -1 & 0 & \gamma\end{pmatrix}, \qquad B_{\f}(\xi;\epsilon,\lambda) = \begin{pmatrix} -\frac{\lambda}{c} & 0 & 0 \\ \lambda - f'(u_\epsilon(\xi)) + f'(u_{\f}(\xi-L_\eps)) & -\frac{\lambda}{c} & 0 \\ 0 & 0 & 0\end{pmatrix}.
\end{align*}
By Proposition~\ref{prop:varred2} any solution $\psi$ to~\eqref{eigenvalueproblem} with initial condition~\eqref{initcond1} must satisfy for $\xi \in \smash{[L_\eps+l_{\epsilon,\nu},\xi_{\mathrm{lf},\eps,\nu}^\mathrm{in}]}$ the variation of constants formula
\begin{align} \label{duhamel1}
\begin{split}
\psi(\xi) &= \alpha_{\f} \Phi_{\f}\left(\xi-L_\eps\right) + \beta_{\f} \Psi_{\f,\nu}(\xi-L_\eps) + \widetilde \T_{\f,\nu}^{\su}(\xi-L_\eps,l_{\epsilon,\nu}) \gamma_{\f}\\ 
&\qquad + \int_{\xi_{\mathrm{lf},\eps,\nu}^\mathrm{in}}^\xi \widetilde \T_{\f,\nu}^{\uu}(\xi-L_\eps,y-L_\eps) \left(\epsilon \mathcal B_0 + B_{\f}(y;\epsilon,\lambda)\right) \psi(y) \de y\\ 
&\qquad + \int_{L_\eps+l_{\epsilon,\nu}}^\xi \widetilde \T_{\f,\nu}^{\su}(\xi-L_\eps,y-L_\eps) \left(\epsilon \mathcal B_0 + B_{\f}(y;\epsilon,\lambda)\right) \psi(y) \de y\\
&\qquad + \epsilon \int_{\xi_{\mathrm{lf},\eps,\nu}^\mathrm{in}}^\xi \widetilde \T_{\f,\nu}^{\cc}(\xi-L_\eps,y-L_\eps) \mathcal B_0 \psi(y) \de y, 
\end{split}
\end{align}
for some $\gamma_{\f} \in \widetilde P_{\f,\nu}^{\su}(l_{\epsilon,\nu})[\C^3]$, where we use that $\widetilde{P}_{\f,\nu}^{\cc}(y-L_\eps) B_{\f}(y;\epsilon,\lambda) = 0$ for any $y \in \smash{[L_\eps+l_{\epsilon,\nu},\xi_{\mathrm{lf},\eps,\nu}^\mathrm{in}]}$. We note that Proposition~\ref{prop:pointwise} yields
\begin{align} \|B_{\f}(\xi;\epsilon,\lambda)\| \leq C_\nu \left(\epsilon^{\frac{2}{3}} + |\lambda|\right), \label{pertest} \end{align}
for $\xi \in \smash{[L_\eps+l_{\epsilon,\nu},\xi_{\mathrm{lf},\eps,\nu}^\mathrm{in}]}$. So, bounding the right-hand side of~\eqref{duhamel1} with the aid of estimate~\eqref{pertest} and Proposition~\ref{prop:varred2}, we find that the solution $\psi$ is linear in $\alpha_{\f},\beta_{\f}$ and $\gamma_{\f}$ and, provided $0 < \epsilon,|\lambda| \ll \nu \ll 1$, it enjoys the estimates
\begin{align} \sup_{\xi \in \left[L_\eps+l_{\epsilon,\nu},\xi_{\mathrm{lf},\eps,\nu}^\mathrm{in}\right]} \|\psi(\xi)\| \leq C_\nu\left(|\alpha_{\f}| + |\beta_{\f}| + \|\gamma_{\f}\|\right),\label{psiest1}\end{align}
and
\begin{align} \sup_{\xi \in \left[L_\eps+l_{\epsilon,\nu},\xi_{\mathrm{lf},\eps,\nu}^\mathrm{in}\right]} \|\psi(\xi) - \alpha_{\f} \Phi_{\f}(\xi-L_\eps)\| \leq C_\nu\left(|\beta_{\f}| + \|\gamma_{\f}\| + \left(\epsilon^{\frac{2}{3}} + |\lambda|\right)|\alpha_\f|\right).\label{psiest2}\end{align}

Next, we express $\psi$ on $\smash{[\xi_{\mathrm{uf},\eps,\nu}^\mathrm{out},L_\eps+l_{\epsilon,\nu}]}$ with the aid of the exponential trichotomy established in Proposition~\ref{prop:slow}. We find that any solution $\psi$ to~\eqref{eigenvalueproblem} with initial condition~\eqref{initcond2} must satisfy
\begin{align} \label{duhamel2}
\psi(\xi) &= \T_{\rr,\epsilon,\lambda,\nu}^{\uu}\left(\xi,L_\eps+l_{\epsilon,\nu}\right) \alpha_{\rr} +  \T_{\rr,\epsilon,\lambda,\nu}^{\cc}\left(\xi,\xi_{\mathrm{uf},\eps,\nu}^\mathrm{out}\right) \beta_{\rr} + \T_{\rr,\epsilon,\lambda,\nu}^{\su}\left(\xi,\xi_{\mathrm{uf},\eps,\nu}^\mathrm{out}\right) \gamma_{\rr}, \quad \xi \in \left[\xi_{\mathrm{uf},\eps,\nu}^\mathrm{out},L_\eps+l_{\epsilon,\nu}\right],
\end{align}
for some $\alpha_{\rr} \in P_{\rr,\epsilon,\lambda,\nu}^{\uu}(L_\eps+l_{\epsilon,\nu})[\C^3]$ and $\beta_{\rr} \in P_{\rr,\epsilon,\lambda,\nu}^{\cc}(L_\eps+l_{\epsilon,\nu})[\C^3]$. 

The next step is to equate the right-hand sides of~\eqref{duhamel1} and~\eqref{duhamel2} at the matching point $\xi = L_\eps+l_{\epsilon,\nu}$, which will yield unique expressions for $\alpha_{\rr}$, $\beta_{\rr}$, and $\gamma_{\f}$. We derive matching conditions by applying the complementary rank-one projections $P_{\rr,\epsilon,\lambda,\nu}^{\uu}(L_\eps+l_{\epsilon,\nu})$, $P_{\rr,\epsilon,\lambda,\nu}^{\su}(L_\eps+l_{\epsilon,\nu})$ and $P_{\rr,\epsilon,\lambda,\nu}^{\cc}(L_\eps+l_{\epsilon,\nu})$. However, we first show that these projections are close to the projections of the exponential trichotomy of the reduced variational problem~\eqref{varred2}, established in Proposition~\ref{prop:varred2}, at $\xi = L_\eps+l_{\epsilon,\nu}$. Propositions~\ref{prop:frontback} and~\ref{prop:pointwise} yield
\begin{align*} \left\|A(L_\eps+l_{\epsilon,\nu};\epsilon,\lambda) - \widetilde{A}_{\f,-\infty}\right\| \leq C_\nu \left(\epsilon^{\frac{2}{3}} + |\lambda|\right),\end{align*}
provided $0 < \epsilon,|\lambda| \ll \nu \ll 1$. Clearly, the same bound holds for the associated spectral projections
\begin{align*} \left\|\mathcal{P}_{\epsilon,\lambda}^j(L_\eps+l_{\epsilon,\nu}) - \widetilde{\mathcal P}_{\f,-\infty}^j\right\| \leq C_\nu \left(\epsilon^{\frac{2}{3}} + |\lambda|\right),\end{align*}
for $j = \su,\uu,\cc$. Hence, combining the latter with estimates~\eqref{projest} and~\eqref{projest2} we arrive at
\begin{align} \label{specprojest2}
\begin{split}
\left\|P_{\rr,\epsilon,\lambda,\nu}^j(L_\eps+l_{\epsilon,\nu}) - \widetilde{P}_{\f,\nu}^j(l_{\epsilon,\nu})\right\| 
\leq C_\nu \left(\epsilon^{\frac{2}{3}} + |\lambda|\right),
\end{split} 
\end{align}
for $j = \su,\uu,\cc$, provided $0 < \epsilon,|\lambda| \ll \nu \ll 1$.

We now apply the projection $P_{\rr,\epsilon,\lambda,\nu}^{\su}(L_\eps+l_{\epsilon,\nu})$ to~\eqref{duhamel1} and~\eqref{duhamel2}, equate the right-hand sides and use Propositions~\ref{prop:varred2} and~\ref{prop:slow} and estimates~\eqref{period},~\eqref{pertest},~\eqref{psiest1} and~\eqref{specprojest2} to obtain the matching condition
\begin{align} \label{match1}
\gamma_{\f} = \F_1(\alpha_{\f},\beta_{\f},\gamma_{\f},\gamma_{\rr}),
\end{align}
where the linear map $\F_1$ satisfies
\begin{align*}
\|\F_1(\alpha_{\f},\beta_{\f},\gamma_{\f},\gamma_{\rr})\| \leq C_{\nu} \left(\left(\epsilon^{\frac{2}{3}} + |\lambda|\right) \left(|\alpha_{\f}| + |\beta_{\f}| + \|\gamma_{\f}\|\right) + \re^{-\vartheta_\nu/\eps} \|\gamma_{\rr}\|\right),
\end{align*}
for some $\epsilon$- and $\lambda$-independent constant $\vartheta_\nu > 0$, provided $0 < \epsilon,|\lambda| \ll \nu \ll 1$. Similarly, applying the projection $P_{\rr,\epsilon,\lambda,\nu}^{\uu}(L_\eps+l_{\epsilon,\nu})$, we find
\begin{align} \label{match2}
\alpha_{\rr} = \F_2(\alpha_{\f},\beta_{\f},\gamma_{\f}),
\end{align}
where the linear map $\F_2$ satisfies
\begin{align*}
\|\F_2(\alpha_{\f},\beta_{\f},\gamma_{\f})\| \leq C_{\nu} \left(\epsilon^{\frac{2}{3}} + |\lambda|\right) \left(|\alpha_{\f}| + |\beta_{\f}| + \|\gamma_{\f}\|\right),
\end{align*}
provided $0 < \epsilon,|\lambda| \ll \nu \ll 1$. Finally, applying the projection 
$P_{\rr,\epsilon,\lambda,\nu}^{\cc}(L_\eps+l_{\epsilon,\nu})$, we arrive at
\begin{align}\label{match3}
\beta_{\rr} = \F_3(\alpha_{\f},\beta_{\f},\gamma_{\f})
\end{align}
where the linear map $\F_3$ satisfies
\begin{align*}
\|\F_3(\alpha_{\f},\beta_{\f},\gamma_{\f})\| \leq C_{\nu} \left(|\beta_\f| + \left(\epsilon^{\frac{2}{3}} + |\lambda|\right) \left(|\alpha_{\f}| + \|\gamma_{\f}\|\right)\right),
\end{align*}
provided $0 < \epsilon,|\lambda| \ll \nu \ll 1$. 

We observe that the matching conditions~\eqref{match1},~\eqref{match2}, and~\eqref{match3} constitute a linear system in the variables $\alpha_\f,\beta_\f,\gamma_\f,\alpha_\rr,\beta_\rr$, and $\gamma_\rr$, which can be uniquely solved for $\beta_\rr,\gamma_\f$ and $\alpha_\rr$ by the above estimates on the linear maps $\F_1,\F_2$ and $\F_3$. We find
\begin{align} \label{matchingbetweenfinal}
\gamma_{\f} = \widetilde{\F}_1(\alpha_\f,\beta_\f,\gamma_\rr),\qquad \alpha_{\rr} = \widetilde{\F}_2(\alpha_\f,\beta_\f,\gamma_\rr), \qquad \beta_{\rr} = \widetilde{\F}_3(\alpha_\f,\beta_\rr,\gamma_\rr),
\end{align}
where the linear maps $\widetilde{\F}_i$, $i = 1,2,3$ enjoy the bounds
\begin{align*}
\left\|\widetilde{\F}_i(\alpha_\f,\beta_\f,\gamma_\rr)\right\| &\leq C_{\nu} \left(\left(\epsilon^{\frac{2}{3}} + |\lambda|\right) \left(|\alpha_{\f}| + |\beta_{\f}|\right) + \re^{-\vartheta_\nu/\eps}\|\gamma_{\rr}\|\right), \qquad i = 1,2,\\
\left\|\widetilde{\F}_3(\alpha_\f,\beta_\f,\gamma_\rr)\right\| &\leq C_{\nu}\left(|\beta_{\f}| + \left(\epsilon^{\frac{2}{3}}  + |\lambda|\right) |\alpha_{\f}| + \re^{-\vartheta_\nu/\eps}\|\gamma_{\rr}\|\right),
\end{align*}
provided $0 < \epsilon,|\lambda| \ll \nu \ll 1$. The estimate~\eqref{matchingbetween} now follows readily by observing 
\begin{align*}
P_{\rr,\epsilon,\lambda,\nu}^{\uu}\left(\xi_{\mathrm{uf},\eps,\nu}^\mathrm{out}\right)\psi\left(\xi_{\mathrm{uf},\eps,\nu}^\mathrm{out}\right) &= \T_{\rr,\epsilon,\lambda,\nu}^{\uu}\left(\xi_{\mathrm{uf},\eps,\nu}^\mathrm{out},L_\eps+l_{\epsilon,\nu}\right)\alpha_{\rr},\\
P_{\rr,\epsilon,\lambda,\nu}^{\cc}\left(\xi_{\mathrm{uf},\eps,\nu}^\mathrm{out}\right)\psi\left(\xi_{\mathrm{uf},\eps,\nu}^\mathrm{out}\right) &= \T_{\rr,\epsilon,\lambda,\nu}^{\cc}\left(\xi_{\mathrm{uf},\eps,\nu}^\mathrm{out},L_\eps+l_{\epsilon,\nu}\right)\beta_{\rr},\\
\widetilde P_{\f,\nu}^{\su}\left(\tfrac{1}{\nu}\right) \psi\left(\xi_{\mathrm{lf},\eps,\nu}^\mathrm{in}\right) &= \widetilde \T_{\f,\nu}^{\su}\left(\tfrac{1}{\nu},l_{\epsilon,\nu}\right) 
\gamma_{\f} + \int_{L_\eps+l_{\epsilon,\nu}}^{\xi_{\mathrm{lf},\eps,\nu}^\mathrm{in}} \widetilde \T_{\f,\nu}^{\su}\left(\tfrac{1}{\nu},y-L_\eps\right) \left(\epsilon \mathcal B_0 + B_{\f}(y;\epsilon,\lambda)\right) \psi(y) \de y,
\end{align*}
and applying Propositions~\ref{prop:varred2} and~\ref{prop:slow} and estimates~\eqref{pertest} and~\eqref{psiest1}. 

For the refined estimate~\eqref{matchingbetweenrefined}, we first note that
\begin{align*}
\Pi_3 \Phi_\f(\xi) = 0, \qquad \Pi_3 \Psi_{\f,\nu}(\xi) = \begin{pmatrix} 0_{2 \times 1} \\ 1\end{pmatrix},\qquad 
\Pi_3 \widetilde{\T}_{\f,\nu}^{\cc}(\xi,y) = \Pi_3, \qquad \Pi_3 \widetilde{P}_{\f,\nu}^{\su}(\xi) = 0,  \qquad \Pi_3 \widetilde{P}_{\f,\nu}^{\uu}(\xi) = 0,
\end{align*}
holds for $\xi,y \in [l_{\epsilon,\nu},\frac{1}{\nu}]$, where 
$\Pi_3 \in \R^{3 \times 3}$ is the orthogonal projection on the third standard basis vector $\mathbf{e}_3$, cf.~Proposition~\ref{prop:varred2}. So, applying the complementary projections $\Pi_3$ and $I_3 - \Pi_3$ to~\eqref{duhamel1} at $\xi = L_\eps+l_{\epsilon,\nu}$, while using Proposition~\ref{prop:varred2} and estimates~\eqref{pertest},~\eqref{psiest1},~\eqref{psiest2} and~\eqref{matchingbetweenfinal}, we arrive at
\begin{align} \label{refinedest1}
\left\|(I_3 - \Pi_3)\psi(L_\eps+l_{\epsilon,\nu})\right\| \leq C_\nu\left(|\beta_\f| + \left(\epsilon^{\frac23}  + |\lambda|\right)|\alpha_f| +  \re^{-\vartheta_\nu/\eps}\|\gamma_{\rr}\|\right),
\end{align}
and
\begin{align}\label{refinedest2}
\begin{split}
&\left\|\Pi_3\psi(L_\eps+l_{\epsilon,\nu}) - \left(\beta_\f \begin{pmatrix} 0\\ 0 \\
1\end{pmatrix} - \alpha_{\f} \epsilon \Pi_3 \mathcal B_0 \int_{-\infty}^{\frac1\nu} \begin{pmatrix} u_{\f}'(\xi) \\ v_{\f}'(\xi)\\ 0\end{pmatrix} \de \xi\right)\right\|\\ 
&\qquad \leq C_\nu\left(\epsilon |\beta_\f| + \epsilon \left(\epsilon^{\frac23}  + |\lambda|\right)|\alpha_f| +  \re^{-\vartheta_\nu/\eps}\|\gamma_{\rr}\|\right),
\end{split}
\end{align}
provided $0 < \epsilon,|\lambda| \ll \nu \ll 1$. We compute
\begin{align*}
-\Pi_3 \mathcal B_0 \int_{-\infty}^{\frac1\nu} \begin{pmatrix} u_{\f}'(\xi) \\ v_{\f}'(\xi)\\ 0\end{pmatrix} \de \xi = \frac{\epsilon}{c}\left(u_{\f}(\tfrac{1}{\nu}) - u_2\right) \begin{pmatrix} 0\\ 0 \\ 1\end{pmatrix}\end{align*}
So, combining estimates~\eqref{refinedest1} and~\eqref{refinedest2} with  Proposition~\ref{prop:slow} yields
\begin{align*}
\left\langle \mathbf{e}_3, P_{\rr,\epsilon,\lambda,\nu}^{\cc}\left(\xi_{\mathrm{uf},\eps,\nu}^\mathrm{out}\right) \psi\left(\xi_{\mathrm{uf},\eps,\nu}^\mathrm{out}\right)\right\rangle  &= \left\langle \mathbf{e}_3, \T_{\rr,\epsilon,\lambda,\nu}^{\cc}\left(\xi_{\mathrm{uf},\eps,\nu}^\mathrm{out},L_\eps+l_{\epsilon,\nu}\right) \Pi_3 \psi(L_\eps+l_{\epsilon,\nu}) \right\rangle\\
&\qquad + \left\langle \mathbf{e}_3, \T_{\rr,\epsilon,\lambda,\nu}^{\cc}\left(\xi_{\mathrm{uf},\eps,\nu}^\mathrm{out},L_\eps+l_{\epsilon,\nu}\right) (I_3 - \Pi_3) \psi(L_\eps+l_{\epsilon,\nu}) \right\rangle\\
&= \frac{u_{\rr}(\nu) - \gamma f(u_{\rr}(\nu)) - a}{u_2 - \gamma f(u_2) - a}\left(\beta_{\f} +  \frac{\epsilon}{c}\left(u_{\f}(\tfrac{1}{\nu}) - u_2\right) \alpha_f\right) + \widetilde{\F}_{31}(\alpha_\f,\beta_\f,\gamma_\rr),
\end{align*}
where the linear map $\widetilde{\F}_{31}$ satisfies
\begin{align*}
|\widetilde{\F}_{31}(\alpha_{\f},\beta_{\f},\gamma_{\rr})| &\leq C_{\nu} \left(\left(\epsilon^{\frac{1}{3}} + |\lambda|\right)\left(|\beta_f| + \epsilon |\alpha_{\f}|\right) + \re^{-\vartheta_\nu/\eps}\|\gamma_{\rr}\|\right),\end{align*}
provided $0 < \epsilon,|\lambda| \ll \nu \ll 1$. This proves the final estimate~\eqref{matchingbetweenrefined}.
\end{proof}

Analogously, one obtains the following result concerning the eigenvalue problem along the back and the left branch of the critical manifold away from the fold points.

\begin{proposition}\label{prop:bvp_left}
Provided $0 < \epsilon, |\lambda| \ll \nu \ll 1$, there exists for each 
$$\gamma_{\lr} \in P_{\lr,\epsilon,\lambda,\nu}^{\su}\left(\xi_{\mathrm{lf},\eps,\nu}^{\mathrm{out},0}\right)[\C^3], \quad \text{ and } \quad \alpha_\bb,\beta_\bb \in \C,$$
a unique solution $\psi \colon\mathcal{I}_\lr \to \C^3$ to the eigenvalue problem~\eqref{eigenvalueproblem} subject to the boundary conditions
\begin{align*}
P_{\lr,\epsilon,\lambda,\nu}^{\su}\left(\xi_{\mathrm{lf},\eps,\nu}^{\mathrm{out},0}\right)\psi\left(\xi_{\mathrm{lf},\eps,\nu}^{\mathrm{out},0}\right) &= \gamma_{\lr}, \quad \ \widetilde{P}_{\bb,\nu}^{\cc}\left(\tfrac{1}{\nu}\right)\psi\left(\xi_{\mathrm{uf},\eps,\nu}^\mathrm{in}\right) = \beta_\bb \Psi_{\bb,\nu}\left(\tfrac{1}{\nu}\right), \quad \  \widetilde P_{\bb,\nu}^{\uu}\left(\tfrac{1}{\nu}\right) \psi\left(\xi_{\mathrm{uf},\eps,\nu}^\mathrm{in}\right) = \alpha_{\bb} \Phi_{\bb}\left(\tfrac{1}{\nu}\right).
\end{align*}
Moreover, there exist $\epsilon$- and $\lambda$-independent constants $C_\nu,\vartheta_\nu > 0$ such that the solution $\psi$ enjoys the estimates
\begin{align*}
\begin{split}
\left\|\widetilde P_{\bb,\nu}^{\su}\left(\tfrac{1}{\nu}\right) \psi\left(\xi_{\mathrm{uf},\eps,\nu}^\mathrm{in}\right)\right\| &\leq C_\nu \left(\left(\epsilon^{\frac{2}{3}}  + |\lambda|\right)\left(|\alpha_{\bb}| + |\beta_{\bb}|\right) + \re^{-\vartheta_\nu/\eps}\|\gamma_{\lr}\|\right), \\
\left\|P_{\lr,\epsilon,\lambda,\nu}^{\uu}\left(\xi_{\mathrm{lf},\eps,\nu}^{\mathrm{out},0}\right)\psi\left(\xi_{\mathrm{lf},\eps,\nu}^{\mathrm{out},0}\right)\right\| &\leq C_\nu \re^{-\vartheta_\nu/\eps} \left(|\alpha_{\bb}| + |\beta_{\bb}| + \|\gamma_{\lr}\|\right),\\
\left\|P_{\lr,\epsilon,\lambda,\nu}^{\cc}\left(\xi_{\mathrm{lf},\eps,\nu}^{\mathrm{out},0}\right)\psi\left(\xi_{\mathrm{lf},\eps,\nu}^{\mathrm{out},0}\right)\right\| &\leq C_\nu \left(|\beta_{\bb}| + \left(\epsilon^{\frac{2}{3}}  + |\lambda|\right) |\alpha_{\bb}| + \re^{-\vartheta_\nu/\eps}\|\gamma_{\lr}\|\right),
\end{split}
\end{align*}
and
\begin{align*}
\begin{split}
&\left|\left\langle \mathbf{e}_3, P_{\lr,\epsilon,\lambda,\nu}^{\cc}\left(\xi_{\mathrm{lf},\eps,\nu}^{\mathrm{out},0}\right)\psi\left(\xi_{\mathrm{lf},\eps,\nu}^{\mathrm{out},0}\right)\right\rangle - \frac{u_{\lr}(\nu) - \gamma f(u_{\lr}(\nu)) - a}{\bar{u}_2 - \gamma f(\bar{u}_2) - a}\left(\beta_{\bb} + \frac{\epsilon}{c} \left(u_{\bb}(\tfrac{1}{\nu}) - \bar{u}_2\right)\alpha_{\bb}\right) \right|\\ 
&\qquad \leq C_\nu \left(\left(\epsilon^{\frac{1}{3}} + |\lambda|\right)\left(|\beta_{\bb}| + \epsilon |\alpha_{\bb}|\right) + \re^{-\vartheta_\nu/\eps}\|\gamma_{\lr}\|\right),
\end{split}
\end{align*}
where $\mathbf{e}_3 = (0,0,1)^\top$ is the third standard basis vector.
\end{proposition}

\subsection{Solving the eigenvalue problem near the fold points}\label{sec:bvp_near_folds}

Propositions~\ref{prop:bvp_right} and~\ref{prop:bvp_left} provide estimates on solutions to the eigenvalue problem~\eqref{eigenvalueproblem} between the fold points, on the intervals $\mathcal{I}_\mathrm{r}=\smash{[\xi_{\mathrm{uf},\eps,\nu}^\mathrm{out},\xi_{\mathrm{lf},\eps,\nu}^\mathrm{in}]}$ and $\mathcal{I}_\mathrm{l}=\smash{[\xi_{\mathrm{lf},\eps,\nu}^{\mathrm{out},0}, \xi_{\mathrm{uf},\eps,\nu}^\mathrm{in}]}$. It remains to obtain estimates on solutions near the fold points, on the intervals $\mathcal{I}_\mathrm{uf}=\smash{[\xi_{\mathrm{uf},\eps,\nu}^\mathrm{in},\xi_{\mathrm{uf},\eps,\nu}^\mathrm{out}]}$ and $\mathcal{I}_\mathrm{lf}=\smash{[\xi_{\mathrm{lf},\eps,\nu}^\mathrm{in},\xi_{\mathrm{lf},\eps,\nu}^{\mathrm{out},L}]}$.  This is a delicate problem due to the loss of normal hyperbolicity at each of the fold points; in particular, it is not possible to establish suitable exponential trichotomies on the intervals $\mathcal{I}_\mathrm{uf}, \mathcal{I}_\mathrm{lf}$ due to the lack of an $(\eps,\lambda)$-independent spectral gap for the center-unstable spatial eigenvalues of~\eqref{eigenvalueproblem}. Blow-up methods are needed in order to track the solution on these intervals -- this procedure is carried out in detail in~\S\ref{sec:folds}, and here we state the main results. In preparation, we define complex-valued functions $\Upsilon_\mathrm{lf}, \Upsilon_\mathrm{uf}$ by
\begin{align*}
     \Upsilon_\mathrm{lf}(z)\coloneqq \frac{z^2}{\theta_\mathrm{lf}c^3\mathrm{Ai}'(-\Omega_0)^2 } \int_{-\Omega_0}^{\infty} \re^{\frac{z^2}{\theta_\mathrm{lf}c^3 }\left(s+\Omega_0\right) }\left(s\mathrm{Ai}(s)^2-\mathrm{Ai}'(s)^2\right)\mathrm{d}s\\
     \Upsilon_\mathrm{uf}(z)\coloneqq \frac{z^2}{\theta_\mathrm{uf}c^3\mathrm{Ai}'(-\Omega_0)^2 } \int_{-\Omega_0}^{\infty} \re^{\frac{z^2}{\theta_\mathrm{uf}c^3 }\left(s+\Omega_0\right) }\left(s\mathrm{Ai}(s)^2-\mathrm{Ai}'(s)^2\right)\mathrm{d}s
\end{align*}
where $-\Omega_0 < 0$ denotes the largest zero of the Airy function $\mathrm{Ai}(z)$ (see Appendix~\ref{app:airy}), and $\theta_\mathrm{lf}, \theta_\mathrm{uf}$ are given by
\begin{align} \label{eq:defthetas}
\begin{split}
    \theta_\mathrm{lf} &\coloneqq  -\frac{(a^2-a+1)^{1/6}(u_1-\gamma f(u_1)-a)^{1/3}}{c}>0\\
    \theta_\mathrm{uf} &\coloneqq  \frac{(a^2-a+1)^{1/6}(\bar{u}_1-\gamma f(\bar{u}_1)-a)^{1/3}}{c}>0
\end{split}
\end{align}
as in~\S\ref{sec:existence}. We first consider the eigenvalue problem near the lower fold on the interval $\mathcal{I}_\mathrm{lf}$. The estimates on solutions at the endpoints of the interval are given in terms of dichotomy projections $P^\mathrm{cu,s}_{\mathrm{lf},\eps,\lambda,\nu}(\xi)$ obtained from a convenient local change of coordinates on the interval $\mathcal{I}_\mathrm{lf}$ which places the system into a normal form for slow passage through a fold. We refer to~\S\ref{sec:folds} for details of these projections. Similarly, near the upper fold, on the interval $\mathcal{I}_\mathrm{uf}$, it is convenient to state estimates on solutions in terms of dichotomy projections $P^\mathrm{cu,s}_{\mathrm{uf},\eps,\lambda,\nu}(\xi)$. For the purposes of the matching procedure, we require estimates which relate the dichotomy projections $P^\mathrm{cu,s}_{\mathrm{uf/lf},\eps,\lambda,\nu}(\xi)$ to those of Propositions~\ref{prop:bvp_right} and~\ref{prop:bvp_left} at the endpoints of the respective intervals $\mathcal{I}_\mathrm{uf/lf}$. All in all, we have the following.

\begin{proposition}\label{prop:fold_bvp_oc_lower}
There exists a continuous function $\eta \colon [0,\infty) \to [0,\infty)$ satisfying $\eta(0)=0$ such that, provided $0 < \epsilon \ll \mu\ll  \nu \ll 1$, there exist $\eps$- and $\mu$-independent constants $C_\nu,\vartheta_\nu > 0$ and complementary projections $P^\mathrm{cu}_{\mathrm{lf},\eps,\lambda,\nu}(\xi), P^\mathrm{s}_{\mathrm{lf},\eps,\lambda,\nu}(\xi) \in \C^{3 \times 3}$ for $\xi \in \mathcal{I}_\mathrm{lf}=\smash{[\xi_{\mathrm{lf},\eps,\nu}^\mathrm{in},\xi_{\mathrm{lf},\eps,\nu}^{\mathrm{out},L}]}$ such that the following holds for $\lambda\in R_1(\mu)$. 
\begin{itemize}
\item[(i)] The projections obey the estimates
\begin{align}
\left\|P^i_{\mathrm{lf},\eps,\lambda,\nu}\left( \xi_{\mathrm{lf},\eps,\nu}^\mathrm{in} \right) - \widetilde P_{\f,\nu}^i\left( \tfrac{1}{\nu} \right)\right\| & \leq C_\nu \left(\epsilon^{\frac{2}{3}}+|\lambda|\right),\label{eq:foldproj_lower_est1}\\
\left\|P^i_{\mathrm{lf},\eps,\lambda,\nu}\left(\xi_{\mathrm{lf},\eps,\nu}^{\mathrm{out},L}\right) - P_{\lr,\epsilon,\lambda,\nu}^i\left(\xi_{\mathrm{lf},\eps,\nu}^{\mathrm{out},0}\right)\right\| & \leq  C_\nu \re^{-\vartheta_\nu/\eps}\label{eq:foldproj_lower_est2}\\
\label{eq:foldprop_thirdrow}
\left\|\mathbf{e}_3^\top\widetilde{P}_{\f,\nu}^{\cc}\left(\tfrac{1}{\nu}\right) P^\mathrm{s}_{\mathrm{lf},\eps,\lambda,\nu}\left(\xi_{\mathrm{lf},\eps,\nu}^\mathrm{in}\right)\right\| &\leq C_\nu \eps,
\end{align}
for $i = \mathrm{s},\mathrm{cu}$, where we denote
\begin{align*}
\widetilde P_{\f,\nu}^\mathrm{cu}(\xi)=\widetilde P_{\f,\nu}^{\cc}(\xi)+\widetilde{P}_{\f,\nu}^{\uu}(\xi), \qquad
    P_{\lr,\epsilon,\lambda,\nu}^{\mathrm{cu}}(\xi) =  P_{\lr,\epsilon,\lambda,\nu}^{\cc}(\xi)+P_{\lr,\epsilon,\lambda,\nu}^{\uu}(\xi),
\end{align*}
and where $\mathbf{e}_3 = (0,0,1)^\top$ is the third standard basis vector.
\item[(ii)] For each
\begin{align*}
\gamma_\mathrm{uf}\in P^\mathrm{s}_{\mathrm{lf},\eps,\lambda,\nu}\left(\xi_{\mathrm{lf},\eps,\nu}^\mathrm{in}  \right) \left[\mathbb{C}^3\right], \qquad \beta_\mathrm{lf}=\begin{pmatrix}\beta_{1,\mathrm{lf}}\\ \beta_{2,\mathrm{lf}} \\ \eps \beta_{3,\mathrm{lf}}  \end{pmatrix}\in P^\mathrm{cu}_{\mathrm{lf},\eps,\lambda,\nu}\left(\xi_{\mathrm{lf},\eps,\nu}^{\mathrm{out},L}  \right)\left[ \mathbb{C}^3\right],
\end{align*}
there exists a solution $\psi \colon \mathcal{I}_\mathrm{lf} \to \C^3$ to the eigenvalue problem~\eqref{eigenvalueproblem} subject to the boundary conditions
\begin{align*}
P^\mathrm{s}_{\mathrm{lf},\eps,\lambda,\nu}\left(\xi_{\mathrm{lf},\eps,\nu}^\mathrm{in}  \right) \psi\left(\xi_{\mathrm{lf},\eps,\nu}^\mathrm{in} \right)&= \gamma_\mathrm{lf},\qquad
P^\mathrm{cu}_{\mathrm{lf},\eps,\lambda,\nu}\left(\xi_{\mathrm{lf},\eps,\nu}^{\mathrm{out},L}  \right)\psi\left(\xi_{\mathrm{lf},\eps,\nu}^{\mathrm{out},L}  \right) = \beta_\mathrm{lf},
\end{align*}
satisfying the estimate
\begin{align*}
\left\|P^\mathrm{s}_{\mathrm{lf},\eps,\lambda,\nu}\left(\xi_{\mathrm{lf},\eps,\nu}^{\mathrm{out},L} \right)\psi\left(\xi_{\mathrm{lf},\eps,\nu}^{\mathrm{out},L} \right) \right\|  &\leq  C_\nu \re^{-\frac{\vartheta_\nu}{\eps}}\|\gamma_\mathrm{lf}\|.
\end{align*}
Furthermore, if $\Re(\lambda) \geq \eps^{1/5}$, then we have 
\begin{align} \label{tame_estimate_lower}
\left\|P^\mathrm{cu}_{\mathrm{lf},\eps,\lambda,\nu}\left(\xi_{\mathrm{lf},\eps,\nu}^\mathrm{in}  \right) \psi\left(\xi_{\mathrm{lf},\eps,\nu}^\mathrm{in}\right)\right\|  &\leq  C_\nu \exp\left(\frac{\nu \Re(\lambda)}{c\eps}\right) \|\beta_\mathrm{lf}\|, 
\end{align}
while if $|\Re(\lambda)| \leq \mu \eps^{1/6}$, then it holds
\begin{align*}
P^\mathrm{cu}_{\mathrm{lf},\eps,\lambda,\nu}\left(\xi_{\mathrm{lf},\eps,\nu}^\mathrm{in}  \right) \psi\left(\xi_{\mathrm{lf},\eps,\nu}^\mathrm{in}\right)&= \begin{pmatrix} U_\mathrm{lf}\\ V_\mathrm{lf}\\ W_\mathrm{lf} \end{pmatrix}(\beta_{\mathrm{lf}}), 
\end{align*}
where $U_\mathrm{lf}, V_\mathrm{lf}$ and $W_\mathrm{lf}$ satisfy
\begin{align*}
\left\|\begin{pmatrix} U_\mathrm{lf}\\ V_\mathrm{lf} \end{pmatrix}(\beta_{\mathrm{lf}})-\alpha^u_\mathrm{lf}(\beta_\mathrm{lf}; \eps,\lambda)\begin{pmatrix} u_\eps'\\ v_\eps'\end{pmatrix}\left(\xi_{\mathrm{lf},\eps,\nu}^\mathrm{in}\right)\right\| &\leq  C_\nu\left(\|\beta_\mathrm{lf}\|+\left|\lambda \log |\lambda|\right|\cdot|\beta_{3,\mathrm{lf}}|\right), \\
\left|W_\mathrm{lf}(\beta_{\mathrm{lf}})-\eps \beta_{3,\mathrm{lf}} \frac{w_\eps'\left(\xi_{\mathrm{lf},\eps,\nu}^\mathrm{in} \right)}{w_\eps'\left(\xi_{\mathrm{lf},\eps,\nu}^{\mathrm{out},L} \right)}\right|  &\leq  C_\nu \eps \|\beta_\mathrm{lf}\|+\left(C_\nu\eps \left|\lambda \log |\lambda|\right|+\eps \eta(\nu)\left|\Upsilon_\mathrm{lf}(\lambda \eps^{-1/6})\right|\right)|\beta_{3,\mathrm{lf}}|
\end{align*}
with
\begin{align*}
\left|\alpha^u_\mathrm{lf}(\beta_\mathrm{lf}; \eps,\lambda)
-\frac{\eps \beta_{3,\mathrm{lf}}}{w_\eps'\left(\xi_{\mathrm{lf},\eps,\nu}^{\mathrm{out},L}\right)} \left(1- \Upsilon_\mathrm{lf}(\lambda \eps^{-1/6})\right)\right|&\leq \eta(\nu)\left|\Upsilon_\mathrm{lf}(\lambda \eps^{-1/6})\right| |\beta_{3,\mathrm{lf}}|.
\end{align*}
\end{itemize}
\end{proposition}
The proof of Proposition~\ref{prop:fold_bvp_oc_lower} will be given in~\S\ref{sec:mainfoldest}. Analogously, we have the following proposition concerning solutions to the eigenvalue problem near the upper fold.

\begin{proposition}\label{prop:fold_bvp_oc_upper}
There exists a continuous function $\eta \colon [0,\infty) \to [0,\infty)$ satisfying $\eta(0)=0$ such that, provided $0 < \epsilon \ll \mu\ll  \nu \ll 1$, there exist $\eps$- and $\mu$-independent constants $C_\nu,\vartheta_\nu > 0$ and complementary projections $P^\mathrm{cu}_{\mathrm{uf},\eps,\lambda,\nu}(\xi), P^\mathrm{s}_{\mathrm{uf},\eps,\lambda,\nu}(\xi) \in \C^{3 \times 3}$ for $\xi \in \mathcal{I}_\mathrm{uf}=\smash{[\xi_{\mathrm{uf},\eps,\nu}^\mathrm{in},\xi_{\mathrm{uf},\eps,\nu}^\mathrm{out}]}$ such that the following holds for $\lambda\in R_1(\mu)$. 
\begin{itemize}
\item[(i)] The projections obey the estimates
\begin{align*}
\left\|P^i_{\mathrm{uf},\eps,\lambda,\nu}\left( \xi_{\mathrm{uf},\eps,\nu}^\mathrm{in} \right) - \widetilde P_{\bb,\nu}^i\left( \tfrac{1}{\nu} \right)\right\| & \leq C_\nu \left(\epsilon^{\frac{2}{3}}+|\lambda|\right),
\\
\left\|P^i_{\mathrm{uf},\eps,\lambda,\nu}\left(\xi_{\mathrm{uf},\eps,\nu}^\mathrm{out}\right) - P_{\rr,\epsilon,\lambda,\nu}^i\left( \xi_{\mathrm{uf},\eps,\nu}^\mathrm{out}\right)\right\| & \leq  C_\nu \re^{-\vartheta_\nu/\eps}
\\
\left\|\mathbf{e}_3^\top\widetilde{P}_{\bb,\nu}^{\cc}\left(\tfrac{1}{\nu}\right) P^\mathrm{s}_{\mathrm{uf},\eps,\lambda,\nu}\left(\xi_{\mathrm{uf},\eps,\nu}^\mathrm{in}\right)\right\| &\leq C_\nu \eps,
\end{align*}
for $i = \mathrm{s},\mathrm{cu}$, where we denote
\begin{align*}
\widetilde P_{\bb,\nu}^\mathrm{cu}(\xi)=\widetilde P_{\bb,\nu}^{\cc}(\xi)+\widetilde P_{\bb,\nu}^{\uu}(\xi), \qquad
    P_{\rr,\epsilon,\lambda,\nu}^{\mathrm{cu}}(\xi) =  P_{\rr,\epsilon,\lambda,\nu}^{\cc}(\xi)+P_{\rr,\epsilon,\lambda,\nu}^{\uu}(\xi),
\end{align*}
and where $\mathbf{e}_3 = (0,0,1)^\top$ is the third standard basis vector.
\item[(ii)] For each
\begin{align*}
\gamma_\mathrm{uf}\in P^\mathrm{s}_{\mathrm{uf},\eps,\lambda,\nu}\left(\xi_{\mathrm{uf},\eps,\nu}^\mathrm{in}  \right) \left[\mathbb{C}^3\right], \qquad \beta_\mathrm{uf}=\begin{pmatrix}\beta_{1,\mathrm{uf}}\\ \beta_{2,\mathrm{uf}} \\ \eps \beta_{3,\mathrm{uf}}  \end{pmatrix}\in P^\mathrm{cu}_{\mathrm{uf},\eps,\lambda,\nu}\left(\xi_{\mathrm{uf},\eps,\nu}^\mathrm{out}  \right)\left[ \mathbb{C}^3\right],
\end{align*}
there exists a solution $\psi \colon \mathcal{I}_\mathrm{uf} \to \C^3$ to the eigenvalue problem~\eqref{eigenvalueproblem} subject to the boundary conditions
\begin{align*}
P^\mathrm{s}_{\mathrm{uf},\eps,\lambda,\nu}\left(\xi_{\mathrm{uf},\eps,\nu}^\mathrm{in} \right) \psi\left(\xi_{\mathrm{uf},\eps,\nu}^\mathrm{in} \right)&= \gamma_\mathrm{uf},\qquad
P^\mathrm{cu}_{\mathrm{uf},\eps,\lambda,\nu}\left(\xi_{\mathrm{uf},\eps,\nu}^\mathrm{out} \right)\psi\left(\xi_{\mathrm{uf},\eps,\nu}^\mathrm{out} \right) = \beta_\mathrm{uf},
\end{align*}
satisfying the estimate
\begin{align*}
\left\|P^\mathrm{s}_{\mathrm{uf},\eps,\lambda,\nu}\left(\xi_{\mathrm{uf},\eps,\nu}^\mathrm{out} \right)\psi\left(\xi_{\mathrm{uf},\eps,\nu}^\mathrm{out}\right) \right\|  &\leq  C_\nu \re^{-\frac{\vartheta_\nu}{\eps}}\|\gamma_\mathrm{uf}\|.
\end{align*}
Furthermore, if $\Re(\lambda) \geq \eps^{1/5}$, then we have 
\begin{align} \label{tame_estimate_upper}
\left\|P^\mathrm{cu}_{\mathrm{uf},\eps,\lambda,\nu}\left(\xi_{\mathrm{uf},\eps,\nu}^\mathrm{in}\right) \psi\left(\xi_{\mathrm{uf},\eps,\nu}^\mathrm{in}\right)\right\|  &\leq  C_\nu \exp\left(\frac{\nu \Re(\lambda)}{c\eps}\right) \|\beta_\mathrm{uf}\|,
\end{align}
while if $|\Re(\lambda)| \leq \mu \eps^{1/6}$, then it holds
\begin{align*}
P^\mathrm{cu}_{\mathrm{uf},\eps,\lambda,\nu}\left(\xi_{\mathrm{uf},\eps,\nu}^\mathrm{in}  \right) \psi\left(\xi_{\mathrm{uf},\eps,\nu}^\mathrm{in}\right)&= \begin{pmatrix} U_\mathrm{uf}\\ V_\mathrm{uf}\\ W_\mathrm{uf} \end{pmatrix}(\beta_{\mathrm{uf}}), 
\end{align*}
where $U_\mathrm{uf}, V_\mathrm{uf}$ and $W_\mathrm{uf}$ satisfy
\begin{align*}
\left\|\begin{pmatrix} U_\mathrm{uf}\\ V_\mathrm{uf} \end{pmatrix}(\beta_{\mathrm{uf}})-\alpha^u_\mathrm{uf}(\beta_\mathrm{uf}; \eps,\lambda)\begin{pmatrix} u_\eps'\\ v_\eps'\end{pmatrix}\left(\xi_{\mathrm{uf},\eps,\nu}^\mathrm{in}\right)\right\| &\leq  C_\nu\left(\|\beta_\mathrm{uf}\|+\left|\lambda \log |\lambda|\right|\cdot|\beta_{3,\mathrm{uf}}|\right), \\
\left|W_\mathrm{uf}(\beta_{\mathrm{uf}})-\eps \beta_{3,\mathrm{uf}} \frac{w_\eps'\left( \xi_{\mathrm{uf},\eps,\nu}^\mathrm{in}\right)}{w_\eps'\left(\xi_{\mathrm{uf},\eps,\nu}^\mathrm{out}\right)}\right|  &\leq  C_\nu \eps \|\beta_\mathrm{uf}\|\\
&\qquad +\,\left(C_\nu\eps \left|\lambda \log |\lambda|\right|+\eps \eta(\nu)\left|\Upsilon_\mathrm{uf}(\lambda \eps^{-1/6})\right|\right)|\beta_{3,\mathrm{uf}}|
\end{align*}
with
\begin{align*}
\left|\alpha^u_\mathrm{uf}(\beta_\mathrm{uf}; \eps,\lambda)
-\frac{\eps \beta_{3,\mathrm{uf}}}{w_\eps'\left(\xi_{\mathrm{uf},\eps,\nu}^\mathrm{out} \right)} \left(1- \Upsilon_\mathrm{uf}(\lambda \eps^{-1/6})\right)\right|&\leq \eta(\nu)\left|\Upsilon_\mathrm{uf}(\lambda \eps^{-1/6})\right| |\beta_{3,\mathrm{uf}}|.
\end{align*}
\end{itemize}
\end{proposition}

\subsection{Derivation of the main formula}\label{sec:criticalcurve}
We recall that $\lambda \in \C$ lies in the spectrum of the linearization $\El_\eps$ if and only if there exists $\rho \in [-\pi/L_\eps,\pi/L_\eps)$ such that the eigenvalue problem~\eqref{eigenvalueproblem} with the associated Floquet boundary condition~\eqref{eigenvalueproblemBCshift} admits a nontrivial solution. To access results from complex function theory, we complexify $\rho$ and consider $\rho$-values in a bounded $\eps$- and $\lambda$-independent open set $U \subset \C$ containing the interval $[-\pi/L_\eps,\pi/L_\eps)$. We employ Lin's method to construct nontrivial solutions to~\eqref{eigenvalueproblem}-\eqref{eigenvalueproblemBCshift} for $\rho \in U$ and $\lambda \in R_1(\mu)$. Specifically, we match solutions to the eigenvalue problem~\eqref{eigenvalueproblem} on the intervals $\mathcal{I}_\rr$, $\mathcal{I}_\lr$, $\mathcal{I}_{\lf}$, and $\mathcal{I}_{\uf}$, established in Propositions~\ref{prop:bvp_right},~\ref{prop:bvp_right},~\ref{prop:fold_bvp_oc_lower}, and~\ref{prop:fold_bvp_oc_upper}, respectively. The outcome of this matching procedure, which is given in the following result, leads to an implicit transcendental equation, which we call the ``main formula'', relating $\lambda,\eps$, and $\rho$. 

\begin{proposition}\label{prop:mainformula}
There exists a continuous function $\eta \colon [0,\infty) \to [0,\infty)$ satisfying $\eta(0)=0$ such that, provided $0 < \epsilon \ll \mu\ll  \nu \ll 1$, there exists an $\eps$- and $\mu$-independent constant $C_\nu > 0$
such that the eigenvalue problem~\eqref{eigenvalueproblem}-\eqref{eigenvalueproblemBC} admits a nontrivial solution for $\rho \in U$ and $\lambda \in R_{1,\eps}(\mu) \coloneqq  R_1(\mu) \cap \{\lambda \in \C : |\Re(\lambda)| \leq \mu \eps^{1/6}\}$ if and only if we have
\begin{align}\label{eq:mainformula}
\re^{-\left(\ri\rho - \frac{\lambda}{c}\right)L_\eps}&=\left(1 +  \frac{(u_1-u_2)\Upsilon_\mathrm{lf}(\lambda \eps^{-1/6})}{u_2 - \gamma f(u_2) - a}\right)\left(1 +  \frac{(\bar{u}_1-\bar{u}_2)\Upsilon_\mathrm{uf}(\lambda \eps^{-1/6})}{\bar{u}_2 - \gamma f(\bar{u}_2) - a}\right)+\mathcal{R}_\eps(\lambda),
\end{align}
where the residual $\mathcal{R}_\eps \colon R_{1,\eps}(\mu) \to \C$ obeys the bound
\begin{align}\label{eq:mainformula_residual}
\left|\mathcal{R}_\eps(\lambda, \eps)\right|\leq C_\nu\left(\eps^{\frac13}+|\lambda \log |\lambda||\right)+\eta(\nu)\left(|\Upsilon_\mathrm{uf}(\lambda \eps^{-1/6})| +|\Upsilon_\mathrm{lf}(\lambda \eps^{-1/6})|  \right).
\end{align}
Furthermore, if $\lambda\in R_1(\mu)$ satisfies $\Re(\lambda)\geq \eps^{1/5}$, the eigenvalue problem~\eqref{eigenvalueproblem}-\eqref{eigenvalueproblemBC} admits no nontrivial solution. 
\end{proposition}
\begin{proof} In this proof, $C\geq 1$ denotes a $\lambda$-, $\xi$-, $\nu$-, $\mu$-, and $\eps$-independent constant, which will be taken larger, if necessary. Moreover, $C_\nu \geq 1$ and $\vartheta_\nu > 0$ denote $\lambda$-, $\xi$-, $\mu$-, and $\eps$-independent constants, which will be taken larger and smaller, respectively, if necessary.

Combining the results of Propositions~\ref{prop:bvp_right},~\ref{prop:bvp_left},~\ref{prop:fold_bvp_oc_lower}, and~\ref{prop:fold_bvp_oc_upper}, we impose matching conditions at 
$\xi=\xi_{\mathrm{uf},\eps,\nu}^\mathrm{in},\xi_{\mathrm{uf},\eps,\nu}^\mathrm{out},\xi_{\mathrm{lf},\eps,\nu}^\mathrm{in}$, and
\begin{align*}
\xi=\xi_{\mathrm{lf},\eps,\nu}^{\mathrm{out},L}=\xi^L_{\eps,\nu}=\xi^0_{\eps,\nu}+L_\eps=\xi_{\mathrm{lf},\eps,\nu}^{\mathrm{out},0}+L_\eps
\end{align*}
to reduce the existence of a nontrivial solution $\psi(\xi)$ to the  eigenvalue problem~\eqref{eigenvalueproblem} satisfying the Floquet boundary condition~\eqref{eigenvalueproblemBCshift} for some $\rho \in U$ to a system of equations in the free variables $\alpha_\mathrm{f/b}, \beta_\mathrm{f/b}, \gamma_\mathrm{l/r},\gamma_\mathrm{uf/lf}$, and $\beta_\mathrm{uf/lf}$, where we define
\begin{align*}
\alpha_{\bb} \Phi_{\bb}\left(\tfrac{1}{\nu}\right)&=\widetilde P_{\bb,\nu}^{\uu}\left(\tfrac{1}{\nu}\right) \psi\left(\xi_{\mathrm{uf},\eps,\nu}^\mathrm{in}\right),\quad 
\beta_\bb \Psi_{\bb,\nu}\left(\tfrac{1}{\nu}\right)=\widetilde{P}_{\bb,\nu}^{\cc}\left(\tfrac{1}{\nu}\right)\psi\left(\xi_{\mathrm{uf},\eps,\nu}^\mathrm{in}\right), \quad \gamma_\mathrm{uf}=P^\mathrm{s}_{\mathrm{uf},\eps,\lambda,\nu}\left(\xi_{\mathrm{uf},\eps,\nu}^\mathrm{in} \right) \psi\left(\xi_{\mathrm{uf},\eps,\nu}^\mathrm{in}  \right)
\end{align*}
at $\xi = \xi_{\mathrm{uf},\eps,\nu}^\mathrm{in}$,
\begin{align*}
\beta_\mathrm{uf}&=\begin{pmatrix}  \beta_{1,\mathrm{uf}}\\\beta_{2,\mathrm{uf}}\\\eps \beta_{3,\mathrm{uf}}\end{pmatrix}=P^\mathrm{cu}_{\mathrm{uf},\eps,\lambda,\nu}\left(\xi_{\mathrm{uf},\eps,\nu}^\mathrm{out}  \right)\psi\left(\xi_{\mathrm{uf},\eps,\nu}^\mathrm{out}  \right) , \quad \gamma_{\rr}=P_{\rr,\epsilon,\lambda,\nu}^{\su}\left(\xi_{\mathrm{uf},\eps,\nu}^\mathrm{out}\right)\psi\left(\xi_{\mathrm{uf},\eps,\nu}^\mathrm{out}\right)
\end{align*}
at $\xi = \xi_{\mathrm{uf},\eps,\nu}^\mathrm{out}$,
\begin{align*}
\alpha_{\f} \Phi_{\f}\left(\tfrac{1}{\nu}\right)&= \widetilde{P}_{\f,\nu}^{\uu}\left(\tfrac{1}{\nu}\right) \psi\left(\xi_{\mathrm{lf},\eps,\nu}^\mathrm{in} \right), \quad
\beta_\f \Psi_{\f,\nu}\left(\tfrac{1}{\nu}\right)=\widetilde{P}_{\f,\nu}^{\cc}\left(\tfrac{1}{\nu}\right)\psi\left(\xi_{\mathrm{lf},\eps,\nu}^\mathrm{in} \right),\quad 
\gamma_\mathrm{lf}= P^\mathrm{s}_{\mathrm{lf},\eps,\lambda,\nu}\left(\xi_{\mathrm{lf},\eps,\nu}^\mathrm{in}  \right) \psi\left(\xi_{\mathrm{lf},\eps,\nu}^\mathrm{in} \right)
\end{align*}
at $\xi = \xi_{\mathrm{lf},\eps,\nu}^\mathrm{in}$, and
\begin{align*}
\beta_\mathrm{lf}^0&=\begin{pmatrix}  \beta_{1,\mathrm{lf}}^0\\\beta_{2,\mathrm{lf}}^0\\\eps \beta_{3,\mathrm{lf}}^0\end{pmatrix}=P^\mathrm{cu}_{\mathrm{lf},\eps,\lambda,\nu}\left(\xi_{\mathrm{lf},\eps,\nu}^{\mathrm{out},L}\right)\psi\left(\xi_{\mathrm{lf},\eps,\nu}^{\mathrm{out},0} \right), \quad \gamma_{\lr}^0 = P_{\lr,\epsilon,\lambda,\nu}^{\su}\left(\xi_{\mathrm{lf},\eps,\nu}^{\mathrm{out},0}\right)\psi\left(\xi_{\mathrm{lf},\eps,\nu}^{\mathrm{out},0}\right),\\
\beta_\mathrm{lf}^L&=\begin{pmatrix}  \beta_{1,\mathrm{lf}}^L\\\beta_{2,\mathrm{lf}}^L\\\eps \beta_{3,\mathrm{lf}}^L\end{pmatrix}=P^\mathrm{cu}_{\mathrm{lf},\eps,\lambda,\nu}\left(\xi_{\mathrm{lf},\eps,\nu}^{\mathrm{out},L}  \right)\psi\left(\xi_{\mathrm{lf},\eps,\nu}^{\mathrm{out},L} \right),\quad 
  \gamma_{\lr}^L = P_{\lr,\epsilon,\lambda,\nu}^{\su}\left(\xi_{\mathrm{lf},\eps,\nu}^{\mathrm{out},0}\right)\psi\left(\xi_{\mathrm{lf},\eps,\nu}^{\mathrm{out},L}\right)
\end{align*}
at $\xi = \xi_{\mathrm{lf},\eps,\nu}^{\mathrm{out},L} = \xi_{\mathrm{lf},\eps,\nu}^{\mathrm{out},0}+L_\eps$.
Moreover, we set
\begin{align} \label{betalr_def}
\beta_\rr = \widetilde P_{\f,\nu}^{\mathrm{cu}}\left(\tfrac{1}{\nu}\right) \psi\left(\xi_{\mathrm{lf},\eps,\nu}^\mathrm{in} \right), \qquad \beta_\lr = \widetilde P_{\bb,\nu}^{\mathrm{cu}}\left(\tfrac{1}{\nu}\right) \psi\left(\xi_{\mathrm{uf},\eps,\nu}^\mathrm{in}\right).
\end{align}
Clearly, it holds
\begin{align} \label{betalr_bounds}
\|\beta_{\rr}\| \leq C_\nu\left(|\alpha_\f| + |\beta_\f|\right), \qquad \|\beta_{\lr}\| \leq C_\nu\left(|\alpha_\bb| + |\beta_\bb|\right).
\end{align}
On the other hand, applying the projections $\widetilde P_{j,\nu}^{\mathrm{c}}\left(\tfrac{1}{\nu}\right)$ and $\widetilde P_{j,\nu}^{\mathrm{u}}\left(\tfrac{1}{\nu}\right)$ for $j = \f,\bb$ to~\eqref{betalr_def} and using Proposition~\ref{prop:varred2}, we establish
\begin{align} \label{betalr_bounds_2}
|\alpha_\f|,|\beta_\f| \leq C_\nu \|\beta_{\rr}\|, \qquad |\alpha_\bb|, |\beta_\bb| \leq C_\nu \|\beta_{\lr}\|.
\end{align}

We begin by solving for $\gamma_\mathrm{l/r},\gamma_\mathrm{up/dn}$ in terms of $\beta_\mathrm{l/r},\beta_\mathrm{up/dn}$. By Proposition~\ref{prop:fold_bvp_oc_lower}, we have that 
\begin{align*}
\left\|P^\mathrm{s}_{\mathrm{lf},\eps,\lambda,\nu}\left(\xi_{\mathrm{lf},\eps,\nu}^{\mathrm{out},L}  \right)\psi\left(\xi_{\mathrm{lf},\eps,\nu}^{\mathrm{out},L}  \right) \right\|  &\leq  C_\nu \re^{-\frac{\vartheta_\nu}{\eps}}\|\gamma_\mathrm{lf}\|,
\end{align*}
provided $0 < \eps, |\lambda| \ll \nu \ll 1$. Using the fact that 
\begin{align*}
\psi\left(\xi_{\mathrm{lf},\eps,\nu}^{\mathrm{out},L}  \right) &= P^\mathrm{cu}_{\mathrm{lf},\eps,\lambda,\nu}\left(\xi_{\mathrm{lf},\eps,\nu}^{\mathrm{out},L}  \right)\psi\left(\xi_{\mathrm{lf},\eps,\nu}^{\mathrm{out},L}\right)  +P^\mathrm{s}_{\mathrm{lf},\eps,\lambda,\nu}\left(\xi_{\mathrm{lf},\eps,\nu}^{\mathrm{out},L}  \right) \psi\left(\xi_{\mathrm{lf},\eps,\nu}^{\mathrm{out},L} \right)\\
&=\beta_\mathrm{lf}^L+P^\mathrm{s}_{\mathrm{lf},\eps,\lambda,\nu}\left(\xi_{\mathrm{lf},\eps,\nu}^{\mathrm{out},L}  \right) \psi\left(\xi_{\mathrm{lf},\eps,\nu}^{\mathrm{out},L}  \right),
\end{align*}
applying the projection $P_{\lr,\epsilon,\lambda,\nu}^{\su}\left(\xi_{\mathrm{lf},\eps,\nu}^{\mathrm{out},0}\right)$, and employing Propositions~\ref{prop:slow} and~\ref{prop:fold_bvp_oc_lower}, we obtain
\begin{align}
    \gamma_\mathrm{l}^L= F_1(\beta_\mathrm{lf}^L, \gamma_\mathrm{lf})\label{eq:gammaid1}
\end{align}
provided $0 < \eps, |\lambda| \ll \nu \ll 1$, where $F_1$ is a linear map satisfying
\begin{align}\label{eq:gammaests1}
\left\|F_1(\beta_\mathrm{lf}^L, \gamma_\mathrm{lf})\right\|\leq C_\nu \re^{-\vartheta_\nu/\eps}\left(\|\beta_\mathrm{lf}^L\|+ \|\gamma_\mathrm{lf}\|   \right).
\end{align}
Similarly, we establish
\begin{align}\label{eq:gammaid2}
    \gamma_\mathrm{r}= F_2(\beta_\mathrm{uf}, \gamma_\mathrm{uf}),
\end{align}
where $F_2$ is a linear map satisfying
\begin{align}\label{eq:gammaests2}
\left\|F_2(\beta_\mathrm{uf}, \gamma_\mathrm{uf})\right\|\leq C_\nu \re^{-\vartheta_\nu/\eps}\left(\|\beta_\mathrm{uf}\|+ \|\gamma_\mathrm{uf}\|   \right).
\end{align}
Next, using estimate~\eqref{betalr_bounds_2}, Propositions~\ref{prop:bvp_right} and~\ref{prop:bvp_left}, and the fact that $\gamma_\lr^0 = \re^{-\left(\ri\rho - \frac{\lambda}{c}\right)L_\eps}\gamma_\lr^L$ holds by the Floquet--Bloch boundary condition~\eqref{eigenvalueproblemBCshift}, we have
\begin{align}
\begin{split}
\left\|\widetilde P_{\f,\nu}^{\su}\left(\tfrac{1}{\nu}\right) \psi\left(\xi_{\mathrm{lf},\eps,\nu}^\mathrm{in} \right)\right\| &\leq C_\nu \left(\left(\epsilon^{\frac23} + |\lambda|\right)\|\beta_\rr\| + \re^{-\vartheta_\nu/\eps}\|\gamma_{\rr}\|\right),\\
\left\|\widetilde P_{\bb,\nu}^{\su}\left(\tfrac{1}{\nu}\right), \psi\left(\xi_{\mathrm{uf},\eps,\nu}^\mathrm{in}\right)\right\| &\leq C_\nu \left(\left(\epsilon^{\frac23} + |\lambda|\right)\|\beta_\lr\| + \re^{-\vartheta_\nu/\eps}\|\gamma_{\lr}^L\|\right),
\end{split} \label{stable_bounds_1}
\end{align}
where we used $0 < \eps, |\lambda| \ll \nu \ll 1$ and~\eqref{period}. Therefore, applying Proposition~\ref{prop:fold_bvp_oc_upper} and using~\eqref{eq:gammaid1} and~\eqref{eq:gammaests1}, we obtain
\begin{align} \label{eq:gammaid3}
\gamma_\mathrm{uf} = F_3(\beta_\mathrm{\lr}, \beta_\mathrm{lf}^L, \gamma_\mathrm{lf}),
\end{align}
provided $0 < \eps, |\lambda| \ll \nu \ll 1$, where $F_3$ is a linear map satisfying
\begin{align}\label{eq:gammaests3}
\left\|F_3(\beta_\lr, \beta_\mathrm{lf}^L, \gamma_\mathrm{lf})\right\|\leq C_\nu \left(\left(\epsilon^{\frac23} + |\lambda|\right)\|\beta_\lr\| +  \re^{-\vartheta_\nu/\eps}\left(\|\beta_\mathrm{lf}^L\|+ \|\gamma_\mathrm{lf}\|   \right)\right).
\end{align}
Similarly, we have
\begin{align} \label{eq:gammaid4}
\gamma_\mathrm{lf} = F_4(\beta_\mathrm{r}, \beta_\mathrm{uf}, \gamma_\mathrm{uf}),
\end{align}
where $F_4$ is a linear map satisfying
\begin{align}\label{eq:gammaests4}
\left\|F_4(\beta_\rr, \beta_\mathrm{uf}, \gamma_\mathrm{uf})\right\|\leq C_\nu \left(\left(\epsilon^{\frac23} + |\lambda|\right)\|\beta_\rr\|+  \re^{-\vartheta_\nu/\eps}\left(\|\beta_\mathrm{uf}\|+ \|\gamma_\mathrm{uf}\|   \right)\right).
\end{align}
Combining the estimates~\eqref{eq:gammaests1},~\eqref{eq:gammaests2},~\eqref{eq:gammaests3}, and~\eqref{eq:gammaests4}, we solve the system~\eqref{eq:gammaid1},~\eqref{eq:gammaid2},~\eqref{eq:gammaid3} and~\eqref{eq:gammaid4} of linear equations for $\gamma_\mathrm{uf}, \gamma_\mathrm{lf}$ and $\gamma_\mathrm{l}, \gamma_\mathrm{r}$ and obtain
\begin{align}\label{eq:gamma_solve_estimates}
\begin{split}
\gamma_\mathrm{uf} &= F_5\left(\beta_{\lr},\beta_\rr, \beta_\mathrm{uf}, \beta_\mathrm{lf}^L\right),\\
\gamma_\mathrm{lf} &= F_6\left(\beta_{\lr},\beta_\rr, \beta_\mathrm{uf}, \beta_\mathrm{lf}^L\right),\\
\gamma_\mathrm{l}^L &= F_7\left(\beta_{\lr},\beta_\rr, \beta_\mathrm{uf}, \beta_\mathrm{lf}^L\right),\\
\gamma_\mathrm{r} &= F_8\left(\beta_{\lr},\beta_\rr, \beta_\mathrm{uf}, \beta_\mathrm{lf}^L\right),
\end{split}
\end{align}
provided $0 < \eps, |\lambda| \ll \nu \ll 1$, where $F_5,F_6,F_7$ and $F_8$ are linear maps obeying
\begin{align*}
\|F_5\left(\beta_{\lr},\beta_\rr, \beta_\mathrm{uf}, \beta_\mathrm{lf}^L\right)\| &\leq C_\nu \left(\left(\epsilon^{\frac23} + |\lambda|\right)\|\beta_\lr\|+  \re^{-\vartheta_\nu/\eps}\left(\|\beta_\rr\|+\|\beta_\mathrm{uf}\|+ \|\beta_\mathrm{lf}^L\|   \right)\right), \\
\|F_6\left(\beta_{\lr},\beta_\rr, \beta_\mathrm{uf}, \beta_\mathrm{lf}^L\right)\| &\leq C_\nu \left(\left(\epsilon^{\frac23} + |\lambda|\right)\|\beta_\rr\| +  \re^{-\vartheta_\nu/\eps}\left(\|\beta_\lr\| + \|\beta_\mathrm{uf}\|+ \|\beta_\mathrm{lf}^L\|   \right)\right),
\end{align*}
and
\begin{align*}
\|F_j\left(\beta_{\lr},\beta_\rr, \beta_\mathrm{uf}, \beta_\mathrm{lf}^L\right)\| &\leq C_\nu \re^{-\vartheta_\nu/\eps}\left(\|\beta_\lr\| + \|\beta_\rr\| +\|\beta_\mathrm{uf}\|+ \|\beta_\mathrm{lf}^L  \|   \right)
\end{align*}
for $j=7,8$.

\paragraph{Precluding spectrum for $\Re(\lambda) \geq \eps^{1/5}$.} For $\Re(\lambda) \geq \eps^{1/5}$ the estimates~\eqref{tame_estimate_lower} and~\eqref{tame_estimate_upper} hold, providing a tame bound on the backward growth in the center-unstable direction of the solution $\psi(\xi)$ along the lower and upper fold points. We show that these estimates preclude the existence of a nontrivial solution to~\eqref{eigenvalueproblem}-\eqref{eigenvalueproblemBCshift} for any $\rho \in U$. More specifically, we derive a homogeneous linear system in the remaining free variables $\beta_{\mathrm{l/r}}$ and $\beta_{\mathrm{uf/lf}}$ and show that its determinant is nonzero. Consequently, it must hold $\beta_{\mathrm{l/r}}, \beta_{\mathrm{uf/lf}} =0$, implying $\gamma_{\mathrm{l/r}}, \gamma_{\mathrm{uf/lf}} = 0$ via~\eqref{eq:gamma_solve_estimates}. This, in turn, yields that $\psi(\xi)$ must be identically zero.

Provided $0 < \eps,|\lambda| \ll \nu \ll 1$ and $\Re(\lambda) \geq \eps^{1/5}$, Proposition~\ref{prop:fold_bvp_oc_upper} yields
\begin{align*}
\left\|P^\mathrm{cu}_{\mathrm{uf}}\left(\xi_{\mathrm{uf},\eps,\nu}^\mathrm{in}  \right) \psi\left(\xi_{\mathrm{uf},\eps,\nu}^\mathrm{in}\right)\right\| &\leq  C_\nu \re^{\frac{\nu \Re(\lambda)}{c\eps}} \|\beta_\mathrm{uf}\|. 
\end{align*}
Hence, applying Proposition~\ref{prop:fold_bvp_oc_upper} and using~\eqref{stable_bounds_1} and~\eqref{eq:gamma_solve_estimates}, we arrive at
\begin{align} \label{eq:betaid1}
\beta_\lr = G_1(\beta_\lr,\beta_{\rr},\beta_{\mathrm{uf}},\beta_{\mathrm{lf}}^L),
\end{align}
where $G_1$ is a linear map satisfying
\begin{align} \label{eq:betaest1}
\|G_1(\beta_\lr,\beta_{\rr},\beta_{\mathrm{uf}},\beta_{\mathrm{lf}}^L)\| \leq C_\nu \left(\re^{\frac{\nu \Re(\lambda)}{c\eps}} \|\beta_{\mathrm{uf}}\| + \left(\eps^{\frac{2}{3}} + |\lambda|\right) \|\beta_{\lr}\| + \re^{-\frac{\vartheta_\nu}{\eps}}\left(\|\beta_\rr\| + \|\beta_{\mathrm{lf}}^L\|\right)\right).
\end{align}
Similarly, we have
\begin{align} \label{eq:betaid2}
\beta_\rr = G_2(\beta_\lr,\beta_{\rr},\beta_{\mathrm{uf}},\beta_{\mathrm{lf}}^L),
\end{align}
where $G_2$ is a linear map obeying
\begin{align} \label{eq:betaest2}
\|G_2(\beta_\lr,\beta_{\rr},\beta_{\mathrm{uf}},\beta_{\mathrm{lf}}^L)\| \leq C_\nu \left(\re^{\frac{\nu \Re(\lambda)}{c\eps}} \|\beta_{\mathrm{lf}}^L\| + \left(\eps^{\frac{2}{3}} + |\lambda|\right)\|\beta_{\rr}\| + \re^{-\frac{\vartheta_\nu}{\eps}}\left(\|\beta_{\mathrm{l}}\| + \|\beta_{\mathrm{uf}}\|\right)\right).
\end{align}
Next, by estimate~\eqref{betalr_bounds_2}, Propositions~\ref{prop:bvp_right} and~\ref{prop:bvp_left}, and the fact that $\gamma_\lr^0 = \re^{-\left(\ri\rho - \frac{\lambda}{c}\right)L_\eps}\gamma_\lr^L$, we have
\begin{align*}
\left\|P_{\rr,\eps,\lambda,\nu}^{\mathrm{cu}}\left(\xi_{\mathrm{uf},\eps,\nu}^\mathrm{out}\right) \psi\left(\xi_{\mathrm{uf},\eps,\nu}^\mathrm{out}\right)\right\| &\leq C_\nu \left(\|\beta_\rr\| + \re^{-\vartheta_\nu/\eps}\|\gamma_{\rr}\|\right),\\
\left\| P_{\lr,\eps,\lambda,\nu}^{\mathrm{cu}}\left(\xi_{\mathrm{lf},\eps,\nu}^{\mathrm{out},0}\right) \psi\left(\xi_{\mathrm{lf},\eps,\nu}^{\mathrm{out},0}\right)\right\| &\leq C_\nu \left(\|\beta_\lr\| + \re^{-\vartheta_\nu/\eps}\|\gamma_{\lr}^L\|\right),
\end{align*}
where we used $0 < \eps, |\lambda| \ll \nu \ll 1$ and~\eqref{period}. So, applying Propositions~\ref{prop:slow},~\ref{prop:fold_bvp_oc_lower}, and~\ref{prop:fold_bvp_oc_upper} and recalling the identities $\smash{\beta_{\mathrm{lf}}^0 = \re^{-\left(\ri\rho -\frac{\lambda}{c}\right)L_\eps}\beta_{\mathrm{lf}}^L}$ and $\smash{\gamma_\lr^0 = \re^{-\left(\ri\rho -\frac{\lambda}{c}\right)L_\eps}\gamma_\lr^L}$, while using~\eqref{period},~\eqref{eq:gamma_solve_estimates}, $0 < \eps, |\lambda| \ll \nu \ll 1$ and $\Re(\lambda) \geq \eps^{1/5}$, we establish
\begin{align} \label{eq:betaid3}
\begin{split}
\beta_{\mathrm{uf}} &= G_3(\beta_\lr,\beta_{\rr},\beta_{\mathrm{uf}},\beta_{\mathrm{lf}}^L),\\
\beta_{\mathrm{lf}}^L &= G_4(\beta_\lr,\beta_{\rr},\beta_{\mathrm{uf}},\beta_{\mathrm{lf}}^L),
\end{split}
\end{align}
where $G_3$ and $G_4$ are linear maps satisfying
\begin{align} \label{eq:betaest3}
\begin{split}
\|G_3(\beta_\lr,\beta_{\rr},\beta_{\mathrm{uf}},\beta_{\mathrm{lf}}^L)\| &\leq C_\nu \left(\|\beta_\rr\| + \re^{-\frac{\vartheta_\nu}{\eps}}\left(\|\beta_{\mathrm{l}}\| + \|\beta_{\mathrm{uf}}\| + \|\beta_{\mathrm{lf}}^L\|\right)\right),\\
\|G_4(\beta_\lr,\beta_{\rr},\beta_{\mathrm{uf}},\beta_{\mathrm{lf}}^L)\| &\leq C_\nu \left(\re^{-\frac{\Re(\lambda)}{c} L_\eps} \|\beta_{\mathrm{l}}\| + \re^{-\frac{\vartheta_\nu}{\eps}}\left(\|\beta_{\mathrm{r}}\| + \|\beta_{\mathrm{uf}}\| + \|\beta_{\mathrm{lf}}^L\|\right)\right).
\end{split}
\end{align}
The equations~\eqref{eq:betaid1},~\eqref{eq:betaid2} and~\eqref{eq:betaid3} comprise a homogeneous system of four linear equations in the variables $\beta_{\lr/\rr}$ and $\beta_{\mathrm{uf/lf}}$ of the form
\begin{align} \label{betasys}
\begin{pmatrix}
1 + a_{1,\eps,\lambda,\nu} & b_{1,\eps,\lambda,\nu} & \re^{\frac{\nu\Re(\lambda)}{c\eps}} c_{1,\eps,\lambda,\nu} & b_{2,\eps,\lambda,\nu} \\ 
b_{3,\eps,\lambda,\nu} & 1 + a_{2,\eps,\lambda,\nu} & b_{4,\eps,\lambda,\nu} & \re^{\frac{\nu\Re(\lambda)}{c\eps}} c_{2,\eps,\lambda,\nu}\\
b_{5,\eps,\lambda,\nu} & c_{3,\eps,\lambda,\nu} & 1 + b_{6,\eps,\lambda,\nu} &
b_{7,\eps,\lambda,\nu} \\
\re^{-\frac{\Re(\lambda)}{c} L_\eps} c_{4,\eps,\lambda,\nu} & b_{8,\eps,\lambda,\nu} & b_{9,\eps,\lambda,\nu} &  1 + b_{10,\eps,\lambda,\nu}\end{pmatrix} \begin{pmatrix} \beta_\lr \\ \beta_{\rr} \\ \beta_{\mathrm{uf}} \\ \beta_\mathrm{lf} \end{pmatrix} = 0,
\end{align}
where, by estimates~\eqref{eq:betaest1},~\eqref{eq:betaest2} and~\eqref{eq:betaest3}, the coefficients obey
\begin{align*}
\left|a_{j,\eps,\lambda,\nu}\right| &\leq C_\nu\big(\eps^{\frac23} + |\lambda|\big), & & j = 1,2,\\
\left|b_{j,\eps,\lambda,\nu}\right| &\leq C_\nu \re^{-\frac{\vartheta_\nu}{\eps}}, & & j = 1,\ldots,10,\\
\left|c_{j,\eps,\lambda,\nu}\right| &\leq C_\nu,  & & j = 1,\ldots,4.
\end{align*}
So, combining the latter with estimate~\eqref{period}, one readily obtains that, provided $0 < \eps, |\lambda| \ll \nu \ll 1$ and $\Re(\lambda) \geq \eps^{1/5}$, the determinant $d_{\eps,\lambda,\nu}$ of the $(4\times4)$-matrix system~\eqref{betasys} can be approximated as
\begin{align*}
|d_{\eps,\lambda,\nu} - 1| \leq C_\nu\big(\eps^{\frac23} + |\lambda|\big),
\end{align*}
yielding $d_{\eps,\lambda,\nu} \neq 0$. This implies that $\beta_{\mathrm{l/r}}, \beta_{\mathrm{uf/lf}} = 0$ and, via~\eqref{eq:gamma_solve_estimates}, also that $\gamma_{\mathrm{l/r}}, \gamma_{\mathrm{uf/lf}} = 0$. Therefore, $\psi(\xi)$ vanishes identically. We conclude that~\eqref{eigenvalueproblem}-\eqref{eigenvalueproblemBCshift} possesses no nontrivial solution for any $\rho \in U$.

\paragraph{Matching at $\xi = \xi_{\mathrm{lf},\eps,\nu}^\mathrm{in} $ and $\xi = \xi_{\mathrm{uf},\eps,\nu}^\mathrm{in}$.} We continue with the derivation of the main formula~\eqref{eq:mainformula}. Having expressed the variables $\gamma_{\lr/\rr}$ and $\gamma_{\mathrm{uf}/\mathrm{lf}}$ in terms of $\alpha_{\f/\bb},\beta_{\f/\bb}$, and $\beta_{\mathrm{uf}/\mathrm{lf}}$ through~\eqref{eq:gamma_solve_estimates}, we proceed with solving for $\alpha_\mathrm{f}$ and $\beta_\mathrm{f}$ by matching at $\xi = \xi_{\mathrm{lf},\eps,\nu}^\mathrm{in}$. By Proposition~\ref{prop:fold_bvp_oc_lower}, we have that 
\begin{align*}
   P^\mathrm{cu}_{\mathrm{lf},\eps,\lambda,\nu}\left(\xi_{\mathrm{lf},\eps,\nu}^\mathrm{in}  \right) \psi\left(\xi_{\mathrm{lf},\eps,\nu}^\mathrm{in} \right)&= \begin{pmatrix} U_\mathrm{lf}\\ V_\mathrm{lf}\\ W_\mathrm{lf} \end{pmatrix}(\beta_\mathrm{lf}^L),
\end{align*}
provided $\lambda \in R_{1,\eps}(\mu)$ and $0 < \eps \ll \mu \ll \nu \ll 1$, where the linear maps $U_\mathrm{lf}, V_\mathrm{lf}, W_\mathrm{lf}$ satisfy
\begin{align} \label{eq:match1}
\begin{split}
\left\|\begin{pmatrix} U_\mathrm{lf}\\ V_\mathrm{lf} \end{pmatrix}(\beta_\mathrm{lf}^L)-\alpha^u_\mathrm{lf}(\beta_\mathrm{lf}^L; \eps,\lambda)\begin{pmatrix} u_\eps'\\ v_\eps'\end{pmatrix}\left(\xi_{\mathrm{lf},\eps,\nu}^\mathrm{in} \right)\right\| &\leq  C_\nu\left(\|\beta^L_\mathrm{lf}\|+\left|\lambda \log |\lambda|\right|\cdot|\beta^L_{3,\mathrm{lf}}|\right), \\
\left|W_\mathrm{lf}(\beta_\mathrm{lf}^L)-\eps \beta^L_{3,\mathrm{lf}} \frac{w_\eps'\left(\xi_{\mathrm{lf},\eps,\nu}^\mathrm{in}  \right)}{w_\eps'\left(\xi_{\mathrm{lf},\eps,\nu}^{\mathrm{out},L} \right)}\right|  &\leq \left(C_\nu\eps \left|\lambda \log |\lambda|\right|+\eps \eta(\nu)\left|\Upsilon_\mathrm{lf}(\lambda \eps^{-1/6})\right|\right)|\beta^L_{3,\mathrm{lf}}|\\
&\qquad +\, C_\nu \eps \|\beta^L_\mathrm{lf}\|
\end{split}
\end{align}
and where $\alpha^u_\mathrm{lf}$ and $\eta(\nu)$ are as in Proposition~\ref{prop:fold_bvp_oc_lower}. By the definition of $\beta_\mathrm{f}$, we have 
\begin{align*}
\beta_\f \Psi_{\f,\nu}\left(\tfrac{1}{\nu}\right)=\widetilde{P}_{\f,\nu}^{\cc}\left(\tfrac{1}{\nu}\right)\psi\left(\xi_{\mathrm{lf},\eps,\nu}^\mathrm{in} \right)
= \widetilde{P}_{\f,\nu}^{\cc}\left(\tfrac{1}{\nu}\right)\left(\begin{pmatrix} U_\mathrm{lf}\\ V_\mathrm{lf}\\ W_\mathrm{lf} \end{pmatrix}(\beta_{\mathrm{lf}}^L)+\gamma_\mathrm{lf}\right).
\end{align*}
To determine $\beta_\mathrm{f}$, we note that the third component of both $\Psi_{\f,\nu}\left(\tfrac{1}{\nu}\right)$ and $\widetilde{P}_{\f,\nu}^{\mathrm{c}}(\tfrac{1}{\nu})\mathbf{e}_3$ equals $1$ by Proposition~\ref{prop:varred2}, so we can ignore the first two components and write
\begin{align*}
 \begin{pmatrix} *\\ *\\\beta_\f \end{pmatrix}
&= \begin{pmatrix} *\\ *\\ W_\mathrm{lf}(\beta_{\mathrm{lf}}^L) \end{pmatrix}+\widetilde{P}_{\f,\nu}^{\cc}\left(\tfrac{1}{\nu}\right)\gamma_\mathrm{lf}.
\end{align*}
Hence, by using estimate~\eqref{eq:foldprop_thirdrow} in Proposition~\ref{prop:fold_bvp_oc_lower} and the bounds~\eqref{betalr_bounds},~\eqref{eq:gamma_solve_estimates}, and~\eqref{eq:match1}, we arrive at
\begin{align}\label{eq:betaf_Festimates}
\beta_\mathrm{f}=F_9(\alpha_\mathrm{f}, \beta_\mathrm{f},\alpha_\mathrm{b}, \beta_\mathrm{b}, \beta_\mathrm{uf}, \beta_\mathrm{lf}^L),
\end{align}
where the linear map $F_9$ satisfies
\begin{align*}
\begin{split}
&\left|F_9(\alpha_\mathrm{f}, \beta_\mathrm{f},\alpha_\mathrm{b}, \beta_\mathrm{b}, \beta_\mathrm{uf}, \beta_\mathrm{lf}^L)-\eps \beta^L_{3,\mathrm{lf}} \frac{w_\eps'\left(\xi_{\mathrm{lf},\eps,\nu}^\mathrm{in} \right)}{w_\eps'\left(\xi_{\mathrm{lf},\eps,\nu}^{\mathrm{out},L}\right)}\right|\leq\left(C_\nu\eps \left|\lambda \log |\lambda|\right|+\eps \eta(\nu)\left|\Upsilon_\mathrm{lf}(\lambda \eps^{-1/6})\right|\right)|\beta^L_{3,\mathrm{lf}}|\\
&\qquad\qquad  +C_\nu \eps \|\beta^L_\mathrm{lf}\| + C_\nu \left(\eps \left(\epsilon^{\frac{2}{3}} + |\lambda|\right)\left(|\alpha_\mathrm{f}|+|\beta_\mathrm{f}|\right)+  \re^{-\vartheta_\nu/\eps}\left(|\alpha_\mathrm{b}|+|\beta_\mathrm{b}|+\|\beta_\mathrm{uf}\|+ \|\beta_\mathrm{lf}^L\|   \right)\right).
\end{split}
\end{align*}
Similarly, using Proposition~\ref{prop:varred2}, we compute
\begin{align*}
\alpha_{\f} \Phi_{\f}\left(\tfrac{1}{\nu}\right)&=\widetilde{P}_{\f,\nu}^{\uu}\left(\tfrac{1}{\nu}\right) \psi\left(\xi_{\mathrm{lf},\eps,\nu}^\mathrm{in} \right)\\&= \widetilde{P}_{\f,\nu}^{\uu}\left(\tfrac{1}{\nu}\right)\begin{pmatrix} U_\mathrm{lf}\\ V_\mathrm{lf}\\ 0 \end{pmatrix} + \widetilde{P}_{\f,\nu}^{\uu}\left(\tfrac{1}{\nu}\right)P^\mathrm{s}_{\mathrm{lf},\eps,\lambda,\nu}\left(\xi_{\mathrm{lf},\eps,\nu}^\mathrm{in} \right)\psi\left(\xi_{\mathrm{lf},\eps,\nu}^\mathrm{in} \right)\\
&=\alpha^u_\mathrm{lf}(\beta_\mathrm{lf}^L; \eps,\lambda)\Phi_{\f}\left(\tfrac{1}{\nu}\right)+   \widetilde{P}_{\f,\nu}^{\uu}\left(\tfrac{1}{\nu}\right)\alpha^u_\mathrm{lf}(\beta_\mathrm{lf}^L; \eps,\lambda)\left(\begin{pmatrix} u_\eps'\left(\xi_{\mathrm{lf},\eps,\nu}^\mathrm{in} \right)\\ v_\eps'\left(\xi_{\mathrm{lf},\eps,\nu}^\mathrm{in} \right)\\0\end{pmatrix}-  \Phi_{\f}\left(\tfrac{1}{\nu}\right)\right)\\
&\qquad + \widetilde{P}_{\f,\nu}^{\uu}\left(\tfrac{1}{\nu}\right)\left(\begin{pmatrix} U_\mathrm{lf}\\ V_\mathrm{lf}\\ 0 \end{pmatrix} - \alpha^u_\mathrm{lf}(\beta_\mathrm{lf}^L; \eps,\lambda)\begin{pmatrix} u_\eps'\left(\xi_{\mathrm{lf},\eps,\nu}^\mathrm{in}\right)\\ v_\eps'\left(\xi_{\mathrm{lf},\eps,\nu}^\mathrm{in}\right)\\0\end{pmatrix} \right) +\widetilde{P}_{\f,\nu}^{\uu}\left(\tfrac{1}{\nu}\right)\gamma_\mathrm{lf}
\end{align*}
so that, by Propositions~\ref{prop:pointwise} and~\ref{prop:fold_bvp_oc_lower}, and estimates~\eqref{betalr_bounds},~\eqref{eq:gamma_solve_estimates}, and~\eqref{eq:match1}, we have
\begin{align}\label{eq:alphaf_Festimates}
\alpha_\mathrm{f}=F_{10}(\alpha_\mathrm{f}, \beta_\mathrm{f},\alpha_\mathrm{b}, \beta_\mathrm{b}, \beta_\mathrm{uf}, \beta_\mathrm{lf}^L),
\end{align}
provided $\lambda \in R_{1,\eps}(\mu)$ and $0 < \eps \ll \mu \ll \nu \ll 1$, where the linear map $F_{10}$ satisfies
\begin{align}
\begin{split}\label{eq:matching_alphaf_equation}
&\left|F_{10}(\alpha_\mathrm{f}, \beta_\mathrm{f},\alpha_\mathrm{b}, \beta_\mathrm{b}, \beta_\mathrm{uf}, \beta_\mathrm{lf}^L)-\alpha^u_\mathrm{lf}(\beta_\mathrm{lf}^L; \eps,\lambda)\right| \leq C_\nu \left(\eps^{\frac23}|\alpha^u_\mathrm{lf}(\beta_\mathrm{lf}^L; \eps,\lambda)|+\|\beta^L_\mathrm{lf}\|+ \left|\lambda \log |\lambda|\right||\beta^L_{3,\mathrm{lf}}|\right.\\
&\qquad \qquad \left.+ \left(\epsilon^{\frac{2}{3}} + |\lambda|\right)^2\left(|\alpha_\mathrm{f}|+|\beta_\mathrm{f}|\right) + \re^{-\vartheta_\nu/\eps}\left(|\alpha_\mathrm{b}|+|\beta_\mathrm{b}|+\|\beta_\mathrm{uf}\|  \right)\right).
\end{split}
\end{align}
Proposition~\ref{prop:pointwise} yields
\begin{align} \label{eq:wder_bound}
\begin{split}
\left|\eps^{-1} w_\eps'\left(\xi_{\mathrm{lf},\eps,\nu}^{\mathrm{out},L}\right) + \frac{1}{c} (u_{\lr}(\nu) - \gamma f(u_\lr(\nu)) - a)\right| &\leq C_\nu \eps^{\frac23}, \\
\left|\eps^{-1} w_\eps'\left(\xi_{\mathrm{lf},\eps,\nu}^{\mathrm{out},L}\right) + \frac{1}{c} (u_1 - \gamma f(u_1) - a)\right| &\leq C \delta_0(\nu).
\end{split}
\end{align}
Combining the latter with~\eqref{eq:matching_alphaf_equation} and Proposition~\ref{prop:fold_bvp_oc_lower}, while using the fact that $u_1 - \gamma f(u_1) - a \neq 0$, we obtain
\begin{align*}
\begin{split}
&\left|F_{10}(\alpha_\mathrm{f}, \beta_\mathrm{f},\alpha_\mathrm{b}, \beta_\mathrm{b}, \beta_\mathrm{uf}, \beta_\mathrm{lf}^L)+c \beta_{3,\mathrm{lf}}^L\frac{1-\Upsilon_\mathrm{lf}(\lambda\eps^{-1/6})}{u_{\lr}(\nu) - \gamma f(u_\lr(\nu)) - a}\right|\\
&\qquad \leq  \left(C_\nu \left(\eps^{\frac23}+\left|\lambda \log |\lambda|\right|\right)+ C\eta(\nu)\left|\Upsilon_\mathrm{lf}(\lambda\eps^{-1/6})\right|\right) \left|\beta_{3,\mathrm{lf}}^L\right|\\
&\qquad \qquad + \, C_\nu \left(\|\beta_\mathrm{lf}^L\|+\left(\epsilon^{\frac{2}{3}} + |\lambda|\right)^2\left(|\alpha_\mathrm{f}|+|\beta_\mathrm{f}|\right)+  \re^{-\vartheta_\nu/\eps}\left(|\alpha_\mathrm{b}|+|\beta_\mathrm{b}|+\|\beta_\mathrm{uf}\|   \right)\right),
\end{split}
\end{align*}
provided $\lambda \in R_{1,\eps}(\mu)$ and $0 < \eps \ll \mu \ll \nu \ll 1$.
Solving~\eqref{eq:betaf_Festimates} and~\eqref{eq:alphaf_Festimates} for $(\alpha_\f, \beta_\f)$ by the implicit function theorem and using~\eqref{eq:wder_bound} and Proposition~\ref{prop:pointwise}, we establish
\begin{align}
\begin{split}\label{eq:matching_f_solve}
\alpha_{\f}&=-c \beta_{3,\mathrm{lf}}^L\frac{1-\Upsilon_\mathrm{lf}(\lambda\eps^{-1/6})}{u_{\lr}(\nu) - \gamma f(u_\lr(\nu)) - a}+F_{11}(\alpha_\bb, \beta_\bb, \beta_\mathrm{uf}, \beta_\mathrm{lf}^L),\\
\beta_{\f}&= \eps \beta_{3,\mathrm{lf}}^L\frac{u_\f(\tfrac{1}{\nu}) - \gamma f(u_2) - a}{u_{\lr}(\nu) - \gamma f(u_\lr(\nu)) - a}+\eps F_{12}(\alpha_\bb, \beta_\bb, \beta_\mathrm{uf}, \beta_\mathrm{lf}^L),
\end{split}
\end{align} 
provided $\lambda \in R_{1,\eps}(\mu)$ and $0 < \eps \ll \mu \ll \nu \ll 1$, where the linear maps $F_{11}, F_{12}$ satisfy
\begin{align*}
\begin{split}
\left|F_j(\alpha_\bb, \beta_\bb, \beta_\mathrm{uf}, \beta_\mathrm{lf}^L)\right| &\leq C_\nu \left(\|\beta_\mathrm{lf}^L\|+  \re^{-\vartheta_\nu/\eps}\left(|\alpha_\mathrm{b}|+|\beta_\mathrm{b}|+\|\beta_\mathrm{uf}\|   \right)\right)\\
&\qquad +\, \left(C_\nu\left(\eps^{\frac23}+|\lambda \log |\lambda||\right)+C\eta(\nu)\left|\Upsilon_\mathrm{lf}(\lambda\eps^{-1/6})\right|\right)|\beta_{3,\mathrm{lf}}^L|.
\end{split}
\end{align*} 
for $j=11,12$. For the matching at $\xi = \xi_{\mathrm{uf},\eps,\nu}^\mathrm{in}$, we proceed similarly as above and solve for $(\alpha_\bb, \beta_\bb)$ in terms of the other variables, from which we obtain 
\begin{align}
\begin{split}\label{eq:matching_b_solve}
\alpha_{\bb}&=-c\beta_{3,\mathrm{uf}}\frac{1-\Upsilon_\mathrm{uf}(\lambda\eps^{-1/6})}{u_{\rr}(\nu) - \gamma f(u_\rr(\nu)) - a}+F_{13}(\alpha_\f, \beta_\f, \beta_\mathrm{uf}, \beta_\mathrm{lf}^L),\\
\beta_{\bb}&=\eps \beta_{3,\mathrm{lf}}^L\frac{u_\bb(\tfrac{1}{\nu}) - \gamma f(\bar{u}_2) - a}{u_{\rr}(\nu) - \gamma f(u_\rr(\nu)) - a}+\eps F_{14}(\alpha_\f, \beta_\f, \beta_\mathrm{uf}, \beta_\mathrm{lf}^L),
\end{split}
\end{align} 
where the linear maps $F_{13}, F_{14}$ satisfy
\begin{align*}
\begin{split}
\left|F_j(\alpha_\f, \beta_\f, \beta_\mathrm{uf}, \beta_\mathrm{lf}^L)\right| &\leq C_\nu \left(\|\beta_\mathrm{uf}\|+  \re^{-\vartheta_\nu/\eps}\left(|\alpha_\mathrm{f}|+|\beta_\mathrm{f}|+\|\beta_\mathrm{lf}^L\|   \right)\right)\\
&\qquad+\,\left(C_\nu\left(\eps^{\frac23}+|\lambda \log |\lambda||\right)+C\eta(\nu)\left|\Upsilon_\mathrm{uf}(\lambda\eps^{-1/6})\right|\right)|\beta_{3,\mathrm{uf}}|,
\end{split}
\end{align*}
for $j=13,14$.

\paragraph{Matching at $\xi = \xi_{\mathrm{uf},\eps,\nu}^\mathrm{out}$ and $\xi =\xi_{\mathrm{lf},\eps,\nu}^{\mathrm{out},L}$.}  By Proposition~\ref{prop:bvp_right}, at $\xi = \xi_{\mathrm{uf},\eps,\nu}^\mathrm{out}$ we have the estimates 
\begin{align} \label{eq:Pr_estimates}
\begin{split}
\left\|P_{\rr,\epsilon,\lambda,\nu}^{\uu}\left(\xi_{\mathrm{uf},\eps,\nu}^\mathrm{out}\right)\psi\left(\xi_{\mathrm{uf},\eps,\nu}^\mathrm{out}\right)\right\| &\leq C_\nu \re^{-\vartheta_\nu/\eps} \left(|\alpha_{\f}| + |\beta_{\f}| + \|\gamma_{\rr}\|\right),\\
\left\|P_{\rr,\epsilon,\lambda,\nu}^{\cc}\left(\xi_{\mathrm{uf},\eps,\nu}^\mathrm{out}\right)\psi\left(\xi_{\mathrm{uf},\eps,\nu}^\mathrm{out}\right)\right\| &\leq C_\nu \left(|\beta_{\f}| + \left(\epsilon^{\frac{2}{3}} + |\lambda|\right) |\alpha_{\f}| + \re^{-\vartheta_\nu/\eps}\|\gamma_{\rr}\|\right),
\end{split}
\end{align}
provided $0 < \eps, |\lambda| \ll \nu \ll 1$. 
Recalling the definition of $\beta_\mathrm{uf}$, we  write 
\begin{align} \label{eq:beta_decomp2}
\begin{split}
    \beta_\mathrm{uf} &=P_{\rr,\epsilon,\lambda,\nu}^{\mathrm{cu}}\left(\xi_{\mathrm{uf},\eps,\nu}^\mathrm{out}\right) \left(1- P^\mathrm{s}_{\mathrm{uf},\eps,\lambda,\nu}\left(\xi_{\mathrm{uf},\eps,\nu}^\mathrm{out} \right) \right)\psi\left(\xi_{\mathrm{uf},\eps,\nu}^\mathrm{out}\right) + \left[P^\mathrm{cu}_{\mathrm{uf},\eps,\lambda,\nu}\left(\xi_{\mathrm{uf},\eps,\nu}^\mathrm{out} \right)-P_{\rr,\epsilon,\lambda,\nu}^{\mathrm{cu}}\left(\xi_{\mathrm{uf},\eps,\nu}^\mathrm{out}  \right)\right] \beta_\mathrm{uf}\\
    &=P_{\rr,\epsilon,\lambda,\nu}^{\mathrm{cu}}\left(\xi_{\mathrm{uf},\eps,\nu}^\mathrm{out} \right) \psi\left(\xi_{\mathrm{uf},\eps,\nu}^\mathrm{out}\right) + \left[P^\mathrm{cu}_{\mathrm{uf},\eps,\lambda,\nu}\left(\xi_{\mathrm{uf},\eps,\nu}^\mathrm{out} \right)-P_{\rr,\epsilon,\lambda,\nu}^{\mathrm{cu}}\left(\xi_{\mathrm{uf},\eps,\nu}^\mathrm{out}  \right)\right] \beta_\mathrm{uf} \\
    &\qquad - \, P_{\rr,\epsilon,\lambda,\nu}^{\mathrm{u}}\left(\xi_{\mathrm{uf},\eps,\nu}^\mathrm{out}\right)\left(P^\mathrm{s}_{\mathrm{uf},\eps,\lambda,\nu}\left(\xi_{\mathrm{uf},\eps,\nu}^\mathrm{out}\right) - P_{\rr,\epsilon,\lambda,\nu}^{\mathrm{s}}\left(\xi_{\mathrm{uf},\eps,\nu}^\mathrm{out}  \right) \right)\left(\gamma_{\rr} + P_{\rr,\epsilon,\lambda,\nu}^{\mathrm{cu}}\left(\xi_{\mathrm{uf},\eps,\nu}^\mathrm{out}  \right) \psi\left(\xi_{\mathrm{uf},\eps,\nu}^\mathrm{out} \right)\right).
\end{split}
\end{align}
So, using Propositions~\ref{prop:slow} and~\ref{prop:fold_bvp_oc_upper} and applying estimates~\eqref{betalr_bounds},~\eqref{eq:gamma_solve_estimates}, and~\eqref{eq:Pr_estimates}, we have that
\begin{align}\label{eq:betaup_normest}
\|\beta_\mathrm{uf}\|\leq C_\nu \left(|\beta_{\f}| + \left(\epsilon^{\frac{2}{3}} + |\lambda|\right) |\alpha_{\f}|+\re^{-\vartheta_\nu/\eps}\left(|\alpha_\mathrm{b}|+|\beta_\mathrm{b}|+ \|\beta_\mathrm{lf}^L\|   \right)\right),
\end{align}
provided $\lambda \in R_{1,\eps}(\mu)$ and $0 < \eps \ll \mu \ll \nu \ll 1$. We obtain more detailed estimates by bounding the fast and slow components of $\beta_{\mathrm{uf}} = (\beta_{1,\mathrm{uf}},\beta_{2,\mathrm{uf}},\eps \beta_{3,\mathrm{uf}})^\top$ separately. To bound the fast components of $\beta_{\mathrm{uf}}$, we write
\begin{align} \label{eq:beta_decomp}
\begin{split}
    \beta_\mathrm{uf} &=  P_{\rr,\epsilon,\lambda,\nu}^{\mathrm{cu}}\left(\xi_{\mathrm{uf},\eps,\nu}^\mathrm{out} \right) \beta_\mathrm{uf}+\left[P^\mathrm{cu}_{\mathrm{uf},\eps,\lambda,\nu}\left(\xi_{\mathrm{uf},\eps,\nu}^\mathrm{out} \right)-P_{\rr,\epsilon,\lambda,\nu}^{\mathrm{cu}}\left(\xi_{\mathrm{uf},\eps,\nu}^\mathrm{out} \right)\right] \beta_\mathrm{uf}\\
    &=P_{\rr,\epsilon,\lambda,\nu}^{\mathrm{c}}\left(\xi_{\mathrm{uf},\eps,\nu}^\mathrm{out}  \right) \beta_\mathrm{uf}+ \left[P^\mathrm{cu}_{\mathrm{uf},\eps,\lambda,\nu}\left(\xi_{\mathrm{uf},\eps,\nu}^\mathrm{out} \right)-P_{\rr,\epsilon,\lambda,\nu}^{\mathrm{cu}}\left(\xi_{\mathrm{uf},\eps,\nu}^\mathrm{out}  \right)\right] \beta_\mathrm{uf}\\
    &\qquad + \, P_{\rr,\epsilon,\lambda,\nu}^{\mathrm{u}}\left(\xi_{\mathrm{uf},\eps,\nu}^\mathrm{out}\right) \left(1- P^\mathrm{s}_{\mathrm{uf},\eps,\lambda,\nu}\left(\xi_{\mathrm{uf},\eps,\nu}^\mathrm{out} \right) \right)\psi\left(\xi_{\mathrm{uf},\eps,\nu}^\mathrm{out}\right)\\
    &=P_{\rr,\epsilon,\lambda,\nu}^{\mathrm{c}}\left(\xi_{\mathrm{uf},\eps,\nu}^\mathrm{out} \right) \beta_\mathrm{uf} +\left[P^\mathrm{cu}_{\mathrm{uf},\eps,\lambda,\nu}\left(\xi_{\mathrm{uf},\eps,\nu}^\mathrm{out} \right)-P_{\rr,\epsilon,\lambda,\nu}^{\mathrm{cu}}\left(\xi_{\mathrm{uf},\eps,\nu}^\mathrm{out}  \right)\right] \beta_\mathrm{uf} +P_{\rr,\epsilon,\lambda,\nu}^{\mathrm{u}}\left(\xi_{\mathrm{uf},\eps,\nu}^\mathrm{out}\right)\psi\left(\xi_{\mathrm{uf},\eps,\nu}^\mathrm{out} \right)\\
    &\qquad - \, P_{\rr,\epsilon,\lambda,\nu}^{\mathrm{u}}\left(\xi_{\mathrm{uf},\eps,\nu}^\mathrm{out}\right)\left(P^\mathrm{s}_{\mathrm{uf},\eps,\lambda,\nu}\left(\xi_{\mathrm{uf},\eps,\nu}^\mathrm{out}\right) - P_{\rr,\epsilon,\lambda,\nu}^{\mathrm{s}}\left(\xi_{\mathrm{uf},\eps,\nu}^\mathrm{out}  \right) \right)\left(\gamma_{\rr} + P_{\rr,\epsilon,\lambda,\nu}^{\mathrm{cu}}\left(\xi_{\mathrm{uf},\eps,\nu}^\mathrm{out}  \right) \psi\left(\xi_{\mathrm{uf},\eps,\nu}^\mathrm{out} \right)\right).
\end{split}
\end{align}
Applying Proposition~\ref{prop:fold_bvp_oc_upper} and the bounds~\eqref{betalr_bounds},~\eqref{eq:gamma_solve_estimates},~\eqref{eq:Pr_estimates}, and~\eqref{eq:betaup_normest} to~\eqref{eq:beta_decomp}, as well as the estimates~\eqref{projest2} and~\eqref{specprojest} in Proposition~\ref{prop:slow} and its proof, we arrive at
\begin{align*}
    \beta_{1,\mathrm{uf}}&= F_{15}(\beta_\mathrm{uf},\alpha_\f, \alpha_\bb, \beta_\f, \beta_\bb, \beta_\mathrm{lf}^L),\\
    \beta_{2,\mathrm{uf}}&= F_{16}(\beta_\mathrm{uf},\alpha_\f, \alpha_\bb, \beta_\f, \beta_\bb, \beta_\mathrm{lf}^L),
\end{align*}
where the linear maps $F_j, j=15,16$ satisfy
\begin{align*}
    &|F_j(\beta_\mathrm{uf},\alpha_\f, \alpha_\bb, \beta_\f, \beta_\bb, \beta_\mathrm{lf}^L)|\\
    &\qquad \leq C_\nu \left(\eps|\beta_{3,\mathrm{uf}}|+\left(\eps^{\frac23}+|\lambda|\right)\left(|\beta_{1,\mathrm{uf}}|+|\beta_{2,\mathrm{uf}}|   \right)+\re^{-\vartheta_\nu/\eps}\left(|\alpha_\mathrm{f}|+|\beta_\mathrm{f}|+|\alpha_\mathrm{b}|+|\beta_\mathrm{b}|+ \|\beta_\mathrm{lf}^L  \|   \right)\right).
\end{align*}
Solving this system for $\beta_{1,\mathrm{uf}}$ and $\beta_{2,\mathrm{uf}}$, we obtain
\begin{align*}
    \beta_{1,\mathrm{uf}}&= F_{17}(\beta_{3,\mathrm{uf}},\alpha_\f, \alpha_\bb, \beta_\f, \beta_\bb, \beta_\mathrm{lf}^L),\\
    \beta_{2,\mathrm{uf}}&= F_{18}(\beta_{3,\mathrm{uf}},\alpha_\f, \alpha_\bb, \beta_\f, \beta_\bb, \beta_\mathrm{lf}^L),
\end{align*}
provided $\lambda \in R_{1,\eps}(\mu)$ and $0 < \eps \ll \mu \ll \nu \ll 1$, where the linear maps $F_j, j=17,18$ satisfy
\begin{align*}
    |F_j(\beta_\mathrm{uf},\alpha_\f, \alpha_\bb, \beta_\f, \beta_\bb, \beta_\mathrm{lf}^L)|\leq C_\nu \left(\eps|\beta_{3,\mathrm{uf}}|+\re^{-\vartheta_\nu/\eps}\left(|\alpha_\mathrm{f}|+|\beta_\mathrm{f}|+|\alpha_\mathrm{b}|+|\beta_\mathrm{b}|+ \|\beta_\mathrm{lf}^L  \|   \right)\right),
\end{align*}
whence we obtain the estimate
\begin{align} \label{eq:betauf_est}
    \|\beta_\mathrm{uf} \|&\leq C_\nu \left(\eps|\beta_{3,\mathrm{uf}}|+\re^{-\vartheta_\nu/\eps}\left(|\alpha_\mathrm{f}|+|\beta_\mathrm{f}|+|\alpha_\mathrm{b}|+|\beta_\mathrm{b}|+ \|\beta_\mathrm{lf}^L  \|   \right)\right).
\end{align}
Focusing on the third (slow) component in~\eqref{eq:beta_decomp2} and using Propositions~\ref{prop:slow},~\ref{prop:bvp_right}, and~\ref{prop:fold_bvp_oc_upper} and estimates~\eqref{betalr_bounds},~\eqref{eq:gamma_solve_estimates}, and~\eqref{eq:betauf_est}, we find that
\begin{align}\label{eq:matching_up_estimates}
\begin{split}
&\left|\eps\beta_{3,\mathrm{uf}} - \frac{u_{\rr}(\nu) - \gamma f(u_{\rr}(\nu)) - a}{u_2 - \gamma f(u_2) - a}\left(\beta_{\f} + \frac{\epsilon}{c} \left(u_{\f}(\tfrac{1}{\nu}) - u_2\right)\alpha_{\f}\right) \right|\\ 
&\qquad \leq C_\nu \left(\left(\epsilon^{\frac{1}{3}} + |\lambda|\right)\left(|\beta_{\f}| + \epsilon |\alpha_{\f}|\right) + \re^{-\vartheta_\nu/\eps}\left(\|\gamma_{\rr}\|+\|\beta_\mathrm{uf}\|\right)\right)\\
&\qquad \leq C_\nu\re^{-\vartheta_\nu/\eps}\left(|\alpha_\mathrm{b}|+|\beta_\mathrm{b}|+\|\beta_\mathrm{uf}\|+ \|\beta_\mathrm{lf}^L\|\right) + C_\nu \left(\epsilon^{\frac{1}{3}} + |\lambda|\right)\left(|\beta_{\f}| + \epsilon |\alpha_{\f}|\right),
\end{split}
\end{align}
provided $\lambda \in R_{1,\eps}(\mu)$ and $0 < \eps \ll \mu \ll \nu \ll 1$. 
At $\xi = \xi_{\mathrm{lf},\eps,\nu}^{\mathrm{out},L}$, proceeding similarly as above,
we obtain
\begin{align}\label{eq:betadn_normest}
\|\beta_\mathrm{lf}^0\|\leq C_\nu \left(\eps|\beta_{3,\mathrm{lf}}^0|+\re^{-\vartheta_\nu/\eps}\left(|\alpha_\mathrm{f}|+|\beta_\mathrm{f}|+|\alpha_\mathrm{b}|+|\beta_\mathrm{b}|+\|\beta_\mathrm{uf}  \|+ \|\beta_\mathrm{lf}^L  \|   \right)\right)
\end{align}
and
\begin{align}\label{eq:matching_dn_estimates}
\begin{split}
&\left| \eps \beta_{3,\mathrm{lf}}^0 - \frac{u_{\lr}(\nu) - \gamma f(u_{\lr}(\nu)) - a}{\bar{u}_2 - \gamma f(\bar{u}_2) - a}\left(\beta_{\bb} + \frac{\epsilon}{c} \left(u_{\bb}(\tfrac{1}{\nu}) - \bar{u}_2\right)\alpha_{\bb}\right) \right|\\
&\qquad \leq C_\nu \left(\left(\epsilon^{\frac{1}{3}} + |\lambda|\right)\left(|\beta_{\bb}| + \epsilon |\alpha_{\bb}|\right) + \re^{-\vartheta_\nu/\eps}\left(\|\gamma_{\lr}^0\|+\|\beta_\mathrm{lf}^L\|\right)\right)\\
&\qquad \leq C_\nu\re^{-\vartheta_\nu/\eps}\left(|\alpha_{\f}|+|\beta_{\f}|+\|\beta_\mathrm{uf}\|+ \|\beta_\mathrm{lf}^L\|\right) + C_\nu \left(\epsilon^{\frac{1}{3}} + |\lambda|\right)\left(|\beta_\mathrm{b}|+\eps|\alpha_\mathrm{b}|\right).
\end{split}
\end{align}
From~\eqref{eq:betaup_normest} and~\eqref{eq:betadn_normest}, and employing the Floquet condition~\eqref{eigenvalueproblemBCshift} and possibly shrinking $\vartheta_\nu > 0$, we arrive at
\begin{align}\begin{split}\label{eq:beta_updn_estnorm}
\|\beta_\mathrm{uf}\|&\leq C_\nu \left(\eps|\beta_{3,\mathrm{uf}}|+\re^{-\vartheta_\nu/\eps}\left(|\alpha_\mathrm{f}|+|\beta_\mathrm{f}|+|\alpha_\mathrm{b}|+|\beta_\mathrm{b}| +|\beta_{3,\mathrm{lf}}^L| \right)\right),\\
\|\beta_\mathrm{lf}^0\|&\leq C_\nu \left(\eps|\beta_{3,\mathrm{lf}}^0|+\re^{-\vartheta_\nu/\eps}\left(|\alpha_\mathrm{f}|+|\beta_\mathrm{f}|+|\alpha_\mathrm{b}|+|\beta_\mathrm{b}| +|\beta_{3,\mathrm{uf}}| \right)\right),\\
\|\beta_\mathrm{lf}^L\|&\leq C_\nu \left(\eps|\beta_{3,\mathrm{lf}}^L|+\re^{-\vartheta_\nu/\eps}\left(|\alpha_\mathrm{f}|+|\beta_\mathrm{f}|+|\alpha_\mathrm{b}|+|\beta_\mathrm{b}| +|\beta_{3,\mathrm{uf}}| \right)\right),
\end{split}
\end{align}
provided $\lambda \in R_{1,\eps}(\mu)$ and $0 < \eps \ll \mu \ll \nu \ll 1$. 

\paragraph{The main formula.} We begin by solving~\eqref{eq:matching_f_solve} and~\eqref{eq:matching_b_solve} for $\alpha_\f,\beta_\f, \alpha_\bb$, and $\beta_\bb$ in terms of $\beta_{3,\mathrm{uf}}$ and $\beta_{3,\mathrm{lf}}^L$. Using~\eqref{eq:beta_updn_estnorm}, we obtain
\begin{align*}
\alpha_{\f}+c \beta_{3,\mathrm{lf}}^L\frac{1-\Upsilon_\mathrm{lf}(\lambda\eps^{-1/6})}{u_{\lr}(\nu) - \gamma f(u_\lr(\nu)) - a} &= F_{19}(\beta_{3,\mathrm{uf}},\beta_{3,\mathrm{lf}}^L),\\ \beta_\mathrm{f}-\eps \beta_{3,\mathrm{lf}}^L\frac{u_\f(\tfrac{1}{\nu}) - \gamma f(u_2) - a}{u_{\lr}(\nu) - \gamma f(u_\lr(\nu)) - a} &= \eps F_{20}(\beta_{3,\mathrm{uf}},\beta_{3,\mathrm{lf}}^L),
\end{align*}
and
\begin{align*}
\alpha_{\bb}+c \beta_{3,\mathrm{uf}}\frac{1-\Upsilon_\mathrm{uf}(\lambda\eps^{-1/6})}{u_{\rr}(\nu) - \gamma f(u_\rr(\nu)) - a} &= F_{21}(\beta_{3,\mathrm{uf}},\beta_{3,\mathrm{lf}}^L),\\
\beta_\bb-\eps \beta_{3,\mathrm{lf}}^L\frac{u_\bb(\tfrac{1}{\nu}) - \gamma f(\bar{u}_2) - a}{u_{\rr}(\nu) - \gamma f(u_\rr(\nu)) - a} &= \eps F_{22}(\beta_{3,\mathrm{uf}},\beta_{3,\mathrm{lf}}^L),
\end{align*}
provided $\lambda \in R_{1,\eps}(\mu)$ and $0 < \eps \ll \mu \ll \nu \ll 1$, where $F_j$, $j = 19,20,21,22$ are linear maps satisfying
\begin{align*}
\left|F_{j}(\beta_{3,\mathrm{uf}},\beta_{3,\mathrm{lf}}^L)\right| &\leq 
\left(C_\nu\left(\eps^{\frac23}+|\lambda \log |\lambda||\right)+C\eta(\nu)\left|\Upsilon_\mathrm{uf}(\lambda\eps^{-1/6})\right|\right) |\beta_{3,\mathrm{lf}}^L| 
+ C_\nu \re^{-\vartheta_\nu/\eps}|\beta_{3,\mathrm{uf}}|
\end{align*}
for $j = 19,20$ and
\begin{align*}
\left|F_{j}(\beta_{3,\mathrm{uf}},\beta_{3,\mathrm{lf}}^L)\right| &\leq 
\left(C_\nu\left(\eps^{\frac23}+|\lambda \log |\lambda||\right)+C\eta(\nu)\left|\Upsilon_\mathrm{uf}(\lambda\eps^{-1/6})\right|\right) |\beta_{3,\mathrm{uf}}| 
+ C_\nu \re^{-\vartheta_\nu/\eps}|\beta_{3,\mathrm{lf}}^L|
\end{align*}
for $j = 21,22$. Employing these estimates in~\eqref{eq:matching_up_estimates} and~\eqref{eq:matching_dn_estimates}, applying Proposition~\ref{prop:pointwise}, using estimates~\eqref{eq:wder_bound} and~\eqref{eq:beta_updn_estnorm}, and dividing by $\eps$, we arrive at
\begin{align*}
&\left|\beta_\mathrm{3,uf} - \beta_{3,\mathrm{lf}}^L \frac{u_{\rr}(\nu) - \gamma f(u_{\rr}(\nu)) - a}{u_{\lr}(\nu) - \gamma f(u_\lr(\nu)) - a} \left(1 +  \frac{(u_1 -u_2)\Upsilon_\mathrm{lf}(\lambda\eps^{-1/6})}{u_2 - \gamma f(u_2) - a}\right)\right|\\
&\qquad\leq  \left(C_\nu\left(\eps^{\frac13}+|\lambda \log |\lambda||\right)+C\eta(\nu)\left|\Upsilon_\mathrm{uf}(\lambda\eps^{-1/6})\right|\right) |\beta_{3,\mathrm{lf}}^L| 
+ C_\nu \re^{-\vartheta_\nu/\eps}|\beta_{3,\mathrm{uf}}|.
\end{align*} 
Similarly, we obtain
\begin{align*}
&\left|\beta_\mathrm{3,lf}^0 - \beta_{3,\mathrm{uf}} \frac{u_{\lr}(\nu) - \gamma f(u_{\lr}(\nu)) - a}{u_{\rr}(\nu) - \gamma f(u_\rr(\nu)) - a} \left(1 +  \frac{(\bar{u}_1 -\bar{u}_2)\Upsilon_\mathrm{uf}(\lambda\eps^{-1/6})}{\bar{u}_2 - \gamma f(\bar{u}_2) - a}\right)\right|\\
&\qquad\leq  \left(C_\nu\left(\eps^{\frac13}+|\lambda \log |\lambda||\right)+C\eta(\nu)\left|\Upsilon_\mathrm{uf}(\lambda\eps^{-1/6})\right|\right) |\beta_{3,\mathrm{uf}}| 
+ C_\nu \re^{-\vartheta_\nu/\eps}|\beta_{3,\mathrm{lf}}^L|.
\end{align*} 
Using Proposition~\ref{prop:pointwise}, employing the Floquet condition~\eqref{eigenvalueproblemBCshift}, solving for $\beta_\mathrm{3,uf}$ in terms of $\beta_{3,\mathrm{lf}}^L$, and substituting into the second estimate, we factor out $\beta_\mathrm{3,lf}^L$ and deduce that
\begin{align*}
\re^{-\left(\ri\rho - \frac{\lambda}{c}\right)L_\eps}&=\left(1 +  \frac{(u_1-u_2)\Upsilon_\mathrm{lf}(\lambda\eps^{-1/6})}{u_2 - \gamma f(u_2) - a}\right)\left(1 +  \frac{(\bar{u}_1-\bar{u}_2)\Upsilon_\mathrm{uf}(\lambda\eps^{-1/6})}{\bar{u}_2 - \gamma f(\bar{u}_2) - a}\right)+\mathcal{R}_\eps(\lambda),
\end{align*}
provided $\lambda \in R_{1,\eps}(\mu)$ and $0 < \eps \ll \mu \ll \nu \ll 1$, where $\mathcal{R}_\eps(\lambda)$ is as in~\eqref{eq:mainformula_residual}.
\end{proof}

\subsection{Analysis of the main formula}\label{sec:mainformula}
We proceed by analyzing the main formula~\eqref{eq:mainformula} of Proposition~\ref{prop:mainformula} which controls the spectrum in the region $R_1(\mu)$. It is helpful to consider the following four subregions
\begin{align*}
    R_{1,1,\eps}(\mu)&\coloneqq B_0\left(\tfrac{1}{2}\mu \eps^{\frac16} \right)\\
        R_{1,2,\eps}(\varsigma, \mu,M)&\coloneqq \left\{\lambda\in R_1(\mu)\setminus R_{1,1,\eps}(\mu): |\Re(\lambda)|\leq \varsigma \eps^{\frac16}, |\Im(\lambda)|\leq M\eps^{\frac16} \right \}\\
     R_{1,3,\eps}(\varsigma,\mu, M)&\coloneqq \left\{\lambda\in R_1(\mu): |\Re(\lambda)|\leq \varsigma \eps^{\frac16}, |\Im(\lambda)|\geq M\eps^{\frac16} \right \},\\
     R_{1,4,\eps}(\mu)&\coloneqq \left\{\lambda\in R_1(\mu): \Re(\lambda)\geq \eps^{\frac15} \right \},
\end{align*}
where $0<\varsigma \ll \mu \ll 1 \ll M$ are chosen independently of $\eps > 0$; see Figure~\ref{fig:R1_regions}. The fact that there is no spectrum in the region $ R_{1,4,\eps}(\mu)$ follows immediately from Proposition~\ref{prop:mainformula}. Furthermore, the union of the remaining regions satisfies 
$$R_{1,1,\eps}(\mu)\cup R_{1,2,\eps}(\varsigma, \mu,M)\cup R_{1,3,\eps}(\varsigma,\mu, M) \subset R_{1,\eps}(\mu),$$ where $R_{1,\eps}(\mu)$ is defined in Proposition~\ref{prop:mainformula}, so that the main formula~\eqref{eq:mainformula} is valid throughout these three regions. We consider each region separately as the associated estimates are slightly different in each.

\begin{figure}
\centering
\includegraphics[width=0.55\linewidth]{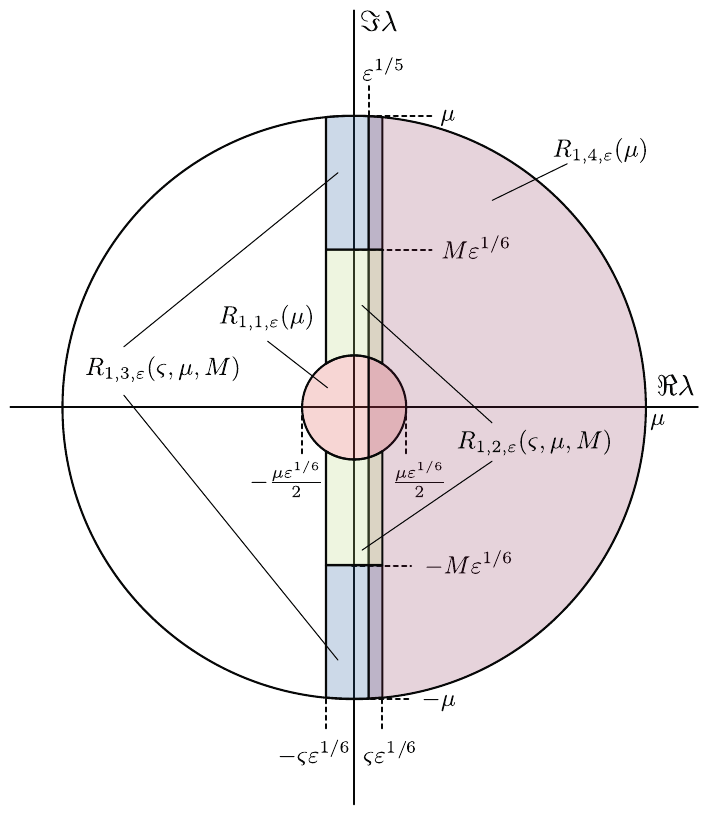}
\caption{Shown are the subregions $R_{1,1,\eps}(\mu),R_{1,2,\eps}(\varsigma,\mu,M), R_{1,3,\eps}(\varsigma,\mu,M),$ and $ R_{1,4,\eps}(\mu) $, of the ball $R_1(\mu)\subset \mathbb{C}$.}
\label{fig:R1_regions}
\end{figure}

\subsubsection{The region \texorpdfstring{$R_{1,1,\eps}(\mu)$}{R11}}
We begin with the analysis of the main formula in the region $R_{1,1,\eps}(\mu)$, which is the most delicate as it contains the critical spectral curve which touches the imaginary axis at the origin. In the following, we use analytic function theory to obtain an expansion for the critical curve and verify the quadratic tangency which is needed for diffusive spectral stability of the wave trains.
\begin{proposition}\label{prop:region_lambda1}
Let $0 < a < \frac12$, $0 < \gamma < \gamma_*(a)$, and $c > c_*(a)$. Fix $\delta>0$. Provided $0<\eps\ll\mu\ll1$, there exist an open interval $I_\eps \subset \R$ containing $0$ and a smooth function $\lambda_{\eps} \colon I_\eps \to \C$ such that $\lambda\in R_{1,1,\eps}(\mu)$ is a solution of the main formula~\eqref{eq:mainformula} for some $\rho \in \R$ if and only if we have $\lambda = \lambda_\eps(\tilde\rho)$ for some $\tilde\rho \in I_\eps$. In addition, it holds $\lambda_\eps(0) = \Re(\lambda_\eps'(0)) = 0$, $\Re(\lambda_\eps(\rho)) < 0$ for $\rho \in I_\eps \setminus \{0\}$, and
\begin{align}\label{eq:critical_lambda_prime}
 \left|\lambda_{\eps}'(0)- \ri c\right| \leq \delta, \qquad \left|\lambda_{\eps}''(0) + \frac{2\kappa c^3}{L_0} \eps^{\frac23}\right| \leq \delta \eps^{\frac23}
\end{align}
where $L_0 = L_{\rr} + L_{\lr} > 0$ is defined by~\eqref{periodexpr1} and~\eqref{periodexpr2}, and $\kappa>0$ is given by
\begin{align}\label{eq:quad_coeff}
\begin{split}
    \kappa &\coloneqq -\frac{2\Omega_0 (1-a+a^2)^{1/3}}{3c^2 }\left(\frac{1}{(u_1-\gamma f(u_1)-a)^{1/3}(u_2-\gamma f(u_2)-a)}\right.\\
    &\qquad \left. + \ \frac{1}{(\bar{u}_1-\gamma f(\bar{u}_1)-a)^{1/3}(\bar{u}_2-\gamma f(\bar{u}_2)-a)}\right),
\end{split}
\end{align}
  with $-\Omega_0 < 0$ denoting the largest zero of the Airy function $\mathrm{Ai}(z)$ (see Appendix~\ref{app:airy}).
\end{proposition}
\begin{proof} 
In this proof, $C\geq 1$ denotes a $\lambda$-, $\nu$-, $\mu$- and $\eps$-independent constant, which will be taken larger, if necessary. Moreover, $C_\nu \geq 1$ denotes a $\lambda$-, $\mu$-, and $\eps$-independent constant, which will be taken larger if necessary.

We use Proposition~\ref{prop:I0properties} to expand
\begin{align} \label{eq:Upsilon_exp}
\Upsilon_{j}(z) = I_0\left(-\frac{z^2}{\theta_{j} c^3}\right) = -\frac{2 \Omega_0}{3\theta_{j} c^3} \, z^2 + \mathcal{O}(z^4)
\end{align}
for $z \in \C$ with $|z| \ll 1$ and $j = \mathrm{uf},\mathrm{lf}$, where $I_0$ is given by~\eqref{eq:defI0}. On the other hand, Proposition~\ref{prop:pointwise} yields the expansion
\begin{align} \label{eq:L0_exp}
\frac{L_0}{L_\eps} = \eps + \mathcal{O}(\eps^\frac43),
\end{align}
for $0 < \eps \ll 1$. We take the principal complex logarithm of the main formula~\eqref{eq:mainformula}, multiply by $L_0/L_\eps$, and insert the expansions~\eqref{eq:Upsilon_exp} and~\eqref{eq:L0_exp}. Using the identities~\eqref{eq:differenceu1u2} and~\eqref{eq:defthetas} and the residual estimate~\eqref{eq:mainformula_residual}, we arrive at
\begin{align} \label{eq1}
\ri \rho L_0 - \frac{\lambda}{c} L_0 + \kappa\epsilon^{\frac23} \lambda^2 = h_{\epsilon,\nu}(\lambda)
\end{align}
for $\rho \in U$, $0 < \eps \ll \mu \ll \nu \ll 1$, and $\lambda \in B_0(\mu \epsilon^{1/6})$, where $U \subset \C$ is, as in~\S\ref{sec:criticalcurve}, a $\mu$-, $\eps$-, $\lambda$-, and $\nu$-independent bounded open set containing the interval $[-\pi/L_\eps,\pi/L_\eps)$ and the function $h_{\epsilon,\nu} \colon B_0(\mu \epsilon^{1/6}) \to \C$ enjoys the estimate
\begin{align} \label{est1}
\left|h_{\epsilon,\nu}(\lambda)\right| \leq C_\nu\left(\epsilon^{\frac43} + |\lambda|^4 \epsilon^{\frac13} + |\lambda| \epsilon |\log(\epsilon)|\right)+C\eta(\nu)\eps^{\frac23}|\lambda|^2.
\end{align}
By translational invariance, the eigenvalue problem~\eqref{eigenvalueproblem}-\eqref{eigenvalueproblemBC} admits the nontrivial solution $\Psi(\xi) = (u_\eps'(\xi),v_\eps'(\xi),w_\eps'(\xi))^\top$ at $(\lambda,\rho) = (0,0)$. Hence, $(\lambda,\rho) = (0,0)$ must be a solution of~\eqref{eq1} by Proposition~\ref{prop:mainformula}, implying $h_{\epsilon,\nu}(0) = 0$.

Next, we show that $h_{\eps,\nu}$ is an analytic function. First, we observe that~\eqref{eigenvalueproblem}-\eqref{eigenvalueproblemBC} has a nontrivial solution for $\lambda,\rho \in \C$ if and only if $\re^{\ri\rho L_\eps}$ is an eigenvalue of the monodromy matrix $\T_\eps(L_\eps,0;\lambda)$, where $\T_\eps(\xi,y;\lambda)$ denotes the evolution of system~\eqref{eigenvalueproblem}. By~\cite[Lemma~2.1.4]{KapitulaPromislow} the evolution $\T_\eps(\xi,y;\lambda)$ depends analytically on $\lambda \in \C$, since its coefficient matrix does. Denote by $\smash{\re^{\ri \rho_{1,\eps}(\lambda)L_\eps}}$, $\smash{\re^{\ri \rho_{2,\eps}(\lambda)L_\eps}}$ and $\smash{\re^{\ri \rho_{3,\eps}(\lambda)L_\eps}}$ the eigenvalues of $\T_\eps(L_\eps,0;\lambda)$. By analytic perturbation theory, cf.~\cite[Chapter~II.1]{Kato1995Perturbation}, $\rho_{1,\eps},\rho_{2,\eps}$ and $\rho_{3,\eps}$ can be chosen to be continuous functions and the matrix $\T_\eps(L_\eps,0;\lambda)$ has a constant number of distinct eigenvalues for all $\lambda \in \C$, except for a discrete set of exceptional points. 

Arguing by contradiction, we suppose that $\lambda_0 \in B_0(\mu \eps^{1/6})$ is a point where we have $\rho_{i,\eps}(\lambda_0) = \rho_{j,\eps}(\lambda_0) \in U$ for $i,j \in \{1,2,3\}$ with $i \neq j$. Then, using the continuity of $\rho_{i,\eps}$ and $\rho_{j,\eps}$ and applying Proposition~\ref{prop:mainformula}, we infer that there exists an open neighborhood $\mathcal{U}_\eps \subset B(\mu\eps^{1/6})$ of $\lambda_0$ such that the pairs $(\rho_{i,\eps}(\lambda),\lambda), (\rho_{j,\eps}(\lambda),\lambda)$ lie in $U \times \mathcal{U}_\eps$ and must obey the main formula~\eqref{eq:mainformula} for all $\lambda \in \mathcal{U}_\eps$. Hence, we have $\rho_{i,\eps}(\lambda) = \rho_{j,\eps}(\lambda)$ for each $\lambda \in \mathcal{U}_\eps$. Since the set of exceptional points is discrete in $\C$, this implies that there are at most two distinct eigenvalues of $\T_\eps(L_\eps,0;\lambda)$ for all $\lambda \in \C$. This can however not be true at $\lambda = 0$, since at $\lambda = 0$ the eigenvalue problem~\eqref{eigenvalueproblem} coincides with the variational equation about the periodic orbit $(u_\eps(\xi),v_\eps(\xi),w_\eps(\xi))^\top$ in~\eqref{eq:fast}, which is hyperbolic with a two-dimensional stable and a two-dimensional unstable manifold by~\cite[Proposition~4.7]{CASCH}. That is, the monodromy matrix $\T_\eps(L_\eps,0;0)$ possesses three distinct eigenvalues: one neutral eigenvalue $1$, one eigenvalue with modulus $> 1$ and one eigenvalue with modulus $< 1$. We conclude that for each $\lambda \in B_0(\mu \eps^{1/6})$ eigenvalues $\re^{\ri \rho L_\eps}$ of $\T_\eps(L_\eps,0;\lambda)$ with $\rho \in U$ are algebraically simple. Hence, by analytic perturbation theory, cf.~\cite[Chapter~II.1]{Kato1995Perturbation}, they depend analytically on $\lambda$. Since~\eqref{eq1} has a unique solution $\rho \in U$ for each $\lambda \in B_0(\mu \eps^{1/6})$ by estimate~\eqref{est2}, there must be a $j_0 \in \{1,2,3\}$ such that 
\begin{align*}
h_{\eps,\nu}(\lambda) = \ri \rho_{j_0,\eps}(\lambda) L_0 - \frac{\lambda}{c} L_0 + \kappa\epsilon^{\frac23} \lambda^2
\end{align*}
holds for each $\lambda \in B_0(\mu \eps^{1/6})$. We conclude that $h_{\eps,\nu}$ is analytic and $\rho_{j_0,\eps} \colon B_0(\mu \epsilon^{1/6}) \to \C$ is given by
\begin{align*}
\rho_{j_0,\eps}(\lambda) = -\ri\frac{\lambda}{c} + \ri \frac{1}{L_0}\left( \kappa \epsilon^{\frac23} \lambda^2 -  h_{\epsilon,\nu}(\lambda)\right).
\end{align*}
By estimate~\eqref{est1} it holds
\begin{align*}
 \sup_{\lambda \in B_0\left(\mu\epsilon^{1/6}\right)} \left|\kappa \epsilon^{\frac23} \lambda^2 - h_{\epsilon,\nu}(\lambda)\right| \leq C_\nu \epsilon
\end{align*}
for all $0 < \eps \ll \mu \ll \nu \ll 1$. Hence, using Cauchy estimates on the first derivative, we obtain 
\begin{align}
 \left|\rho_{j_0,\eps}'(\lambda) + \ri c^{-1}\right| \leq \frac{\sup\left\{\left|\kappa \epsilon^{\frac23} z^2 - h_{\epsilon,\nu}(z)\right| : z \in R_{1,1,\eps}(\mu)\right\}}{\frac12 \mu  \epsilon^{1/6}} \leq C_\nu  \epsilon^{\frac56} \label{est0}
\end{align}
for all $0 < \eps \ll \mu \ll \nu \ll 1$ and $\lambda \in R_{1,1,\eps}(\mu)$. Therefore, we find that, provided $0 < \eps \ll \mu \ll \nu \ll 1$, the analytic function $\rho_{j_0,\eps}$ maps the ball $R_{1,1,\eps}(\mu)$ bijectively onto its image $V_\eps$, which must be open by the open mapping theorem and contains the origin as $\rho_{j_0,\eps}(0) = 0$. By the holomorphic inverse function theorem, its inverse $\lambda_{\epsilon} \colon V_\eps \to R_{1,1,\eps}(\mu)$ must be analytic, too. Using that the interval $[-\pi/L_\eps,\pi/L_\eps)$ is contained in $U$, we conclude that there exists an open interval $I_\eps \subset \R$ such that $\lambda \in R_{1,1,\eps}(\mu)$ is a solution of the main formula~\eqref{eq:mainformula} for some $\rho \in \R $ if and only if we have $\lambda = \lambda_\eps(\tilde\rho)$ for some $\tilde\rho \in I_\eps$.

To obtain bounds on the spectral curve $\lambda_\eps|_{I_\eps}$, we start by bounding the second derivative of $\rho_{j_0,\eps}$ using Cauchy estimates. Thus, with the aid of estimate~\eqref{est1} we arrive at
\begin{align} \label{est2}
\begin{split}
 \left|\rho_{j_0,\eps}''(\lambda) - 2\ri\frac{\kappa}{L_0}\epsilon^{\frac23}\right| &\leq \frac{2! \sup\left\{\left|h_{\epsilon,\nu}(z)\right| : z \in R_{1,1,\eps}(\mu)\right\}}{\frac14 \mu^2 \epsilon^{1/3}} \\
 &\leq C_\nu\epsilon^{\frac23}\left(\mu^2 + \mu^{-2} \epsilon^{\frac13} + \mu^{-1} \epsilon^{\frac16}|\log(\epsilon)|\right)+C\eta(\nu)\eps^{\frac23}
\end{split}
\end{align}
for all $0 < \eps \ll \mu \ll \nu \ll 1$ and $\lambda \in R_{1,1,\eps}(\mu)$. So, provided $0 < \eps \ll \mu \ll \nu \ll 1$, we find that the real function $\smash{\rho^R_{\epsilon} \colon (-\tfrac12 \mu \epsilon^{1/6},\tfrac12 \mu \epsilon^{1/6}) \to \R}$ given by $\smash{\rho^R_{\epsilon}(\omega)} = \Re \rho_{j_0,\eps}(\ri \omega)$ is strictly convex. Moreover, by real symmetry of the problem, it must hold $\rho_{j_0,\epsilon}'(0) \in \R$, implying $\smash{\rho^R_{\epsilon}(0) = 0 = \left(\rho^R_{\epsilon}\right)'(0)}$. We conclude that $\smash{\rho^R_{\epsilon}(\omega)}$ attains a strict minimum at $\omega = 0$ and is thus strictly positive for all $\omega \in \smash{(-\tfrac12 \mu \epsilon^{1/6},\tfrac12 \mu \epsilon^{1/6})} \setminus \{0\}$. That is, the spectral curve $\lambda_\eps|_{I_\eps}$ does not touch the imaginary axis other than at the origin, at which we have
\begin{align}\label{eq:lambda_inverse_FT}
 \lambda_{\epsilon}(0) = 0, \qquad \lambda_{\epsilon}'(0) = \frac{1}{\rho_{j_0,\eps}'(0)}, \qquad \lambda_{\epsilon}''(0) = -\frac{\rho_{j_0,\eps}''(0)}{\left(\rho_{j_0,\eps}'(0)\right)^3}
\end{align}
by the inverse function theorem. Using estimate~\eqref{est2} and recalling $\rho_{j_0,\eps}'(0) \in \R$, we observe that $\smash{\lambda^R_{\epsilon}} \colon  I_\eps \to \R$ with $\smash{\lambda^R_{\epsilon}(\rho)} = \Re  \lambda_{\epsilon}(\rho)$ satisfies $\smash{ \left(\lambda^R_{\epsilon}\right)'(0) = 0}$ and $\smash{ \left(\lambda^R_{\epsilon}\right)''(0) < 0}$, as desired. Finally, using~\eqref{est0},~\eqref{est2}, and~\eqref{eq:lambda_inverse_FT}, we obtain~\eqref{eq:critical_lambda_prime}.
\end{proof}

\subsubsection{The region \texorpdfstring{$R_{1,2,\eps}(\varsigma, \mu,M)$}{R12}}

We preclude the existence of unstable spectrum in the region $R_{1,2,\eps}(\varsigma, \mu,M)$ by showing that any solution $\lambda \in R_{1,2,\eps}(\varsigma, \mu,M)$ to the main formula~\eqref{eq:mainformula} must have strictly negative real part. To this end, we use the Airy function expansions from Appendix~\ref{app:airy} and exploit monotonicity properties of the functions $\Upsilon_{\mathrm{uf}}$ and $\Upsilon_{\mathrm{lf}}$.

\begin{proposition}\label{prop:region_lambda2}
Let $0<a<\tfrac{1}{2}$ and $0<\gamma<\gamma_*(a)$. Fix $\mu > 0$. Provided $0 < \eps \ll \varsigma \ll 1$, any solution $\lambda\in R_{1,2,\eps}(\varsigma, \mu,\varsigma^{-1/2})$ of the main formula~\eqref{eq:mainformula} satisfies $\Re(\lambda)<0$. 
\end{proposition}
\begin{proof}
In this proof, $C\geq 1$ denotes a $\lambda$-, $\eps$-, $\nu$-, and $\varsigma$-independent constant, which will be taken larger, if necessary. Throughout the proof, we use the abbreviations $\lambda_r=\eps^{-1/6}\Re(\lambda)$ and $\lambda_i=\eps^{-1/6}\Im(\lambda)$. 

Provided $0 < \eps \ll \varsigma \ll 1$, we have, for $\lambda \in R_{1,2,\eps}(\varsigma, \mu,\varsigma^{-1/2})$, that $\tfrac{1}{4}\mu \leq |\lambda_i|\leq \varsigma^{-1/2} $ and $|\lambda_r|\leq \varsigma$ so that 
\begin{align*}
\left|\frac{\lambda^2}{\eps^{1/3}} + \lambda_i^2\right|\leq C \sqrt{\varsigma}, \qquad \Re\left(\frac{\lambda^2}{\eps^{1/3}}\right) \leq \varsigma^2.
\end{align*}
Hence, using the mean value theorem in combination with Proposition~\ref{prop:I0properties}, we estimate
\begin{align} \label{eq:Upsilonbounds}
\left|\Upsilon_j\left(\frac{\lambda}{\eps^{1/6}}\right) - I_0\left(\frac{\lambda_i^2}{\theta_j c^3}  \right)\right| = \left|I_0\left(-\frac{\lambda^2}{\theta_j c^3 \eps^{1/3}} \right) - I_0\left(\frac{\lambda_i^2}{\theta_j c^3}  \right)\right| \leq C\sqrt{\varsigma}, \qquad \left| I_0\left(\frac{\lambda_i^2}{\theta_j c^3}  \right)\right| \leq C
\end{align}
for $j = \mathrm{lf}, \mathrm{uf}$, provided $0 < \eps \ll \varsigma \ll 1$ and $\lambda \in R_{1,2,\eps}(\varsigma, \mu,\varsigma^{-1/2})$, where the function $I_0$ is given by~\eqref{eq:defI0}. Thus, taking absolute values in the main formula~\eqref{eq:mainformula} and using~\eqref{eq:Upsilonbounds} and Proposition~\ref{prop:mainformula}, we arrive at
\begin{align}\label{eq:mainformula_lambda2_abs}
\re^{\frac{\Re(\lambda)}{c}L_\eps}&\leq \left|\left(1 -  \frac{(u_2-u_1)I_0\left(\tfrac{\lambda_i^2}{\theta_\mathrm{lf}c^3} \right)}{u_2 - \gamma f(u_2) - a}\right)\left(1 - \frac{(\bar{u}_2-\bar{u}_1)I_0\left(\tfrac{\lambda_i^2}{\theta_\mathrm{uf}c^3} \right)}{\bar{u}_2 - \gamma f(\bar{u}_2) - a}\right)\right|+ C\eta(\nu)
\end{align}
provided $0 < \eps \ll \varsigma \ll \nu \ll 1$ and $\lambda \in R_{1,2,\eps}(\varsigma, \mu,\varsigma^{-1/2})$. We proceed by showing that the right-hand side of~\eqref{eq:mainformula_lambda2_abs} is strictly less than one for $\lambda \in R_{1,2,\eps}(\varsigma, \mu,\varsigma^{-1/2})$. By Proposition~\ref{prop:I0properties}, the function $I_0(z)$ is strictly increasing on $[0,\infty)$ with $I_0(0) = 0$ and $I_0(z)\to1$ as $z\to\infty$. Hence, using that
\begin{align*}
    \frac{u_2-u_1}{u_2 - \gamma f(u_2) - a}&>0, \qquad  \frac{\bar{u}_2-\bar{u}_1}{\bar{u}_2 - \gamma f(\bar{u}_2) - a} >0,
\end{align*}
we observe that the functions
\begin{align*}
A_\mathrm{lf}(z)\coloneqq 1 -  \frac{(u_2-u_1)I_0\left(\tfrac{z^2}{\theta_\mathrm{lf}c^3} \right)}{u_2 - \gamma f(u_2) - a}, \qquad A_\mathrm{uf}(z)\coloneqq 1 -  \frac{(\bar{u}_2-\bar{u}_1)I_0\left(\tfrac{z^2}{\theta_\mathrm{uf}c^3} \right)}{\bar{u}_2 - \gamma f(\bar{u}_2) - a}
\end{align*}
are strictly decreasing on $[0,\infty)$. Therefore, we have that
\begin{align*}
1 = A_{\mathrm{lf}}(0) > A_{\mathrm{lf}}\left(\tfrac14\mu\right) \geq A_{\mathrm{lf}}(z) \geq \lim_{w \to \infty} A_{\mathrm{lf}}(w) = \frac{u_1 - \gamma f(u_1) - a}{u_2 - \gamma f(u_2) - a}
\end{align*}
and
\begin{align*}
    1 = A_{\mathrm{uf}}(0) > A_{\mathrm{uf}}\left(\tfrac14\mu\right) \geq A_{\mathrm{uf}}(z) \geq \lim_{w \to \infty} A_{\mathrm{uf}}(w)=\frac{\bar{u}_1 - \gamma f(\bar{u}_1) - a}{\bar{u}_2 - \gamma f(\bar{u}_2) - a}
\end{align*}
for $z \geq \frac14\mu$, where we used that $f(u_1) = f(u_2)$ and $f(\bar{u}_1) = f(\bar{u}_2)$. Recalling the expressions~\eqref{defu1},~\eqref{defu1s}, and~\eqref{defu2} for $u_1,\bar{u}_1,u_2$, and $\bar{u}_2$ and using that $f(u_1) = f(u_2) < 0$, $f(\bar{u}_1) = f(\bar{u}_2) > 0$, $0 < a < \frac12$, and $\gamma > 0$, we establish
\begin{align*}
    0>\frac{u_1 - \gamma f(u_1) - a}{u_2 - \gamma f(u_2) - a}&> \frac{u_1 - a}{u_2 - a} > -\frac{1}{2}, \qquad 0>\frac{\bar{u}_1 - \gamma f(\bar{u}_1) - a}{\bar{u}_2 - \gamma f(\bar{u}_2) - a}>\frac{\bar{u}_1}{\bar{u}_2}> -2
\end{align*}
so that
\begin{align*}
A_\mathrm{lf}(z)&>-\frac{1}{2}, \qquad A_\mathrm{uf}(z) >-2
\end{align*}
for $z \geq 0$. Hence, if the product $|A_\mathrm{lf}(z)||A_\mathrm{uf}(z)| $ is ever equal to one for $z \geq \frac14\mu$, it must occur at a value of $z_* \geq \frac14\mu$ such that $A_\mathrm{lf}(z_*)\geq \tfrac{1}{2}>0$ and $A_\mathrm{uf}(z_*)\leq -1<0$, so that $A_\mathrm{lf}(z_*)A_\mathrm{uf}(z_*)=-1$. Let $z_* \geq \frac14\mu$ be such that $A_\mathrm{lf}(z_*)>0$ and $A_\mathrm{uf}(z_*)<0$. Then, there exists $0 < z_- < z_* < z_+$ such that $A_\mathrm{uf}(z_-)=0$ and $A_\mathrm{lf}(z_+)=0$, so that $A_\mathrm{lf}(z)A_\mathrm{uf}(z)$ achieves its minimum on the interval $(z_-,z_+)$. We show that this minimum is strictly greater than $-1$. For $z\in (z_-,z_+)$, we have
\begin{align*}
A_\mathrm{lf}(z)A_\mathrm{uf}(z)&=\left(1 -  \frac{(u_2-u_1)I_0\left(\tfrac{z^2}{\theta_\mathrm{lf}c^3} \right)}{u_2 - \gamma f(u_2) - a}\right)\left(1 - \frac{(\bar{u}_2-\bar{u}_1)I_0\left(\tfrac{z^2}{\theta_\mathrm{uf}c^3} \right)}{\bar{u}_2 - \gamma f(\bar{u}_2) - a}\right)\\
&>\left(1 -  \frac{(u_2-u_1)I_0\left(\tfrac{z^2}{\theta_\mathrm{uf}c^3} \right)}{u_2 - \gamma f(u_2) - a}\right)\left(1 - \frac{(\bar{u}_2-\bar{u}_1)I_0\left(\tfrac{z^2}{\theta_\mathrm{uf}c^3} \right)}{\bar{u}_2 - \gamma f(\bar{u}_2) - a}\right)
\end{align*}
due to the fact that $\theta_\mathrm{uf}>\theta_\mathrm{lf}$ for $0 < a < \frac12$ and $0 < \gamma < \gamma_*(a)$; see~\eqref{eq:defthetas}. The right-hand side of the latter is quadratic in $\smash{I_0(z^2/(\theta_\mathrm{uf}c^3))}$. Using~\eqref{eq:differenceu1u2}, one finds that it reaches its minimum when 
\begin{align*}
    (u_2-u_1)I_0\left(\frac{z^2}{\theta_\mathrm{uf}c^3} \right) = \frac{1}{2}\left( u_2 - \gamma f(u_2) - \bar{u}_2 + \gamma f(\bar{u}_2)\right),
\end{align*}
so that, using $0<a<\tfrac{1}{2}$ and $0<\gamma<\gamma_*(a)<6$, we have that
\begin{align*}
A_\mathrm{lf}(z)A_\mathrm{uf}(z)
&>\frac{\left(u_2 - \gamma f(u_2) - a +\bar{u}_2 - \gamma f(\bar{u}_2) - a\right)^2}{4(\bar{u}_2 - \gamma f(\bar{u}_2) - a)(u_2 - \gamma f(u_2) - a)}\\
&=-\frac{(1 - 2 a)^2 (9 + (-2 - a + a^2) \gamma)^2}{
 27 (9 + 2 (2 - 5 a + 5 a^2) \gamma + (-1 + a)^2 a^2 \gamma^2)} >-\frac{1}{3}.
\end{align*}
Therefore, the quantity $|A_\mathrm{lf}(z)A_\mathrm{uf}(z)|$ is uniformly bounded away from $1$ for $z \geq \tfrac{1}{4}\mu$. Thus, by taking $0 < \eps \ll \varsigma \ll \nu \ll 1$, we can ensure that the right-hand side of~\eqref{eq:mainformula_lambda2_abs} is strictly bounded away from $1$ for $\lambda \in R_{1,2,\eps}(\varsigma, \mu,\varsigma^{-1/2})$ and hence any solution of~\eqref{eq:mainformula} in this region must satisfy $\Re(\lambda)<0$.
\end{proof}

\subsubsection{The region \texorpdfstring{$R_{1,3,\eps}(\varsigma, \mu,M)$}{R13}} 

Finally, we prove that any solution $\lambda\in R_{1,3,\varepsilon}(\varsigma,\mu,M)$ of the main formula~\eqref{eq:mainformula} 
must have strictly negative real part. This follows by studying the limiting behavior of the functions $\Upsilon_{\mathrm{uf}}(z)$ and 
$\Upsilon_{\mathrm{lf}}(z)$ as $z \to \infty$.

\begin{proposition}\label{prop:region_lambda3}
Let $0<a<\tfrac{1}{2}$ and $0<\gamma<\gamma_*(a)$. Provided $0 < \eps \ll \varsigma, \mu, 1/M \ll 1$, any solution $\lambda\in R_{1,3,\eps}(\varsigma, \mu,M)$ of the main formula~\eqref{eq:mainformula} satisfies $\Re(\lambda)<0$. 
\end{proposition}
\begin{proof}
We proceed similarly to the proof of Proposition~\ref{prop:region_lambda2}. Again $C\geq 1$ denotes a $\lambda$-, $\eps$-, $\nu$-, $M$-, and $\varsigma$-independent constant, which will be taken larger, if necessary. Moreover, we use the abbreviations $\lambda_r=\eps^{-1/6}\Re(\lambda)$ and $\lambda_i=\eps^{-1/6}\Im(\lambda)$. 

Provided $0 < \eps \ll \varsigma, 1/M, \mu \ll 1$, we have, for $\lambda \in R_{1,3,\eps}(\varsigma, \mu,M)$, that $M\leq |\lambda_i|\leq \mu \eps^{-1/6} $ and $|\lambda_r|\leq \varsigma$ so that 
\begin{align*}
\Re\left(\frac{\lambda^2}{\eps^{1/3}}\right) =   \lambda_r^2-\lambda_i^2 \leq -\frac{M^2}{2}.
\end{align*}
Therefore, Proposition~\ref{prop:I0properties} yields
\begin{align*}
\left|\Upsilon_j(\lambda \eps^{-1/6}) - 1\right| &= \left|I_0\left(-\frac{\lambda^2}{\theta_j c^3 \eps^{1/3}}\right) - 1\right| \leq \frac{1}{\mathrm{Ai}'(-\Omega_0)^2}\int^\infty_{-\Omega_0}\re^{-\frac{M^2}{2\theta_j c^3}\left( s+\Omega_0\right)}\mathrm{Ai}(s)^2\mathrm{d}s\\ &= 1 - I_0\left(\frac{M^2}{2 \theta_j c^3} \right) \leq \frac{C}{M^4},
\end{align*}
for $j = \mathrm{uf},\mathrm{lf}$, provided $0 < \eps \ll \varsigma, 1/M, \mu \ll 1$ and $\lambda \in R_{1,3,\eps}(\varsigma, \mu,M)$, where the function $I_0$ is given by~\eqref{eq:defI0}. Thus, Proposition~\ref{prop:mainformula} implies
\begin{align}\label{eq:main_formula3}
\re^{\frac{\Re(\lambda)}{c} L_\eps}&= \left|\left(1 +  \frac{u_1-u_2}{u_2 - \gamma f(u_2) - a}\right)\left(1 +  \frac{\bar{u}_1-\bar{u}_2}{\bar{u}_2 - \gamma f(\bar{u}_2) - a}\right)\right| + C\eta(\nu),
\end{align}
provided $0 < \eps \ll \varsigma, \mu, 1/M \ll \nu \ll 1$ and $\lambda \in R_{1,3,\eps}(\varsigma, \mu,M)$.

We proceed by showing that the right-hand side of~\eqref{eq:main_formula3} is strictly less than one from which we can deduce that any solution $\lambda \in R_{1,3,\eps}(\varsigma, \mu,M)$ of~\eqref{eq:mainformula} must satisfy $\Re(\lambda)<0$. Recalling the formulae~\eqref{defu1},~\eqref{defu1s}, and~\eqref{defu2} for $u_1,\bar{u}_1,u_2$, and $\bar{u}_2$, we compute 
\begin{align*}
\left(1 +  \frac{u_1-u_2}{u_2 - \gamma f(u_2) - a}\right)\left(1 +  \frac{\bar{u}_1-\bar{u}_2}{\bar{u}_2 - \gamma f(\bar{u}_2) - a}\right)&= \frac{a(1 - a) \left(9 - \left(2 + 4 a - 4 a^2\right) \gamma + a(1 - a) \gamma^2\right)}{9 + 
 2 \left(2 - 5 a + 5 a^2\right) \gamma + (1-a)^2 a^2 \gamma^2}.
\end{align*}
We note that the prefactors in front of the $\gamma$-terms in both the numerator and denominator have no zeros in the region $0<a<\tfrac{1}{2}$, so every term has fixed sign. We also note that $\gamma_*(a)<6$ for $0<a<\tfrac{1}{2}$. Hence, we estimate
\begin{align*}
\frac{a(1 - a) \left(9 - \left(2 + 4 a - 4 a^2\right) \gamma + a(1 - a) \gamma^2\right)}{9 + 
 2 \left(2 - 5 a + 5 a^2\right) \gamma + (1-a)^2 a^2 \gamma^2}&\leq \frac{a(1 - a) \left(9 + a(1 - a) \gamma^2\right)}{9}\leq \frac{1}{2}
\end{align*}
and
\begin{align*}
\frac{a(1 - a) \left(9 - \left(2 + 4 a - 4 a^2\right) \gamma + a(1 - a) \gamma^2\right)}{9 + 
 2 \left(2 - 5 a + 5 a^2\right) \gamma + (1-a)^2 a^2 \gamma^2}&\geq -\frac{a(1 - a) \left(2 + 4 a - 4 a^2\right) \gamma}{9}\geq -\frac{1}{2},
\end{align*}
from which we deduce that the right-hand side of~\eqref{eq:main_formula3} is strictly less than $1$, provided $0 < \eps \ll \varsigma \ll \mu \ll \nu \ll 1 \ll M$ and $\lambda \in R_{1,3,\eps}(\varsigma, \mu,M)$, and hence $\Re(\lambda)<0$.
\end{proof}

\subsection{Proof of Proposition~\ref{prop:region_r1}}\label{sec:region_r1_proof}

The result follows by combining Propositions~\ref{prop:mainformula},~\ref{prop:region_lambda1},~\ref{prop:region_lambda2}, and~\ref{prop:region_lambda3}. 

\section{Eigenvalue problem near the fold points}\label{sec:folds}

The proof of Proposition~\ref{prop:region_r1}, particularly the analysis of the main formula in~\S\ref{sec:mainformula}, relies on the estimates of Propositions~\ref{prop:fold_bvp_oc_lower}-\ref{prop:fold_bvp_oc_upper} regarding the behavior of the eigenvalue problem~\eqref{eigenvalueproblem} on the intervals $\mathcal{I}_\mathrm{lf}$ and $\mathcal{I}_\mathrm{uf}$, along which the wave train passes through small neighborhoods of the lower left and upper right folds, respectively. In this section, we prove Propositions~\ref{prop:fold_bvp_oc_lower}-\ref{prop:fold_bvp_oc_upper} through a careful analysis of~\eqref{eigenvalueproblem} along these intervals. We develop an approach based on geometric desingularization techniques~\cite{CSosc, krupaszmolyan2001}, which we apply to the system obtained by coupling the linearized
eigenvalue problem to the existence problem for the wave trains~\cite{CASCH}.

The key challenge is the lack of a uniform $(\eps,\lambda)$-independent spectral gap between the center and unstable spatial eigenvalues as the wave train passes near the nonhyperbolic folds; it is therefore not possible to construct $(\eps,\lambda)$-uniform exponential trichotomies for the eigenvalue problem~\eqref{eigenvalueproblem} for small $\lambda\in R_1(\mu)$ on the intervals $\mathcal{I}_\mathrm{lf}$ and $\mathcal{I}_\mathrm{uf}$. A spectral gap is nevertheless present between the center-unstable and stable dynamics. Therefore, guided by the structure of the existence problem, as described in~\S\ref{sec:foldsetup}, we derive a change of coordinates which separates the two-dimensional (nonhyperbolic) center-unstable dynamics from the (hyperbolic) stable dynamics. Based on the structure of the resulting linearized problem, in~\S\ref{sec:toymodel} we propose a simplified model which identifies the key terms, particularly those capturing the leading-order effect of the eigenvalue parameter $\lambda$, from which we derive the expected leading-order behavior of the eigenvalue problem~\eqref{eigenvalueproblem} near the folds. Informed by the behavior of the toy model, in~\S\ref{sec:fold_blowup_CU} we return to the full problem, and we derive precise estimates for the behavior of the eigenvalue problem in the (nonhyperbolic) center-unstable space. Finally, we obtain tame estimates in~\S\ref{sec:fold_tame} for the dynamics in the case $\Re(\lambda)\geq \eps^{1/5}$, and complete the proof of Propositions~\ref{prop:fold_bvp_oc_lower}-\ref{prop:fold_bvp_oc_upper} in~\S\ref{sec:mainfoldest}. 

\paragraph{Notation.} Throughout this section, we use the notation $f=\mathcal{O}(g)$ to denote $|f|\leq C|g|$, where $C$ is a constant which can be taken independent of $0<\eps,|\lambda|\ll\nu\ll1$.

\subsection{Setup and local coordinates}\label{sec:foldsetup}
We study the eigenvalue problem~\eqref{eigenvalueproblem} in a neighborhood of the lower fold point $(u,w)=(u_1,f(u_1))$, which we rewrite as
\begin{align}\label{eq:fhn_linstab1r}
\begin{split}
\Psi_\xi = \left(A_0(\xi;\eps)+\lambda B\right)\Psi, \qquad A_0(\xi,\eps)\coloneqq \begin{pmatrix} 0& 1& 0\\ -f'(u_\eps(\xi)) & -c& 1\\ -\frac{\eps}{c} &0 & \frac{\eps\gamma}{c}  \end{pmatrix}, \qquad B\coloneqq \begin{pmatrix} -\frac{1}{c}& 0& 0\\ 1 & -\frac{1}{c}& 0\\ 0 &0 & 0  \end{pmatrix}.
\end{split}
\end{align}

We recall from~\S\ref{sec:existence} that in a neighborhood of the fold point, in the existence problem~\eqref{eq:fast}, there exists a smooth change of coordinates which reduces the nontrivial dynamics to a two-dimensional center manifold, in which one can employ blow-up desingularization techniques to track the wave-train solution around the nonhyperbolic fold. Here, we derive an analogous local coordinate system to analyze~\eqref{eq:fhn_linstab1r}.

The wave-train solution $U_\eps(\xi)=(u_\eps,u_\eps',w_\eps)(\xi)$ satisfies the first-order system~\eqref{eq:fast}, which we write in the form
\begin{align}\label{eq:existencefold_fo}
U_\xi&= \mathcal{F}(U),
\end{align}
where $U=(u,v,w)$, $\mathcal{F}(U) = \left(v,-cv-f(u)+w,-\eps/c(u-\gamma w-a)\right)$. As described in~\S\ref{sec:existence}, for any $k>0$, there exists a neighborhood of the origin $\mathcal{V}\subset \mathbb{R}^3$ and a $C^k$-change of coordinates $\mathcal{N}_\eps:\mathcal{V}\rightarrow \mathbb{R}^3$ such that the map $U= (u_1,0,f(u_1))^\top+\mathcal{N}_\eps(V)$ transforms~\eqref{eq:fast} to the system~\eqref{eq:fold_normalform}, in which the nonhyperbolic center dynamics on a local (non-unique) two-dimensional center manifold $\mathcal{W}^\mathrm{c}\coloneqq \{(x,y,z)\in\mathcal{V}:z=0\}$ are decoupled from the normally hyperbolic $z$-dynamics. Following the analysis in~\cite{CASCH}, tracking the two-dimensional unstable manifold $\mathcal{W}^\mathrm{u}(\Gamma_\eps)$ of the periodic orbit $\Gamma_\eps$ corresponding to the wave train into a neighborhood of the fold, $\mathcal{W}^\mathrm{u}(\Gamma_\eps)$ enters such a neighborhood aligned exponentially close in $\eps$ to the unstable manifold $\mathcal{W}^{\mathrm{u}}(\mathcal{M}^{\rr}_\eps)$ of the right slow manifold, and transverse to the strong stable fibers of the two-dimensional center manifold $\mathcal{W}^\mathrm{c}$; hence $\mathcal{W}^\mathrm{u}(\Gamma_\eps)$ aligns along $\mathcal{W}^\mathrm{c}$ and subsequently along the unstable manifold $\mathcal{W}^\mathrm{u}(\mathcal{M}^r_\eps)$ of the right slow manifold $\mathcal{M}^r_\eps$. In a neighborhood of the fold point, we exploit the invariance of the manifold $\mathcal{W}^\mathrm{u}(\Gamma_\eps)$ and choose coordinates relative to this manifold which will simplify the analysis for the existence and stability problems.

In particular, we note that $\mathcal{W}^\mathrm{u}(\Gamma_\eps)$ is itself a two-dimensional normally attracting locally invariant center-like manifold in a neighborhood of the fold point (and hence represents a choice of the non-unique center manifold $\mathcal{W}^\mathrm{c}$). Therefore, choosing coordinates relative to this manifold, and straightening the corresponding strong stable fibers, results in a vector field of the same form as~\eqref{eq:fold_normalform}, in which now the manifold $\mathcal{W}^\mathrm{u}(\Gamma_\eps)$ is represented by $z=0$. Slightly abusing notation, we continue to denote this coordinate transformation by $\mathcal{N}_\eps$. 

Summarizing the above discussion, there exist an $\eps$-independent neighborhood of the origin $\mathcal{V}\subset \mathbb{R}^3$ and a change of coordinates $\mathcal{N}_\eps:\mathcal{V}\rightarrow \mathbb{R}^3$ (see also e.g.~\cite{CSosc}) such that the map $U= (u_1,0,f(u_1))^\top+\mathcal{N}_\eps(V)$ transforms the existence problem~\eqref{eq:existencefold_fo} to the system
\begin{align}\label{eq:fold_existence_compact}
V_\xi&= \mathcal{G}(V),
\end{align}
where $V=(x,y,z)$ and
\begin{align*}
\mathcal{G}(V) &= \begin{pmatrix}g_1(x,y;\eps)\\ g_2(x,y;\eps)\\ g_3(x,y,z;\eps)  \end{pmatrix} = \begin{pmatrix}\frac{f''(u_1)}{2\beta_1c}x^2-\frac{\beta_1}{\beta_2}y+\mathcal{O}\left(\eps,xy,x^3\right)\\ -\frac{\eps \beta_2}{c^2}\left(u_1-\gamma f(u_1)-a-\frac{x}{\beta_1}-\frac{\gamma c y}{\beta_2}+\mathcal{O}\left(\eps, x^2,xy,y^2\right)  \right)\\ z\left(-c+\mathcal{O}\left(x,y,z,\eps \right)\right) \end{pmatrix}.
\end{align*}
The map $\mathcal{N}_\eps$ satisfies 
\begin{align*}
N_0\coloneqq  \mathcal{N}_0'(0)= \begin{pmatrix} -\frac{1}{\beta_1}& 0& 1\\ 0 &\frac{1}{\beta_2}& -c\\ 0& \frac{c}{\beta_2}& 0 \end{pmatrix}, \qquad N_0^{-1}\coloneqq  \begin{pmatrix} -\beta_1& -\frac{\beta_1}{c}& \frac{\beta_1}{c^2}\\ 0 &0& \frac{\beta_2}{c}\\ 0& -\frac{1}{c}& \frac{1}{c^2} \end{pmatrix},
\end{align*}
where
\begin{align*}
\beta_1 = (a^2-a+1)^{1/3}(u_1-\gamma f(u_1)-a)^{-1/3}, \qquad \beta_2=c(a^2-a+1)^{1/6}(u_1-\gamma f(u_1)-a)^{-2/3}.
\end{align*}
Thus, the linear map $\mathcal{N}_0'(0)$ transforms the system into its Jordan normal form for $\eps=0$. In the $(x,y,z)$-coordinates, $\mathcal{W}^\mathrm{u}(\Gamma_\eps)$ is given by the subspace $z=0$, and is parameterized by the $(x,y)$-coordinates. In these coordinates, the wave train $\Gamma_\eps(\xi)$ is given by $U_\eps(\xi)= (u_1,0,f(u_1))^\top+\mathcal{N}_\eps(V_\eps(\xi))$ where $V_\eps(\xi) = (x_\eps(\xi),y_\eps(\xi),0)$. Furthermore, we have that 
\begin{align}\label{eq:fold_linearized_transf_lambda0}
   \mathcal{N}_\eps'(V_\eps) = N_0+\left( n_{ij}(x,y;\eps)\right)_{i,j=1,2,3},
\end{align} where
\begin{align*}
n_{ij}(x,y;\eps) &= \mathcal{O}(|x|+|y|+\eps), \qquad i=1,2, j=1,2,3\\
n_{3j}(x,y;\eps) &= \mathcal{O}(\eps), \qquad j=1,2,3.
\end{align*}
As described in~\S\ref{sec:existence}, in this local coordinate system, for $\eps=0$ the critical manifold $\mathcal{M}_0$ is determined by the conditions $z=0$ and $g_1(x,y;0)=0$, locally taking the form of an upward facing parabola centered at the origin $(x,y)=(0,0)$, with the left and right branches of the parabola representing $\mathcal{M}^\lr_0$ and $\mathcal{M}^\mathrm{m}_0$, respectively; see Figure~\ref{fig:fold_singular}. the dynamics of~\eqref{eq:fold_existence_compact} in the subspace $z=0$ can be analyzed using blow-up desingularization techniques; in particular, by tracking the perturbed manifold $\mathcal{M}^{\lr}_\eps$ around the fold~\cite[\S4]{CSosc}, it is known that the extended manifold $\mathcal{M}^{\lr,+}_0$~\eqref{eq:ml_plus} obtained by appending the positive $x$-axis to $\mathcal{M}^{\lr}_0$ perturbs to a locally invariant manifold $\mathcal{M}^{\lr,+}_\eps$ which is $\mathcal{O}(\eps^{2/3})$-close in $C^0$ to $\mathcal{M}^{\lr,+}_0$ and $\mathcal{O}(\eps^{1/3})$-close in $C^1$ to $\mathcal{M}^{\lr,+}_0$. Furthermore, following the analysis in~\cite[Proposition 4.7]{CASCH}, tracking the periodic orbit $\Gamma_\eps$ backwards around the fold, we deduce that $\Gamma_\eps$ is $C^1$-$\mathcal{O}(\re^{-\vartheta_\nu/\eps})$-close to $\mathcal{M}^{\lr,+}_\eps$ in the neighborhood $\mathcal{V}$ for some $\vartheta_\nu>0$. Hence we can characterize the behavior of the wave train $\Gamma_\eps$ in these local coordinates in the same manner as the manifold $\mathcal{M}^{\lr,+}_\eps$, up to exponentially small errors. For sufficiently small $\eps_0>0$, we denote $\mathcal{M}^{\lr,+}\coloneqq \cup_{0\leq \eps\leq \eps_0}\mathcal{M}^{\lr,+}_\eps$

We have the following, due to~\cite{CSosc, CASCH}, Proposition~\ref{prop:pointwise}, and the above discussion. 
\begin{lemma}\label{lem:fold_existence}
There exists a continuous function $\eta \colon [0,\infty) \to [0,\infty)$ satisfying $\eta(0) = 0$ such that, provided $0<\eps\ll\nu\ll1$, there exist $\eps$-independent constants $C_\nu,\vartheta_\nu>0$ such that the following holds. On the interval $\mathcal{I}_\mathrm{lf}$ the periodic orbit $\Gamma_\eps$ of~\eqref{eq:fast} corresponding to the wave-train solution $(u_\eps,w_\eps)(\xi)$ is $C^1$-$\mathcal{O}(\re^{-\vartheta_\nu/\eps})$-close to $\mathcal{M}^{\lr,+}_\eps$, and can be represented by a solution $(x,y,z) = (x_\eps(\xi),y_\eps(\xi),0)$ of~\eqref{eq:fold_normalform} satisfying
\begin{align*}
    \sup_{\xi\in \mathcal{I}_\mathrm{lf}}|x_\eps(\xi)|+|y_\eps(\xi)| \leq \eta(\nu)
\end{align*}
and 
\begin{align*}
    (x_\eps, y_\eps)\left( \xi_{\mathrm{lf},\eps,\nu}^\mathrm{in} \right)&=(x_\mathrm{in}(\nu, \eps), y_\mathrm{in}(\nu, \eps)),\\
     (x_\eps, y_\eps)\left(\xi_{\mathrm{lf},\eps,\nu}^{\mathrm{out},L} \right)&=(x_\mathrm{out}(\nu, \eps), y_\mathrm{out}(\nu, \eps)),
\end{align*}
where $x_\mathrm{in/out},y_\mathrm{in/out}$ are continuous functions satisfying
\begin{align*}
    |x_\mathrm{in}(\nu, \eps)-x_\mathrm{in}(\nu,0)|,|y_\mathrm{in}(\nu,\eps)|, |x_\mathrm{out}(\nu, \eps)-x_\mathrm{out}(\nu,0)|,|y_\mathrm{out}(\nu,\eps)-y_\mathrm{out}(\nu,0)|&\leq C_\nu\eps^{2/3},\\
    0<x_\mathrm{in}(\nu,0)< \eta(\nu), \qquad 
    -\eta(\nu)<x_\mathrm{out}(\nu,0)<0<y_\mathrm{out}(\nu,0)&< \eta(\nu)
\end{align*}
and $g_1(x_\mathrm{out}(\nu,0),y_\mathrm{out}(\nu,0);0)=0$.
\end{lemma}

We now turn to the linearized stability problem~\eqref{eq:fhn_linstab1r}, which we transform to this same local coordinate system, in the following lemma.

\begin{lemma}\label{lem:fold_transf}
For any $k>0$ and $0 < \nu \ll 1$, there exist $\mu, \eps_0>0$ and a $C^k$-change of coordinates $R_1(\mu)\times (0,\eps_0)\times \mathcal{I}_\mathrm{lf}\times \mathbb{C}^3\to \mathbb{C}^3$ defined by $\Psi = N_{\eps,\lambda}(\xi)\Phi$, where $N_{\eps,\lambda}(\xi) = \mathcal{N}'_\eps(V_\eps(\xi))+\mathcal{O}(\lambda)$, and a rescaling $\zeta = \theta_\mathrm{lf}\xi$ which transforms the linearized equation~\eqref{eq:fhn_linstab1r} to the system 
\begin{align}
\begin{split}\label{eq:transformed_lin_fold_diag}
X_\zeta &=X\left(-2x-\frac{\lambda^2}{\theta_\mathrm{lf}c^3} +\mathcal{O}(x^2,y,\eps, \lambda x, \lambda^3)\right)+Y\left(1+\mathcal{O}(x,y,\eps,\lambda)\right),\\
Y_\zeta &=\mathcal{O}(\eps X,\eps Y),\\
Z_\zeta &=\left(-\frac{c}{\theta_\mathrm{lf}}+\mathcal{O}(x,y,\eps,\lambda)\right)Z,
\end{split}
\end{align}
where the wave train $(x,y)(\zeta)=(x_\eps,y_\eps)(\zeta/\theta_\mathrm{lf})$ satisfies the equation 
\begin{align*}
\begin{split}
x_\zeta &= -x^2+y+\mathcal{O}(xy, y^2, x^3,\eps),\\
y_\zeta &=\eps\left( 1+\mathcal{O}(x,y,\eps)\right)
\end{split}
\end{align*}
on the interval $\xi = \zeta/\theta_\mathrm{lf}\in \mathcal{I}_\mathrm{lf}$.
\end{lemma}
\begin{proof}
We first transform $\Psi= \mathcal{N}'_\eps(V_\eps(\xi))\tilde{\Phi}$, where $\tilde{\Phi} = (\tilde{X},\tilde{Y},\tilde{Z})^\top$, which results in the system
\begin{align}\label{eq:transformed_lin_fold}
\begin{split}
\tilde{\Phi}_\xi &= \left(\bar{A}(x_\eps(\xi),y_\eps(\xi);\eps)+\lambda \bar{B}(x_\eps(\xi),y_\eps(\xi);\eps)  \right)\tilde{\Phi}, 
\end{split}
\end{align}
where $\bar{A}(x,y;\eps) = \mathcal{G}'\left((x,y,0)^\top\right)$ and
\begin{align*}
\bar{B}(x,y;\eps)&=\left[\mathcal{N}_\eps'\left((x,y,0)^\top\right)\right]^{-1}B\mathcal{N}_\eps'\left((x,y,0)^\top\right)\\
&= \begin{pmatrix} b_{11}(x,y,\eps)&  \frac{\beta_1}{\beta_2c^2}+b_{12}(x,y,\eps) & -1+b_{13}(x,y,\eps)\\ \eps b_{21}(x,y,\eps) &\eps b_{22}(x,y,\eps)&\eps b_{23}(x,y,\eps)\\ \frac{1}{\beta_1 c}+b_{31}(x,y,\eps)&\frac{1}{c^2\beta_2}+ b_{32}(x,y,\eps)& -\frac{2}{c}+b_{33}(x,y,\eps) \end{pmatrix}
\end{align*}
with $b_{ij}(x,y,\eps) = \mathcal{O}(x,y,\eps)$. Since the wave-train solution is confined to the invariant manifold $\mathcal{W}^\mathrm{u}(\Gamma_\eps)$ given by the subspace $z=0$, we can restrict our attention to this subspace and note that solutions of the eigenvalue problem~\eqref{eq:transformed_lin_fold} are given by solutions of the system 
\begin{align}\label{eq:transformed_lin_fold_red}
\begin{split}
\tilde{\Phi}_\xi &= \left(\bar{A}(x,y;\eps)+\lambda \bar{B}(x,y;\eps)  \right)\tilde{\Phi}, 
\end{split}
\end{align}
where $(x,y)=\left(x_\eps(\xi),y_\eps(\xi)\right)$ satisfies
\begin{align*}
\begin{split}
x_\xi &= g_1(x,y;\eps),\\
y_\xi&= g_2(x,y;\eps).
\end{split}
\end{align*}This reduction effectively factors out the local hyperbolic dynamics in the existence problem.

We now aim to similarly factor out the corresponding hyperbolic dynamics in the linearized $\tilde{\Phi}$-variables. Note that when $(x,y)=(x_\eps(\xi),y_\eps(\xi))$ and $\lambda=0$, the derivative $\tilde{\Phi} = V_\eps'(\xi)$ of the wave train satisfies the linearized equation $\tilde{\Phi}_\xi = \bar{A}(x,y;\eps)\tilde{\Phi}$, which has a normally attracting invariant subspace given by $\{\tilde{Z}=0\}$, foliated by strong stable fibers. In general, this manifold (and its stable foliation) persists as a  normally attracting invariant manifold in~\eqref{eq:transformed_lin_fold_red} for small $|\lambda|$. We therefore seek a coordinate transformation $(\tilde{X},\tilde{Y},\tilde{Z})\mapsto \Phi=(X,Y,Z)$, linear in the variables $\tilde{\Phi}=(\tilde{X},\tilde{Y},\tilde{Z})$, which shifts this manifold to the subspace $\{Z=0\}$ for small $|\lambda|$, and straightens its stable fibers. The goal of this transformation is to preserve the linearity of the vector field in the linearized coordinates, while block diagonalizing the system in order to separate the strong stable hyperbolic dynamics from the non-hyperbolic dynamics near the fold.

To achieve this, we first consider the alternative system
\begin{align}
\begin{split}\label{eq:eigenvalue_fold_aux}
V_\xi&= \mathcal{G}(V),\\
\tilde{W}_\xi &= \mathcal{G}(V+\tilde{W})-\mathcal{G}(V)+\lambda \bar{B}(x,y;\eps)\tilde{W}, 
\end{split}
\end{align}
where $V=(x,y,z)$, and $\tilde{W}=(\tilde{W}_1,\tilde{W}_2,\tilde{W}_3)$.
Note that linearizing the $\tilde{W}$-equation about the solution $(V,\tilde{W}) = (V_\eps,0)$ results in precisely the eigenvalue problem~\eqref{eq:transformed_lin_fold}. Considering~\eqref{eq:eigenvalue_fold_aux} in the invariant subspace $z=0$, when $\lambda=0$, there exists a normally attracting invariant submanifold given by the subspace $\tilde{W}_3=0$. For $\lambda \in R_1(\mu)$, for $\mu>0$ sufficiently small, we therefore obtain a $\lambda$-dependent local center manifold of~\eqref{eq:eigenvalue_fold_aux} which is contained in the subspace $z=0$~\cite{chow1994center}. We now perform a near-identity coordinate transformation $\tilde{W}= \mathcal{H}_{\eps,\lambda}(W,V) = W+\mathcal{O}(\lambda)$, where $W=(W_1,W_2,W_3)$ which transforms this center manifold to the subspace $W_3=0$ and straightens its strong stable fibers. A short computation shows that this transformation satisfies 
\begin{align}\label{eq:fold_linearized_transf_lambdapart}
    D_W\mathcal{H}_{\eps,\lambda}(0,V_\eps)= I+\lambda \left( \mathcal{H}_{ij}(x,y;\eps,\lambda)\right)_{i,j=1,2,3},
\end{align} where
\begin{align*}
H_{1j}(x,y;\eps,\lambda) &= \mathcal{O}\left(|x|+|y|+|\eps|+|\lambda|\right),\\
H_{2j}(x,y;\eps,\lambda) &= \mathcal{O}(\eps),\\
H_{31}(x,y;\eps,\lambda) &=\frac{1}{\beta_1c^2}+\mathcal{O}\left(|x|+|y|+|\eps|+|\lambda|\right),\\
H_{32}(x,y;\eps,\lambda) &= \frac{1}{\beta_2c^3}+\mathcal{O}\left(|x|+|y|+|\eps|+|\lambda|\right),\\
H_{33}(x,y;\eps,\lambda) &= 0
\end{align*}
for $j = 1,2,3$.
Linearizing the resulting system about the solution $(V,W) = (V_\eps,0)$ reduces to the linearized equations
\begin{align}
\begin{split}\label{eq:eq:transformed_lin_fold_diag_xi}
X_\xi &=\theta_\mathrm{lf}X\left(-2x-\frac{\lambda^2}{\theta_\mathrm{lf}c^3} +\mathcal{O}(x^2,y,\eps, \lambda x, \lambda^3)\right)+\theta_\mathrm{lf}Y\left(1+\mathcal{O}(x,y,\eps,\lambda)\right),\\
Y_\xi  &=\mathcal{O}(\eps X,\eps Y),\\
Z_\xi  &=\left(-c+\mathcal{O}(x,y,\eps,\lambda)\right)Z,
\end{split}
\end{align}
where
\begin{align*}
-\frac{f''(u_1)}{2\beta_1c} = -\frac{\beta_1}{\beta_2} = \theta_\mathrm{lf} \coloneqq  -\frac{(a^2-a+1)^{1/6}(u_1-\gamma f(u_1)-a)^{1/3}}{c}>0.
\end{align*} 
The corresponding transformation to obtain~\eqref{eq:eq:transformed_lin_fold_diag_xi} directly from~\eqref{eq:transformed_lin_fold} is therefore achieved by setting
\begin{align*}
\tilde{\Phi} &= \begin{pmatrix} X\\Y\\ Z\end{pmatrix} = H_{\eps,\lambda}(\xi)\Phi,
\end{align*}
where the matrix $H_{\eps,\lambda}(\xi) \coloneqq  D_W\mathcal{H}_{\eps,\lambda}(0,V_\eps(\xi))$. Returning to~\eqref{eq:fhn_linstab1r}, we therefore obtain a $C^k$-change of coordinates $\Psi=N_{\eps,\lambda}(\xi) \Phi$, linear in $\Phi$, where
\begin{align}\label{eq:fold_linearized_transf_def}
    N_{\eps,\lambda}(\xi) \coloneqq  \mathcal{N}_\eps'(V_\eps(\xi)) H_{\eps,\lambda}(\xi)
\end{align}
and a rescaling $\zeta = \theta_\mathrm{lf}\xi$ which transforms~\eqref{eq:fhn_linstab1r} to~\eqref{eq:transformed_lin_fold_diag}.
\end{proof}

Results analogous to Lemma~\ref{lem:fold_existence} and Lemma~\ref{lem:fold_transf} hold for the wave train $\Gamma_\eps$ and~\eqref{eq:fhn_linstab1r} near the upper fold $(u,w)=(\bar{u}_1, f(\bar{u}_1))$ on the interval $\mathcal{I}_\mathrm{uf}$.

\subsection{A simplified system}\label{sec:toymodel}
Motivated by the results of Lemma~\ref{lem:fold_transf}, we first ignore higher order terms and consider a simpler toy model for the eigenvalue problem in the $XYZ$-coordinates near the fold
\begin{align}\label{eq:stability_fold_toy}
\begin{split}
X_\zeta &= -2xX+Y-\frac{\lambda^2}{\theta_\mathrm{lf}c^3}X, \\
Y_\zeta &= 0,\\
Z_\zeta &= -\frac{c}{\theta_\mathrm{lf}}Z.
\end{split}
\end{align}
The wave-train solution $(x,y)(\zeta)=(x_\eps,y_\eps)(\zeta/\theta_\mathrm{lf})$ is assumed to satisfy the corresponding existence problem with higher order terms ignored
\begin{align}\label{eq:existence_fold_toy}
\begin{split}
x_\xi &= -x^2+y,\\
y_\xi &= \eps.
\end{split}
\end{align}
We abbreviate $\xi_\mathrm{in} = \xi_{\mathrm{lf},\eps,\nu}^\mathrm{in}$, $\xi_\mathrm{out} = \xi_{\mathrm{lf},\eps,\nu}^{\mathrm{out},L}$ and define $\zeta_\mathrm{in} = \theta_\mathrm{lf}\xi_\mathrm{in}, \zeta_\mathrm{out} = \theta_\mathrm{lf}\xi_\mathrm{out}$, and we let $x_\mathrm{in}\coloneqq x(\zeta_\mathrm{in}), y_\mathrm{out}=y(\zeta_\mathrm{out})$, noting that $x_\mathrm{in}, y_\mathrm{out}>0$ are bounded away from zero independently of $\eps$ by Lemma~\ref{lem:fold_existence}. We consider the problem of solving~\eqref{eq:stability_fold_toy} subject to the boundary conditions
\begin{align*}
    X\left(\zeta_\mathrm{out}\right) &= X_\mathrm{out}, \qquad Y\left(\zeta_\mathrm{out}\right) = \eps Y_\mathrm{out}, \qquad 
Z\left(\zeta_\mathrm{in}\right)= Z_\mathrm{in},
\end{align*}
where $X_\mathrm{out}, Y_\mathrm{out}, Z_\mathrm{in}\in \mathbb{C}$. The choice of scaling the $Y$-coordinate by $\eps$ is for convenience (see below). From the form of the equations, we immediately see that $Y(\zeta) \equiv \eps Y_\mathrm{out}$, while $Z(\zeta)=\smash{Z_\mathrm{in} \re^{-c(\zeta-\zeta_\mathrm{in})/\theta_{\mathrm{lf}}}}$; in particular the stable $Z$-dynamics are completely decoupled from the center-unstable $XY$-dynamics, in which $Y$ remains constant. Hence it remains to determine $X(\zeta)$.

By the discussion in~\S\ref{sec:foldsetup}, the wave-train solution is exponentially close to the perturbed manifold $\mathcal{M}^{\mathrm{l},+}_\eps$ obtained by tracking the perturbed critical manifold $\mathcal{M}^\mathrm{l}_0$ backwards around the fold. In the simplified system~\eqref{eq:existence_fold_toy}, $\mathcal{M}^\mathrm{l}_0$  is simply given by the branch of the parabola $y=x^2$ in the region $x<0$. Thus $x(\zeta)$ is described by the equation
\begin{align*}
    x_\zeta = -x^2 +y_\mathrm{out}+\eps(\zeta-\zeta_\mathrm{out}).
\end{align*}
Up to exponentially small errors, we obtain that
\begin{align*}
\begin{split}
    x(\zeta) &= \eps^{1/3} x_R\left(\frac{y_\mathrm{out}}{\eps^{2/3}}+\eps^{1/3}\left(\zeta-\zeta_\mathrm{out}\right)\right),\\
    y(\zeta) &= y_\mathrm{out}+\eps(\zeta-\zeta_\mathrm{out})
    \end{split}
\end{align*}
is the solution of~\eqref{eq:existence_fold_toy} corresponding to $\mathcal{M}^{\mathrm{l},+}_\eps$; here the bijective function $x_R \colon (-\Omega_0,\infty) \to \R$ is the unique monotonically increasing solution of the Riccati equation
\begin{align}\label{eq:riccati0}
x_s = -x^2+s
\end{align}
that satisfies
\begin{align}
\begin{split}\label{eq:riccati_conditions}
s&\sim x_R^2 +\frac{1}{2x_R}+\mathcal{O}\left( \frac{1}{x_R^3} \right),\qquad x_R\to-\infty,\\
s&\sim -\Omega_0+\frac{1}{x_R}+\mathcal{O}\left( \frac{1}{x_R^3} \right),\qquad x_R\to\infty,
\end{split}
\end{align}
where $\Omega_0$ is the smallest positive zero of $\bar{J}(s)\coloneqq \smash{J_{-\frac{1}{3}}(2s^{3/2}/3)+J_{\frac{1}{3}}(2s^{3/2}/3)}$, where $\smash{J_{\pm\frac{1}{3}}}$ are Bessel functions of the first kind~\cite{krupaszmolyan2001}. By writing $x=\tfrac{\phi_s}{\phi}$, we can re-express~\eqref{eq:riccati0} as the Airy equation
\begin{align*}
    \phi_{ss} = s\phi,
\end{align*}
where the conditions~\eqref{eq:riccati_conditions} ensure that $\phi(s) = \mathrm{Ai}(s)$, the Airy function satisfying $\mathrm{Ai}(s)\to0$ as $s\to\infty$, where $-\Omega_0$ is the first  zero of $\mathrm{Ai}(s)$.

We note that $Y(\zeta) = \eps Y_\mathrm{out} = Y_\mathrm{out} y'(\zeta)$, and set
\begin{align*}
    X(\zeta) = Y_\mathrm{out} x'(\zeta) +\tilde{X}(\zeta),
\end{align*}
which results in the equation
\begin{align*}
\tilde{X}_\zeta &= -2x\tilde{X}-\frac{\lambda^2}{\theta_\mathrm{lf}c^3}\tilde{X}-\frac{\lambda^2}{\theta_\mathrm{lf}c^3}Y_\mathrm{out} x'(\zeta),
\end{align*}
from which we obtain the solution
\begin{align*}
\tilde{X}(\zeta) &= X_\mathrm{out}\re^{-\int_{\zeta_\mathrm{out}}^\zeta2x(\tilde{\zeta})+\frac{\lambda^2}{\theta_\mathrm{lf}c^3}\mathrm{d}\tilde{\zeta} } -\frac{\lambda^2}{\theta_\mathrm{lf}c^3}Y_\mathrm{out} \int_{\zeta_\mathrm{out}}^\zeta \re^{\int_{-\tilde{\zeta}}^\zeta2x(\bar{\zeta})+\frac{\lambda^2}{\theta_\mathrm{lf}c^3}\mathrm{d}\bar{\zeta} }x'(\tilde{\zeta})\mathrm{d}\tilde{\zeta}.
\end{align*}
 Letting $s\coloneqq \frac{y_\mathrm{out}}{\eps^{2/3}}+\eps^{1/3}(\zeta-\zeta_\mathrm{out})$ and similarly for $\tilde{\zeta}, \bar{\zeta}$, with
\begin{align}\label{eq:toy_sbounds}
(s_\mathrm{in},s_\mathrm{out})\coloneqq \left(\frac{y_\mathrm{out}}{\eps^{2/3}}+\eps^{1/3}(\zeta_\mathrm{in}-\zeta_\mathrm{out}), \frac{y_\mathrm{out}}{\eps^{2/3}}\right),
\end{align}
and using the fact that $x_R(s) = \tfrac{\mathrm{Ai}'(s)}{\mathrm{Ai}(s)}$, we have that
\begin{align*}
\tilde{X}(\zeta_\mathrm{in}) &= X_\mathrm{out}\frac{\mathrm{Ai}(s_\mathrm{out})^2}{\mathrm{Ai}(s_\mathrm{in})^2}\re^{\frac{\lambda^2}{\eps^{1/3}\theta_\mathrm{lf}c^3}\left(s_\mathrm{out}-s_\mathrm{in}\right)} \\
&\qquad +\frac{\lambda^2\eps^{1/3}}{\theta_\mathrm{lf}c^3 \mathrm{Ai}(s_\mathrm{in})^2}Y_\mathrm{out} \int_{s_\mathrm{in}}^{s_\mathrm{out}} \re^{\frac{\lambda^2}{\eps^{1/3}\theta_\mathrm{lf}c^3 }\left(s-s_\mathrm{in}  \right) }\left(s\mathrm{Ai}(s)^2-\mathrm{Ai}'(s)^2\right)\mathrm{d}s.
\end{align*}
Using~\eqref{eq:riccati_conditions} and~\eqref{eq:toy_sbounds}, Lemma~\ref{lem:fold_existence}, and properties of the Airy function $\mathrm{Ai}$ (see Lemma~\ref{lem:airyfunction}\ref{lem:airybound}) we have that 
\begin{align*}
    |s_\mathrm{in}+\Omega_0|\leq C_\nu\eps^{1/3}
\end{align*}
and
\begin{align*}
\tilde{X}(\zeta_\mathrm{in}) &= X_\mathrm{out}\frac{x'(\zeta_\mathrm{in})}{x'(\zeta_\mathrm{out})}\frac{s_\mathrm{out}\mathrm{Ai}(s_\mathrm{out})^2-\mathrm{Ai}'(s_\mathrm{out})^2 }{s_\mathrm{in}\mathrm{Ai}(s_\mathrm{in})^2-\mathrm{Ai}'(s_\mathrm{in})^2 }\re^{\frac{\lambda^2}{\eps^{1/3}\theta_\mathrm{lf}c^3}\left(s_\mathrm{out}-s_\mathrm{in}\right)} \\
&\qquad +\frac{\lambda^2Y_\mathrm{out}x'(\zeta_\mathrm{in})}{\eps^{1/3}\theta_\mathrm{lf}c^3 \left(s_\mathrm{in}\mathrm{Ai}(s_\mathrm{in})^2-\mathrm{Ai}'(s_\mathrm{in})^2  \right)} \int_{s_\mathrm{in}}^{s_\mathrm{out}} \re^{\frac{\lambda^2}{\eps^{1/3}\theta_\mathrm{lf}c^3 }\left(s-s_\mathrm{in}  \right) }\left(s\mathrm{Ai}(s)^2-\mathrm{Ai}'(s)^2\right)\mathrm{d}s,
\end{align*}
We deduce that, for $\lambda$ satisfying $|\lambda \eps^{-1/6}|\ll1$ or $\Re (\lambda^2)<0$ (or equivalently, $|\Re (\lambda)|< |\Im (\lambda)|$), there exist $C_\nu, \vartheta_\nu>0$ such that
\begin{align*}
\left|\tilde{X}(\zeta_\mathrm{in})+\Upsilon_\mathrm{lf}\left(\lambda\eps^{-1/6}\right)Y_\mathrm{out}x'(\zeta_\mathrm{in})\right| &\leq  C_\nu \left(\re^{-\vartheta_\nu/\eps} |X_\mathrm{out}|+|\lambda|^2|Y_\mathrm{out}|\right)\\
&\qquad +\, C_\nu (\eps^{1/3}+|\lambda|^2)\left|I_\mathrm{lf}\left(\lambda\eps^{-1/6}\right)Y_\mathrm{out}x'(\zeta_\mathrm{in})\right|,
\end{align*}
where
\begin{align*}
    \Upsilon_\mathrm{lf}(z)\coloneqq \frac{z^2}{\theta_\mathrm{lf}c^3\mathrm{Ai}'(-\Omega_0)^2 } \int_{-\Omega_0}^{\infty} \re^{\frac{z^2}{\theta_\mathrm{lf}c^3 }\left(s+\Omega_0\right) }\left(s\mathrm{Ai}(s)^2-\mathrm{Ai}'(s)^2\right)\mathrm{d}s.
\end{align*}
The properties of the function $\Upsilon_\mathrm{lf}(z)$ are described in Appendix~\ref{app:airy}.

Returning to the boundary value problem associated with~\eqref{eq:stability_fold_toy}, given exit conditions $X_\mathrm{out}, Y_\mathrm{out}\in \mathbb{C}$, we therefore obtain a solution of~\eqref{eq:stability_fold_toy} satisfying $Y(\zeta)\equiv \eps Y_\mathrm{out}$, $X(\zeta_\mathrm{out})=X_\mathrm{out}$, and corresponding entry conditions at $\zeta=\zeta_\mathrm{in}$
\begin{align}
\begin{split}\label{eq:toy_solution}
X(\zeta_\mathrm{in}) &= Y_\mathrm{out}x'(\zeta_\mathrm{in})\left(1-\Upsilon_\mathrm{lf}\left(\lambda \eps^{-1/6}\right)   \right) + \text{h.o.t.}\\
Y(\zeta_\mathrm{in}) &= Y_\mathrm{out}\eps = Y_\mathrm{out}y'(\zeta_\mathrm{in}).
\end{split}
\end{align}
In effect, at $\zeta=\zeta_\mathrm{in}$ the corresponding solution to the linearized problem~\eqref{eq:stability_fold_toy} resembles a multiple of the derivative $(x',y')(\zeta)$ (which solves~\eqref{eq:stability_fold_toy} for $\lambda=0$), up to a leading-order correction in $X(\zeta)$ described by the function $\Upsilon_\mathrm{lf}\left(\lambda\eps^{-1/6}\right)$ due to the presence of the $\lambda^2 X$ term in~\eqref{eq:stability_fold_toy}. In the next section, we return to~\eqref{eq:transformed_lin_fold_diag} and focus on the center-unstable $XY$-dynamics. Guided by~\eqref{eq:toy_solution}, we solve an analogous boundary value problem in the full $XY$ dynamics of~\eqref{eq:transformed_lin_fold_diag} including all higher-order terms, and we precisely characterize the resulting estimates on the solution.

\subsection{Blow-up analysis of the center-unstable
\texorpdfstring{$XY$}{XY}-dynamics}\label{sec:fold_blowup_CU}
Guided by the analysis in~\S\ref{sec:toymodel}, we now consider the full $XY$-dynamics. Coupling the existence problem to the eigenvalue problem through the fold, rescaling $\theta_\mathrm{lf}\xi =\zeta$, and restricting to the invariant subspace $Z=0$, we have the following
\begin{align}
\begin{split}\label{eq:stability_fold_XY}
x_\zeta &= -x^2+y+\mathcal{O}(xy, y^2, x^3,\eps),\\
y_\zeta &= \eps\left( 1+\mathcal{O}(x,y,\eps)\right),\\
X_\zeta &=X\left(-2x-\frac{\lambda^2}{\theta_\mathrm{lf}c^3} +\mathcal{O}(x^2,y,\eps, \lambda x, \lambda^3)\right)+Y\left(1+\mathcal{O}(x,y,\eps,\lambda)\right),\\
Y_\zeta &=\mathcal{O}(\eps X,\eps Y).
\end{split}
\end{align}
Here, we note that the wave-train solution $(x_\eps,y_\eps)(\zeta/\theta_\mathrm{lf})$ satisfies the first two equations on the rescaled interval $\zeta\in [\zeta_\mathrm{in}, \zeta_\mathrm{out}]$ corresponding to $\xi\in \mathcal{I}_\mathrm{lf}$,  where we abbreviate $\smash{\xi_\mathrm{in}:=\xi_{\mathrm{lf},\eps,\nu}^\mathrm{in},  \xi_\mathrm{out}:=\xi_{\mathrm{lf},\eps,\nu}^{\mathrm{out},L}}$ and write $\zeta_\mathrm{in}\smash{=\theta_\mathrm{lf}\xi_\mathrm{in}, \zeta_\mathrm{out}=\theta_\mathrm{lf}\xi_\mathrm{out}}$. The derivative $(x,y)'(\zeta)$ satisfies the linearized equations in the last two component of~\eqref{eq:stability_fold_XY} when taking $(x,y)(\zeta) = (x_\eps,y_\eps)(\zeta/\theta_\mathrm{lf})$ and $\lambda=0$. Guided by the results of~\S\ref{sec:toymodel}, for $0<\mu<M$, we define $$\Lambda_\eps(\mu,M) =\Lambda_{\mathrm{r},\eps}(\mu,M)\cup \Lambda_{\mathrm{c},\eps}(\mu,M),$$ where
\begin{align*}
    \Lambda_{\mathrm{r},\eps}(\mu,M)&=\left\{\lambda\in \mathbb{C}: |\Re(\lambda)|\leq \mu \eps^{1/6}, |\Im(\lambda)|\leq M \eps^{1/6} \right \},\\
     \Lambda_{\mathrm{c},\eps}(\mu,M)&=\left\{\lambda\in \mathbb{C}: |\Re(\lambda)|\leq \mu M^{-1} |\Im(\lambda)|, M \eps^{1/6}\leq |\Im(\lambda)|\leq \mu \right\}
\end{align*}
define rectangular and (partial) cone-shaped regions, respectively, near $\lambda=0$. See Figure~\ref{fig:fold_regions} for a depiction of the regions $\Lambda_{\mathrm{r},\eps}(\mu,M)\cup \Lambda_{\mathrm{c},\eps}(\mu,M) = \Lambda_\eps(\mu,M)$. Note that the union $\Lambda_\eps(\mu,M)$ contains the subregion $R_{1,\eps}(\mu)=\left\{\lambda \in R_1(\mu): |\Re(\lambda)|\leq \mu \eps^{1/6}\right\}$ of Proposition~\ref{prop:mainformula}, and in particular contains the union of the three subregions $R_{1,i,\eps}, i=1,2,3$ analyzed in~\S\ref{sec:mainformula}. We recall the function
\begin{align*}
\Upsilon_\mathrm{lf}(z)=\frac{z^2}{\theta_\mathrm{lf}c^3}\frac{1}{\mathrm{Ai}'(-\Omega_0)^2}\int^\infty_{-\Omega_0}\re^{\frac{z^2}{\theta_\mathrm{lf} c^3}\left( s+\Omega_0\right)   } \left(s\mathrm{Ai}(s)^2-\mathrm{Ai}'(s)^2\right)\mathrm{d}s,
\end{align*}
where $\mathrm{Ai}(s)$ denotes the Airy function and where $-\Omega_0$ is largest zero of $\mathrm{Ai}(s)$, or equivalently $\Omega_0$ is the smallest positive zero of $\bar{J}(z)\coloneqq \smash{J_{-\frac{1}{3}}(2z^{3/2}/3)+J_{\frac{1}{3}}(2z^{3/2}/3)}$, where $\smash{J_{\pm\frac{1}{3}}}$ are Bessel functions of the first kind. We have the following.

\begin{proposition}\label{prop:fold_bvp}
Fix $M>0$. There exists a continuous function $\eta \colon [0,\infty) \to [0,\infty)$ with $\eta(0) = 0$  such that, provided $0 < \eps \ll \nu \ll \mu \ll 1$, for each $(X_\mathrm{out},Y_\mathrm{out})\in \mathbb{C}^2$ and $\lambda\in \Lambda_\eps(\mu,M)$, there exists a solution $\psi_\mathrm{lf}^0=(x,y,X_\mathrm{lf}^0,Y_\mathrm{lf}^0) \colon \left[\zeta_\mathrm{in}, \zeta_\mathrm{out} \right] \to \R^2 \times \C^2$ of~\eqref{eq:stability_fold_XY} with $(x,y)(\zeta) = (x_\eps,y_\eps)(\zeta/\theta_{\mathrm{lf}})$, which satisfies
\begin{align*}
X_\mathrm{lf}^0\left(\zeta_\mathrm{out}\right) &= X_\mathrm{out},\\
 Y_\mathrm{lf}^0\left(\zeta_\mathrm{out}\right) &= \eps Y_\mathrm{out}.
\end{align*}
Moreover, there exist $(\eps,\lambda)$-independent constants $C_\nu, \vartheta_\nu > 0$ such that the solution satisfies the following estimates
\begin{align*}
\left|X_\mathrm{lf}^0\left(\zeta_\mathrm{in}\right)-\alpha^x_{\mathrm{lf},\nu}(Y_\mathrm{out}; \eps,\lambda)x_\eps'\left(\zeta_\mathrm{in}/\theta_\mathrm{lf}\right)\right| &\leq  \mathcal{R}^x_{\mathrm{lf},\nu}(X_\mathrm{out}, Y_\mathrm{out};\eps, \lambda), \\
\left|Y_\mathrm{lf}^0\left(\zeta_\mathrm{in}\right)-\eps Y_\mathrm{out}\frac{y_\eps'\left(\zeta_\mathrm{in}/\theta_\mathrm{lf}\right)}{y_\eps'\left(\zeta_\mathrm{out}/\theta_\mathrm{lf}\right)}\right|  &\leq  \mathcal{R}^y_{\mathrm{lf},\nu}(X_\mathrm{out}, Y_\mathrm{out};\eps,\lambda) , 
\end{align*}
where $\alpha^x_{\mathrm{lf},\nu}$ and $\mathcal{R}^x_{\mathrm{lf},\nu}, \mathcal{R}^y_{\mathrm{lf},\nu}$ satisfy
\begin{align*}
\alpha_{\mathrm{lf},\nu}^x(Y_\mathrm{out}; \eps,\lambda)
&=\frac{\eps Y_\mathrm{out}}{y_\eps'\left(\zeta_\mathrm{out}/\theta_\mathrm{lf}\right)}\left(1- \Upsilon_\mathrm{lf}\left(\lambda \eps^{-1/6}\right)+\mathcal{O}\left(\eta(\nu)\Upsilon_\mathrm{lf}\left(\lambda \eps^{-1/6}\right)\right)  \right),\\
\mathcal{R}^x_{\mathrm{lf},\nu}(X_\mathrm{out}, Y_\mathrm{out}; \eps,\lambda)&=\begin{cases}C_\nu\left(\left|X_\mathrm{out}\right|+ \left(\eps+|\lambda \log \eps|\right)\left|Y_\mathrm{out}\right|\right),& \lambda \in \Lambda_{\mathrm{r},\eps}(\mu, M),\\
C_\nu\left(|X_\mathrm{out}|+(\eps+| \lambda \log |\lambda||)|Y_\mathrm{out}|\right), &\lambda \in \Lambda_{\mathrm{c},\eps}(\mu, M),
\end{cases}\\
\mathcal{R}^y_{\mathrm{lf},\nu}(X_\mathrm{out}, Y_\mathrm{out}; \eps,\lambda)&=\begin{cases} C_\nu \eps\left|X_\mathrm{out}\right|+ \left( C_\nu\eps \left(\eps+|\lambda \log \eps|\right) +\eps \eta(\nu)\left|\Upsilon_\mathrm{lf}\left(\lambda \eps^{-1/6}\right)\right|\right)\left|Y_\mathrm{out}\right|, & \lambda \in \Lambda_{\mathrm{r},\eps}(\mu, M),\\
C_\nu \eps |X_\mathrm{out}|+(C_\nu\eps  \left(\eps+|\lambda \log |\lambda||\right)+\eps \eta(\nu))|Y_\mathrm{out}|, & \lambda \in \Lambda_{\mathrm{c},\eps}(\mu, M),
\end{cases}
\end{align*}
\end{proposition}

\begin{figure}
\centering
\includegraphics[width=0.45\linewidth]{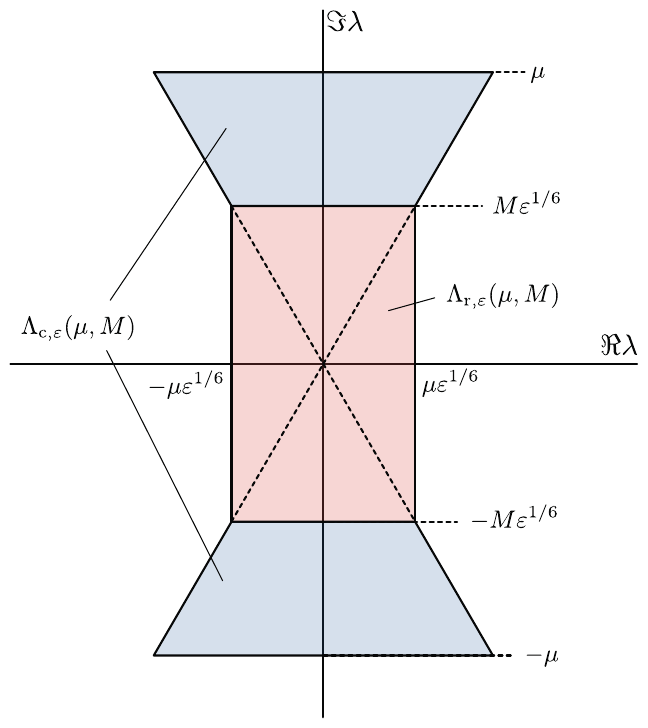}\hspace{0.04\linewidth }\includegraphics[width=0.49\linewidth]{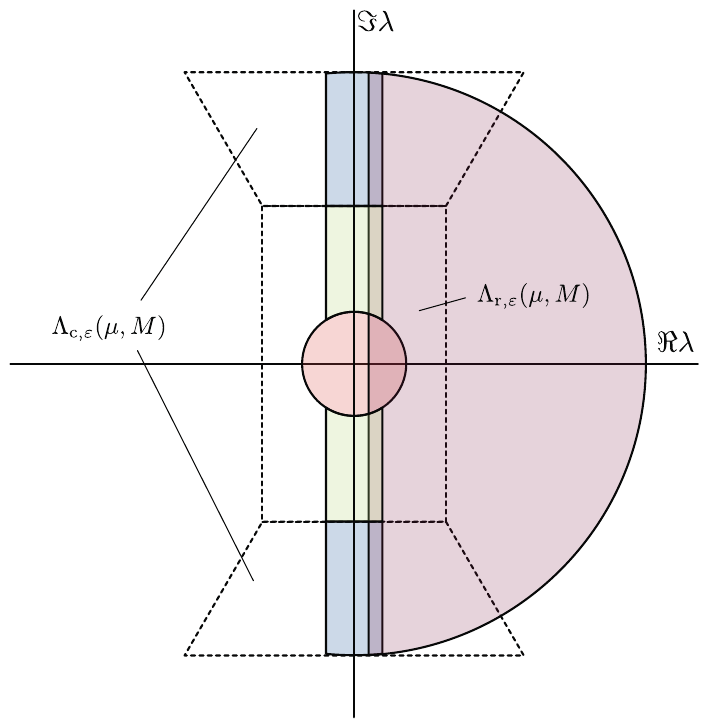}
\caption{(Left) Depicted are the regions $\Lambda_{\mathrm{r},\eps}(\mu, M)\cup\Lambda_{\mathrm{c},\eps}(\mu,M) = \Lambda_\eps(\mu,M)$. (Right) Shown are the regions $\Lambda_{\mathrm{r},\eps}(\mu, M)$ and $ \Lambda_{\mathrm{c},\eps}(\mu,M)$ overlaid on the (unlabeled) regions $R_{1,1,\eps}(\mu),R_{1,2,\eps}(\varsigma,\mu,M), R_{1,3,\eps}(\varsigma,\mu,M),$ and $ R_{1,4,\eps}(\mu) $ from~\S\ref{sec:mainformula}; compare with Figure~\ref{fig:R1_regions}.}
\label{fig:fold_regions}
\end{figure}

To solve the full eigenvalue problem~\eqref{eq:stability_fold_XY} when $(x,y)(\zeta) = (x_\eps,y_\eps)(\zeta/\theta_\mathrm{lf})$ and $\lambda\neq0$, we use blow-up methods as in the existence problem; see~\S\ref{sec:existence_overview}. In the existence problem, as with the blow-up of a canonical fold point~\cite{krupaszmolyan2001}, the strategy is to append the equation $\eps_\zeta=0$ and perform a quasi-homogeneous spherical blow-up transformation for the three coordinates $(x,y,\eps)$ given by
\begin{align*}
    (\bar{x},\bar{y},\bar{\eps},\bar{r})\mapsto (x,y,\eps)=(\bar{r}\bar{x}, \bar{r}^2\bar{y}, \bar{r}^3 \bar{\eps}),
\end{align*}
where $(\bar{x},\bar{y},\bar{\eps})\in S^2$. Here, we perform the same procedure for the coupled existence-stability problem~\eqref{eq:stability_fold_XY}, which necessitates simultaneously applying a blow-up transformation to $(X,Y,\lambda)$. Based on the analysis of~\S\ref{sec:toymodel}, we expect the linearized solution $(X,Y)$ to behave like the derivative $(x_\zeta, y_\zeta)$ of the solution to the existence problem, hence the weights for $(X,Y)$ are chosen accordingly. Again inspired by the analysis in~\S\ref{sec:toymodel}, we further choose the weight for $\lambda$ to retain the anticipated leading-order $\lambda^2 X$-term in the $X$-equation in~\eqref{eq:stability_fold_XY}. We correspondingly adjust the scalings  from the existence problem to avoid fractional powers of the scaling variable in each chart, due to the appearance of $\lambda^2$ in the linearized equations. This results in the blow-up transformation
\begin{align*}
    \left(\bar{x},\bar{y},\bar{\eps}, \bar{X}, \bar{Y}, \bar{\lambda} \bar{r}\right)\mapsto \left(x,y,\eps, X, Y, \lambda\right)=\left(\bar{r}^2\bar{x}, \bar{r}^4\bar{y}, \bar{r}^6 \bar{\eps}, \bar{r}^4 \bar{X},\bar{r}^6 \bar{Y}, \bar{r}\bar{\lambda} \right),
\end{align*}
where $(\bar{x},\bar{y},\bar{\eps})\in S^2$ and $(\bar{X}, \bar{Y}, \bar{\lambda})\in \mathbb{C}^3$.

The solution must then be tracked through the three coordinate charts $\mathcal{K}_i,i=1,2,3$ as in the existence problem; see~\cite{CASCH,CSosc,krupaszmolyan2001}.  An additional chart $\mathcal{K}_4$ will be needed to capture values of $\lambda \in \Lambda_{\mathrm{c},\eps}(\mu,M)$, corresponding to values of $\lambda$ outside a small neighborhood of $\mathcal{O}(\eps^{1/6})$ up to small $\mathcal{O}(1)$ values of $\Im(\lambda)$, independent of $\eps$.

\begin{figure}
\centering
\includegraphics[width=0.6\linewidth]{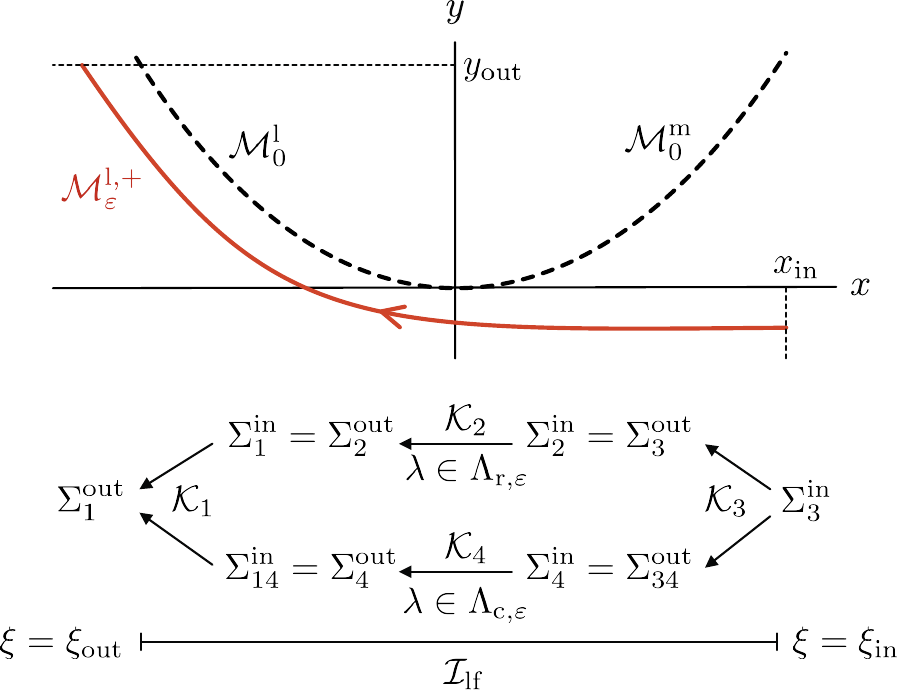}
\caption{Shown is the setup for the proof of Proposition~\ref{prop:fold_bvp}, outlining the charts $\mathcal{K}_i,i =1,2,3,4$ where the solution is tracked to solve the boundary value problem for the center-unstable $(XY)$-dynamics on the interval $\xi\in\mathcal{I}_\mathrm{lf}$. The sections $\Sigma_*^{\mathrm{in}/\mathrm{out}}$ represent the entry/exit sections for solutions passing through the charts $\mathcal{K}_i, i=1,2,3,4,$ which are analyzed in~\S\ref{sec:chartK1}--\S\ref{sec:chartK3}.  \label{fig:foldcharts}}
\end{figure}

In the chart $\mathcal{K}_1$, the transformation takes the form
\begin{align}\label{eq:chartK1_coords}
x=r_1^2x_1, \quad y = r_1^4,\quad \eps = r_1^6\eps_1, \quad X = r_1^4X_1, \quad Y=r_1^6Y_1, \quad \lambda = r_1\lambda_1.
\end{align}
 Similarly in the chart $\mathcal{K}_2$, the transformation takes the form
\begin{align}\label{eq:chartK2_coords}
x=r_2^2x_2, \quad y = r_2^4y_2,\quad \eps = r_2^6, \quad X = r_2^4X_2, \quad Y=r_2^6Y_2, \quad \lambda = r_2\lambda_2
\end{align}
and in the chart $\mathcal{K}_3$, the transformation takes the form
\begin{align}\label{eq:chartK3_coords}
x=r_3^2, \quad y = r_3^4y_3,\quad \eps = r_3^6\eps_3, \quad X = r_3^4X_3, \quad Y=r_3^6Y_3, \quad \lambda = r_3\lambda_3.
\end{align}
Finally in the chart $\mathcal{K}_4$, we have the transformation
\begin{align}\label{eq:chartK4_coords}
x_4=r_4^2x_4, \quad y = r_4^4y_4,\quad \eps = r_4^6\eps_4, \quad X = r_4^4X_4, \quad Y=r_4^6Y_4, \quad \lambda = r_4(\lambda_4+\ri).
\end{align}
Note that this chart amounts to a rescaling by the real quantity $r_4\coloneqq \Im(\lambda)$.

Based on the existence analysis in~\cite{CASCH} (and the corresponding theory for passage through a generic fold point~\cite{krupaszmolyan2001}), the wave train $(x,y)(\zeta) = (x_\eps,y_\eps)(\zeta/\theta_\mathrm{lf})$ passes through each of the three charts, $\mathcal{K}_i,i=1,2,3$. We denote by 
\begin{align*}
I_i : = [\zeta_{\mathrm{in},j}, \zeta_{\mathrm{out},j}],\qquad i=1,2,3
\end{align*}
the $\zeta$-intervals over which the wave train lies within each chart. Since we will refer to the geometric setup for known results concerning passage through a nondegenerate fold point~\cite{krupaszmolyan2001, CSosc}, we have chosen the labelling of charts to correspond to the notation in those works, though due to the direction of the flow (see Figures~\ref{fig:fold_existence} and~\ref{fig:foldcharts}), the wave train actually passes through each of the charts in the order $\mathcal{K}_3\rightarrow \mathcal{K}_2 \rightarrow \mathcal{K}_1$, so that  that $\zeta_{\mathrm{in},3}=\zeta_\mathrm{in}$ and $\zeta_{\mathrm{out},1}=\zeta_\mathrm{out}$ and $[\zeta_\mathrm{in}, \zeta_\mathrm{out}] = I_3\cup I_2\cup I_1$. For values of $\lambda\in\Lambda_{\mathrm{r},\eps}(\mu,M)$, the eigenvalue problem can be analyzed using these three charts; however, values of $\lambda\in \Lambda_{\mathrm{c},\eps}(\mu,M)$ will require a detour via the chart $\mathcal{K}_4$, instead of the chart $\mathcal{K}_2$; see Figure~\ref{fig:foldcharts}.

We will also make use of the fact that the derivative of the wave train, which we denote by $(X_\eps,Y_\eps)(\zeta)\coloneqq (x,y)'(\zeta)$ where $(x,y)(\zeta) = (x_\eps,y_\eps)(\zeta/\theta_\mathrm{lf})$ satisfies the linearized equations when $\lambda=0$. Motivated by the analysis in~\S\ref{sec:toymodel} and the anticipated form of the solution~\eqref{eq:toy_solution}, we write $(X,Y) = (\alpha X_\eps,\alpha Y_\eps)+(\tilde{X},\tilde{Y})$ for $\alpha$ to be determined, and obtain the system
\begin{align}
\begin{split}\label{eq:fold_centerXY_dynamics_alpha}
x_\zeta &= -x^2+y+\mathcal{O}(xy, y^2, x^3,\eps),\\
y_\zeta &= \eps\left( 1+\mathcal{O}(x,y,\eps)\right),\\
X_\zeta &=X\left(-2x-\frac{\lambda^2}{\theta_\mathrm{lf}c^3} +\mathcal{O}(x^2,y,\eps, \lambda x, \lambda^3)\right)+Y\left(1+\mathcal{O}(x,y,\eps,\lambda)\right),\\
&\qquad +\alpha X_\eps\left(-\frac{\lambda^2}{\theta_\mathrm{lf}c^3}+\mathcal{O}\left(\lambda (|x|+|y|+|\eps|+|\lambda|^2  \right)\right)+\mathcal{O}(\alpha \lambda Y_\eps),\\
Y_\zeta &=\mathcal{O}(\eps X,\eps Y, \alpha \eps \lambda X_\eps, \alpha \eps \lambda Y_\eps),
\end{split}
\end{align}
where we have dropped the tildes on $(X,Y)$. 

In the remainder of this section, we analyze~\eqref{eq:fold_centerXY_dynamics_alpha} by deriving estimates satisfied by solutions of boundary value problems in each of the charts $\mathcal{K}_1$~(\S\ref{sec:chartK1}), $\mathcal{K}_2$~(\S\ref{sec:chartK2}), $\mathcal{K}_4$~(\S\ref{sec:chartK4}), and $\mathcal{K}_3$~(\S\ref{sec:chartK3}). These solutions are then matched to obtain a solution of~\eqref{eq:stability_fold_XY} satisfying the estimates of Proposition~\ref{prop:fold_bvp}, the proof of which is presented in~\S\ref{sec:fold_bvp_proof}. 

\subsubsection{Chart \texorpdfstring{$\mathcal{K}_1$}{K1}}\label{sec:chartK1}
We begin in the chart $\mathcal{K}_1$, which upon applying the coordinate transformation~\eqref{eq:chartK1_coords}, results in the system 
\begin{align*}
\begin{split}
(r_1)_\zeta &= \frac{1}{4}r_1^3\eps_1F_1(x_1,r_1,\eps_1),\\
(x_1)_\zeta &= r_1^2\left(1-x_1^2-\frac{1}{2}\eps_1x_1+\mathcal{O}(r_1^2)\right),\\
(\eps_1)_\zeta &= -\frac{3}{2}r_1^2\eps_1^2F_1(x_1,r_1,\eps_1),\\
(X_1)_\zeta &=r_1^2X_1\left(-2x_1-\frac{\lambda_1^2}{\theta_\mathrm{lf}c^3}-\eps_1 +\mathcal{O}(r_1\lambda_1, r_1^2)\right)+r_1^2Y_1\left(1+\mathcal{O}(r_1\lambda_1, r_1^2)\right)\\
&\qquad +\alpha \left(\frac{1}{2}r_1^2\eps_1F_1(x_1,r_1,\eps_1)x_1+(x_1)_\zeta \right)\left(-\frac{\lambda_1^2}{\theta_\mathrm{lf}c^3} +\mathcal{O}(r_1\lambda_1)\right)+ \mathcal{O}(\alpha r_1^3\eps_1\lambda_1),\\
(Y_1)_\zeta &=-\frac{3}{2}r_1^2\eps_1Y_1F_1(x_1,r_1,\eps_1)+\mathcal{O}(r_1^4\eps_1X_1,r_1^6\eps_1Y_1, \alpha r_1\eps_1\lambda_1X_\eps,\alpha r_1\eps_1\lambda_1Y_\eps), \\
(\lambda_1)_\zeta&= -\frac{1}{4}r_1^2\eps_1\lambda_1F_1(x_1,r_1,\eps_1),
\end{split}
\end{align*}
where $F_1(x_1,r_1,\eps_1)=1+\mathcal{O}(r_1^2)$. We now desingularize the system via the rescaling $\mathrm{d}z_1 = r_1^2 \mathrm{d} \zeta$ and obtain
\begin{align*}
\begin{split}
r_1' &= \frac{1}{4}r_1\eps_1F_1(x_1,r_1,\eps_1),\\
x_1' &=1-x_1^2-\frac{1}{2}\eps_1x_1+\mathcal{O}(r_1^2),\\
\eps_1' &= -\frac{3}{2}\eps_1^2F_1(x_1,r_1,\eps_1),\\
X_1' &=X_1\left(-2x_1-\frac{\lambda_1^2}{\theta_\mathrm{lf}c^3}-\eps_1 +\mathcal{O}(r_1\lambda_1, r_1^2)\right)+Y_1\left(1+\mathcal{O}(r_1\lambda_1, r_1^2)\right)\\
&\qquad +\alpha\left(\frac{1}{2}\eps_1F_1(x_1,r_1,\eps_1)x_1+x_1' \right)\left(-\frac{\lambda_1^2}{\theta_\mathrm{lf}c^3} +\mathcal{O}(r_1\lambda_1)\right)+ \mathcal{O}(\alpha r_1\eps_1\lambda_1),\\
Y_1' &=-\frac{3}{2}\eps_1Y_1F_1(x_1,r_1,\eps_1)+\mathcal{O}(r_1^2\eps_1X_1,r_1^4\eps_1Y_1, \alpha r_1^{-1}\eps_1\lambda_1X_\eps,\alpha r_1^{-1}\eps_1\lambda_1Y_\eps),\\
\lambda_1'&= -\frac{1}{4}\eps_1\lambda_1F_1(x_1,r_1,\eps_1),
\end{split}
\end{align*}
where $'=\frac{\mathrm{d}}{\mathrm{d}z_1}$. Focusing first on the existence problem in the coordinates $(r_1,x_1,\eps_1)$, we rescrict attention to the flow between the sections $\Sigma_1^\mathrm{in}$ and $\Sigma_1^\mathrm{out}$ defined by
\begin{align*}
\Sigma_1^\mathrm{in}&=\{|x_1+1|\leq \bar{x}_1, 0\leq r_1\leq r_{1,\eps}(\delta), \eps_1= \delta  \},\\
\Sigma_1^\mathrm{out}&=\{|x_1+1|\leq \bar{x}_1, r_1= r_{1,\eps}(\eps_1), 0\leq\eps_1\leq \delta \},
\end{align*}
where $r_{1,\eps}(\eps_1)^4 = y_\mathrm{out}(\nu,\eps)$ satisfies $| r_{1,\eps}(\eps_1)-\bar{r}_{1}|\leq C_\nu\eps_1^{2/3}$ and where $\bar{r}_{1}^4:= r_{1,\eps}(0)^4= y_\mathrm{out}(\nu,0)>0$ by Lemma~\ref{lem:fold_existence}, and $\delta, \bar{x}_1>0$  are sufficiently small constants. In the $\mathcal{K}_1$-coordinates, the wave train is represented by a solution $(x_{1,\eps}, r_{1,\eps},\eps_{1,\eps})(z_1)$ which reaches these sections at values of $z_1=z_1^\mathrm{in},z_1^\mathrm{out}$, respectively. The fixed point $(x_1,r_1,\eps_1)=(-1,0,0)$ admits a center-unstable manifold $M_1^+$ given by the graph
\begin{align*}
    M_1^+ = \left\{x_1 =x_1^+(r_1,\eps_1)= -1+\mathcal{O}\left(r_1^2,\eps_1\right), 0\leq r_1\leq r_{1,\eps}(\eps_1),  0\leq\eps_1\leq \delta \right\}.
\end{align*}
The manifold $M_1^+$ is the representation of $\mathcal{M}^{\lr,+}$ in the $\mathcal{K}_1$-coordinates; see e.g.~\cite[Proposition 4.2]{CSosc}. Therefore, by Lemma~\ref{lem:fold_existence}, in between the sections $\Sigma_1^\mathrm{in}, \Sigma_1^\mathrm{out}$, the solution $(x_{1,\eps}, r_{1,\eps},\eps_{1,\eps})(z_1)$ remains $\mathcal{O}(\re^{-\vartheta_\nu/\eps})$-close to $M_1^+$. The corresponding eigenvalue problem can therefore be written as
\begin{align}
\begin{split}\label{eq:k1_transf}
X_1' &=X_1\left(2-\frac{\lambda_1^2}{\theta_\mathrm{lf}c^3} +\mathcal{O}\left(r_1\lambda_1, r_1^2, \eps_1\right)\right)+Y_1\left(1+\mathcal{O}\left(r_1\lambda_1, r_1^2\right)\right)\\
&\qquad +\alpha \left(-\frac{1}{2}\eps_1+\mathcal{O}\left(r_1^2\eps_1, \eps_1^2\right)\right)\left(-\frac{\lambda_1^2}{\theta_\mathrm{lf}c^3} +\mathcal{O}(r_1\lambda_1)\right)+ \mathcal{O}(\alpha r_1\eps_1\lambda_1),\\
Y_1' &=-\frac{3}{2}\eps_1Y_1\bar{F}_1(r_1,\eps_1)+\mathcal{O}(r_1^2\eps_1X_1,r_1^4\eps_1Y_1, \alpha r_1^3\eps_1^2\lambda_1),\\
\lambda_1'&= -\frac{1}{4}\eps_1\lambda_1\bar{F}_1(r_1,\eps_1)
\end{split}
\end{align}
over the interval $z_1\in[z_1^\mathrm{in}, z_1^\mathrm{out}]$, where $\bar{F}_1(r_1,\eps_1)=F_1\left(x_1^+(r_1,\eps_1),r_1,\eps_1\right)+\mathcal{O}(\re^{-\vartheta_\nu/(r_1^6\eps_1)})$. We consider~\eqref{eq:k1_transf} on the invariant set $r_1=Y_1=0$
\begin{align*}
\begin{split}
\eps_1' &= -\frac{3}{2}\eps_1^2,\\
X_1' &=X_1\left(2-\frac{\lambda_1^2}{\theta_\mathrm{lf}c^3} +\mathcal{O}(\eps_1)\right)+\alpha\frac{\lambda_1^2\eps_1}{2\theta_\mathrm{lf}c^3} \left(1+\mathcal{O}(\eps_1)\right),\\
\lambda_1'&= -\frac{1}{4}\eps_1\lambda_1.
\end{split}
\end{align*}
Fix $\mu>0$ sufficiently small. Given a solution of the $(\eps_1,\lambda_1)$-subsystem restricted to the region 
\begin{align*}
    \lambda_1\in \Lambda_1(\delta, \mu, M)\coloneqq \left\{\lambda_1\in \mathbb{C}: |\Re(\lambda_1)| \leq\mu \delta^{1/6}, |\Im(\lambda_1)| \leq  M \delta^{1/6}  \right\},
\end{align*}
there is a unique solution $X_1 = X_1^*=\mathcal{O}\left( \alpha \lambda_1^2\eps_1 \right)$ which is bounded as $z_1\to\infty$ given by
\begin{align}
\begin{split}\label{eq:k1_x1star}
X_1^*(z_1) &=\frac{\alpha}{2\theta_\mathrm{lf}c^3}\int_\infty^{z_1}\exp\left(\int_s^{z_1}2-\frac{\lambda_1(\bar{s})^2}{\theta_\mathrm{lf}c^3} +\mathcal{O}(\eps_1(\bar{s}))\mathrm{d}\bar{s}\right) \lambda_1(s)^2\eps_1(s) \left(1+\mathcal{O}(\eps_1(s))\right)\mathrm{d}s,
\end{split}
\end{align}
where we note that $\left|X_1^*(z_1^\mathrm{out})\right|\leq C_{\bar{r}_1}|\alpha \eps \lambda|$. We set $X_1 = X_1^*+\tilde{X}_1$ in the full system~\eqref{eq:k1_transf}, which results in the system
\begin{align}
\begin{split}\label{eq:k1_transf_tilde}
\tilde{X}_1' &=\tilde{X}_1h_{1,X}(r_1,\eps_1,\lambda_1)+Y_1h_{1,Y}(r_1,\eps_1,\lambda_1)+ \alpha h_{1,\alpha}(r_1,\eps_1,\lambda_1),\\
Y_1' &=-\frac{3}{2}\eps_1Y_1\bar{F}_1(r_1,\eps_1)+\eps_1\tilde{X}_1g_{1,X}(r_1,\eps_1,\lambda_1)+\eps_1Y_1g_{1,Y}(r_1,\eps_1,\lambda_1)+\eps_1\alpha g_{1,\alpha}(r_1,\eps_1,\lambda_1),\\
\lambda_1'&= -\frac{1}{4}\eps_1\lambda_1\bar{F}_1(r_1,\eps_1),
\end{split}
\end{align}
where
\begin{align*}
h_{1,X}(r_1,\eps_1,\lambda_1)&=2-\frac{\lambda_1^2}{\theta_\mathrm{lf}c^3} +\mathcal{O}(r_1\lambda_1, r_1^2, \eps_1),\\
h_{1,Y}(r_1,\eps_1,\lambda_1)&= 1+\mathcal{O}(r_1\lambda_1, r_1^2),\\
h_{1,\alpha}(r_1,\eps_1,\lambda_1)&=\mathcal{O}(r_1\eps_1\lambda_1),\\
g_{1,X}(r_1,\eps_1,\lambda_1)&=\mathcal{O}(r_1^2),\\
g_{1,Y}(r_1,\eps_1,\lambda_1)&=\mathcal{O}(r_1^4),\\
g_{1,\alpha}(r_1,\eps_1,\lambda_1)&=\mathcal{O}(r_1^2\eps_1\lambda_1^2,  r_1^3\eps_1\lambda_1).
\end{align*}

We have the following.

\begin{proposition}\label{prop:K1_estimates}
Consider~\eqref{eq:k1_transf} with $(x_1,r_1,\eps_1)=(x_{1,\eps}, r_{1,\eps},\eps_{1,\eps})(z_1)$, and fix $M>0$. There exists $C, C_{\bar{r}_1}, \theta_{\bar{r}_1},\mu>0$ such that the following holds. Given $X_1^\mathrm{out}\in \mathbb{C}$ and any $\lambda\in \mathbb{C}$ such that $\lambda_1(z_\mathrm{in})=:\lambda_{1,0}\in\Lambda_1(\delta,\mu,M)$, there exists a solution $(X_1,Y_1,\lambda_1): [z_1^\mathrm{in}, z_1^\mathrm{out}]\to \mathbb{C}^2\times \Lambda_1(\delta, \mu,M)$ of~\eqref{eq:k1_transf} satisfying 
\begin{align*}
X_1(z_1^\mathrm{out}) = X_1^*(z_1^\mathrm{out})+X_1^\mathrm{out}, \qquad Y_1(z_1^\mathrm{in}) = 0,
\end{align*}
as well as the estimates
\begin{align*}
\left|X_1(z_1^\mathrm{in}) - X_1^*(z_1^\mathrm{in})\right|\leq C_{\bar{r}_1} \left(|X^\mathrm{out}_1|\re^{-\theta_{\bar{r}_1}/\eps}+ |\alpha \lambda| \right), \qquad |Y_1(z_1^\mathrm{out})| \leq C\frac{\eps}{\bar{r}_1^4}|X^\mathrm{out}_1|+ C_{\bar{r}_1} |\alpha \eps \lambda|.
\end{align*}
\end{proposition}

\begin{proof}
We focus on solving the system~\eqref{eq:k1_transf_tilde}. We first construct a solution to the inhomogeneous equation satisfying the boundary conditions $\tilde{X}_1(z_1^\mathrm{out}) = 0, Y_1(z_1^\mathrm{in})=0$. Using the fact that
\begin{align}\int_s^{z_1} -\frac{3}{2}\eps_1(\tilde{s})\bar{F}_1(r_1(\tilde{s},\eps_1(\tilde{s}) \de s = \int_s^{z_1} \eps_1'(\tilde{s})/\eps_1(\tilde{s}) \de \tilde{s} = \log(\eps_1(z_1)) - \log(\eps_1(s))
\end{align}
so that
\begin{align}\exp\left(\int_s^{z_1} -\frac{3}{2}\eps_1(\tilde{s})\bar{F}_1(r_1(\tilde{s}),\eps_1(\tilde{s})) \de \tilde{s} \right) = \frac{\eps_1(z_1)}{\eps_1(s)},
\end{align}
 we can write the solution of the inhomogeneous equation satisfying $\tilde{X}_1(z_1^\mathrm{out}) = 0, Y_1(z_1^\mathrm{in})=0$ as an integral equation
\begin{align}
\begin{split}\label{eq:k1_transf_int_ih}
\tilde{X}_1(z_1) &=\int_{z_1^\mathrm{out}}^{z_1} \re^{\beta_1(z_1,s)}\left(Y_1(s)h_{1,Y}(s)+\alpha h_{1,\alpha}(s)\right) \mathrm{d}s,\\
Y_1(z_1) &=\eps_1(z_1)\int_{z_1^\mathrm{in}}^{z_1}\tilde{X}_1(s)g_{1,X}(s)+Y_1(s)g_{1,Y}(s)+\alpha g_{1,\alpha}(s)\mathrm{d}s.
\end{split}
\end{align}
where
\begin{align*}
\beta_1(z_1,z_0) &= \int_{z_0}^{z_1}h_{1,X}(s)\mathrm{d}s
\end{align*}
and we denote
\begin{align*}
    g_{1,*}(s)&\coloneqq g_{1,*}(r_1(s),\eps_1(s),\lambda_1(s)),\\
    h_{1,*}(s)&\coloneqq h_{1,*}(r_1(s),\eps_1(s),\lambda_1(s))
\end{align*}
for $*=X,Y,\alpha$. We consider the integral equation~\eqref{eq:k1_transf_int_ih} on the space $\tilde{X}_1,Y_1\in C([z_1^\mathrm{in}, z_1^\mathrm{out}])$ with the norm
\begin{align*}
    \|(\tilde{X}_1,Y_1)\|_{1,I} = \|\tilde{X}_1\|_{\eps_1}+\|Y_1\|_{r_1\eps_1}= \sup_{z_1\in [z_1^\mathrm{in}, z_1^\mathrm{out}]}\left|\eps_1(z_1)^{-1}\tilde{X}_1(z_1)  \right| +\sup_{z_1\in [z_1^\mathrm{in}, z_1^\mathrm{out}]}\left|(r_1(z_1)\eps_1(z_1))^{-1}Y_1(z_1)  \right|,
\end{align*}
Using a fixed-point argument, the equation~\eqref{eq:k1_transf_int_ih} admits a unique solution $(\tilde{X}_1^I,Y_1^I)(z_1)$ satisfying the estimates 
\begin{align*}
\begin{split}
\left|\tilde{X}_1^I(z_1^\mathrm{in})\right| &\leq C_{\bar{r}_1} |\alpha \lambda|, \qquad \left|Y_1^I(z_1^\mathrm{out})\right| \leq C_{\bar{r}_1}|\alpha \eps \lambda|.
\end{split}
\end{align*}
We next construct a solution of the homogeneous equation, given by~\eqref{eq:k1_transf_tilde} with $\alpha=0$, satisfying the boundary conditions $\tilde{X}_1(z_1^\mathrm{out}) = X_1^\mathrm{out}, Y_1(z_1^\mathrm{in})=0$, as a solution of the integral equation
\begin{align*}
\begin{split}
\tilde{X}_1(z_1) &=X^\mathrm{out}_1\re^{\beta_1(z_1,z_1^\mathrm{out})}+\int_{z_1^\mathrm{out}}^{z_1} \re^{\beta_1(t,s)}Y_1(s)h_{1,Y}(s)\mathrm{d}s,\\
Y_1(z_1) &=\eps_1(z_1)\int_{z_1^\mathrm{in}}^{z_1}\tilde{X}_1(s)g_{1,X}(s)+Y_1g_{1,Y}(s)\mathrm{d}s.
\end{split}
\end{align*}
Similarly, for small $\kappa>0$ fixed independently of $\bar{r}_1,\delta, \eps, \lambda$, a fixed point argument on the space $\tilde{X}_1,Y_1\in C([z_1^\mathrm{in}, z_1^\mathrm{out}])$ with the norm 
\begin{align*}
    \|(\tilde{X}_1,Y_1)\|_{1,\kappa} = \sup_{z_1\in [z_1^\mathrm{in}, z_1^\mathrm{out}]}\left|\re^{\beta_1^\kappa(z_1^\mathrm{out},z_1)}\tilde{X}_1(z_1)  \right| +\sup_{z_1\in [z_1^\mathrm{in}, z_1^\mathrm{out}]}\left|\eps_1(z_1)^{-1}\re^{\beta_1^\kappa(z_1^\mathrm{out},z_1)}Y_1(z_1)  \right|,
\end{align*}
where
\begin{align*}
\beta_1^\kappa(z_1,z_0) &= \int_{z_0}^{z_1}(h_{1,X}(s)-\kappa)\mathrm{d}s.
\end{align*}
yields a solution  $(\tilde{X}_1^H,Y_1^H)(z_1)$ satisfying the estimates
\begin{align*}
|\tilde{X}_1^H(z_1)|\leq C|X_1^\mathrm{out}| \re^{\beta_1^\kappa(z_1,z_1^\mathrm{out})}, \qquad  |Y_1^H(z_1)|\leq C|X_1^\mathrm{out}|\eps_1(z_1) \re^{\beta_1^\kappa(z_1,z_1^\mathrm{out})},
\end{align*}
and in particular we find that
\begin{align*}
|Y_1^H(z_1^\mathrm{out})| &\leq C\frac{\eps}{\bar{r}_1^4}\left|X_1^\mathrm{out}\right|.
\end{align*}

Thus, we obtain a solution of~\eqref{eq:k1_transf} given by
\begin{align*}
\begin{split}
X_1(z_1) &=X_1^*(z_1)+\tilde{X}_1^I(z_1)+\tilde{X}_1^H(z_1),\\
Y_1(z_1) &=Y_1^I(z_1)+Y_1^H(z_1)
\end{split}
\end{align*}
satisfying
\begin{align*}
\begin{split}
X_1(z_1^\mathrm{out})&=X_1^*(z_1^\mathrm{out})+\tilde{X}_1^I(z_1^\mathrm{out})+\tilde{X}_1^H(z_1^\mathrm{out})+X^\mathrm{out}_1= X_1^*(z_1^\mathrm{out})+X^\mathrm{out}_1,\\
\left|Y_1(z_1^\mathrm{out})\right| &=\left|Y_1^I(z_1^\mathrm{out})+Y_1^H(z_1^\mathrm{out})\right|\leq C\frac{\eps}{\bar{r}_1^4}|X^\mathrm{out}_1|+ C_{\bar{r}_1} |\alpha \eps \lambda|,
\end{split}
\end{align*}
and 
\begin{align*}
\begin{split}
\left|X_1(z_1^\mathrm{in})-X_1^*(z_1^\mathrm{in})\right| &=\left|\tilde{X}_1^I(z_1^\mathrm{in})+\tilde{X}_1^H(z_1^\mathrm{in})\right|\leq C_{\bar{r}_1} \left(|X^\mathrm{out}_1|\re^{-\vartheta_\nu/\eps}+ |\alpha \lambda| \right),\\
Y_1(z_1^\mathrm{in}) &=Y_1^I(z_1^\mathrm{in})+Y_1^H(z_1^\mathrm{in})= 0,
\end{split}
\end{align*}
as claimed.
\end{proof}

We consider the condition on $\lambda$ present in Proposition~\ref{prop:K1_estimates}. Note that since both the real and imaginary parts of $\lambda_1$ decrease strictly in absolute value on the interval $[z_1^\mathrm{in}, z_1^\mathrm{out}]$, $\lambda_1(z_1)$ remains in the set $\Lambda_1(\delta,\mu,M)$ for $z_1\in [z_1^\mathrm{in}, z_1^\mathrm{out}]$. This restricts consideration to values of $\lambda$ satisfying $|\Re(\lambda)| \leq\mu \eps^{1/6}$ and $ |\Im(\lambda)| \leq M \eps^{1/6}$, where $\mu$ may need to be taken small; that is, we restrict to values of $\lambda\in \Lambda_{\mathrm{r},\eps}(\mu,M)$. The solutions constructed in Proposition~\ref{prop:K1_estimates} for such values of $\lambda$ can be matched with solutions from chart $\mathcal{K}_2$.

In order to extend the argument to values of $\lambda\in \Lambda_{\mathrm{c},\eps}(\mu,M)$ (that is, small, but $\eps$-independent values of $|\Im(\lambda)|$), we must consider solutions for which $\lambda_1(z_1)$ lies outside of $\Lambda_1(\delta,\mu, M)$ for some part of the interval $z_1\in[z_1^\mathrm{in}, z_1^\mathrm{out}]$. The corresponding solutions will satisfy slightly different estimates and will instead be matched with solutions from the chart $\mathcal{K}_4$. We focus on orbits which enter $\mathcal{K}_1$  via the section
\begin{align*}
  \Sigma^\mathrm{in}_{14} =\left\{\Im(\lambda_1) = M \delta^{1/6}\right\},
\end{align*} 
in particular, those solutions for which $\Im\lambda_1(z_{14}^\mathrm{in})=M \delta^{1/6}$ and $|\Re \lambda_1(z_{14}^\mathrm{in})|\leq \mu \delta^{1/6}$ for some $z_1^\mathrm{in}<z_{14}^\mathrm{in}<z_1^\mathrm{out}$. 
These solutions will be matched with solutions from the chart $\mathcal{K}_4$ (the case $\Im(\lambda_1) =-M\delta^{1/6}$ is similar), while those which enter via the boundary $|\Re(\lambda_1)| = \mu \delta^{1/6}$ are not relevant. We have the following

\begin{proposition}\label{prop:K1_estimates4}
Consider~\eqref{eq:k1_transf} with $(x_1,r_1,\eps_1)=(x_{1,\eps}, r_{1,\eps},\eps_{1,\eps})(z_1)$, and fix $M>0$. There exist $C, C_{\bar{r}_1},\mu>0$ such that the following holds. Given $X_1^\mathrm{out}\in \mathbb{C}$ and any $\lambda\in \Lambda_{\mathrm{c},\eps}(\mu, M)$ such that $\Im\lambda_1(z_{14}^\mathrm{in})=M \delta^{1/6}$ and $|\Re \lambda_1(z_{14}^\mathrm{in})|\leq \mu \delta^{1/6}$ for some $z_1^\mathrm{in}<z_{14}^\mathrm{in}<z_1^\mathrm{out}$, there exists a solution $(X_1,Y_1,\lambda_1): [z_{14}^\mathrm{in}, z_1^\mathrm{out}]\to \mathbb{C}^2\times \Lambda_1(\delta,\mu,M)$ of~\eqref{eq:k1_transf} satisfying 
\begin{align*}
X_1(z_1^\mathrm{out}) = X_1^*(z_1^\mathrm{out})+X_1^\mathrm{out} , \qquad Y_1(z_{14}^\mathrm{in}) = 0,
\end{align*}
as well as the estimates
\begin{align*}
\left|X_1(z_{14}^\mathrm{in}) - X_1^*(z_{14}^\mathrm{in})\right|\leq C_{\bar{r}_1} \left(|X^\mathrm{out}_1|+ |\alpha \lambda| \right), \qquad |Y_1(z_1^\mathrm{out})| \leq C\frac{\eps}{\bar{r}_1^4}|X^\mathrm{out}_1|+ C_{\bar{r}_1} |\alpha \eps \lambda|.
\end{align*}
\end{proposition}

\begin{proof}
The argument is similar to the proof of Proposition~\ref{prop:K1_estimates}. 
\end{proof}

\subsubsection{Chart \texorpdfstring{$\mathcal{K}_2$}{K2}}\label{sec:chartK2}
We use the change of coordinates~\eqref{eq:chartK2_coords} and set $z_2 = r_2^2\zeta$. Denoting $'=\frac{\mathrm{d}}{\mathrm{d}z_2}$, we arrive at the system
\begin{align}
\begin{split}\label{eq:K2eqn}
x_2' &= -x_2^2+y_2+\mathcal{O}(r_2^2),\\
y_2' &= 1+\mathcal{O}(r_2^2),\\
X_2' &=X_2\left(-2x_2-\frac{\lambda_2^2}{\theta_\mathrm{lf}c^3} +\mathcal{O}(r_2^2,r_2\lambda_2)\right)+Y_2\left(1+\mathcal{O}(r_2^2,r_2\lambda_2)\right)\\
&\qquad +\alpha\left( -x_2^2+y_2+\mathcal{O}(r_2^2)\right)\left(-\frac{\lambda_2^2}{\theta_\mathrm{lf}c^3} +\mathcal{O}(r_2\lambda_2)\right)+\mathcal{O}(\alpha r_2\lambda_2), \\
Y_2' &=\mathcal{O}(r_2^2X_2, r_2^4Y_2, \alpha r_2^3\lambda_2),
\end{split}
\end{align}
which we consider on the interval $z_2\in [z_2^\mathrm{in}, z_2^\mathrm{out}]$, on which the wave-train solution is represented by a solution $(x_{\eps,2}, y_{\eps,2})$ in the $\mathcal{K}_2$-coordinates which traverses between the sections
\begin{align*}
\Sigma_2^\mathrm{in}&=\left\{x_2=\delta^{-1/3}\right\},\\
\Sigma_2^\mathrm{out}&=\left\{y_2=\delta^{-2/3}\right\}.
\end{align*}
We note that $\lambda_2=\lambda_1\delta^{-1/6}$ so that $\lambda_2$ satisfies $|\Re(\lambda_2)| \leq\mu, |\Im(\lambda_2)| \leq M$. When $r_2=0$, we arrive at the system 
\begin{align}
\begin{split}\label{eq:K2eqn_0}
x_2' &= -x_2^2+y_2,\\
y_2' &= 1,\\
X_2' &=X_2\left(-2x_2-\frac{\lambda_2^2}{\theta_\mathrm{lf}c^3}\right)+Y_2 -\alpha\frac{\lambda_2^2}{\theta_\mathrm{lf}c^3}\left( -x_2^2+y_2\right), \\
Y_2' &=0,
\end{split}
\end{align}
which is a rescaled version of the toy model which was analyzed in~\S\ref{sec:toymodel}. The first pair of equations have dynamics organized by the unique solution $x_R$ of the Riccati equation
\begin{align}\label{eq:riccati}
x' = -x^2+z_2,
\end{align}
satisfying 
\begin{align*}
\begin{split}
z_2&\sim x_R^2 +\frac{1}{2x_R}+\mathcal{O}\left( \frac{1}{x_R^4} \right),\qquad x_R\to-\infty,\\
z_2&\sim -\Omega_0+\frac{1}{x_R}+\mathcal{O}\left( \frac{1}{x_R^3} \right),\qquad x_R\to\infty,
\end{split}
\end{align*}
where $\Omega_0$ is the smallest positive zero of $\bar{J}(z)\coloneqq \smash{J_{-\frac{1}{3}}(2z^{3/2}/3)+J_{\frac{1}{3}}(2z^{3/2}/3)}$, where $\smash{J_{\pm\frac{1}{3}}}$ are Bessel functions of the first kind. For small $r_2>0$, this solution perturbs (in a regular fashion) to a solution of the $(x_2,y_2)$-system in~\eqref{eq:K2eqn} satisfying $(x_2,y_2)(z_2)=(x_R(z_2),z_2)+\mathcal{O}(r_2^2)$ corresponding to a slice of the manifold $\mathcal{M}^{\lr,+}$ in the $\mathcal{K}_2$-coordinates; see~\cite[Remark 4.3]{CSosc}. By Lemma~\ref{lem:fold_existence}, the wave-train solution is exponentially close to this solution, and is thus represented in the $\mathcal{K}_2$-coordinates by a solution $(x_{\eps,2}, y_{\eps,2})=(x_R(z_2),z_2)+\mathcal{O}(r_2^2)$ for small $r_2$ on the interval $z_2\in[z_2^\mathrm{in},z_2^\mathrm{out}]$. Hence in the limit $r_2\to0$, we identify the wave-train solution $(x_{\eps,2}, y_{\eps,2})$ in the $\mathcal{K}_2$-coordinates with the solution $(x_2,y_2) = (x_R(z_2),z_2)$, and we write the linearized problem as 
\begin{align*}
\begin{split}
X_2' &=X_2\left(-2x_R-\frac{\lambda_2^2}{\theta_\mathrm{lf}c^3}\right)+Y_2 -\alpha\frac{\lambda_2^2}{\theta_\mathrm{lf}c^3}x_R', \\
Y_2' &=0,
\end{split}
\end{align*}
When $Y_2=0$, the $X_2$-equation has a unique solution $X_2=X_2^*(z_2)$ which is bounded as $z_2\to\infty$, given by 
\begin{align*}
\begin{split}
X_2^*(z_2)&=-\alpha\frac{\lambda_2^2}{\theta_\mathrm{lf},c^3}\int_\infty^{z_2}\exp\left(\int_s^{z_2}-2x_R(\bar{s})-\frac{\lambda_2^2}{\theta_\mathrm{lf}c^3}\mathrm{d}\bar{s}  \right)x_R'(s)\mathrm{d}s,
\end{split}
\end{align*}
which corresponds to the solution $X_1^*$ from the chart $\mathcal{K}_1$, now represented in the $\mathcal{K}_2$-coordinates. We set $X_2 = X_2^*+\tilde{X}_2$ which results in the equation
\begin{align}
\begin{split}\label{eq:K2eqn_tilde}
x_2' &= -x_2^2+y_2+\mathcal{O}(r_2^2),\\
y_2' &= 1+\mathcal{O}(r_2^2),\\
\tilde{X}_2' &=\tilde{X}_2\left(-2x_2-\frac{\lambda_2^2}{\theta_\mathrm{lf}c^3} +\mathcal{O}(r_2^2,r_2\lambda_2)\right)+Y_2\left(1+\mathcal{O}(r_2^2,r_2\lambda_2)\right)+\mathcal{O}(\alpha r_2\lambda_2), \\
Y_2' &=\mathcal{O}\left(r_2^2\tilde{X}_2, r_2^4Y_2, \alpha r_2^3\lambda_2, \alpha r_2^2\lambda_2^2\right).
\end{split}
\end{align}

We have the following.

\begin{proposition}\label{prop:K2_estimates}
Consider~\eqref{eq:K2eqn} with $(x_2,y_2)=(x_{2,\eps}, y_{2,\eps})(z_2)$. Given $(X_2^\mathrm{out}, Y_2^\mathrm{out})\in \mathbb{C}^2$ and any $\lambda_2\in\mathbb{C}$ satisfying $|\Re(\lambda_2)| <\mu, |\Im(\lambda_2)| <M$, there exists a solution $(X_2,Y_2): [z_2^\mathrm{in}, z_2^\mathrm{out}]\to\mathbb{C}^2$ of~\eqref{eq:K2eqn} satisfying 
\begin{align*}
X_2(z_2^\mathrm{out}) = X_2^\mathrm{out}+X_2^*(z_2^\mathrm{out}), \qquad Y_2(z_2^\mathrm{out}) = 0,
\end{align*}
as well as the estimates
\begin{align*}
X_2(z_2^\mathrm{in}) &= X_2^\mathrm{out}\exp\left(\int_{z_2^\mathrm{out}}^{z_2^\mathrm{in}}-2x_R(\tilde{z})-\frac{\lambda_2^2}{\theta_\mathrm{lf}c^3}\mathrm{d}\tilde{z}\right)+X_2^*(z_2^\mathrm{in}) +\mathcal{O}\left(r_2^2|X_2^\mathrm{out}|, r_2\lambda_2|X_2^\mathrm{out}|, \alpha r_2\lambda_2 \right), \\
Y_2(z_2^\mathrm{in}) &= \mathcal{O}\left(r_2^2|X_2^\mathrm{out}|, \alpha r_2^2\lambda_2^2, \alpha r_2^3\lambda_2\right),
\end{align*}
for all sufficiently small $r_2>0$.
\end{proposition}
\begin{proof}
Using the fact that $(x_{2,\eps},y_{2,\eps})(z_2) = (x_R(z_2),-z_2)+\mathcal{O}(r_2^2)$, and the fact that the transition time $|z_2^\mathrm{in}-z_2^\mathrm{out}|$ is bounded independently of $r_2,\lambda_2$, the estimates follow from a regular perturbation argument applied to~\eqref{eq:K2eqn_tilde} for sufficiently small $r_2$.
\end{proof}

\subsubsection{Chart \texorpdfstring{$\mathcal{K}_4$}{K4}}\label{sec:chartK4}
We transform to the $\mathcal{K}_4$-coordinates~\eqref{eq:chartK4_coords}, desingularize through the scaling $\mathrm{d}z_4 = r_4^2\mathrm{d}\zeta$, and denote ${}^\prime = \frac{\mathrm{d}}{\mathrm{d}z_4}$, to obtain the system
\begin{align}
\begin{split}\label{eq:K4eqn}
x_4' &= F_4(x_4,y_4,\eps_4,r_4),\\
y_4' &= \eps_4\left(1+\mathcal{O}(r_4^2)\right),\\
X_4' &=X_4\left(-2x_4-\frac{(\lambda_4+\ri)^2}{\theta_\mathrm{lf}c^3} +\mathcal{O}(r_4)\right)+Y_4\left(1+\mathcal{O}(r_4)\right)\\
&\qquad +\alpha F_4(x_4,y_4,\eps_4,r_4)\left(-\frac{(\lambda_4+\ri)^2}{\theta_\mathrm{lf}c^3} +\mathcal{O}(r_4)\right)+\mathcal{O}(\alpha r_4\eps_4), \\
Y_4' &=\mathcal{O}(r_4^2\eps_4X_4, r_4^4\eps_4Y_4, \alpha r_4^3\eps_4),
\end{split}
\end{align}
where $F_4(x_4,y_4,\eps_4,r_4) = -x_4^2+y_4+\mathcal{O}(r_4^2)$. The chart $\mathcal{K}_4$ concerns values of $\eps_4\leq M^{-6}$. We fix a small constant $\delta_4>0$ and first consider~\eqref{eq:K4eqn} for values of $\delta_4\leq \eps_4\leq M^{-6}$ for sufficiently small values of $r_4$ on the interval $[z_4^\mathrm{in}, z_4^\mathrm{out}]$ encompassing the transition between the sections 
\begin{align*}
\Sigma_4^\mathrm{in}&=\left\{x_4 = M^{-2}\delta^{-1/3}  \right\}, \qquad \Sigma_4^\mathrm{out}=\left\{y_4 = M^{-4}\delta^{-2/3}  \right\}.
\end{align*}
Note that the section $\Sigma_4^\mathrm{out}$ corresponds to the section $\Sigma^\mathrm{in}_{14}$ from the chart $\mathcal{K}_1$; see~\S\ref{sec:chartK1}. Setting $r_4=0$, we obtain the system
\begin{align*}
\begin{split}
x_4' &= -x_4^2+y_4,\\
y_4' &= \eps_4,\\
X_4' &=X_4\left(-2x_4-\frac{(\lambda_4+\ri)^2}{\theta_\mathrm{lf}c^3}\right)+Y_4-\alpha\frac{(\lambda_4+\ri)^2}{\theta_\mathrm{lf}c^3} \left(-x_4^2+y_4\right),\\
Y_4' &=0,
\end{split}
\end{align*}
which corresponds to the system~\eqref{eq:K2eqn_0} transformed to the $\mathcal{K}_4$-coordinates. The $(x_4,y_4)$ system corresponds to a rescaled Riccati equation~\eqref{eq:riccati} given by
\begin{align*}
x_4' &= -x_4^2+\eps_4z_4,
\end{align*}
which admits the rescaled distinguished solution $x_4(z_4)=\eps_4^{1/3}x_R(\eps_4^{1/3}z_4)$. Similarly to the analysis in the $\mathcal{K}_2$ in~\S\ref{sec:chartK2}, the wave-train solution can be represented by a solution $(x_4,y_4)=(x_{4,\eps}, y_{4,\eps})(z_4)$ of~\eqref{eq:K4eqn} satisfying $\smash{(x_{4,\eps}, y_{4,\eps})(z_4)=(\eps_4^{1/3}x_R(\eps_4^{1/3}z_4), \eps_4z_4)+\mathcal{O}(r_4^2)}$.

When $Y_4=0$ we can construct the unique solution $X_4 = X_4^*(z_4)$ given by
\begin{align*}
X_4^*(z_4)\coloneqq  -\alpha\frac{(\lambda_4+\ri)^2}{\theta_\mathrm{lf}c^3}\int_\infty^{z_4}\exp\left( \int_s^{z_4} -2\eps_4^{1/3}x_R(\eps_4^{1/3}\tilde{s})-\frac{(\lambda_4+\ri)^2}{\theta_\mathrm{lf}c^3} \mathrm{d}\tilde{s} \right) \eps_4^{2/3}x_R'\big(\eps_4^{1/3}s\big)\mathrm{d}s,
\end{align*}
which corresponds to the solution $X_1^*$ from the chart $\mathcal{K}_1$.
We set $X_4 = X_4^*+\tilde{X}_4$ which results in the equation
\begin{align}
\begin{split}\label{eq:K4eqn_tilde}
x_4' &= F_4(x_4,y_4,\eps_4,r_4),\\
y_4' &= \eps_4\left(1+\mathcal{O}(r_4^2)\right),\\
\tilde{X}_4' &=\tilde{X}_4\left(-2x_4-\frac{(\lambda_4+\ri)^2}{\theta_\mathrm{lf}c^3} +\mathcal{O}(r_4)\right)+Y_4\left(1+\mathcal{O}(r_4)\right)+\mathcal{O}(\alpha r_4), \\
Y_4' &=\mathcal{O}(r_4^2\eps_4\tilde{X}_4, r_4^4\eps_4Y_4, \alpha r_4^2).
\end{split}
\end{align}

We have the following.

\begin{proposition}\label{prop:K4_estimates}
Consider~\eqref{eq:K4eqn} with $(x_4,y_4)=(x_{4,\eps}, y_{4,\eps})(z_4)$, and fix $M>0$ and $\delta_4>0$ sufficiently small. For all sufficiently small $r_4>0$, the following holds. Given $(X_4^\mathrm{out}, Y_4^\mathrm{out})\in \mathbb{C}^2$, any $\eps_4$ satisfying $\delta_4\leq \eps_4\leq M^{-6}$ and any sufficiently small $\lambda_4\in\mathbb{R}$, there exists a solution $(X_4,Y_4): [z_4^\mathrm{in}, z_4^\mathrm{out}]\to\mathbb{C}^2$ of~\eqref{eq:K4eqn} satisfying 
\begin{align*}
X_4(z_4^\mathrm{out}) = X_4^\mathrm{out}+X_4^*(z_4^\mathrm{out}), \qquad Y_4(z_4^\mathrm{out}) = 0,
\end{align*}
as well as the estimates
\begin{align*}
X_4(z_4^\mathrm{in}) &= X_4^\mathrm{out}\exp\left( \int_{z_4^\mathrm{out}}^{z_4^\mathrm{in}} -2\eps_4^{1/3}x_R(\eps_4^{1/3}\tilde{s})-\frac{(\lambda_4+\ri)^2}{\theta_\mathrm{lf}c^3} \mathrm{d}\tilde{s} \right)+X_4^*(z_4^\mathrm{in}) +\mathcal{O}\left(r_4|X_4^\mathrm{out}|, \alpha r_4 \right), \\
Y_4(z_4^\mathrm{in}) &= \mathcal{O}\left(r_4^2|X_4^\mathrm{out}|, \alpha r_4^2\right).
\end{align*}

\end{proposition}
\begin{proof}
Using the fact that $(x_{4,\eps},y_{4,\eps})(z_2) = (\eps^{1/3}x_R(\eps^{1/3}z_4),\eps_4z_4)$ when $r_4=0$, and the fact that the transition time $|z_4^\mathrm{in}-z_4^\mathrm{out}|$ is bounded independently of $r_4$, the estimates follow from a regular perturbation argument applied to~\eqref{eq:K4eqn_tilde} for sufficiently small $r_4$.
\end{proof}

Proposition~\ref{prop:K4_estimates} relies on a perturbation argument for sufficiently small $r_4$; in particular, one needs to be able to bound (independently of $r_4$) the transition time, which behaves as $\eps_4^{-1}$. For small values of $0<\eps_4\leq \delta_4\ll1$, we note that the existence problem in the $\mathcal{K}_4$-coordinates
\begin{align*}
\begin{split}
x_4' &= F_4(x_4,y_4,\eps_4,r_4),\\
y_4' &= \eps_4\left(1+\mathcal{O}(r_4^2)\right)
\end{split}
\end{align*}
amounts to slow passage through a fold point (in reverse time) with small parameter $\eps_4$. In the limit $\eps_4\to0$, for sufficiently small $r_4>0$ the critical manifold $F_4(x_4,y_4,0,r_4)=0$ in the region $|x_4|\leq M^{-2}\delta^{-1/3}, |y_4|\leq M^{-4}\delta^{-2/3}$ takes the form of an upward-facing parabola. The union of the left branch of this parabola with the positive $x$-axis corresponds to the manifold $\mathcal{M}^{\lr,+}_0$ in the $\mathcal{K}_4$-coordinates. This manifold perturbs for small $\eps_4>0$ using blow-up desingularization techniques, much like the system~\eqref{eq:fold_normalform}. The manifold $\mathcal{M}^{\lr,+}$ therefore has a natural representation in these coordinates as the corresponding union of these manifolds for small $\eps_4$. By Lemma~\ref{lem:fold_existence}, the wave train is again represented by a solution $(x_4,y_4)=(x_{4,\eps}, y_{4,\eps})(z_4)$ which is exponentially close to $\mathcal{M}^{\lr,+}$ in these coordinates. We have the following.
\begin{proposition}\label{prop:K4_estimates2}
Consider~\eqref{eq:K4eqn} with $(x_4,y_4)=(x_{4,\eps}, y_{4,\eps})(z_4)$. There exist $\delta_4, \theta_4>0$ such that for all $0<\eps_4\leq \delta_4$ and all  sufficiently small $r_4,\lambda_4>0$, the following holds. Given $X_4^\mathrm{out}\in \mathbb{C}^2$, there exists a solution $(X_4,Y_4): [z_4^\mathrm{in}, z_4^\mathrm{out}]\to\mathbb{C}^2$ of~\eqref{eq:K4eqn} satisfying 
\begin{align*}
X_4(z_4^\mathrm{out}) = X_4^\mathrm{out}, \qquad Y_4(z_4^\mathrm{out}) = 0,
\end{align*}
as well as the estimates
\begin{align*}
X_4(z_4^\mathrm{in}) &= X_4^\dagger(z_4^\mathrm{in};r_4)+\mathcal{O}\left(\alpha r_4\eps_4, r_4^2\eps_4|X_4^\mathrm{out}|,\re^{-\theta_4/\eps_4}|X_4^\mathrm{out}| \right), \\
Y_4(z_4^\mathrm{in}) &= \mathcal{O}\left(\alpha r_4^2\eps_4, r_4^2\eps_4|X_4^\mathrm{out}|\right),
\end{align*}
for all $0<\eps_4\leq \delta_4$, where $X_4^\dagger(z_4;r_4)$ is a solution which satisfies
\begin{align*}
    X_4^\dagger(z_4^\mathrm{in};r_4) &=\alpha x_{4,\eps}'(z_4^\mathrm{in})\left(-1+\mathcal{O}\left(r_4,\eps_4^{1/3}\right)\right
)+ \mathcal{O}(\alpha r_4 \eps_4).
\end{align*}
\end{proposition}
\begin{proof} We rewrite~\eqref{eq:K4eqn} as
\begin{align*}
\begin{split}
x_4' &= F_4(x_4,y_4,\eps_4,r_4),\\
y_4' &= \eps_4\left(1+\mathcal{O}(r_4^2)\right),\\
X_4' &=X_4h_{4,X}(x_4,y_4, r_4,\eps_4,\lambda_4)+Y_4h_{4,Y}(x_4,y_4, r_4,\eps_4,\lambda_4)\\
&\qquad +\alpha F_4(x_4,y_4,\eps_4,r_4)\left(-\frac{(\lambda_4+\ri)^2}{\theta_\mathrm{lf}c^3} +\mathcal{O}(r_4)\right)+\alpha h_{4,\alpha}(x_4,y_4, r_4,\eps_4,\lambda_4), \\
Y_4' &=X_4g_{4,X}(x_4,y_4, r_4,\eps_4,\lambda_4)+Y_4g_{4,Y}(x_4,y_4, r_4,\eps_4,\lambda_4)+\alpha g_{4,\alpha}(x_4,y_4, r_4,\eps_4,\lambda_4).
\end{split}
\end{align*}
where 
\begin{align*}
h_{4,X}(x_4,y_4, r_4,\eps_4,\lambda_4)&=-2x_4-\frac{(\lambda_4+\ri)^2}{\theta_\mathrm{lf}c^3} +\mathcal{O}(r_4),\\
h_{4,Y}(x_4,y_4, r_4,\eps_4,\lambda_4)&= 1+\mathcal{O}(r_4),\\
h_{4,\alpha}(x_4,y_4, r_4,\eps_4,\lambda_4)&= \mathcal{O}( r_4\eps_4),\\
g_{4,X}(x_4,y_4, r_4,\eps_4,\lambda_4)&=\mathcal{O}(r_4^2\eps_4),\\
g_{4,Y}(x_4,y_4, r_4,\eps_4,\lambda_4)&=\mathcal{O}(r_4^4\eps_4),\\
g_{4,\alpha}(x_4,y_4, r_4,\eps_4,\lambda_4)&=\mathcal{O}(r_4^3\eps_4).
\end{align*}
We set $h_{4,*}(s) = h_{4,*}(x_4(s),y_4(s), r_4,\eps_4,\lambda_4)$ and $g_{4,*}(s) = g_{4,*}(x_4(s),y_4(s), r_4,\eps_4,\lambda_4)$ for $* = X,Y,\alpha$, and we define
\begin{align*}
\beta_4(z_4, z_0) \coloneqq  \int_{z_0}^{z_4}h_{4,X}(s)\mathrm{d}s.
\end{align*}
We set $X_4 = X_4^\dagger+\tilde{X}_4$, where
\begin{align*}
\begin{split}
X_4^\dagger(z_4;r_4) &=\int_{z_4^\mathrm{out}}^{z_4} \re^{\beta_4(z_4,s)}\left(\alpha x_{4,\eps}'(s)\left(-\frac{(\lambda_4+\ri)^2}{\theta_\mathrm{lf}c^3}+\mathcal{O}(r_4)\right)+\alpha h_{4,\alpha}(s)\right)\mathrm{d}s,
\end{split}
\end{align*}
and we rewrite the eigenvalue problem as the integral equation
\begin{align}
\begin{split}\label{eq:k4_int}
\tilde{X}_4(z_4) &=X_4^\mathrm{out}\re^{\beta_4(z_4,z_4^\mathrm{out}})+\int_{z_4^\mathrm{out}}^{z_4} \re^{\beta_4(z_4,s)}Y_4h_{4,Y}(s)\mathrm{d}s,\\
Y_4(z_4) &=\int_{z_4^\mathrm{out}}^{z_4}\left(X_4^\dagger(s;r_4)+\tilde{X}_4\right)g_{4,X}(s) + Y_4(s) g_{4,Y}(s) +\alpha g_{4,\alpha}(s)\mathrm{d}s.
\end{split}
\end{align}
We first estimate $X_4^\dagger(z_4;r_4)$ by noting that 
\begin{align*}
h_{4,X}(s)&=-2x_{4,\eps}(s)-\frac{(\lambda_4+\ri)^2}{\theta_\mathrm{lf}c^3} +\mathcal{O}(r_4)\\
&=\frac{x_{4,\eps}''(s)}{x_{4,\eps}'(s)}-\frac{y_{4,\eps}'(s)}{x_{4,\eps}'(s)}-\frac{(\lambda_4+\ri)^2}{\theta_\mathrm{lf}c^3} +\mathcal{O}(r_4),
\end{align*}
so that 
\begin{align*}
\begin{split}
X_4^\dagger(z_4;r_4) &=\alpha x_{4,\eps}'(z_4)\left(-\frac{(\lambda_4+\ri)^2}{\theta_\mathrm{lf}c^3}+\mathcal{O}(r_4)\right
)\int_{z_4^\mathrm{out}}^{z_4} e^{\tilde{\beta}_4(z_4,s)}\mathrm{d}s+ \mathcal{O}(\alpha r_4 \eps_4),
\end{split}
\end{align*}
where 
\begin{align*}
   \tilde{\beta}_4(z_4,s)\coloneqq \beta_4(z_4, s)- \int_s^{z_4}\frac{x_{4,\eps}''(\tilde{s})}{x_{4,\eps}'(\tilde{s})}\mathrm{d}\tilde{s}=\int_s^{z_4} -\frac{y_{4,\eps}'(\tilde{s})}{x_{4,\eps}'(\tilde{s})}-\frac{(\lambda_4+\ri)^2}{\theta_\mathrm{lf}c^3} +\mathcal{O}(r_4)\mathrm{d}\tilde{s}.
\end{align*}
We note that the first term of the integrand represents the slope of the graph formed by the solution $(x_4,y_4)=(x_{4,\eps}, y_{4,\eps})(z_4)$ which satisfies
\begin{align}\label{eq:K4_slope_bound}
  -\frac{\mathrm{d}y_4}{\mathrm{d}x_4}(\tilde{s})= -\frac{y_{4,\eps}'(\tilde{s})}{x_{4,\eps}'(\tilde{s})} \geq C \eps_4,
\end{align}
see for example~\cite[Remark 4.1]{CSosc}. Integrating by parts  -- where we integrate the term $\smash{\exp\left({\frac{(\lambda_4+\ri)^2}{\theta_\mathrm{lf}c^3}(s-z_4)}\right)}$ -- we obtain
\begin{align}
\begin{split}\label{eq:x4star_est_parts}
X_4^\dagger(z_4^\mathrm{in};r_4) &=\alpha x_{4,\eps}'(z_4^\mathrm{in})\left(-1+\mathcal{O}\left(r_4,\eps_4\right)\right
)+ \mathcal{O}(\alpha r_4 \eps_4)\\
&\qquad +\alpha x_{4,\eps}'(z_4^\mathrm{in})\int_{z_4^\mathrm{out}}^{z_4^\mathrm{in}} \left(\frac{y_{4,\eps}'(\tilde{s})}{x_{4,\eps}'(\tilde{s})} +\mathcal{O}(r_4) \right)e^{\tilde{\beta}_4(z_4^\mathrm{in},s)}\mathrm{d}s.
\end{split}
\end{align}
Finally, following the estimates in~\cite[\S4.7]{CSosc}, we find that 
\begin{align*}
    \left|\frac{y_{4,\eps}'(\tilde{s})}{x_{4,\eps}'(\tilde{s})}\right|\leq C\eps_4^{1/3}
\end{align*}
on an interval of width $\smash{\mathcal{O}(\eps_4^{-1/3})}$ near $\tilde{s}=\smash{z_4^\mathrm{in}}$; on the remainder of the integration interval, the integrand is exponentially small in $\smash{\eps_4^{-1/3}}$ due to the exponential term and~\eqref{eq:K4_slope_bound}, from which we obtain 
\begin{align}
\begin{split}\label{eq:x4star_est2}
X_4^\dagger(z_4^\mathrm{in};r_4) &=\alpha x_{4,\eps}'(z_4^\mathrm{in})\left(-1+\mathcal{O}\left(r_4,\eps_4^{1/3}\right)\right
)+ \mathcal{O}(\alpha r_4 \eps_4).
\end{split}
\end{align}
 The result then follows by applying a fixed-point argument to the integral equation~\eqref{eq:k4_int} on the space $\tilde{X}_4,Y_4\in C([z_4^\mathrm{in}, z_4^\mathrm{out}])$.
\end{proof}

\subsubsection{Chart \texorpdfstring{$\mathcal{K}_3$}{K3}} \label{sec:chartK3}

We transform to the $\mathcal{K}_3$-coordinates~\eqref{eq:chartK3_coords}, desingularize via $\mathrm{d}z_3 = r_3^2\mathrm{d}\zeta$, and denote $'=\frac{\mathrm{d}}{\mathrm{d}z_3}$, which results in the system
\begin{align*}
\begin{split}
r_3' &= -\frac{1}{2}r_3F_3(r_3,y_3,\eps_3),\\
y_3' &= 2y_3F_3(r_3,y_3,\eps_3)+\eps_3\left(1+\mathcal{O}(r_3^2)\right),\\
\eps_3' &= 3\eps_3F_3(r_3,y_3,\eps_3),\\
X_3' &=2X_3F_3(r_3,y_3,\eps_3)+X_3\left(-2-\frac{\lambda_3^2}{\theta_\mathrm{lf}c^3} +\mathcal{O}(r_3^2, r_3\lambda_3)\right)+Y_3\left(1+\mathcal{O}(r_3^2, r_3\lambda_3)\right)\\
&\qquad +\alpha \left(\frac{\lambda_3^2}{\theta_\mathrm{lf}c^3} +\mathcal{O}(r_3\lambda_3)\right)F_3(r_3,y_3,\eps_3)+ \mathcal{O}(\alpha r_3\eps_3\lambda_3),\\
Y_3' &=3Y_3F_3(r_3,y_3,\eps_3)+\mathcal{O}(r_3^2\eps_3X_3,r_3^4\eps_3Y_3,\alpha r_3\lambda_3 \epsilon_3 ),\\
\lambda_3' &= \frac{1}{2}\lambda_3F_3(r_3,y_3,\eps_3),
\end{split}
\end{align*}
where $F_3(r_3,y_3,\eps_3) = 1-y_3+\mathcal{O}(r_3^2)$. Considering first the existence problem in the $(r_3,y_3,\eps_3)$-coordinates, we restrict attention to the flow between the sections $\Sigma_3^\mathrm{in}$ and $\Sigma_3^\mathrm{out}$ defined by
\begin{align*}
\Sigma_3^\mathrm{in}&=\{r_3=r_{3,\eps}(\eps_3), |y_3|\leq \bar{y}_3,0\leq\eps_3\leq \delta  \},\\
\Sigma_3^\mathrm{out}&=\{0\leq r_3\leq r_{3,\eps}(\delta), |y_3|\leq \bar{y}_3, \eps_3= \delta \},
\end{align*}
where by Lemma~\ref{lem:fold_existence}, $r_{3,\eps}(\eps_3)^2 = x_\mathrm{in}(\nu,\eps)$ satisfies $| \smash{r_{3,\eps}(\eps_3)}-\bar{r}_{3}|\leq C_\nu\smash{\eps_3^{2/3}}$ and where $\smash{\bar{r}_{3}^2}:= r_{3,\eps}(0)^2= x_\mathrm{in}(\nu,0)$, and $\bar{y}_3$ is a small constant. The wave train is represented in the $\mathcal{K}_3$-coordinates by a solution $(r_{3,\eps},y_{3,\eps},\eps_{3,\eps})(z_3)$, which reaches these sections at $z=\smash{z_3^\mathrm{in},z_3^\mathrm{out}}$, respectively. Using~\cite[Lemma 2.10]{krupaszmolyan2001} and the form of the equations, for sufficiently small $\bar{r}_3$, the manifold $\mathcal{M}^{\lr,+}$ can be represented in the $\mathcal{K}_3$-coordinates as a graph $y_3 = -\smash{\Omega_0\eps_3^{2/3}+\mathcal{O}(r_3^2\eps_3\log \eps_3, \eps_3)}$, and thus by Lemma~\ref{lem:fold_existence}, the wave train lies exponentially close to this graph.

We therefore consider the corresponding eigenvalue problem
\begin{align}
\begin{split}\label{eq:k3_transf}
X_3' &=2X_3F_3(r_3,y_3,\eps_3)+X_3\left(-2-\frac{\lambda_3^2}{\theta_\mathrm{lf}c^3} +\mathcal{O}(r_3^2, r_3\lambda_3)\right)+Y_3\left(1+\mathcal{O}(r_3^2, r_3\lambda_3)\right)\\
&\qquad +\alpha \left(\frac{\lambda_3^2}{\theta_\mathrm{lf}c^3} +\mathcal{O}(r_3\lambda_3)\right)F_3(r_3,y_3,\eps_3)+ \mathcal{O}(\alpha r_3\eps_3\lambda_3),\\
Y_3' &=3Y_3F_3(r_3,y_3,\eps_3)+\mathcal{O}(r_3^2\eps_3X_3,r_3^4\eps_3Y_3, \alpha r_3\lambda_3 \epsilon_3),\\
\lambda_3' &= \frac{1}{2}\lambda_3F_3(r_3,y_3,\eps_3)
\end{split}
\end{align}
over the interval $z_3\in[z_3^\mathrm{in}, z_3^\mathrm{out}]$, for values of $\lambda_3\in \Lambda_3(\delta,\mu,M) = \{\lambda_3\in \mathbb{C}: |\Re(\lambda_3)|\leq  \mu \delta^{1/6}, |\Im(\lambda_3)|\leq  M \delta^{1/6}\}$. We first consider solutions which depart $\mathcal{K}_3$ via the section $\Sigma^\mathrm{out}_3$; these solutions will be matched with solutions from the chart $\mathcal{K}_2$. Rescaling $\mathrm{d}\tilde{z}_3 = F_3(r_3,y_3,\eps_3) \mathrm{d}z_3$ and denoting the corresponding transformed interval by $\tilde{z}_3\in[\tilde{z}_3^\mathrm{in}, \tilde{z}_3^\mathrm{out}]$, we first consider the system on the invariant subspace $r_3=0$
\begin{align*}
\begin{split}
\dot{y}_3 &= 2y_3+\frac{\eps_3}{1-y_3},\\
\dot{\eps}_3 &= 3\eps_3,\\
\dot{X}_3 &=X_3\frac{\left(-\frac{\lambda_3^2}{\theta_\mathrm{lf}c^3}-2y_3\right)}{1-y_3}+\frac{Y_3}{1-y_3}+\alpha\frac{\lambda_3^2}{\theta_\mathrm{lf}c^3},\\
\dot{Y}_3 &=3Y_3,\\
\dot{\lambda}_3 &= \frac{1}{2}\lambda_3,
\end{split}
\end{align*}
where $\cdot=\tfrac{\mathrm{d}}{\mathrm{d}\tilde{z}_3}$. In the subspace $Y_3=0$, let $X_3=X_3^*$ be a solution to this equation which satisfies $\lim_{\tilde{z}_3\to-\infty}\smash{X_3^*(\tilde{z}_3)=X_3^{\infty}}$. We use the fact that $\lambda_3=\smash{\lambda_2 \eps_3^{1/6}}$ and $y_3 = \smash{-\Omega_0\eps_3^{2/3}+\mathcal{O}(r_3^2\eps_3\log \eps_3, \eps_3)}$ to find that in the subspace $r_3=0$,
\begin{align}
    \frac{\mathrm{d}X_3}{\mathrm{d}\eps_3} = \frac{\dot{X}_3}{\dot{\eps}_3} = \frac{X_3}{3\eps_3^{2/3}}\left(-\frac{\lambda_2^2}{\theta_\mathrm{lf}c^3}+\mathcal{O}\left(\eps_3^{1/3}, \lambda_2^2\eps_3^{2/3}\right)\right)+\alpha\frac{\lambda_2^2}{3\eps_3^{2/3}\theta_\mathrm{lf}c^3},
\end{align}
so that $X_3^*$ admits the expansion in terms of $\eps_3$ as
\begin{align*}
X_3^* &= X_3^{\infty}\re^{-\frac{\lambda_2^2}{\theta_{\mathrm{lf}}c^3}\eps_3^{1/3}}\left(1+\mathcal{O}\left(\eps_3^{2/3}, \lambda_2^2\eps_3\right)\right)+\alpha\left(1-\re^{-\frac{\lambda_2^2}{\theta_{\mathrm{lf}}c^3}\eps_3^{1/3}}\right)\left(1+\mathcal{O}\left(\eps_3^{2/3}, \lambda_2^2\eps_3\right)\right)\\
&= X_3^{\infty}\re^{-\frac{\lambda_3^2}{\theta_{\mathrm{lf}}c^3}}\left(1+\mathcal{O}\left(\eps_3^{2/3}\right)\right)+\alpha\left(1-\re^{-\frac{\lambda_3^2}{\theta_{\mathrm{lf}}c^3}}\right)\left(1+\mathcal{O}\left(\eps_3^{2/3}\right)\right),
\end{align*}
as $\eps_3\to0$. 
In order to determine the value of $X^\infty_3$ such that this solution corresponds to the distinguished solution $X_2^*$ from the chart $\mathcal{K}_2$, we transform $X_2^*$ to the chart $\mathcal{K}_3$, where $X_2^*$ admits the expansion in $\eps_3$
\begin{align*}
X_3 = \alpha\frac{\lambda_2^2}{\theta_\mathrm{lf}c^3}\frac{1}{\mathrm{Ai}'(-\Omega_0)^2}\left(1+\mathcal{O}\left(\eps_3^{1/3}\right)\right)\int_\infty^{-\Omega_0}\re^{\frac{\lambda_2^2}{\theta_\mathrm{lf} c^3}\left( s+\Omega_0\right)   } \left(\mathrm{Ai}'(s)^2-s\mathrm{Ai}(s)^2\right)\mathrm{d}s,
\end{align*}
from which we find that $X_3^*$ and $X_2^*$ correspond to the same solution when \begin{align*}
X^\infty_3=\alpha \Upsilon_\mathrm{lf}(\lambda_2)=\alpha\frac{\lambda_2^2}{\theta_\mathrm{lf}c^3}\frac{1}{\mathrm{Ai}'(-\Omega_0)^2}\int_\infty^{-\Omega_0}\re^{\frac{\lambda_2^2}{\theta_\mathrm{lf} c^3}\left( s+\Omega_0\right)   } \left(\mathrm{Ai}'(s)^2-s\mathrm{Ai}(s)^2\right)\mathrm{d}s.
\end{align*}
We note that $|\Upsilon_\mathrm{lf}(\lambda_2)|$ is well-defined and uniformly bounded in the region $|\Re(\lambda_2)|<\mu, |\Im(\lambda_2)|<M$ for sufficiently small $\mu>0$; see Appendix~\ref{app:airy}.

Setting $X_3 = X_3^*+\tilde{X}_3$, we obtain the system
\begin{align}
\begin{split}\label{eq:k3_transf_tilde}
\dot{\tilde{X}}_3 &=\tilde{X}_3h_{3,X}(r_3,\eps_3,y_3,\lambda_3)+Y_3h_{3,Y}(r_3,\eps_3,y_3,\lambda_3)+\alpha h_{3,\alpha}(r_3,\eps_3,y_3,\lambda_3),\\
\dot{Y}_3 &=3Y_3+\eps_3\left( \tilde{X}_3 g_{3,X}(r_3,\eps_3,y_3,\lambda_3) + Y_3 g_{3,Y}(r_3,\eps_3,y_3,\lambda_3) +\alpha g_{3,\alpha}(r_3,\eps_3,y_3,\lambda_3) \right),\\
\dot{\lambda}_3 &= \frac{1}{2}\lambda_3,
\end{split}
\end{align}
where
\begin{align*}
h_{3,X}(r_3,\eps_3,y_3,\lambda_3)&=-\frac{\lambda_3^2}{\theta_\mathrm{lf}c^3} +\mathcal{O}\left(r_3^2, r_3\lambda_3, y_3\right),\\
h_{3,Y}(r_3,\eps_3,y_3,\lambda_3)&= 1+\mathcal{O}\left(r_3^2, r_3\lambda_3, y_3\right),\\
h_{3,\alpha}(r_3,\eps_3,y_3,\lambda_3)&=\mathcal{O}\left(r_3\lambda_3,r_3^2 \Upsilon_\mathrm{lf}(\lambda_2)\right),\\
g_{3,X}(r_3,\eps_3,y_3,\lambda_3)&=\mathcal{O}\left(r_3^2\right),\\
g_{3,Y}(r_3,\eps_3,y_3,\lambda_3)&=\mathcal{O}\left(r_3^4\right),\\
g_{3,\alpha}(r_3,\eps_3,y_3,\lambda_3)&=\mathcal{O}\left(r_3\lambda_3 ,  r_3^2\Upsilon_\mathrm{lf}(\lambda_2)\right).
\end{align*}
We have the following.
\begin{proposition}\label{prop:K3_estimates}
Consider~\eqref{eq:k3_transf} with $(r_3,y_3,\eps_3)=(r_{3,\eps},y_{3,\eps},\eps_{3,\eps})(z_3)$, and fix $M>0$. There exist $C, C_{\bar{r}_3},\mu>0$ such that the following holds. Given $(X_3^\mathrm{in}, Y_3^\mathrm{out})\in \mathbb{C}^2$ and any $\lambda\in \Lambda_{\mathrm{r},\eps}(\mu,M)$, there exists a solution $(X_3,Y_3,\lambda_3): [z_3^\mathrm{in}, z_3^\mathrm{out}]\to\mathbb{C}^2\times \Lambda_3(\delta,\mu,M)$ of~\eqref{eq:k3_transf} satisfying 
\begin{align*}
X_3(z_3^\mathrm{out}) = X_3^*(z_3^\mathrm{out})+\tilde{X}_3^\mathrm{out}, \qquad Y_3(z_3^\mathrm{out}) = Y_3^\mathrm{out},
\end{align*}
as well as the estimates
\begin{align*}
\left|X_3(z_3^\mathrm{in})-\Upsilon_\mathrm{lf}(\lambda_2)F_3(r_3^\mathrm{in},y_3^\mathrm{in},\eps_3^\mathrm{in})\right| &\leq C_{\bar{r}_3}\left(\left|\tilde{X}_3^\mathrm{out}\right| +\left|Y_3^\mathrm{out}\right|+ |\alpha \lambda \log \eps|\right)+C|\alpha \bar{r}_3^2 \Upsilon_\mathrm{lf}(\lambda_2)|,\\
 \left|Y_3(z_3^\mathrm{in})\right| &\leq C_{\bar{r}_3} \eps \left(|\tilde{X}_3^\mathrm{out}| +|Y_3^\mathrm{out}|+ |\alpha \lambda \log \eps|\right)+C\frac{\eps }{\bar{r}_3^4}\left|\alpha \Upsilon_\mathrm{lf}(\lambda_2)\right|,
\end{align*}
where $(r_3^\mathrm{in},y_3^\mathrm{in},\eps_3^\mathrm{in}):= (r_3,y_3,\eps_3)(z_3^\mathrm{in})$.

\end{proposition}
\begin{proof}
Using the relation $X_3 = X_3^*+\tilde{X}_3$, we solve~\eqref{eq:k3_transf_tilde} subject to the boundary conditions
$\tilde{X}_3(\tilde{z}_3^\mathrm{out})=\tilde{X}_3^\mathrm{out}$ and $Y_3(\tilde{z}_3^\mathrm{out}) = Y_3^\mathrm{out}$. 
We write~\eqref{eq:k3_transf_tilde} as the corresponding integral equation
\begin{align}
\begin{split}\label{eq:k3_transf_bar_int}
\tilde{X}_3(\tilde{z}_3) &=\tilde{X}_3^\mathrm{out}\re^{\beta_3(\tilde{z}_3,\tilde{z}_3^\mathrm{out}})+\int_{\tilde{z}_3^\mathrm{out}}^{\tilde{z}_3} \re^{\beta_3(\tilde{z}_3,s)}\left(Y_3h_{3,Y}(s)+\alpha h_{3,\alpha}(s)\right)\mathrm{d}s,\\
Y_3(\tilde{z}_3) &=Y_3^\mathrm{out}\frac{\eps_3(\tilde{z}_3)}{\delta}+\eps_3(\tilde{z}_3)\int_{\tilde{z}_3^\mathrm{out}}^{z_3}\tilde{X}_3(s) g_{3,X}(s) + Y_3(s) g_{3,Y}(s) +\alpha g_{3,\alpha}(s)\mathrm{d}s,
\end{split}
\end{align}
where we used the fact that $\eps_3(\tilde{z}_3^\mathrm{out}) = \delta$, and where
\begin{align*}
\beta_3(\tilde{z}_3, \tilde{z}_0) = \int_{\tilde{z}_0}^{\tilde{z}_3}h_{3,X}(s)\mathrm{d}s.
\end{align*}
We note that since $\lambda_3$ remains in the set $\Lambda_3(\delta,\mu,M)$ by assumption; by taking $\mu$ sufficiently small, we can ensure that $|\re^{\beta_3(\tilde{z}_3,\tilde{z}_0)}|$ is uniformly bounded, independently of $M,\delta,\bar{r}_3, \mu$, whenever $\tilde{z}_3^\mathrm{in}\leq \tilde{z}_3\leq \tilde{z}_0\leq \tilde{z}_3^\mathrm{out}$.

Considering~\eqref{eq:k3_transf_bar_int} as a fixed point equation on the space $\tilde{X}_3,Y_3\in C([\tilde{z}_3^\mathrm{in}, \tilde{z}_3^\mathrm{out}])$ with the norm 
\begin{align*}
    \|(\tilde{X}_3,Y_3)\| = \sup_{\tilde{z}_3\in [\tilde{z}_3^\mathrm{in}, \tilde{z}_3^\mathrm{out}]}\left|\tilde{X}_3(\tilde{z}_3)  \right| +\sup_{\tilde{z}_3\in [\tilde{z}_3^\mathrm{in}, \tilde{z}_3^\mathrm{out}]}\left|\eps_3(\tilde{z}_3)^{-1}Y_3(\tilde{z}_3)  \right|,
\end{align*}
we obtain a solution satisfying the estimates 
\begin{align*}
|\tilde{X}_3(\tilde{z}_3^\mathrm{in})| &\leq C_{\bar{r}_3} \left(|\tilde{X}_3^\mathrm{out}|+ |Y_3^\mathrm{out}|+ |\alpha \lambda \log \eps|\right)+C\bar{r}_3^2 |\alpha \Upsilon_\mathrm{lf}(\lambda_2)|,\\
|Y_3(\tilde{z}_3^\mathrm{in})| &\leq C_{\bar{r}_3} \eps \left(|\tilde{X}_3^\mathrm{out}|+ |Y_3^\mathrm{out}|+ |\alpha \lambda \log \eps|\right)+C\frac{\eps}{\bar{r}_3^4} |\alpha \Upsilon_\mathrm{lf}(\lambda_2)|.
\end{align*}
Noting that
\begin{align*}
X_3^*(\tilde{z}_3^\mathrm{in}) &=X_3^{\infty}F_3(r_3^\mathrm{in},\eps_3^\mathrm{in},y_3^\mathrm{in})+\mathcal{O}\left(X_3^\infty (\eps_3^\mathrm{in})^{2/3}, X_3^\infty (\lambda_3^\mathrm{in})^2\right),
\end{align*}
where $\eps_3^\mathrm{in}= \eps/(r_3^\mathrm{in})^6$ and $\lambda_3^\mathrm{in}= \lambda/r_3^\mathrm{in}$
and expressing the solution in terms of the original independent variable $z_3$, we therefore obtain a solution of~\eqref{eq:k3_transf} satisfying
\begin{align*}
\left|X_3(z_3^\mathrm{in})-\alpha \Upsilon_\mathrm{lf}(\lambda_2)F_3(r_3^\mathrm{in},y_3^\mathrm{in},\eps_3^\mathrm{in})\right| &= \left|X_3^*(z_3^\mathrm{in})+\tilde{X}_3(z_3^\mathrm{in})-\alpha \Upsilon_\mathrm{lf}(\lambda_2)F_3(r_3^\mathrm{in},y_3^\mathrm{in},\eps_3^\mathrm{in})\right|\\
&\leq C_{\bar{r}_3} \left(|\tilde{X}_3^\mathrm{out}|+ |Y_3^\mathrm{out}|+ |\alpha \lambda \log \eps|\right)+C\bar{r}_3^2|\alpha \Upsilon_\mathrm{lf}(\lambda_2)|,
\end{align*}
as claimed.
\end{proof}

We separately consider solutions which depart $\mathcal{K}_3$ via the section $\smash{\Sigma_{34}^\mathrm{out}}=\{\Im(\lambda_3)=M\delta^{1/6}\}$, corresponding to the section $\Sigma_{4}^\mathrm{in}$ from the chart $\mathcal{K}_4$; the case $\Im(\lambda_3)=-\smash{M\delta^{1/6}}$ is similar. 
 We split into the cases covered by Propositions~\ref{prop:K4_estimates} and~\ref{prop:K4_estimates2} separately, corresponding to orbits meeting $\Sigma_{34}^\mathrm{out}$ for values of $\delta_4 M^6 \leq \eps_3/\delta \leq 1$ and $0<\eps_3/\delta\leq \delta_4 M^6$, respectively, where we obtain slightly different estimates in each case. 

The case of $\delta_4 M^6\leq \eps_3/\delta \leq 1$, governed by Proposition~\ref{prop:K4_estimates}, is nearly identical to the argument above. Transforming the solution $X_4^*$ from $\mathcal{K}_4$ to $\mathcal{K}_3$ results in the expansion
\begin{align*}
X_3 = \alpha\frac{(\lambda_4+\ri)^2}{\eps_4^{1/3}\theta_\mathrm{lf}c^3}\frac{1}{\mathrm{Ai}'(-\Omega_0)^2}\left(1+\mathcal{O}(\lambda_3^2\eps_4^{1/3})\right)\int_\infty^{-\Omega_0}\re^{\frac{(\lambda_4+\ri)^2}{\eps_4^{1/3}\theta_\mathrm{lf} c^3}\left( s+\Omega_0\right)   } \left(\mathrm{Ai}'(s)^2-s\mathrm{Ai}(s)^2\right)\mathrm{d}s.
\end{align*}
This solution corresponds to the solution $X_3^*$ upon taking $X_3^\infty = \alpha \Upsilon_\mathrm{lf}((\lambda_4+\ri)\eps_4^{-1/6})$, now in the limit $\lambda_3\to 0$. This essentially extends the definition of $X_3^*$ above to values of $\lambda_2$ satisfying $|\Im(\lambda_2)|\leq \smash{\delta_4^{-1/6}}$, noting that $\Upsilon_\mathrm{lf}((\lambda_4+\ri)\smash{\eps_4^{-1/6}})$ is well defined as $\eps_4\to0$ provided $|\lambda_4|$ is sufficiently small.

On the other hand, for the case $0<\eps_3/\delta\leq \delta_4 M^6$, we recall the estimate~\eqref{eq:x4star_est2} satisfied by $\smash{X_4^\dagger(z_4^\mathrm{in};r_4)}$ from the proof of Proposition~\ref{prop:K4_estimates2}. In particular, we see that when $r_4=0$, $\smash{X_4^\dagger}$ corresponds to a solution in the chart $\mathcal{K}_3$ satisfying
\begin{align*}
    X_3=\alpha \left(1+\mathcal{O}\left(\lambda_3^2\eps_4^{2/3}, \eps_4^{1/3}\right)\right),
\end{align*}
which in turn corresponds to the solution $X_3^*$ in $\mathcal{K}_3$ upon taking
\begin{align*}
    X_3^\infty = \alpha \left(1+\mathcal{O}\left(\eps_4^{1/3}\right)\right)=\alpha \Upsilon_\mathrm{lf}\left((\lambda_4+\ri)\eps_4^{-1/6}\right)\left(1+\mathcal{O}\left(\eps_4^{1/3}\right)\right).
\end{align*}
 Analogously to Proposition~\ref{prop:K3_estimates}, we have the following.
\begin{proposition}\label{prop:K3_estimates4}
Consider~\eqref{eq:k3_transf} with $(r_3,y_3,\eps_3)=(r_{3,\eps},y_{3,\eps},\eps_{3,\eps})(z_3)$, and fix $M>0$. There exist $C_{\bar{r}_3}, C, \mu>0$ such that the following holds. Given $(X_3^\mathrm{in}, Y_3^\mathrm{out})\in \mathbb{C}^2$ and any $\lambda \in \Lambda_{\mathrm{c},\eps}(\mu,M)$, there exists a solution $(X_3,Y_3,\lambda_3): [z_3^\mathrm{in}, z_{34}^\mathrm{out}]\to\mathbb{C}^3$ of~\eqref{eq:k3_transf} satisfying 
\begin{align*}
X_3(z_{34}^\mathrm{out}) = X_3^*(z_{34}^\mathrm{out})+\tilde{X}_3^\mathrm{out}, \qquad Y_3(z_{34}^\mathrm{out}) = \eps_3(z_{34}^\mathrm{out})Y_3^\mathrm{out},
\end{align*}
as well as the estimates
\begin{align*}
\left|X_3(z_3^\mathrm{in})-\alpha \Upsilon_\mathrm{lf}(\lambda\eps^{-1/6}) F_3(r_3^\mathrm{in},y_3^\mathrm{in},\eps_3^\mathrm{in})\right| &\leq C_{\bar{r}_3}\left(|\tilde{X}_3^\mathrm{out}|+ |Y_3^\mathrm{out}| +|\alpha \lambda \log |\lambda||\right)+C\bar{r}_3^2 |\alpha| \left|\Upsilon_\mathrm{lf}(\lambda\eps^{-1/6})\right|,\\
 \left|Y_3(z_3^\mathrm{in})\right| &\leq C_{\bar{r}_3}\eps \left(|\tilde{X}_3^\mathrm{out}| +|Y_3^\mathrm{out}|+ |\alpha \lambda \log |\lambda||\right)+C \frac{\eps}{\bar{r}_3^4}|\alpha|,
\end{align*}
where $(r_3^\mathrm{in},y_3^\mathrm{in},\eps_3^\mathrm{in}):= (r_3,y_3,\eps_3)(z_3^\mathrm{in})$.
\end{proposition}
\begin{proof}
We again rewrite the system~\eqref{eq:k3_transf_tilde} as an integral equation
\begin{align*}
\begin{split}
\tilde{X}_3(\tilde{z}_3) &=\tilde{X}_3^\mathrm{out}\re^{\beta_3(\tilde{z}_3,\tilde{z}_{34}^\mathrm{out}})+\int_{\tilde{z}_{34}^\mathrm{out}}^{\tilde{z}_3} \re^{\beta_3(z_3,s)}\left(Y_3h_{3,Y}(s)+\alpha h_{3,\alpha}(s)\right)\mathrm{d}s,\\
Y_3(z_3) &=Y_3^\mathrm{out}\eps_3(\tilde{z}_3)+\eps_3(\tilde{z}_3)\int_{\tilde{z}_{34}^\mathrm{out}}^{\tilde{z}_3}\bar{X}_3(s) g_{3,X}(s) + Y_3(s) g_{3,Y}(s) +\alpha g_{3,\alpha}(s)\mathrm{d}s
\end{split}
\end{align*}
on the transformed interval $\tilde{z}_3\in [\tilde{z}_3^\mathrm{in}, \tilde{z}_{34}^\mathrm{out}]$
where the functions $g_*, h_*, \beta_3$ are as in~\eqref{eq:k3_transf_tilde} except that $h_{3,\alpha}, g_{3,\alpha}$ now satisfy the slightly modified estimates
\begin{align*}
h_{3,\alpha}(r_3,\eps_3,y_3,\lambda_3)&=\mathcal{O}(r_3\lambda_3 ,  r_3^2),\\
    g_{3,\alpha}(r_3,\eps_3,y_3,\lambda_3)&=\mathcal{O}(r_3\lambda_3 ,  r_3^2).
\end{align*}
Continuing as in the proof of Proposition~\ref{prop:K3_estimates}, the corresponding solution now satisfies the estimates
\begin{align*}
|\tilde{X}_3(\tilde{z}_3^\mathrm{in})| &\leq C_{\bar{r}_3} \left(|\tilde{X}_3^\mathrm{out}|+ |Y_3^\mathrm{out}|+ |\alpha \lambda \log |\lambda||\right)+C\bar{r}_3^2 \left|\alpha \Upsilon_\mathrm{lf}\left((\lambda_4+\ri)\eps_4^{-1/6}\right)\right|,\\
|Y_3(\tilde{z}_3^\mathrm{in})| &\leq C_{\bar{r}_3} \eps \left(|\tilde{X}_3^\mathrm{out}|+ |Y_3^\mathrm{out}|+ |\alpha \lambda \log |\lambda||\right)+C\frac{\eps}{\bar{r}_3^4} \left|\alpha \Upsilon_\mathrm{lf}\left((\lambda_4+\ri)\eps_4^{-1/6}\right)\right|,
\end{align*}
where we note that the interval $[\tilde{z}_3^\mathrm{in}, \tilde{z}_{34}^\mathrm{out}]$ is now of length $\mathcal{O}(\log |\lambda|)$. The result follows as in the proof of Proposition~\ref{prop:K3_estimates}.
\end{proof}

\subsubsection{Proof of Proposition~\ref{prop:fold_bvp}}\label{sec:fold_bvp_proof}
We can now complete the proof of Proposition~\ref{prop:fold_bvp}, concerning the estimates for the eigenvalue problem through the fold.

\begin{proof}[Proof of Proposition~\ref{prop:fold_bvp}]
We combine the results of Propositions~\ref{prop:K1_estimates}-\ref{prop:K3_estimates4}. For values of $\lambda\in \Lambda_{\mathrm{r},\eps}(\mu, M)$, solutions are tracked through the chart $\mathcal{K}_2$, while for $\lambda\in \Lambda_{\mathrm{c},\eps}(\mu, M)$ we use the chart $\mathcal{K}_4$. We first consider the former.

We begin by matching the solutions in the charts $\mathcal{K}_1,\mathcal{K}_2$. By Proposition~\ref{prop:K1_estimates}, in the section $\Sigma_1^\mathrm{in}$, the solution $(X_1, Y_1)$ satisfies
\begin{align*}
\left|X_1(z_1^\mathrm{in}) - X_1^*(z_1^\mathrm{in})\right|\leq C_{\bar{r}_1}\left(|X_1^\mathrm{out}|\re^{-\theta_{\bar{r}_1}/\eps}+ |\alpha \lambda|\right), \qquad Y_1(z_1^\mathrm{in})=0,
\end{align*} 
which, transformed into the $\mathcal{K}_2$-coordinates, corresponds to a solution satisfying
\begin{align*}
\left|X_2(z_2^\mathrm{out}) - X_2^*(z_2^\mathrm{out})\right|\leq C_{\bar{r}_1}\left(|X_1^\mathrm{out}|\re^{-\theta_{\bar{r}_1}/\eps}+ |\alpha \lambda|\right), \qquad Y_2(z_2^\mathrm{out})=0
\end{align*} 
in the section $\Sigma_2^\mathrm{out}$. Thus, we can match with a solution from Proposition~\ref{prop:K2_estimates} by choosing $X_2^\mathrm{out}$ appropriately, satisfying 
\begin{align*}
\left|X_2^\mathrm{out}\right|\leq C_{\bar{r}_1}\left(|X_1^\mathrm{out}|\re^{-\theta_{\bar{r}_1}/\eps}+ |\alpha \lambda|\right).
\end{align*}
Recalling Proposition~\ref{prop:K2_estimates}, and using the fact that the transition time $|z_2^\mathrm{in}-z_2^\mathrm{out}|$ is bounded independently of $r_2,\lambda_2$ to bound the exponential factor in the solution $X_2(z_2^\mathrm{in})$,  the solution therefore satisfies the estimates
\begin{align*}
\begin{split}
\left|X_2(z_2^\mathrm{in}) - X_2^*(z_2^\mathrm{in})\right|&\leq C_{\bar{r}_1}\left(|X_1^\mathrm{out}|\re^{-\theta_{\bar{r}_1}/\eps}+ |\alpha r_2\lambda_2|\right),\\
\left|Y_2(z_2^\mathrm{in})\right|&\leq C_{\bar{r}_1}\left(|X_1^\mathrm{out}|\re^{-\theta_{\bar{r}_1}/\eps}+ |\alpha r_2\lambda_2|\left(|r_2|+|\lambda_2|\right)\right)
\end{split}
\end{align*} 
in the section $\Sigma_2^\mathrm{in}$, where the constant $C_{\bar{r}_1}$ may be taken larger if necessary.  
Transforming into the $\mathcal{K}_3$-coordinates, in the section $\Sigma_3^\mathrm{out}$, this corresponds to a solution satisfying
\begin{align*}
\begin{split}
\left|X_3(z_3^\mathrm{out}) - X_3^*(z_3^\mathrm{out})\right|&\leq C_{\bar{r}_1}\left(|X_1^\mathrm{out}|\re^{-\theta_{\bar{r}_1}/\eps}+ |\alpha r_2\lambda_2|\right),\\
\left|Y_3(z_3^\mathrm{out})\right|&\leq C_{\bar{r}_1}\left(|X_1^\mathrm{out}|\re^{-\theta_{\bar{r}_1}/\eps}+ |\alpha r_2\lambda_2|\left(|r_2|+|\lambda_2|\right)\right),
\end{split}
\end{align*}
which in turn corresponds to a solution of Proposition~\ref{prop:K3_estimates} for appropriate choice of $\tilde{X}_3^\mathrm{out},Y_3^\mathrm{out}\in \mathbb{C}$, satisfying
\begin{align*}
\begin{split}
\left|\tilde{X}_3^\mathrm{out}\right| &\leq C_{\bar{r}_1}\left(|X_1^\mathrm{out}|\re^{-\theta_{\bar{r}_1}/\eps}+ |\alpha \lambda|\right),\\
\left|Y_3^\mathrm{out}\right|&\leq C_{\bar{r}_1}\left(|X_1^\mathrm{out}|\re^{-\theta_{\bar{r}_1}/\eps}+\left|\alpha \eps^{1/3}\lambda\right|+ \left|\alpha \lambda^2\right|\right).
\end{split}
\end{align*}
In the section $\Sigma_3^\mathrm{in}$, this solution therefore satisfies the estimates
\begin{align*}
\begin{split}
\left|X_3(z_3^\mathrm{in})-\alpha \Upsilon_\mathrm{lf}(\lambda_2)F_3(r_3^\mathrm{in},y_3^\mathrm{in},\eps_3^\mathrm{in})\right| &\leq C_{\bar{r}_1, \bar{r}_3}\left(\left|X_1^\mathrm{out}\right|\re^{-\theta_{\bar{r}_1}/\eps} + |\alpha \lambda \log \eps|\right)+C|\alpha \bar{r}_3^2 \Upsilon_\mathrm{lf}(\lambda_2)|,\\
 \left|Y_3(z_3^\mathrm{in})\right| &\leq C_{\bar{r}_1, \bar{r}_3} \eps \left(\left|X_1^\mathrm{out}\right|\re^{-\theta_{\bar{r}_1}/\eps}+ |\alpha \lambda \log \eps|\right)+C\frac{\eps }{\bar{r}_3^4}|\alpha \Upsilon_\mathrm{lf}(\lambda_2)|,
 \end{split}
\end{align*}
which, transformed into the original (blow-down) coordinates, satisfies
\begin{align}
\begin{split}\label{fold:match_lambdar_blowdown}
\left|\bar{r}_3^4 X_3(z_3^\mathrm{in})+\alpha_3(\alpha; \lambda,\eps)X_\eps(\xi_\mathrm{in})\right| &\leq C_{\bar{r}_1, \bar{r}_3}\left(\left|X_1^\mathrm{out}\right|\re^{-\theta_{\bar{r}_1}/\eps}\right),\\
\left|\bar{r}_3^6Y_3(z_3^\mathrm{in})\right| &\leq C_{\bar{r}_1, \bar{r}_3} \eps \left(\left|X_1^\mathrm{out}\right|\re^{-\theta_{\bar{r}_1}/\eps}+ |\alpha \lambda \log \eps|\right)+C\bar{r}_3^2 \eps|\alpha \Upsilon_\mathrm{lf}(\lambda \eps^{-1/6})|.
\end{split}
\end{align}
where
\begin{align*}
\left|\alpha_3(\alpha; \lambda,\eps)- \alpha \Upsilon_\mathrm{lf}(\lambda \eps^{-1/6})\right|\leq C_{\bar{r}_1, \bar{r}_3}|\alpha \lambda \log \eps|+C|\alpha \bar{r}_3^2 \Upsilon_\mathrm{lf}(\lambda\eps^{-1/6})|.
\end{align*}
Finally, we match the full solution with the exit conditions at $\zeta=\zeta_\mathrm{out}$ by transforming from the $\mathcal{K}_1$-coordinates in the section $\Sigma_1^\mathrm{out}$, namely 
\begin{align*}
\begin{split}
X_\mathrm{out}&=\alpha X_\eps(\zeta_\mathrm{out})+\bar{r}_1^4 X_1^\mathrm{out}+R_1^x\left(\alpha, X_1^\mathrm{out};\lambda,\eps\right),\\
\eps Y_\mathrm{out}&=\alpha Y_\eps(\zeta_\mathrm{out})+\eps R_1^y\left(\alpha,X_1^\mathrm{out};\lambda,\eps\right),
\end{split}
\end{align*}
where 
\begin{align*}
    \left|R_1^x\left(\alpha, X_1^\mathrm{out};\lambda,\eps\right)\right|&\leq C_{\bar{r}_1}|\alpha \eps\lambda|,\\
    \left|R_1^y\left(\alpha,X_1^\mathrm{out};\lambda,\eps\right)\right|&\leq C_{\bar{r}_1}|\alpha\lambda| +C\left| X_1^\mathrm{out} \right|.
\end{align*}
From the equation for $Y_\mathrm{out}$, we find that
\begin{align}\label{eq:K1solve-start}
Y_\mathrm{out}&=\alpha \frac{Y_\eps(\zeta_\mathrm{out})}{\eps}+R_1^y\left(\alpha,X_1^\mathrm{out};\lambda,\eps\right),
\end{align}
which, noting that $Y_\eps(\zeta_\mathrm{out})=\eps(1+\mathcal{O}(\bar{r}_1,\eps))$ can be solved for $\alpha$, satisfying the estimate
\begin{align*}
\left|\alpha-\frac{\eps Y_\mathrm{out}}{Y_\eps(\zeta_\mathrm{out})}\right|\leq C_{\bar{r}_1}\left(\left|X_1^\mathrm{out}\right|+|\lambda|\left|Y_\mathrm{out}\right|\right).
\end{align*}
Substituting into the equation for $X_\mathrm{out}$, we find that 
\begin{align*}
X_\mathrm{out}&=\bar{r}_1^4 X_1^\mathrm{out}+\tilde{R}_1^x\left( X_1^\mathrm{out}, Y_\mathrm{out};\lambda,\eps\right),
\end{align*}
where 
\begin{align*}
    \left|\tilde{R}_1^x\left( X_1^\mathrm{out}, Y_\mathrm{out};\lambda,\eps\right)\right|&\leq C_{\bar{r}_1} \eps\left(\left|X_1^\mathrm{out}\right|+\left|Y_\mathrm{out}\right|  \right),
\end{align*}
which we can solve for $X_1^\mathrm{out}$ satisfying the estimate
\begin{align} \label{eq:K1solve-end}
\left|X_1^\mathrm{out}-\frac{X_\mathrm{out}}{\bar{r}_1^4}\right|\leq  C_{\bar{r}_1} \eps\left(\left|X_1^\mathrm{out}\right|+\left|Y_\mathrm{out}\right|  \right).
\end{align}
Substituting these estimates into~\eqref{fold:match_lambdar_blowdown}, at $\zeta=\zeta_\mathrm{in}$ we obtain the following
\begin{align*}
\begin{split}
\left|X(\zeta_\mathrm{in})-\alpha X_\eps(\zeta_\mathrm{in})+\alpha_3(\alpha; \lambda,\eps)X_\eps(\xi_\mathrm{in})\right| &\leq  C_{\bar{r}_1, \bar{r}_3}\left(\left|X_1^\mathrm{out}\right|\re^{-\theta_{\bar{r}_1}/\eps}\right),\\
\left|Y(\zeta_\mathrm{in})-\alpha Y_\eps(\zeta_\mathrm{in})\right| &\leq C_{\bar{r}_1, \bar{r}_3} \eps \left(\left|X_1^\mathrm{out}\right|\re^{-\theta_{\bar{r}_1}/\eps}+ |\alpha \lambda \log \eps|\right)+C\bar{r}_3^2 \eps|\alpha \Upsilon_\mathrm{lf}(\lambda \eps^{-1/6})|.
\end{split}
\end{align*}
so that
\begin{align*}
\begin{split}
\left|X(\zeta_\mathrm{in})- \alpha_{\mathrm{lf},\bar{r}_1,\bar{r}_3}^x(Y_\mathrm{out}; \eps,\lambda)X_\eps(\zeta_\mathrm{in})\right| &\leq C_{\bar{r}_1, \bar{r}_3} \left(\left|X_\mathrm{out}\right|+ \left(\eps+|\lambda \log \eps|\right)\left|Y_\mathrm{out}\right|\right),\\
\left|Y(\zeta_\mathrm{in})-\frac{\eps Y_\eps(\zeta_\mathrm{in}) }{Y_\eps(\zeta_\mathrm{out})} Y_\mathrm{out}\right| &\leq C_{\bar{r}_1, \bar{r}_3} \eps \left(\left|X_\mathrm{out}\right|+ \left(\eps+|\lambda \log \eps|\right)\left|Y_\mathrm{out}\right|\right)+C\bar{r}_3^2 \eps\left| \Upsilon_\mathrm{lf}(\lambda \eps^{-1/6})\right||Y_\mathrm{out}|.
\end{split}
\end{align*}
where
\begin{align}\label{eq:alphalf-preestimate}
\alpha_{\mathrm{lf},\bar{r}_1,\bar{r}_3}^x(Y_\mathrm{out}; \eps,\lambda)
&=\frac{\eps Y_\mathrm{out}}{Y_\eps(\zeta_\mathrm{out})}\left(1- \Upsilon_\mathrm{lf}(\lambda \eps^{1/6})\left(1+\mathcal{O}\left(\bar{r}_3^2\right)\right) \right),
\end{align}
where we note that the quantities $\bar{r}_1, \bar{r}_3$ satisfy $\bar{r}_1^4 = y_\mathrm{out}(\nu,\eps)$ and $\bar{r}_3^2 = x_\mathrm{in}(\nu,\eps)$ and hence can be taken smaller as $\nu\to0$; see Lemma~\ref{lem:fold_existence}. Therefore, we obtain the desired estimate
\begin{align}\label{eq:finalestimate-lambdar}
\begin{split}
\left|X(\zeta_\mathrm{in})- \alpha_{\mathrm{lf},\nu}^x(Y_\mathrm{out}; \eps,\lambda)x_\eps'\left(\xi_{\mathrm{lf},\eps,\nu}^\mathrm{in}\right)\right| &\leq C_\nu \left(\left|X_\mathrm{out}\right|+ \left(\eps+|\lambda \log \eps|\right)\left|Y_\mathrm{out}\right|\right),\\
\left|Y(\zeta_\mathrm{in})-\eps Y_\mathrm{out}\frac{ y_\eps'\left(\xi_{\mathrm{lf},\eps,\nu}^\mathrm{in}\right) }{y_\eps'\left(\xi_{\mathrm{lf},\eps,\nu}^{\mathrm{out},L}\right)} \right| &\leq C_\nu \eps \left(\left|X_\mathrm{out}\right|+ \left(\eps+|\lambda \log \eps|\right)\left|Y_\mathrm{out}\right|\right)\\
&\qquad \, +\eta(\nu) \eps\left| \Upsilon_\mathrm{lf}(\lambda \eps^{-1/6})\right||Y_\mathrm{out}|,
\end{split}
\end{align}
for $\lambda\in \Lambda_{\mathrm{r},\eps}(\mu,M)$ and $\eta(\nu)$ satisfying $\eta(\nu)\to0$ as $\nu\to0$, where
\begin{align}\label{eq:alphalf-estimate}
\alpha_{\mathrm{lf},\nu}^x(Y_\mathrm{out}; \eps,\lambda)
&=\frac{\eps Y_\mathrm{out}}{y_\eps'\left(\xi_{\mathrm{lf},\eps,\nu}^{\mathrm{out},L}\right)}\left(1- \Upsilon_\mathrm{lf}(\lambda \eps^{1/6})\left(1+\mathcal{O}\left(\eta(\nu)\right)\right) \right),
\end{align}
and we used the relation
\begin{align*}
    \left(X_\eps, Y_\eps  \right)(\zeta) = \frac{1}{\theta_\mathrm{lf}}\left(x_\eps,y_\eps   \right)'\left( \frac{\zeta}{\theta_\mathrm{lf}}  \right)
\end{align*}
and $\zeta_\mathrm{in} = \theta_\mathrm{lf}\xi_{\mathrm{lf},\eps,\nu}^\mathrm{in}$, $\zeta_\mathrm{out} = \theta_\mathrm{lf}\xi_{\mathrm{lf},\eps,\nu}^{\mathrm{out},L}$.

Next, we consider values of $\lambda\in\Lambda_{\mathrm{c},\eps}(\mu,M)$, for which solutions must pass through the chart $\mathcal{K}_4$. We first match the solutions in the charts $\mathcal{K}_1,\mathcal{K}_4$. By Proposition~\ref{prop:K1_estimates4}, in the section $\Sigma_{14}^\mathrm{in}$, the solution $(X_1, Y_1)$ satisfies
\begin{align*}
\left|X_1(z_1^\mathrm{in}) - X_1^*(z_1^\mathrm{in})\right|\leq C_{\bar{r}_1}\left(|X_1^\mathrm{out}|+ |\alpha \lambda|\right), \qquad Y_1(z_1^\mathrm{in})=0,
\end{align*} 
which, transformed into the $\mathcal{K}_4$-coordinates, corresponds to a solution satisfying
\begin{align*}
\left|X_4(z_4^\mathrm{out}) - X_4^*(z_4^\mathrm{out})\right|\leq C_{\bar{r}_1}\left(|X_1^\mathrm{out}|+ |\alpha \lambda|\right), \qquad Y_4(z_4^\mathrm{out})=0
\end{align*} 
in the section $\Sigma_4^\mathrm{out}$. Transforming to the $\mathcal{K}_4$-coordinates, we either match with a solution from Proposition~\ref{prop:K4_estimates} or Proposition~\ref{prop:K4_estimates2}, depending on whether $\eps_4\geq \delta_4$ or $\eps_4\leq \delta_4$. 

\paragraph{Case $ \delta_4\leq \eps_4\leq M^{-6}$: } In this case, we use Proposition~\ref{prop:K4_estimates}, matching with a solution satisfying
\begin{align*}
\left|X_4^\mathrm{out}\right|\leq C_{\bar{r}_1}\left(|X_1^\mathrm{out}|+ |\alpha \lambda|\right).
\end{align*}
Using the fact that the transition time $|z_4^\mathrm{in}-z_4^\mathrm{out}|$ is bounded independently of $r_4$ and $\eps_4\geq \delta_4$ to bound the exponential factor for the solution $X_4(z_4^\mathrm{in})$ in Proposition~\ref{prop:K4_estimates}, the solution therefore satisfies the estimates
\begin{align*}
\begin{split}
\left|X_4(z_4^\mathrm{in}) -  X_4^*(z_4^\mathrm{in})\right|&\leq C_{\bar{r}_1}\left(\left( \re^{-\theta_4/\eps_4}+|r_4| \right)|X_1^\mathrm{out}|+|\alpha r_4|\right), \\
\left|Y_4(z_4^\mathrm{in})\right| &\leq C_{\bar{r}_1}|r_4|^2\left(|X_1^\mathrm{out}|+ |\alpha| \right),
\end{split}
\end{align*} 
Transforming into the $\mathcal{K}_3$-coordinates, in the section $\Sigma_{34}^\mathrm{out}$, this solution satisfies the estimates
\begin{align*}
\begin{split}
\left|X_3(z_{34}^\mathrm{out}) - X_3^*(z_{34}^\mathrm{out})\right|&\leq C_{\bar{r}_1}\left(\left( \re^{-\theta_4/\eps_4}+|r_4| \right)|X_1^\mathrm{out}|+|\alpha r_4|\right),\\
\left|Y_3(z_{34}^\mathrm{out})\right|&=C_{\bar{r}_1}|r_4|^2\left(|X_1^\mathrm{out}|+ |\alpha| \right),
\end{split}
\end{align*}
which in turn corresponds to a solution of Proposition~\ref{prop:K3_estimates4} by taking appropriate $\tilde{X}_3^\mathrm{out}, Y_3^\mathrm{out}$, which satisfy the estimates
\begin{align*}
\begin{split}
\left|\tilde{X}_3^\mathrm{out}\right| &\leq C_{\bar{r}_1}\left(\left( \re^{-\theta_4/\eps_4}+|r_4| \right)|X_1^\mathrm{out}|+|\alpha r_4|\right),\\
\left|Y_3^\mathrm{out}\right|&\leq C_{\bar{r}_1}|r_4|^2\left(|X_1^\mathrm{out}|+ |\alpha| \right).
\end{split}
\end{align*}
In the section $\Sigma_3^\mathrm{in}$, this solution therefore satisfies the estimates
\begin{align}
\begin{split}\label{eq:k3finalests}
\left|X_3(z_3^\mathrm{in})-\alpha \Upsilon_\mathrm{lf}(\lambda\eps^{-1/6}) F_3(r_3^\mathrm{in},y_3^\mathrm{in},\eps_3^\mathrm{in})\right| &\leq C_{\bar{r}_1, \bar{r}_3}\left(|X_1^\mathrm{out}| +|\alpha \lambda \log |\lambda||\right)+C\bar{r}_3^2 |\alpha| \left|\Upsilon_\mathrm{lf}(\lambda\eps^{-1/6})\right|,\\
 \left|Y_3(z_3^\mathrm{in})\right| &\leq C_{\bar{r}_1, \bar{r}_3}\eps \left(|X_1^\mathrm{out}|+ |\alpha \lambda \log |\lambda||\right)+C \frac{\eps}{\bar{r}_3^4}|\alpha|.
\end{split}
\end{align}
\paragraph{Case $0<\eps_4\leq \delta_4$:} We proceed similarly as above, though we match instead with a solution from Proposition~\ref{prop:K4_estimates2} by choosing appropriate $X_4^\mathrm{out}$, which satisfies
\begin{align*}
\left|X_4^\mathrm{out}\right|\leq C_{\bar{r}_1}\left(|X_1^\mathrm{out}|+ \left(|\eps_4|+|\lambda|\right)|\alpha|\right),
\end{align*}
where we used~\eqref{eq:k1_x1star} to estimate $X_1^*(z_1^\mathrm{in})$ in the $\mathcal{K}_4$-coordinates. By Proposition~\ref{prop:K4_estimates2}, this solution therefore satisfies the estimates
\begin{align*}
\begin{split}
\left|X_4(z_4^\mathrm{in})- X_4^\dagger(z_4^\mathrm{in};r_4)\right|&\leq C_{\bar{r}_1} \left(\left(\re^{-\theta_4/\eps_4}  +r_4^2\eps_4\right)|X_1^\mathrm{out}| + \left(\re^{-\theta_4/\eps_4}  +r_4\eps_4\right)|\alpha|\right), \\
\left|Y_4(z_4^\mathrm{in})\right| &\leq C_{\bar{r}_1} r_4^2\eps_4\left(|X_1^\mathrm{out}| + |\alpha|\right)
\end{split}
\end{align*} 
in the section $\Sigma_4^\mathrm{in}$. Transforming into the $\mathcal{K}_3$-coordinates, in the section $\Sigma_{34}^\mathrm{out}$, this corresponds to a solution satisfying
\begin{align*}
\begin{split}
\left|X_3(z_{34}^\mathrm{out}) - X_3^*(z_{34}^\mathrm{out})\right|&\leq C_{\bar{r}_1} \left(\left(\re^{-\theta_4/\eps_4}  +r_4^2\eps_4\right)|X_1^\mathrm{out}| + \left(\re^{-\theta_4/\eps_4}  +r_4\right)|\alpha|\right), \\
\left|Y_3(z_{34}^\mathrm{out})\right|&\leq C_{\bar{r}_1} r_4^2\eps_4\left(|X_1^\mathrm{out}| + |\alpha|\right),
\end{split}
\end{align*}
which in turn corresponds to a solution of Proposition~\ref{prop:K3_estimates} by choosing $\tilde{X}_3^\mathrm{out}, Y_3^\mathrm{out}$ appropriately, satisfying
\begin{align*}
\begin{split}
\left|\tilde{X}_3^\mathrm{out}\right| &\leq  C_{\bar{r}_1} \left(\left(\re^{-\theta_4/\eps_4}  +r_4^2\eps_4\right)|X_1^\mathrm{out}| + \left(\re^{-\theta_4/\eps_4}  +r_4\right)|\alpha|\right),\\
\left|Y_3^\mathrm{out}\right|&\leq C_{\bar{r}_1} r_4^2\eps_4\left(|X_1^\mathrm{out}| + |\alpha|\right).
\end{split}
\end{align*}
In the section $\Sigma_3^\mathrm{in}$, this solution therefore satisfies the estimates~\eqref{eq:k3finalests}, where the correction term $\re^{-\theta_4/\eps_4}|\alpha|$ can be absorbed into the final term in the first inequality of~\eqref{eq:k3finalests} by taking $\delta_4$ sufficiently small relative to $\bar{r}_1, \bar{r}_3$, and using Proposition~\ref{prop:I0properties}.

Taking into account the estimates~\eqref{eq:k3finalests} from the two cases, we have a unified estimate for the solution in $\Sigma_3^\mathrm{in}$ for all values of $0<\eps_4\leq M^{-6}$ (and therefore all $\lambda\in\Lambda_{\mathrm{c},\eps}(\mu,M)$ for $\mu>0$ sufficiently small). Transformed into the original (blow-down) coordinates, the solution therefore satisfies
\begin{align}
\begin{split}\label{eq:K4finalmatch-blowdown}
\left|\bar{r}_3^4X_3(z_3^\mathrm{in})+ \alpha_3(\alpha; \lambda,\eps)X_\eps(\xi_\mathrm{in})\right| &\leq C_{\bar{r}_1, \bar{r}_3}|X_1^\mathrm{out}|,\\
 \left|\bar{r}_3^6Y_3(z_3^\mathrm{in})\right| &\leq C_{\bar{r}_1, \bar{r}_3}\eps \left(|X_1^\mathrm{out}|+ |\alpha \lambda \log |\lambda||\right)+C \bar{r}_3^2\eps |\alpha|,
\end{split}
\end{align}
where
\begin{align*}
\left|\alpha_3(\alpha; \lambda,\eps)- \alpha \Upsilon_\mathrm{lf}(\lambda \eps^{-1/6})\right|\leq C_{\bar{r}_1, \bar{r}_3}|\alpha \lambda \log |\lambda||+C|\alpha \bar{r}_3^2 \Upsilon_\mathrm{lf}(\lambda\eps^{-1/6})|.
\end{align*}
Finally, we match the full solution with the exit conditions at $\zeta=\zeta_\mathrm{out}$ by transforming from the $\mathcal{K}_1$ coodinates in the section $\Sigma_1^\mathrm{out}$, namely 
\begin{align*}
\begin{split}
X_\mathrm{out}&=\alpha X_\eps(\zeta_\mathrm{out})+\bar{r}_1^4 X_1^\mathrm{out}+R_1^x\left(\alpha, X_1^\mathrm{out};\lambda,\eps\right),\\
\eps Y_\mathrm{out}&=\alpha Y_\eps(\zeta_\mathrm{out})+\eps R_1^y\left(\alpha,X_1^\mathrm{out};\lambda,\eps\right),
\end{split}
\end{align*}
where 
\begin{align*}
    \left|R_1^x\left(\alpha, X_1^\mathrm{out};\lambda,\eps\right)\right|&\leq C_{\bar{r}_1}|\alpha \eps\lambda|,\\
    \left|R_1^y\left(\alpha,X_1^\mathrm{out};\lambda,\eps\right)\right|&\leq C_{\bar{r}_1}|\alpha\lambda| +C\left| X_1^\mathrm{out} \right|.
\end{align*}
Proceeding as in~\eqref{eq:K1solve-start}--\eqref{eq:K1solve-end}, we solve to find $\alpha, X_1^\mathrm{out},$ satisfying the estimates
\begin{align*}
\left|\alpha-\frac{\eps Y_\mathrm{out}}{Y_\eps(\zeta_\mathrm{out})}\right|&\leq C_{\bar{r}_1}\left(\left|X_1^\mathrm{out}\right|+|\lambda|\left|Y_\mathrm{out}\right|\right),\\
\left|X_1^\mathrm{out}-\frac{X_\mathrm{out}}{\bar{r}_1^4}\right|&\leq  C_{\bar{r}_1} \eps\left(\left|X_1^\mathrm{out}\right|+\left|Y_\mathrm{out}\right|  \right).
\end{align*}
Substituting these estimates at $\zeta=\zeta_\mathrm{in}$ using~\eqref{eq:K4finalmatch-blowdown}, we obtain the following 
\begin{align*}
\begin{split}
\left|X(\zeta_\mathrm{in})-\alpha X_\eps(\zeta_\mathrm{in})+\alpha_3(\alpha; \lambda,\eps)X_\eps(\zeta_\mathrm{in})\right| &\leq C_{\bar{r}_1, \bar{r}_3}|X_1^\mathrm{out}|,\\
\left|Y(\zeta_\mathrm{in})-\alpha Y_\eps(\zeta_\mathrm{in})\right| &\leq C_{\bar{r}_1, \bar{r}_3} \eps \left(\left|X_1^\mathrm{out}\right|+ |\alpha \lambda \log |\lambda||\right)+C\bar{r}_3^2 \eps|\alpha \Upsilon_\mathrm{lf}(\lambda \eps^{-1/6})|,
\end{split}
\end{align*}
so that
\begin{align*}
\begin{split}
\left|X(\zeta_\mathrm{in})- \alpha_{\mathrm{lf},\bar{r}_1,\bar{r}_3}^x(Y_\mathrm{out}; \eps,\lambda)X_\eps(\zeta_\mathrm{in})\right| &\leq C_{\bar{r}_1, \bar{r}_3}\left(|X_\mathrm{out}|+\left(\eps +|\lambda|\log|\lambda||\right) |Y_\mathrm{out}|\right),\\
\left|Y(\zeta_\mathrm{in})-\frac{\eps Y_\eps(\zeta_\mathrm{in})}{Y_\eps(\zeta_\mathrm{out})}Y_\mathrm{out}\right| &\leq C_{\bar{r}_1, \bar{r}_3} \eps  |X_\mathrm{out}|+\left(C_{\bar{r}_1, \bar{r}_3}\eps|\lambda|\log|\lambda||+C\bar{r}_3^2 \eps \right) |Y_\mathrm{out}|,
\end{split}
\end{align*}
where $\alpha_{\mathrm{lf},\bar{r}_1,\bar{r}_3}^x(Y_\mathrm{out}; \eps,\lambda)$ satisfies~\eqref{eq:alphalf-preestimate}. Again we note that the quantities $\bar{r}_1, \bar{r}_3$ satisfy $\smash{\bar{r}_1^4 = y_\mathrm{out}}(\nu,\eps)$ and $\smash{\bar{r}_3^2 = x_\mathrm{in}}(\nu,\eps)$ and hence can be taken smaller as $\nu\to0$. Analogously to~\eqref{eq:finalestimate-lambdar}, we obtain the desired estimate
\begin{align*}
\begin{split}
\left|X(\zeta_\mathrm{in})- \alpha_{\mathrm{lf},\nu}^x(Y_\mathrm{out}; \eps,\lambda)x_\eps'\left(\xi_{\mathrm{lf},\eps,\nu}^\mathrm{in}\right)\right| &\leq C_\nu\left(|X_\mathrm{out}|+\left(\eps +|\lambda|\log|\lambda||\right) |Y_\mathrm{out}|\right),\\
\left|Y(\zeta_\mathrm{in})-\eps Y_\mathrm{out}\frac{ y_\eps'\left(\xi_{\mathrm{lf},\eps,\nu}^\mathrm{in}\right) }{y_\eps'\left(\xi_{\mathrm{lf},\eps,\nu}^{\mathrm{out},L}\right)} \right| &\leq C_\nu \eps  |X_\mathrm{out}|+\left(C_\nu\eps|\lambda|\log|\lambda||+\eta(\nu) \eps \right) |Y_\mathrm{out}|,
\end{split}
\end{align*}
for $\lambda\in \Lambda_{\mathrm{c},\eps}(\mu,M)$, where $\alpha_{\mathrm{lf},\nu}^x(Y_\mathrm{out}; \eps,\lambda)$ satisfies~\eqref{eq:alphalf-estimate}.
\end{proof}

\subsection{A tame estimate on the center-unstable \texorpdfstring{$(X,Y)$}{(X,Y)}-dynamics for \texorpdfstring{$\Re(\lambda) \geq \eps^{1/5}$}{Re(lambda)>= eps\^(1/5)}}\label{sec:fold_tame}

The center-unstable $(X,Y)$-dynamics in the transformed eigenvalue problem~\eqref{eq:transformed_lin_fold_diag} near the fold are given by
\begin{align} \label{centerunstable}
\Psi_\zeta = \left(\tilde{A}(\zeta;\eps,\lambda) + \mathcal{O}(\eps)\right) \Psi, \qquad \tilde{A}(\zeta;\eps,\lambda) = \begin{pmatrix} a_1(\zeta;\eps) + \tilde{a}_1(\zeta;\eps,\lambda) & a_2(\zeta;\eps,\lambda) \\ 0 & 0 \end{pmatrix}
\end{align}
with $\zeta \in [\zeta_{\mathrm{in}},\zeta_{\mathrm{out}}]$, where we have
\begin{align*}
a_1(\zeta;\eps) &= -2x_\eps(\zeta) +\mathcal{O}(x_\eps(\zeta)^2,y_\eps(\zeta)), \qquad  \tilde{a}_1(\zeta;\eps,\lambda) = - \frac{\lambda^2}{\theta_\mathrm{lf}c^3} + \mathcal{O}(\lambda x_\eps(\zeta),\lambda y_\eps(\zeta), \lambda^3), \end{align*}
and
\begin{align*}
a_2(\zeta;\eps,\lambda) = 1 +\mathcal{O}(x_\eps(\zeta),y_\eps(\zeta),\lambda).
\end{align*}
We study~\eqref{centerunstable} in the regime $0 < \eps, |\lambda| \ll \nu \ll 1$ and $\Re(\lambda) \geq \eps^{1/5}$. The leading-order dynamics of system~\eqref{centerunstable} is given by the upper triangular system
\begin{align} \label{uppertriangularcenterunstable}
\Psi_\zeta = \tilde{A}(\zeta;\eps,\lambda) \Psi,
\end{align}
whose evolution can be explicitly determined in terms of the coefficient functions $a_1$, $\tilde{a}_1$ and $a_2$. Thus, by bounding these coefficient functions, we establish a tame estimate on the backward growth of the evolution of~\eqref{uppertriangularcenterunstable}. This estimate can be transferred to the full system~\eqref{centerunstable} with the aid of the variation of constants formula.

The following lemma provides bounds on the coefficient functions $a_1$ and $\tilde{a}_1$. 

\begin{lemma} \label{lem:auxiliary_green_rectangle}
Provided $0 < \eps, |\lambda| \ll \nu \ll 1$, there exist an $\eps$- and $\lambda$-independent constant $C_\nu > 0$, a $\nu$-, $\lambda$- and $\eps$-independent constant $C> 0$, and a point $\zeta_\mathrm{middle} \in [\zeta_{\mathrm{in}},\zeta_{\mathrm{out}}]$ such that $|\zeta_\mathrm{middle} - \zeta_{\mathrm{in}}| \leq C_\nu\eps^{-1/3}$ and $1 \leq C_\nu\eps |\zeta_\mathrm{out} - \zeta_{\mathrm{middle}}|$, and the coefficients $a_1(\zeta;\eps)$ and $\tilde{a}_1(\zeta;\eps,\lambda)$ in~\eqref{centerunstable} obey
\begin{align*}
\left|\Re\left(\tilde{a}_1(\zeta;\eps,\lambda)\right) + \frac{1}{\theta_\mathrm{lf}c^3}\left(\Re(\lambda)^2 - \Im(\lambda)^2\right)\right| \leq C \left(|\Re(\lambda)| +|\Im(\lambda)|^2\right)\left(|x_\eps(\zeta)|+|y_\eps(\zeta)|\right) + C_\nu |\lambda|^3
\end{align*}
for $\zeta \in [\zeta_\mathrm{in},\zeta_\mathrm{out}]$ and
\begin{align*}
a_1(\zeta;\eps) \geq -C_\nu \eps^{1/3}
\end{align*}
for $\zeta \in [\zeta_\mathrm{middle},\zeta_{\mathrm{out}}]$.
\end{lemma}
\begin{proof}
    The bound on the coefficient $\tilde{a}_1$ can be obtained directly from the coordinate transformation of Lemma~\ref{lem:fold_transf} by taking real parts and using Lemma~\ref{lem:fold_existence}.

    For the bound on $a_1$, we note that the point $\zeta_\mathrm{middle}\in[\zeta_{\mathrm{in}},\zeta_{\mathrm{out}}]$ can be taken as the entry point of the chart $\mathcal{K}_2$. The estimates on $\zeta_\mathrm{middle}$ follow from the fact that the time spent in the chart $\mathcal{K}_1$ is of $\mathcal{O}(\eps^{-1})$, and the time spent in the charts $\mathcal{K}_{2,3}$ is of $\mathcal{O}(\eps^{-1/3})$. For $\xi\in [\zeta_{\mathrm{middle}},\zeta_{\mathrm{out}}]$, the wave train is thus captured by either the chart $\mathcal{K}_1$ or $\mathcal{K}_2$, in which case we can bound $x_\eps(\zeta)\leq C_\nu \eps^{1/3}$. Furthermore, in the chart $\mathcal{K}_1$, we have that $|y_\eps(\zeta)|\leq C_\nu |x_\eps(\zeta)|^2$, while in the chart $\mathcal{K}_2$, we have $|y_\eps(\zeta)|\leq C_\nu \eps^{2/3}$, hence the result follows. 
\end{proof}

We are now ready to establish the tame estimate on the backward growth of solutions to~\eqref{centerunstable} on $[\zeta_\mathrm{in},\zeta_{\mathrm{out}}]$.

\begin{proposition} \label{prop:tameestimatecenterunstable}
Let $0 < \eps, |\lambda| \ll \nu \ll 1$ and $\Re(\lambda) \geq \eps^{1/5}$. Then, for each $\Psi_{\mathrm{out}} \in \C^2$, the solution $\Psi \colon [\zeta_\mathrm{in},\zeta_{\mathrm{out}}] \to \C^2$ with initial condition $\Psi(\zeta_\mathrm{out}) = \Psi_{\mathrm{out}}$ obeys
\begin{align*}
\left\| \Psi(\zeta_\mathrm{in})\right\| \leq \re^{\frac{\Re(\lambda)}{c} (\zeta_\mathrm{out}-\zeta_\mathrm{in})} \|\Psi_{\mathrm{out}}\|.
\end{align*}
\end{proposition}
\begin{proof}
Since $\Psi(\zeta)$ solves~\eqref{centerunstable} with initial condition $\Psi(\zeta_\mathrm{out}) = \Psi_\mathrm{out}$, we find that $\Phi(\zeta) = \re^{\frac{\lambda}{c} (\zeta - \zeta_{\mathrm{out}})} \Psi(\zeta)$ is a solution with initial condition $\Phi(\zeta_\mathrm{out}) = \Psi_\mathrm{out}$ to the weighted problem
\begin{align} \label{weightedcenterunstable}
\Phi_\zeta = \left(\tilde{A}(\zeta;\eps,\lambda) + \mathcal{O}(\eps) + \frac{\lambda}{c}\right) \Phi,
\end{align}
whose coefficient matrix may be written as
\begin{align*}
\begin{pmatrix} a_1(\zeta;\eps) + \tilde{a}_1(\zeta;\eps,\lambda) + \frac{\lambda}{c} & a_2(\zeta;\eps,\lambda) \\ 
0 & \frac{\lambda}{c}\end{pmatrix} + \eps \tilde{B}(\zeta;\eps,\lambda),
\end{align*}
where the matrix $\tilde{B}(\zeta;\eps,\lambda)$ and the coefficient $a_2(\zeta;\eps,\lambda)$ are bounded on $[\zeta_\mathrm{in},\zeta_\mathrm{out}]$ by an $\eps$-, $\lambda$- and $\nu$-independent constant $C_0 > 0$. Provided $0 < \eps,|\lambda| \ll \nu \ll 1$ and $\Re(\lambda) \geq 0$, Lemma~\ref{lem:fold_existence} and Lemma~\ref{lem:auxiliary_green_rectangle} yield
\begin{align} \label{tildea1bound}
\Re\left(\tilde{a}_1(\zeta;\eps,\lambda) + \frac{\lambda}{c}\right) \geq \frac{3\Re(\lambda)}{4c}.
\end{align}
for $\zeta \in [\zeta_\mathrm{in},\zeta_\mathrm{out}]$.
The evolution $\mathcal{\tilde{T}}_{\eps,\lambda}(\zeta,y)$ of system
\begin{align*}
\Psi_\zeta = \left(\tilde{A}(\zeta;\eps,\lambda) + \frac{\lambda}{c}\right) \Psi,
\end{align*}
reads
\begin{align*}
\mathcal{\tilde{T}}_{\eps,\lambda}(\zeta,y) = \begin{pmatrix}
\re^{\left(a_1(\zeta;\eps) + \tilde{a}_1(\zeta;\eps,\lambda) + \frac{\lambda}{c}\right)(\zeta - y)} & \displaystyle \int_y^\zeta \re^{\left(a_1(\zeta;\eps) + \tilde{a}_1(\zeta;\eps,\lambda) + \frac{\lambda}{c}\right)(\zeta - z)} a_2(z;\eps,\lambda) \re^{\frac{\lambda}{c}(z-y)} d z\\
0 & \re^{\frac{\lambda}{c}(\zeta-y)} 
\end{pmatrix}
\end{align*}
for $\zeta,y \in [\zeta_\mathrm{in},\zeta_\mathrm{out}]$. Hence, combining~\eqref{tildea1bound} with Lemma~\ref{lem:auxiliary_green_rectangle}, we establish, provided $0 < \eps, |\lambda| \ll \nu \ll 1$ and $\Re(\lambda) \geq \eps^{1/5}$, a $\nu$-, $\eps$- and $\lambda$-independent constant $C_1 > 0$ such that
\begin{align} \label{evolboundcenterunstable}
\left\|\mathcal{\tilde{T}}_{\eps,\lambda}(\zeta,y)\right\| \leq \frac{C_1}{\Re(\lambda)} \re^{\frac{\Re(\lambda)}{2c} (\zeta - y)}
\end{align}
holds for $\zeta,y \in [\zeta_\mathrm{middle},\zeta_\mathrm{out}]$ with $\zeta \leq y$. We express the solution $\Phi(\zeta)$ to~\eqref{weightedcenterunstable} through the variation of constants formula
\begin{align} \label{varconstcenterunstable}
\Phi(\zeta) = \mathcal{\tilde{T}}_{\eps,\lambda}(\zeta,\zeta_\mathrm{out}) \Psi_\mathrm{out} + \eps \int_{\zeta_\mathrm{out}}^\zeta \mathcal{\tilde{T}}_{\eps,\lambda}(\zeta,y) \tilde{B}(y;\eps,\lambda) \Phi(y) dy
\end{align}
for $\zeta \in [\zeta_\mathrm{middle},\zeta_\mathrm{out}]$. Setting
\begin{align*}
\eta(\zeta) = \sup_{y \in [\zeta,\zeta_\mathrm{out}]} \|\Phi(y)\| \re^{-\frac{\Re(\lambda)}{4c} (y - \zeta_\mathrm{out}) }
\end{align*}
we bound~\eqref{varconstcenterunstable} with the aid of~\eqref{evolboundcenterunstable} and obtain
\begin{align*}
\|\Phi(\zeta)\| \leq \frac{C_1}{\Re(\lambda)} \re^{\frac{\Re(\lambda)}{2c} (\zeta - \zeta_\mathrm{out})} \|\Psi_{\mathrm{out}}\| + \frac{4C_0C_1\eps}{\Re(\lambda)^2} \eta(\zeta) \re^{\frac{\Re(\lambda)}{4c} (\zeta - \zeta_\mathrm{out})}
\end{align*}
for $\zeta \in [\zeta_\mathrm{middle},\zeta_\mathrm{out}]$. Provided $0 < \eps, |\lambda| \ll \nu \ll 1$ and $\Re(\lambda) \geq \eps^{1/5}$, the latter implies
\begin{align*}
\eta(\zeta) \leq 2\frac{C_1}{\Re(\lambda)} \|\Psi_{\mathrm{out}}\|
\end{align*}
for $\zeta \in [\zeta_\mathrm{middle},\zeta_\mathrm{out}]$. Hence, provided $0 < \eps, |\lambda| \ll \nu \ll 1$ and $\Re(\lambda) \geq \eps^{1/5}$, we establish
\begin{align*}
\|\Phi(\zeta_{\mathrm{middle}})\| \leq 2 C_1 \eps^{-1/5} \re^{\frac{\Re(\lambda)}{4c} (\zeta_\mathrm{middle}-\zeta_\mathrm{out})} \|\Psi_{\mathrm{out}}\|.
\end{align*}
Therefore, observing that the coefficient matrix of~\eqref{weightedcenterunstable} is bounded on $[\zeta_\mathrm{in},\zeta_\mathrm{out}]$ by an $\eps$-, $\nu$- and $\lambda$-independent constant $C_2 > 0$, we apply Lemma~\ref{lem:auxiliary_green_rectangle} and Grönwall's inequality to infer
\begin{align*}
\|\Phi(\zeta_{\mathrm{in}})\| \leq \re^{C_2 |\zeta_{\mathrm{middle}} - \zeta_\mathrm{in}|} \|\Phi(\zeta_{\mathrm{middle}})\| \leq  \|\Psi_{\mathrm{out}}\|,
\end{align*}
provided $0 < \eps, |\lambda| \ll \nu \ll 1$ and $\Re(\lambda) \geq \eps^{1/5}$. So, reverting to the original $\Psi$-coordinate, the result follows.
\end{proof}

\subsection{Proof of Propositions~\ref{prop:fold_bvp_oc_lower} and~\ref{prop:fold_bvp_oc_upper}}\label{sec:mainfoldest}
Incorporating the (hyperbolic) $Z$-dynamics in~\eqref{eq:stability_fold_XY}, we have the following linearized problem
\begin{align}
\begin{split}\label{eq:fold_linearized_diag}
X_\zeta &=X\left(-2x-\frac{\lambda^2}{\theta_\mathrm{lf}c^3} +\mathcal{O}(x^2,y,\eps, \lambda x, \lambda^3)\right)+Y\left(1+\mathcal{O}(x,y,\eps,\lambda)\right),\\
Y_\zeta &=\mathcal{O}(\eps X,\eps Y),\\
Z_\zeta &=\left(-\frac{c}{\theta_\mathrm{lf}}+\mathcal{O}(x,y,\eps,\lambda)\right)Z.
\end{split}
\end{align}
We have the following.
\begin{proposition}\label{prop:fold_bvp_Z}
Fix $M>0$ and $\nu>0$. There exists $\mu>0$ such that for each $(X_\mathrm{out},Y_\mathrm{out}, Z_\mathrm{in})\in \mathbb{C}^3$ and $\lambda\in \Lambda_\eps(\mu,M)$, there exists a solution $\psi_\mathrm{uf}=(X_\mathrm{lf},Y_\mathrm{lf},Z_\mathrm{lf}):\mathcal{I}_\mathrm{lf}\to \mathbb{C}^3$ of~\eqref{eq:fold_linearized_diag} which satisfies
\begin{align*}
X_\mathrm{lf}\left(\frac{\nu}{\eps}\right) &= X_\mathrm{out}, \qquad Y_\mathrm{lf}\left(\frac{\nu}{\eps}\right) = \eps Y_\mathrm{out}, \qquad 
 Z_\mathrm{lf}\left(\frac{1}{\nu}\right) = Z_\mathrm{in}
\end{align*}
Moreover, $(X_\mathrm{lf},Y_\mathrm{lf})=(X_\mathrm{lf}^0,Y_\mathrm{lf}^0)$, the corresponding solution guaranteed by Proposition~\ref{prop:fold_bvp}, and there exist $(\eps,\lambda)$-independent constants $C_\nu, \vartheta_\nu$ such that the solution satisfies
\begin{align*}
\left| Z_\mathrm{lf}\left(\xi_{\mathrm{lf},\eps,\nu}^{\mathrm{out},L}\right) \right| \leq C_\nu |Z_\mathrm{in}| \re^{-\vartheta_\nu/\eps}.
\end{align*}
\end{proposition}
\begin{proof}
The results follows from Proposition~\ref{prop:fold_bvp} and the block-diagonal form of the equation~\eqref{eq:fold_linearized_diag}.
\end{proof}
We are now able to complete the proof of Proposition~\ref{prop:fold_bvp_oc_lower}. The proof of Proposition~\ref{prop:fold_bvp_oc_upper} is similar.
\begin{proof}[Proof of Proposition~\ref{prop:fold_bvp_oc_lower} ]
In light of Proposition~\ref{prop:fold_bvp_Z}, it remains to reframe the results in terms of the original $(U,V,W)$-coordinates. Through the fold, we define projections $\smash{Q^{\mathrm{cu/s}}_\mathrm{uf}(\xi)}$ for the system~\eqref{eq:fold_linearized_diag}, given by
\begin{align}\label{eq:foldQ_projections}
Q^\mathrm{cu}_\mathrm{lf}(\xi) = \begin{pmatrix} I_2 & 0\\ 0 & 0 \end{pmatrix}, \qquad Q^\mathrm{s}_\mathrm{lf}(\xi) = \begin{pmatrix} 0 & 0\\ 0 & 1 \end{pmatrix}.
\end{align}
To obtain these projections in the original coordinates, we recall the transformation of Lemma~\ref{lem:fold_transf} and obtain corresponding projections $\smash{P^{\mathrm{cu/s}}_{\mathrm{lf}, \eps, \lambda,\nu}(\xi)}$ as
\begin{align}\label{eq:fold_dich_proj}
P^{\mathrm{cu/s}}_{\mathrm{lf}, \eps, \lambda,\nu}(\xi)&\coloneqq N_{\eps,\lambda}(\xi)Q^\mathrm{cu/s}_\mathrm{lf}(\xi)N_{\eps,\lambda}(\xi)^{-1}.
\end{align}
For (i), we note that near $\smash{\xi = \xi_{\mathrm{lf},\eps,\nu}^{\mathrm{out},L}}=L_\eps+\tfrac{\nu}{\eps}$, when the wave-train solution is close to the left slow manifold $\mathcal{M}^\mathrm{l}_\eps$, the projections $\smash{P^i_{\mathrm{lf}, \eps, \lambda,\nu}, P^i_{\mathrm{l},\eps,\lambda,\nu}}, i=\mathrm{s,cu},$ can both be extended in such a way that they are each well defined on an overlapping interval of width $\tfrac{\nu}{2\eps}$. As the projections $\smash{P^i_{\mathrm{lf}, \eps, \lambda,\nu}, P^i_{\mathrm{l},\eps,\lambda,\nu}}$ along the wave train while it lies inside and outside an arbitrary small neighborhood of the fold point, respectively, one can extend the definition of the projections $\smash{P^i_{\mathrm{l},\eps,\lambda,\nu}(\xi)}$ to $\xi =\tfrac{3\nu}{4\eps}$, and the projections $\smash{P^i_{\mathrm{lf}, \eps, \lambda,\nu}(\xi)}$ to $\xi =L_\eps+\tfrac{5\nu}{4\eps}$, where we note that $\xi = \tfrac{\nu}{\eps}$ and $\xi = L_\eps+ \tfrac{\nu}{\eps}$ can be identified by periodicity. Therefore, applying~\cite[Lemma~B.3]{BengeldeRijk2025}, the estimate~\eqref{eq:foldproj_lower_est2} holds. 

In order to describe the behavior of the projections $P^i_{\mathrm{lf},0,\lambda,\nu}$ as $\eps\to0$, 
we denote
\begin{align*}
P^i_{\mathrm{lf},0,\lambda,\nu}\left(\tfrac{1}{\nu} \right) \coloneqq  \lim_{\eps\to0}P^i_{\mathrm{lf},\eps,\lambda,\nu}\left(L_\eps+\tfrac{1}{\nu} \right).
\end{align*}For the estimate~\eqref{eq:foldproj_lower_est1}, we first note that by construction and the pointwise estimates of Proposition~\ref{prop:pointwise}, we have that 
\begin{align*}
\left\|P^i_{\mathrm{lf},\eps,\lambda,\nu}\left(L_\eps+\tfrac{1}{\nu} \right) - P^i_{\mathrm{lf},0,0,\nu}\left(\tfrac{1}{\nu} \right)\right\| & \leq C_\nu \left(\epsilon^{\frac{2}{3}}+|\lambda|\right).
\end{align*}
We now claim that $P^i_{\mathrm{lf},0,0,\nu}\left( \tfrac{1}{\nu} \right)=P^i_{\f,\nu}\left( \tfrac{1}{\nu} \right)$. To see this, we first recall from Proposition~\ref{prop:varred2} that 
\begin{align*}
  \ker  P^\mathrm{s}_{\f,\nu}\left( \tfrac{1}{\nu} \right) = \mathrm{Ran}   P^\mathrm{cu}_{\f,\nu}\left( \tfrac{1}{\nu} \right)= \mathrm{Sp}\left\{\Phi_\f(\tfrac{1}{\nu}), \Psi_{\f,\nu}(\tfrac{1}{\nu}) \right\},
\end{align*}
where the solutions $\Phi_\f(\tfrac{1}{\nu}), \Psi_{\f,\nu}(\tfrac{1}{\nu})$ are both bounded as $\xi\to-\infty$. Furthermore, since $P^\mathrm{s}_{\f,\nu}\left(\frac1{\nu}\right)$ has rank $1$, we have that 
\begin{align*}
  \mathrm{Ran}  P^\mathrm{s}_{\f,\nu}\left( \tfrac{1}{\nu} \right) = \mathrm{Sp}\left\{\Psi_{\f,*}(\tfrac{1}{\nu}) \right\},
\end{align*}
where $\smash{\Psi_{\f,*}}$ is the unique (up to scalar multiple) solution of~\eqref{varred2} which decays exponentially as $\xi\to\infty$ and is unbounded as $\xi\to-\infty$.

Inspecting~\eqref{eq:eq:transformed_lin_fold_diag_xi} in the limit $\eps\to0$, we see that the one dimensional subspace $X=Y=0$ uniquely captures all solutions which decay exponentially in $\xi$, hence \begin{align*}
   \mathrm{Ran}  P^\mathrm{s}_{\mathrm{lf},0,0,\nu}\left(\tfrac{1}{\nu}\right)= \mathrm{Ran}  P^\mathrm{s}_{\f,\nu}\left( \tfrac{1}{\nu} \right).
\end{align*}
By construction (see Lemma~\ref{lem:fold_transf}), when $\lambda=0$
\begin{align*}
  \mathrm{Ran}  P^\mathrm{cu}_{\mathrm{lf},\eps,0,\nu}\left(L_\eps+ \tfrac{1}{\nu} \right) = \ker P^\mathrm{s}_{\mathrm{lf},\eps,0,\nu}\left(L_\eps+\tfrac{1}{\nu} \right)= T_{\Gamma_\eps\left(\frac{1}{\nu}\right)} \mathcal{W}^\mathrm{u}(\Gamma_\eps),
\end{align*}
where $\smash{T_{\Gamma_\eps(1/\nu)} \mathcal{W}^\mathrm{u}(\Gamma_\eps)}$ denotes the tangent space of the unstable manifold $\mathcal{W}^\mathrm{u}(\Gamma_\eps)$ of the periodic orbit $\Gamma_\eps$ at $\xi=\tfrac{1}{\nu}$. Furthermore, it follows from the proof of~\cite[Proposition 4.7]{CASCH} that the manifolds $\mathcal{W}^\mathrm{u}(\Gamma_\eps)$ and $\mathcal{W}^\mathrm{u}(\mathcal{M}^\mathrm{r}_0)$ are $C^1$-$\mathcal{O}(\eps)$-close upon entering a neighborhood of the upper fold. Thus we have that 
\begin{align*}
  \mathrm{Ran}  P^\mathrm{cu}_{\mathrm{lf},0,0,\nu}\left( \tfrac{1}{\nu} \right) = \ker P^\mathrm{s}_{\mathrm{lf},0,0,\nu}\left(\tfrac{1}{\nu} \right)= T_{\left(u_\f\left(\frac{1}{\nu}\right),v_\f\left(\frac{1}{\nu}\right), f(u_1)\right)} \mathcal{W}^\mathrm{u}(\mathcal{M}^\mathrm{r}_0).
\end{align*}
This space is two dimensional and must consist of solutions which are bounded as $\xi\to-\infty$. Since the reduced fast front $\smash{\left(u_\f\left(\tfrac{1}{\nu}\right),v_\f\left(\tfrac{1}{\nu}\right), f(u_1)\right))^\top}$ lies in $\mathcal{W}^\mathrm{u}(\mathcal{M}^\mathrm{r}_0)$, we have that the derivative $\Phi_\f(\xi)$ of the reduced front solution must lie in $T_{\left(u_\f\left(\xi\right),v_\f\left(\xi\right), f(u_1)\right)} \mathcal{W}^\mathrm{u}(\mathcal{M}^\mathrm{r}_0)$ at $\xi=\tfrac{1}{\nu}$. As the solution $\Psi_{\f,*}$ is unbounded as $\xi\to-\infty$, the second solution in this space must be a linear combination of  $\Phi_\f, \Psi_{\f,\nu}$. Hence 
\begin{align*}
\ker P^\mathrm{s}_{\mathrm{lf},0,0,\nu}\left(\tfrac{1}{\nu} \right)= \ker  P^\mathrm{s}_{\f,\nu}\left( \tfrac{1}{\nu} \right),
\end{align*}
and $\smash{P^i_{\mathrm{lf},0,0,\nu}\left( \tfrac{1}{\nu} \right)=P^i_{\f,\nu}\left( \tfrac{1}{\nu} \right), i=\mathrm{s, cu}}$, as required, which completes the proof of estimate~\eqref{eq:foldproj_lower_est1}.

For the estimate~\eqref{eq:foldprop_thirdrow}, we directly apply the definition of the projection $\smash{\widetilde{P}_{\f,\nu}^{\cc}\left(\tfrac{1}{\nu}\right)}$ from Proposition~\ref{prop:varred2} to~\eqref{eq:fold_dich_proj}, from which we note by~\eqref{eq:foldQ_projections} that only the third row of $\smash{P^\mathrm{s}_{\mathrm{lf},\eps,\lambda,\nu}(\xi^\mathrm{in}_{\mathrm{lf},\eps,\nu})}$ is relevant in determining~\eqref{eq:foldprop_thirdrow}. The result follows upon examination of~\eqref{eq:fold_linearized_transf_def}, noting the structure of the third row of $\mathcal{N}_\eps'(V_\eps(\xi))$ via~\eqref{eq:fold_linearized_transf_lambda0} and the second row of $D_W\mathcal{H}_{\eps,\lambda}(0,V_\eps)$ via~\eqref{eq:fold_linearized_transf_lambdapart}.

Finally, we note that the transformation $N_{\eps,\lambda}(\xi)$ given by~\eqref{eq:fold_linearized_transf_def} is $(\eps,\lambda,\mu)$-uniformly bounded for $\xi \in \mathcal{I}_\mathrm{lf}$ and $\lambda \in R_1(\mu)$, and
\begin{align}\label{eq:deriv_fold_transf}
U_\eps'(\xi) = \mathcal{N}_\eps'(V_\eps(\xi))V_\eps'(\xi)= \mathcal{N}_\eps'(V_\eps(\xi)) \begin{pmatrix}x_\eps'(\xi)\\ y_\eps'(\xi)\\0 \end{pmatrix},
\end{align}
so that 
\begin{align}\label{eq:deriv_fold_transf_lambda}
N_{\eps,\lambda}(\xi)\begin{pmatrix}x_\eps'(\xi)\\ y_\eps'(\xi)\\0 \end{pmatrix}=U_\eps'(\xi)+ \begin{pmatrix}\mathcal{O}(|\lambda||U_\eps'(\xi)|)\\ \mathcal{O}(|\lambda||U_\eps'(\xi)|)\\\mathcal{O}(\eps|\lambda||U_\eps'(\xi)|) \end{pmatrix}.
\end{align}
 Interpreting the results of Proposition~\ref{prop:tameestimatecenterunstable} in the region $\Re(\lambda)\leq \eps^{1/5}$ and Proposition~\ref{prop:fold_bvp_Z} in the region $|\Re(\lambda)| \leq \mu \eps^{1/6}$ in terms of the projections $\smash{P^\mathrm{cu/s}_{\mathrm{lf},\eps,\lambda,\nu}(\xi)}$, the result follows upon converting back to the $(U,V,W)$-coordinates and making use of~\eqref{eq:deriv_fold_transf}-\eqref{eq:deriv_fold_transf_lambda}.
\end{proof}

\section{The region \texorpdfstring{$R_2(\mu,\varpi,\varrho)$}{R2}}\label{sec:R2}

Let $\varrho > 0$ be as in Proposition~\ref{prop: regionR3} and $\mu > 0$ as in Proposition~\ref{prop:region_r1}. In this section, we prove Proposition~\ref{prop:region_r2} by analyzing the spectrum of $\El_\eps$ in the compact intermediate region $R_2(\mu,\varpi,\varrho)$, where $\varpi > 0$ is a sufficiently small $\eps$-independent constant. In particular, we prove that all spectrum of $\El_\eps$ in $R_2(\mu,\varpi,\varrho)$ must lie in the open left-half plane. To this end, it suffices to show, as outlined in~\S\ref{sec:spectral_setup}, that the eigenvalue problem~\eqref{fulleigenvalueproblem_unscaled}-\eqref{Floquet_BC} admits no nontrivial solution for $\lambda \in R_2(\mu,\varpi,\varrho)$ with $\Re(\lambda) \geq 0$. 

We proceed as follows. First, we establish that the fast subsystem
\begin{align} \label{fastsubsystemR2}
\check\Psi_\xi = \check A_\f(\xi;\eps,\lambda) \check\Psi, \qquad \check A_i(\xi;\eps,\lambda) = \begin{pmatrix} 0 & 1 \\ \lambda - f'(u_\eps(\xi)) & -c \end{pmatrix},
\end{align}
of~\eqref{fulleigenvalueproblem_unscaled} admits an exponential dichotomy on $\R$ by relating it to the spectral problem
\begin{align} \label{SturmLiouville}
u_{\xi\xi} + c u_{\xi} + f'(u_i(\xi))u = \lambda u
\end{align}
associated with the traveling front (or back) solution $u_i(x - ct)$, $i = \f,\bb$ to the Fisher--KPP-type equation 
\begin{align*}
u_t = u_{xx} + f(u).
\end{align*}
Next, we divide $R_2(\mu,\varpi,\varrho)$ into two subregions
\begin{align*}
R_{2,1}(\mu,\varpi,\varrho) &= \{\lambda \in \C : |\Re(\lambda)| \leq \varpi, \mu \leq |\lambda| \leq \varrho\},\\
R_{2,2}(\mu,\varpi,\varrho) &= \{\lambda \in \C : \Re(\lambda) \geq \varpi, \mu \leq |\lambda| \leq \varrho\}.
\end{align*}

For $\lambda \in R_{2,1}(\mu,\varpi,\varrho)$ we show that the exponential dichotomy of~\eqref{fastsubsystemR2} is inherited by the rescaled system~\eqref{fastsub}, which allows for a block diagonalization of the full eigenvalue problem with the aid of the Riccati transform. As in the proof of Proposition~\ref{prop:slow}, the dynamics in the slow component of the diagonalized eigenvalue problem can be computed to leading-order. The Floquet boundary condition~\eqref{eigenvalueproblemBC} then leads to an equation for $\lambda$ which has no solutions $\lambda \in R_{2,1}(\mu,\varpi,\varrho)$ with $\Re(\lambda) \geq 0$, provided $\varpi > 0$ is sufficiently small. 

Given a fixed $\varpi > 0$, roughness results yield that for $\lambda \in R_{2,2}(\mu,\varpi,\varrho)$ the exponential dichotomy on $\R$ of~\eqref{fastsubsystemR2} transfers to the full eigenvalue problem~\eqref{fulleigenvalueproblem_unscaled}. This prohibits nontrivial solutions to~\eqref{fulleigenvalueproblem_unscaled} to fulfill the Floquet boundary condition~\eqref{Floquet_BC} and, thus, we find that $\El_\eps$ has no spectrum in $R_{2,2}(\mu,\varpi,\varrho)$.

\subsection{The fast subsystem}

We obtain an exponential dichotomy on $\R$ for the fast subsystem~\eqref{fastsubsystemR2} for each $\lambda \in R_2(\mu,\varpi,\varrho)$. First, we establish exponential dichotomies for~\eqref{fastsubsystemR2} along the front and the back of the wave train by perturbing from the reduced fast subsystems given by
\begin{align} \label{reducedfastR2}
\check\Psi_\xi = \check{\mathcal{A}}(u_i(\xi);\lambda) \check\Psi, \qquad \check{\mathcal{A}}(u;\lambda) = \begin{pmatrix} 0 & 1 \\ \lambda - f'(u) & -c \end{pmatrix}
\end{align}
for $i = \f,\bb$. System~\eqref{reducedfastR2} is the first-order formulation of the Sturm--Liouville problem~\eqref{SturmLiouville}, which admits, provided $\varpi > 0$ is sufficiently small, no nontrivial bounded solutions for $\lambda \in R_2(\mu,\varpi,\varrho)$. This yields an exponential dichotomy on $\R$ of~\eqref{reducedfastR2}, which can be transferred to~\eqref{fastsubsystemR2} using roughness results. On the other hand, the coefficient matrix of~\eqref{fastsubsystemR2} is pointwise hyperbolic and varies slowly along the left and right branch of the critical manifold for each $\lambda \in R_2(\mu,\varpi,\varrho)$. As in the proof of Proposition~\ref{prop:slow}, this yields exponential dichotomies of~\eqref{reducedfastR2} along these branches. Finally, pasting the exponential dichotomies on the various intervals together, we establish an exponential dichotomy on $\R$ for~\eqref{fastsubsystemR2}.

\begin{proposition} \label{prop:expdi fast R2}
Fix constants $\mu,\varrho > 0$ with $\mu < \varrho$. Provided $0 < \eps \ll \varpi \ll 1$, system~\eqref{fastsubsystemR2} possesses for each $\lambda \in R_2(\mu,\varpi,\varrho)$ an exponential dichotomy on $\R$ with $\eps$- and $\lambda$-independent constants.
\end{proposition}
\begin{proof}
We start by establishing an exponential dichotomy for~\eqref{reducedfastR2} on $\R$. By Proposition~\ref{prop:frontback} there exist $\lambda$-independent constants $C,\upsilon > 0$ such that
\begin{align} \label{asympmatrixestR2} 
\begin{split}
\left\|\check{\mathcal{A}}(u_\f(-\xi);\lambda) - \check{\mathcal{A}}(u_2;\lambda)\right\|, \left\|\check{\mathcal{A}}(u_\bb(-\xi);\lambda) - \check{\mathcal{A}}(\bar{u}_2;\lambda)\right\| &\leq C\re^{-\upsilon \xi}, \\
\left\|\check{\mathcal{A}}(u_\f(\xi);\lambda) - \check{\mathcal{A}}(u_1;\lambda)\right\|, \left\|\check{\mathcal{A}}(u_\bb(\xi);\lambda) - \check{\mathcal{A}}(\bar{u}_1;\lambda)\right\| &\leq \frac{C}{1+\xi},\end{split}
\end{align}
for $\xi \geq 0$ and $\lambda \in R_2(\mu,\varpi,\varrho)$. Using that $f'(u_2),f'(\bar{u}_2) < 0$, $f'(u_1), f'(\bar{u}_1) = 0$ and $c > c_*(a) > 0$, we infer, provided $\varpi > 0$ is sufficiently small, that the asymptotic matrices $\check{\mathcal{A}}(u_j;\lambda)$ and $\check{\mathcal{A}}(\bar{u}_j;\lambda)$ are for each $\lambda \in R_2(\mu,\varpi,\varrho)$ and $j = 1,2$ hyperbolic. Hence, by~\cite[Lemma~3.4]{PAL} system~\eqref{reducedfastR2} admits exponential dichotomies on both $(-\infty,0]$ and $[0,\infty)$ for $i = \f,\bb$. Fix $i \in \{\f,\bb\}$. Since $u_i'$ is a solution to the Sturm--Liouville problem~\eqref{SturmLiouville} at $\lambda = 0$ which possesses no zeros by Proposition~\ref{prop:frontback}, it follows from~\cite[p.~344 and Theorem~5.5]{Sattinger} that~\eqref{SturmLiouville} admits, provided $\varpi > 0$ is sufficiently small, no nontrivial bounded solutions for all $\lambda \in R_2(\mu,\varpi,\varrho)$. Therefore, the first-order formulation~\eqref{reducedfastR2} also has no nontrivial bounded solutions for all $\lambda \in R_2(\mu,\varpi,\varrho)$. Hence,~\cite[Proposition~2.1]{PAL} implies that~\eqref{reducedfastR2} has for all $\lambda \in R_2(\mu,\varpi,\varrho)$ an exponential dichotomy on $\R$ with projections $\check{P}_i(\xi;\lambda)$. Since $R_2(\mu,\varpi,\varrho)$ is compact and $\check{\mathcal{A}}(u;\lambda)$ depends continuously on $\lambda$, the constants associated with this exponential dichotomy can be chosen independent of $\lambda$ by roughness of exponential dichotomies, cf.~\cite[Proposition~5.1]{COP}. Finally,~\cite[Lemma~3.4]{PAL} and its proof in conjunction with estimate~\eqref{asympmatrixestR2} yield $\lambda$-independent constants $K,\alpha_0 > 0$ such that
\begin{align} \label{projest1R2} 
\begin{split}
\left\|\check{P}_\f(-\xi;\lambda) - \check{\mathcal P}(u_2;\lambda)\right\|,\left\|\check{P}_\bb(-\xi;\lambda) - \check{\mathcal P}(\bar{u}_2;\lambda)\right\| &\leq K\re^{-\alpha_0 \xi}, \\ 
\left\|\check{P}_\f(\xi;\lambda) - \check{\mathcal P}(u_1;\lambda)\right\|, \left\|\check{P}_\bb(\xi;\lambda) - \check{\mathcal P}(\bar{u}_1;\lambda)\right\| &\leq \frac{K}{1+\xi},
\end{split}
\end{align}
for $\xi \geq 0$ and $\lambda \in R_2(\mu,\varpi,\varrho)$, where $\check{\mathcal{P}}(u;\lambda)$ is the spectral projection onto the stable eigenspace of the matrix $\check{\mathcal{A}}(u;\lambda)$.

We transfer the exponential dichotomy of~\eqref{reducedfastR2} on $\R$ to exponential dichotomies of~\eqref{fastsubsystemR2} on the intervals $\mathcal{I}_\f = \smash{[\frac{\log(\eps)}{\chi},\frac{1}{\chi}]}$ and $\mathcal{I}_\bb = \smash{[L_{\lr,\eps} + \frac{\log(\eps)}{\chi},L_{\lr,\eps} + \frac{1}{\chi}]}$. Provided $0 < \eps \ll \chi \ll 1$, Proposition~\ref{prop:pointwise} in combination with roughness results, cf.~\cite[Proposition~5.1]{COP}, implies that system~\eqref{fastsubsystemR2} admits for each $\lambda \in R_2(\mu,\varpi,\varrho)$ an exponential dichotomy on $\mathcal{I}_i$ with $\lambda$- and $\eps$-independent constants and projections $\check{Q}_{i,\eps}(\xi;\lambda)$ for $i = \f,\bb$. In addition, there exists a $\lambda$- and $\eps$-independent constant $C_\chi > 0$ such that the estimates
\begin{align} \label{projest2R2}
\begin{split}
\left\|\check{Q}_{\f,\eps}(\xi;\lambda) -\check{P}_\f(\xi;\lambda)\right\| &\leq C_\chi \eps^{\frac23}, \qquad \xi \in \mathcal{I}_\f,\\
\left\|\check{Q}_{\bb,\eps}(\xi;\lambda) -\check{P}_\bb(\xi-L_{\lr,\eps};\lambda)\right\| &\leq C_\chi \eps^{\frac23}, \qquad \xi \in \mathcal{I}_\bb
\end{split}
\end{align}
hold for $\lambda \in R_2(\mu,\varpi,\varrho)$.

Next, we establish exponential dichotomies for~\eqref{fastsubsystemR2} on the intervals $\mathcal{J}_{\lr} = \smash{[\frac{1}{\chi},L_{\lr,\eps} + \frac{\log(\eps)}{\chi}]}$ and $\mathcal{J}_{\rr} = \smash{[\frac{1}{\chi}-L_{r,\eps},\frac{\log(\eps)}{\chi}]}$. First, provided $0 < \eps \ll \tilde{\chi} \ll 1$, Proposition~\ref{prop:pointwise} yields that the coefficient matrix $\check{A}_{\f}(\xi;\eps,\lambda)$ of~\eqref{fastsubsystemR2} is hyperbolic with $\xi$-, $\lambda$-, $\eps$- and $\tilde{\chi}$-independent spectral gap and is bounded by a $\xi$-, $\lambda$-, $\eps$- and $\tilde{\chi}$-independent constant for each $\xi \in \mathcal{J}_{\lr,\tilde{\chi}} \cup \mathcal{J}_{\rr,\tilde{\chi}}$ and $\lambda \in R_2(\mu,\varpi,\varrho)$. Second, there exists a $\xi$-, $\lambda$-, $\eps$- and $\tilde{\chi}$-independent constant $C_0 > 0$ such that $\|\partial_\xi \check{A}_\f(\xi,\eps,\lambda)\| \leq C_0\delta_0(\tilde{\chi})$ for $\xi \in \mathcal{J}_{\lr,\tilde{\chi}} \cup \mathcal{J}_{\rr,\tilde{\chi}}$ and $\lambda \in R_2(\mu,\varpi,\varrho)$ by Proposition~\ref{prop:pointwise}. Noting that $\{\xi + y : x \in \mathcal{J}_{i,\chi}, y \in [-1,1]\} \subset 
\mathcal{J}_{i,2\chi}$ for $0 < \eps \ll \chi \ll 1$, the latter two observations in combination with~\cite[Proposition~A.3]{BDR2} imply, provided  $0 < \eps \ll \chi \ll 1$, that system~\eqref{fastsubsystemR2} has for $\lambda \in R_2(\mu,\varpi,\varrho)$ and $i = \lr,\rr$ an exponential dichotomy on $\mathcal{J}_{i,\chi}$ with $\lambda$-, $\eps$- and $\chi$-independent constants and projections $\check{Q}_{i,\eps}(\xi;\lambda)$ satisfying
\begin{align} \label{projest3R2}
\left\|\check{Q}_{i,\eps}(\xi;\lambda) - \check{\mathcal P}(u_\eps(\xi);\lambda)\right\| \leq C_1 \delta_0(2\chi), \qquad \xi \in \mathcal{J}_{i,\chi},
\end{align}
where $C_1 > 0$ is a $\xi$-, $\eps$-, $\lambda$- and $\chi$-independent constant.

Since~\eqref{fastsubsystemR2} is $L_\eps$-periodic with $L_\eps = L_{\lr,\eps}+L_{\rr,\eps}$, any exponential dichotomy of~\eqref{fastsubsystemR2} on an interval $\mathcal{I}$ yields an exponential dichotomy on an $L_\eps$-translate of $\mathcal{I}$. Our aim is to paste the exponential dichotomies of~\eqref{fastsubsystemR2} on the intervals $\mathcal{J}_{\rr}$, $\mathcal{I}_{\f}$, $\mathcal{J}_{\lr}$ and $\mathcal{I}_{\bb}$ and their $L_\eps$-translates together with the aid of~\cite[Lemma~B.2]{BengeldeRijk2025} to establish an exponential dichotomy for~\eqref{fastsubsystemR2} on a double periodicity interval of length $2L_\eps$. This, in turn, yields the desired exponential dichotomy of~\eqref{fastsubsystemR2} on $\R$ by an application of~\cite[Theorem~1]{Palmer1987}. To this end, we combine estimates~\eqref{projest1R2},~\eqref{projest2R2}, and~\eqref{projest3R2} with Proposition~\ref{prop:pointwise} and deduce, provided $0 < \eps \ll \chi \ll 1$, that
\begin{align} \label{projest4R2}
\begin{split}
&\left\|\check{Q}_{\f,\eps}\left(\tfrac{\log(\eps)}{\chi};\lambda\right) - \check{Q}_{\rr,\eps}\left(\tfrac{\log(\eps)}{\chi};\lambda\right)\right\|, \left\|\check{Q}_{\f,\eps}\left(\tfrac{1}{\chi};\lambda\right) - \check{Q}_{\lr,\eps}\left(\tfrac{1}{\chi};\lambda\right)\right\|, \\
&\left\|\check{Q}_{\bb,\eps}\left(L_{\lr,\eps} + \tfrac{\log(\eps)}{\chi};\lambda\right) - \check{Q}_{\lr,\eps}\left(L_{\lr,\eps} + \tfrac{\log(\eps)}{\chi};\lambda\right)\right\|, \left\|\check{Q}_{\bb,\eps}\left(L_{\lr,\eps} + \tfrac{1}{\chi};\lambda\right) - \check{Q}_{\rr,\eps}\left(\tfrac{1}{\chi} - L_{\rr,\eps};\lambda\right)\right\| < \frac12 
\end{split}
\end{align}
for $\lambda \in R_2(\mu,\varpi,\varrho)$. Therefore, applying~\cite[Lemma~A.1]{BengeldeRijk2025} we deduce that $\smash{\check{Q}_{\f,\eps}(\frac{1}{\chi} \log(\eps);\lambda)[\C^n]}$ and $\smash{\ker(\check{Q}_{\rr,\eps}(\tfrac{1}{\chi} \log(\eps);\lambda))}$ are complementary subspaces for $\lambda \in R_2(\mu,\varpi,\varrho)$. In particular, it implies that the projection onto $\smash{\check{Q}_{\f,\eps}(\frac{1}{\chi} \log(\eps){\chi};\lambda)[\C^n]}$ along $\smash{\ker(\check{Q}_{\rr,\eps}(\frac{1}{\chi} \log(\eps);\lambda))}$ is well-defined and can be bounded by an $\eps$-, $\chi$- and $\lambda$-independent constant for $\lambda \in R_2(\mu,\varpi,\varrho)$, because the constants of the exponential dichotomy of~\eqref{fastsubsystemR2} on $\mathcal{I}_\f$ are independent of $\eps$, $\chi$ and $\lambda$. Applying~\cite[Lemma~B.2]{BengeldeRijk2025} we find, provided $0 < \eps \ll \chi \ll 1$, that~\eqref{fastsubsystemR2} possesses for each $\lambda \in R_2(\mu,\varpi,\varrho)$ an exponential dichotomy on $\mathcal{J}_{\rr} \cup \mathcal{I}_\f$ with $\eps$-, $\chi$- and $\lambda$-independent constants and projections $\check{\mathcal{Q}}_{1,\eps}(\xi;\lambda)$ satisfying 
\begin{align}  \label{projest5R2}
\left\|\check{\mathcal{Q}}_{1,\eps}(\tfrac{1}{\chi};\lambda) - \check{Q}_{\f,\eps}\left(\tfrac{1}{\chi};\lambda\right)\right\| \leq \chi.
\end{align}
Combining~\eqref{projest4R2} with~\eqref{projest5R2} we obtain, provided $0 < \eps \ll \chi \ll 1$, that the estimate
\begin{align}\label{projest6R2}
\left\|\check{\mathcal{Q}}_{1,\eps}(\tfrac{1}{\chi};\lambda) - \check{Q}_{\lr,\eps}\left(\tfrac{1}{\chi};\lambda\right)\right\| < \frac12 + \chi < 1
\end{align}
holds for $\lambda \in R_2(\mu,\varpi,\varrho)$. Continuing analogously as before, we can use estimate~\eqref{projest6R2} and~\cite[Lemmas~A.1 and B.2]{BengeldeRijk2025} to prove that, provided $0 < \eps \ll \chi \ll 1$, system~\eqref{fastsubsystemR2} admits for each $\lambda \in R_2(\mu,\varpi,\varrho)$ an exponential dichotomy on $\mathcal{J}_{\rr} \cup \mathcal{I}_\f \cup \mathcal{J}_{\lr}$ with $\eps$-, $\chi$- and $\lambda$-independent constants. Proceeding inductively, while using the estimates~\eqref{projest4R2}, we thus find that system~\eqref{fastsubsystemR2} has for each $\lambda \in R_2(\mu,\varpi,\varrho)$ an exponential dichotomy on an interval of length $2L_\eps$ with $\eps$-, $\chi$- and $\lambda$-independent constants. Hence, using that the $L_\eps$-periodic coefficient matrix of~\eqref{fastsubsystemR2} can be bounded by an $\eps$- and $\lambda$-independent constant for $\lambda \in R_2(\mu,\varpi,\varrho)$ and $\xi \in \R$ by Proposition~\ref{prop:pointwise}, system~\eqref{fastsubsystemR2} has by~\cite[Theorem~1]{Palmer1987} an exponential dichotomy on $\R$ for each $\lambda \in R_2(\mu,\varpi,\varrho)$ with $\lambda$- and $\eps$-independent constants.
\end{proof}

\subsection{The region \texorpdfstring{$R_{2,1}(\mu,\varpi,\varrho)$}{R21}}

The exponential dichotomy of the fast subsystem~\eqref{fastsubsystemR2}, established in Proposition~\ref{prop:expdi fast R2}, allows us to apply the Riccati transformation to block diagonalize the rescaled eigenvalue problem~\eqref{eigenvalueproblem}. By computing the scalar dynamics in the one-dimensional slow component of the diagonalized eigenvalue problem to leading order and employing the Floquet boundary condition~\eqref{eigenvalueproblemBC}, we preclude the existence of any spectrum of nonnegative real part in the region $R_{2,1}(\mu,\varpi,\varrho)$. 

\begin{proposition}\label{prop:region_r21}
Let $0 < a < \frac12$. Fix $0 < \gamma < \gamma_*(a)$ and $c > c_*(a)$. Take constants $\mu,\varrho > 0$ with $\mu < \varrho$. Then, provided $0 < \eps \ll \varpi \ll 1$, the linearization $\El_\eps$ of~\eqref{FHN} about $\phi_\eps(\xi)$ possesses no spectrum of nonnegative real part in the compact set $R_{2,1}(\mu,\varpi,\varrho)$. 
\end{proposition}
\begin{proof}
As outlined in~\S\ref{sec:spectral_setup}, it suffices to show that the eigenvalue problem~\eqref{eigenvalueproblem}-\eqref{eigenvalueproblemBC} admits no nontrivial solution for $\rho \in \R$ and $\lambda \in R_{2,1}(\mu,\varpi,\varrho)$ with $\Re(\lambda) \geq 0$. As in the proof of Proposition~\ref{prop:expdi fast R2}, we adopt the notation $\mathcal{J}_{\lr} = \smash{[\frac{1}{\chi},L_{\lr,\eps} + \frac{\log(\eps)}{\chi}]}$ and $\mathcal{J}_{\rr} = \smash{[\frac{1}{\chi}-L_{r,\eps},\frac{\log(\eps)}{\chi}]}$. Moreover, throughout the proof we denote by $C \geq 1$ any $\eps$-, $\lambda$-, $\chi$-, $\varpi$- and $\xi$-independent constant.

First, Proposition~\ref{prop:pointwise} yields, provided $0 < \eps \ll \chi \ll 1$, that we have
\begin{align} \label{matrixappR2}
\begin{split}
\left\|A_{\f}(\xi;\epsilon,\lambda) - \check{A}_{\lr}(\epsilon \xi;\lambda)\right\| &\leq C\delta_0(\chi), \qquad \xi \in \mathcal{J}_{\lr},\\
\left\|A_{\f}(\xi;\epsilon,\lambda) - \check{A}_{\rr}(\epsilon \xi + L_{\rr,\eps};\lambda)\right\| &\leq C\delta_0(\chi), \qquad \xi \in \mathcal{J}_{\rr}
\end{split}
\end{align}
for $\lambda \in R_{2,1}(\mu,\varpi,\varrho)$, where we denote
\begin{align*}
\check{A}_{i}(y;\lambda) = \begin{pmatrix} -\frac{\lambda}{c} & 1 \\
\lambda - f'(u_i(y)) & -\frac{\lambda}{c}\end{pmatrix}
\end{align*}
for $i = \lr,\rr$. Subsequently, we observe that the spectral projection of 
\begin{align*} \begin{pmatrix} \check{A}_{i}(y;\lambda) & B_0 \\ 0_{1 \times 2} & 0 \end{pmatrix} \end{align*}
onto its center eigenspace is given by
\begin{align*}
\check{\mathcal Q}_i^{\cc}(y;\lambda) = \begin{pmatrix} 0 & 0 & \left(f'(u_{i}(y)) - \lambda + \frac{\lambda^2}{c^2}\right)^{-1} \\ 0 & 0 & \frac{\lambda}{c}\left(f'(u_{i}(y)) - \lambda + \frac{\lambda^2}{c^2}\right)^{-1} \\ 0 & 0 & 1\end{pmatrix}.
\end{align*}
for $y \in [0,L_{i} + 1]$, $\lambda \in R_{2,1}(\mu,\varpi,\varrho)$ and $i = \lr,\rr$. By estimate~\eqref{matrixappR2} we have
\begin{align} \label{specprojestR2}
\begin{split}
\left\|\mathcal{P}^{\cc}_{\epsilon,\lambda}(\xi) - \check{\mathcal Q}^{\cc}_{\lr}(\epsilon \xi;\lambda)\right\| \leq C\delta_0(\chi), \qquad \xi \in \mathcal{J}_{\lr},\\
\left\|\mathcal{P}^{\cc}_{\epsilon,\lambda}(\xi) - \check{\mathcal Q}^{\cc}_{\rr}(\epsilon \xi + L_{\rr,\eps};\lambda)\right\| \leq C\delta_0(\chi), \qquad \xi \in \mathcal{J}_{\rr},
\end{split}
\end{align}
for $\lambda \in R_{2,1}(\mu,\varpi,\varrho)$, where $\mathcal{P}^{\cc}_{\epsilon,\lambda}(\xi)$ is the spectral projection of the coefficient matrix $A(\xi;\eps,\lambda)$ of~\eqref{eigenvalueproblem} onto its center subspace.

Proposition~\ref{prop:expdi fast R2} yields, provided $0 < \eps \ll \varpi \ll 1$, that the fast subsystem~\eqref{fastsubsystemR2} has for each $\lambda \in R_{2,1}(\mu,\varpi,\varrho)$ an exponential dichotomy on $\R$ with $\eps$- and $\lambda$-independent constants $K,\alpha > 0$. Therefore, provided $0 < \eps \ll \varpi \ll 1$, the rescaled fast subsystem~\eqref{fastsub} has for each $\lambda \in R_{2,1}(\mu,\varpi,\varrho)$ also an exponential dichotomy on $\R$ with $\eps$- and $\lambda$-independent constants $K,\alpha - \frac{\varpi}{c} > 0$. In addition, the $L_\eps$-periodic coefficient matrix of~\eqref{eigenvalueproblem} can be bounded on $\R$ by an $\eps$- and $\lambda$-independent constant for $\lambda \in R_{2,1}(\mu,\varpi,\varrho)$ by Proposition~\ref{prop:pointwise}. So, proceeding as in the proof of Proposition~\ref{prop:slow}, we apply the Riccati transformation, cf.~\cite[Theorem~5.1]{BDR}, to yield continuous $L_\eps$-periodic matrix functions $H_{\eps,\lambda} \colon \R \to \C^{3 \times 3}$ and $U_{\eps,\lambda} \colon \R \to \C^{2 \times 1}$ such that $H_{\eps,\lambda}(\xi)$ is invertible for each $\xi \in \R$ and $U_{\eps,\lambda}$ can be bounded on $\R$ by a $\lambda$- and $\eps$-independent constant. Moreover, if $\Psi(\xi)$ is a solution to~\eqref{eigenvalueproblem}, then $\Phi(\xi) = H_{\epsilon,\lambda}(\xi)\Psi(\xi)$ obeys the diagonalized system
\begin{align} \label{blocksystemR2}
\begin{split}
\Phi_\xi = \begin{pmatrix} A_\f(\xi;\eps,\lambda) - \eps U_{\eps,\lambda}(\xi) B_1  & 0_{2 \times 1} \\
0_{1 \times 2} & \eps A_\su + \eps B_1 U_{\eps,\lambda}(\xi)
           \end{pmatrix} \Phi,
\end{split}
\end{align}
for $\xi \in \R$ and $\lambda \in R_{2,1}(\mu,\varpi,\varrho)$. Analogously as the derivation of estimate~\eqref{Uapprox} in the proof of Proposition~\ref{prop:slow}, one establishes the estimates
\begin{align} \label{approxR2lowerblock}
\begin{split}
\left\|U_{\eps,\lambda}(\xi) - \begin{pmatrix} \left(f'(u_{\lr}(\epsilon \xi)) - \lambda + \frac{\lambda^2}{c^2}\right)^{-1} \\ \frac{\lambda}{c} \left(f'(u_{\lr}(\epsilon \xi)) - \lambda + \frac{\lambda^2}{c^2}\right)^{-1} \end{pmatrix}\right\| \leq C \delta_0(\chi), \qquad \xi \in I_{\lr,\chi},\\
\left\|U_{\eps,\lambda}(\xi) - \begin{pmatrix} \left(f'(u_{\rr}(\epsilon \xi + L_{\rr,\eps})) - \lambda + \frac{\lambda^2}{c^2}\right)^{-1} \\ \frac{\lambda}{c} \left(f'(u_{\rr}(\epsilon \xi + L_{\rr,\eps})) - \lambda + \frac{\lambda^2}{c^2}\right)^{-1} \end{pmatrix}\right\| \leq C \delta_0(\chi), \qquad \xi \in I_{\rr,\chi}
\end{split}
\end{align}
for $\lambda \in R_{2,1}(\mu,\varpi,\varrho)$, where we used~\eqref{specprojestR2} instead of~\eqref{specprojest}. Furthermore, by roughness results, cf.~\cite[Theorem~5.1]{COP}, the fast block system
\begin{align} \label{fastsubR21}
\Phi_\xi = \left(A_\f(\xi;\eps,\lambda) - \eps U_{\eps,\lambda}(\xi) B_1\right)\Phi
\end{align}
admits for $\lambda \in R_{2,1}(\mu,\varpi,\varrho)$ an exponential dichotomy on $\R$ with $\lambda$- and $\eps$-independent constants.

Now suppose that $\Psi(\xi)$ is a nontrivial solution to the boundary-value problem~\eqref{eigenvalueproblem}-\eqref{eigenvalueproblemBC} for some $\rho \in \R$ and $\lambda \in  R_{2,1}(\mu,\varpi,\varrho)$. Then, $\Phi(\xi) = H_{\eps,\lambda}(\xi)\Psi(\xi)$ solves~\eqref{blocksystemR2} and obeys
\begin{align}\label{scalarR2}
\Phi(L_\eps + \xi) = \re^{\left(\ri\rho - \frac{\lambda}{c}\right) L_\eps} \Phi(\xi)
\end{align}
for all $\xi \in \R$, where we used that $H_{\eps,\lambda}(0) = H_{\eps,\lambda}(L_\eps)$ is invertible. The first two components of $\Phi(\xi)$ solve the fast block system~\eqref{fastsubR21}, which has an exponential dichotomy on $\R$ with $\lambda$- and $\eps$-independent constants. Therefore, provided $0 < \eps \ll \varpi \ll 1$, the condition~\eqref{scalarR2} yields that $\Phi(\xi)$ must vanish identically in its first two components. Since the last component $\phi(\xi)$ of $\Phi(\xi)$ solves the scalar problem
\begin{align*}
\phi_\xi = \eps \left(A_\su + B_1 U_{\eps,\lambda}(\xi)\right) \phi,
\end{align*}
it must be nonzero for all $\xi \in \R$ and obeys
\begin{align*}
\phi(L_\eps) = \exp\left(\int_0^{L_\eps} \eps \left(A_\su + B_1 U_{\eps,\lambda}(\xi)\right) d\xi\right)\phi(0).
\end{align*}
Substituting the latter into the last component of~\eqref{scalarR2} for $\xi = 0$, dividing by $\phi(0) \neq 0$, taking the complex logarithm on both sides and equating real parts yields
\begin{align*}
\Re\left(\frac{\lambda L_\eps}{c} + \int_0^{L_\eps} \eps \left(A_\su + B_1 U_{\eps,\lambda}(\xi)\right)  d\xi\right) = 0.
\end{align*}
Hence, using the approximations~\eqref{period} and~\eqref{approxR2lowerblock} and the fact that $U_{\eps,\lambda}$ is bounded on $\R$ by an $\eps$- and $\lambda$-independent constant, we infer that, provided $0 < \eps \ll \chi \ll \varpi \ll 1$, the estimate
\begin{align*}
&\left|\Re(\lambda)\frac{L_\eps}{c} + \frac{\gamma}{c} \left(L_\lr + L_\rr\right) - \int_0^{L_\lr} \frac{f'(u_{\lr}(y)) - \Re(\lambda) + \frac{\Re(\lambda)^2 - \Im(\lambda)^2}{c^2} }{c\left|f'(u_{\lr}(y)) - \lambda + \frac{\lambda}{c^2}\right|^2} \de y - \int_0^{L_\rr} \frac{f'(u_{\rr}(y)) - \Re(\lambda) + \frac{\Re(\lambda)^2 - \Im(\lambda)^2}{c^2} }{c\left|f'(u_{\rr}(y)) - \lambda + \frac{\lambda}{c^2}\right|^2} \de y \right|\\ 
&\qquad \leq C \delta_0(\chi)
\end{align*}
holds. Therefore, upon recalling that we have $\gamma,c > 0$ and $f'(u_i(y)) \leq 0$ for $y \in [0,L_i]$ and $i = \lr,\rr$, we deduce that it must hold $\Re(\lambda) < 0$, provided $0 < \eps \ll \chi \ll \varpi \ll 1$, which concludes the proof.
\end{proof}

\subsection{The region \texorpdfstring{$R_{2,2}(\mu,\varpi,\varrho)$}{R22}}

We now show that the linearization $\El_\eps$ possesses no spectrum in the compact region $R_{2,2}(\mu,\varpi,\varrho)$ for fixed $\varrho,\varpi,\mu > 0$. The result follows by establishing an exponential dichotomy for the full eigenvalue problem~\eqref{fulleigenvalueproblem_unscaled} on $\R$.

\begin{proposition}\label{prop:region_r22}
Let $0 < a < \frac12$. Fix $0 < \gamma < \gamma_*(a)$ and $c > c_*(a)$. Take constants $\mu,\varrho,\varpi > 0$ with $\mu < \varrho$. Then, provided $0 < \eps \ll 1$, there is no spectrum of the linearization $\El_\eps$ of~\eqref{FHN} about $\phi_\eps(\xi)$ in the region $R_{2,2}(\mu,\varpi,\varrho)$. 
\end{proposition}
\begin{proof}
Suppose $\check{\Psi}(\xi)$ is a solution to~\eqref{fulleigenvalueproblem_unscaled}-\eqref{Floquet_BC} with $\rho \in \R$. Then, it follows by standard Floquet theory, cf.~\cite{Gardner1993}, that
\begin{align*}
\check\Psi(L_\eps + \xi) = \re^{\ri \rho L_\eps} \check\Psi(\xi)
\end{align*}
holds for all $\xi \in \R$. Consider the invertible matrix
\begin{align*}
S_\eps = \begin{pmatrix} I_2 & 0_{2 \times 1} \\ 0_{1 \times 2} & \sqrt{\eps}\end{pmatrix}.
\end{align*}  
We find that $\check\Phi(\xi) = S_\eps \check{\Psi}(\xi)$ solves the rescaled problem
\begin{align} \label{rescaledR22}
\check\Phi_\xi = \begin{pmatrix} \check{A}_\f(\xi;\eps,\lambda) & \sqrt{\eps} B_0 \\ \sqrt{\eps} B_1 & \frac{1}{c}\left(\eps \gamma + \lambda\right) \end{pmatrix} \check\Phi
\end{align}
and obeys
\begin{align} \label{BCR22}
\check\Phi(L_\eps + \xi) = \re^{\ri \rho L_\eps} \check\Phi(\xi)
\end{align}
for $\xi \in \R$. Clearly, provided $0 < \eps \ll 1$, the diagonal system
\begin{align*}
\Phi_\xi = \begin{pmatrix} \check{A}_\f(\xi;\eps,\lambda) & 0 \\ 0 & \frac{\lambda}{c} \end{pmatrix} \Phi
\end{align*}
admits for each $\lambda \in R_{2,2}(\mu,\varpi,\varrho)$ an exponential dichotomy on $\R$ with $\lambda$- and $\eps$-independent constants by Proposition~\ref{prop:expdi fast R2}. By roughness of exponential dichotomies, cf.~\cite[Proposition~5.1]{COP}, it follows, provided $0 < \eps \ll 1$, that~\eqref{rescaledR22} has for each $\lambda \in R_{2,2}(\mu,\varpi,\varrho)$ an exponential dichotomy on $\R$ with $\lambda$- and $\eps$-independent constants. Combining the latter with~\eqref{BCR22} implies that $\check\Phi$ must be identically $0$. Therefore, the eigenvalue problem~\eqref{fulleigenvalueproblem_unscaled}-\eqref{Floquet_BC} does not admit a nontrivial solution for each $\lambda \in R_{2,2}(\mu,\varpi,\varrho)$ and $\rho \in \R$. By the exposition in~\S\ref{sec:spectral_setup} this implies that $\El_\eps$ does not possess spectrum in $R_{2,2}(\mu,\varpi,\varrho)$.
\end{proof}
\subsection{Proof of Proposition~\ref{prop:region_r2}}
    Choosing $\varrho$ as in Proposition~\ref{prop: regionR3} and $\mu$ as in Proposition~\ref{prop:region_r1} and taking $\varpi$ sufficiently small, the result follows from Propositions~\ref{prop:region_r21} and~\ref{prop:region_r22}.

\appendix

\section{The region \texorpdfstring{$R_{3,\eps}(\varrho)$}{R3}}\label{app:r3}
In this section, we prove Proposition~\ref{prop: regionR3} concerning spectrum in the region $R_{3,\eps}(\varrho)$.
\begin{proof}[Proof of Proposition~\ref{prop: regionR3}]
The principal part of $\El_\eps$ is the diagonal diffusion-advection operator
\begin{align*}
    L_0 = \begin{pmatrix}
   \partial_{\xi \xi} + c \partial_\xi & 0 \\
     0 & c \partial_\xi 
     \end{pmatrix}
\end{align*}
acting on $Y \coloneqq  L^2(\R,\C) \times L^2(\R,\C)$ with dense domain $H^2 (\R,\C) \times H^1 (\R,\C)$. The skew-adjoint operator $A_\alpha = \alpha \partial_\xi$ on the Hilbert space $L^2(\R)$ with domain $H^1(\R)$ generates by Stone's Theorem, cf.~\cite[Theorem~II.3.24]{EngelNagel}, a unitary group for any $\alpha \in \R \setminus \{0\}$. Consequently, by~\cite[Corollary~II.4.9]{EngelNagel} and its proof, the square $\partial_{\xi\xi} = A_1^2 \colon H^2(\R) \subset L^2(\R) \to L^2(\R)$ is sectorial. In particular, there exists a constant $C_0 \geq 1$ such that any $\lambda \in \C \setminus \{0\}$ with $\Re(\lambda) \geq 0$ lies in the resolvent set of $A_1^2$ and we have
\begin{align} \label{e:resolv1}
\left\|\left(A_1^2 - \lambda\right)^{-1}\right\|_{L^2 \to H^2} \leq C_0\left(1 + \frac{1}{|\lambda|}\right).
\end{align}
On the other hand, by Sobolev interpolation there exists for any $\delta > 0$ a constant $C_\delta > 0$ such that $\|u'\|_{L^2} \leq \delta\|u''\|_{L^2} + C_\delta \|u\|_{L^2}$ for all $u \in H^2(\R)$. Hence, for any $\delta > 0$ the operator $A_c$ is $A_1^2$-bounded with $A_1^2$-bound $\delta$. Therefore, using that $A_1^2$ is sectorial, the proof of~\cite[Theorem~III.2.10]{EngelNagel} yields constants $C_1,r_1>0$ such that any $\lambda \in \C$ with $\Re(\lambda) \geq 0$ and $|\lambda|>r_1$ lies in the resolvent set of $A_1^2 + A_c=\partial_{\xi\xi} + c\partial_\xi$ and we have
\begin{align} \label{e:resolv3}
    \left\|\left(A_1^2 + A_c - \lambda\right)^{-1}\right\|_{L^2 \to L^2} &\leq \frac{C_1}{|\lambda|}.
\end{align}
Moreover, since $A_c$ generates a unitary group,~\cite[Corollary~II.3.7]{EngelNagel} implies that any $\lambda \in \C$ with $\Re(\lambda) > 0$ lies in the resolvent set of $A_c$ and it holds
\begin{align} \label{e:resolv2}
\left\|\left(A_c - \lambda\right)^{-1}\right\|_{L^2 \to L^2} &\leq \frac{1}{\Re(\lambda)}.
\end{align}

The residual operator $\El_\eps - L_0 \colon Y \to Y$ is bounded by an $\eps$-independent constant by Proposition~\ref{prop:pointwise}. Combining the latter and estimates~\eqref{e:resolv3} and~\eqref{e:resolv2} with~\cite[Theorem~IV.1.16]{Kato1995Perturbation} yields an $\eps$-independent constant $\varrho_1 > 0$ such that $\El_\eps - \lambda$ is invertible for all $\lambda \in \C$ with $\Re(\lambda) > \varrho_1$. 

We now consider the eigenvalue problem 
\begin{align} \label{e:eigenvalueproblemR3}
(\El_\eps - b - \ri \kappa \Omega)\begin{pmatrix} u \\ w\end{pmatrix} = 0 
\end{align}
with $b \in [-\frac{3}{4} \eps \gamma,\varrho_1]$, $\kappa \in \{\pm 1\}$, and $\Omega \geq 1$. Inspired by the analysis in~\cite[Appendix~A]{FHNpulled}, we multiply~\eqref{e:eigenvalueproblemR3} from the left with the diagonal matrix $\mathrm{diag}(\Omega^{-1},\eps^{-1}\Omega^{-1/2})$ and introduce the rescaled spatial variables $X = \sqrt{\Omega} \xi$ and $W = \eps^{-1} w$. We arrive at
\begin{align*}
\left(\El_{1,\eps,\mu} + \El_{2,\eps,\mu}\right) \begin{pmatrix} u \\ W\end{pmatrix} = 0,
\end{align*}
where we denote $\mu = \Omega^{-1/2} \in (0,1)$ and the operators $\El_{1,\eps,\mu} \colon H^2(\R) \times H^1(\R) \subset Y \to Y$ and $\El_{2,\eps,\mu} \colon H^1(\R) \times L^2(\R) \subset Y \to Y$ are given by
\begin{align*}
\El_{1,\eps,\mu} = \begin{pmatrix}
A_1^2 - i \kappa & 0 \\ 0 & A_c - (b + \eps\gamma) \mu - \frac{i \kappa}{\mu}
\end{pmatrix}, \qquad \El_{2,\eps,\mu} = \begin{pmatrix} A_{\mu c} + \mu^2\left(f'(u_\eps(\mu X)) - b\right) & -\eps\mu^2 \\ \mu & 0 \end{pmatrix}.
\end{align*}
Using that $b \in [-\frac{3}{4}\eps\gamma,\varrho_1]$, it follows from~\eqref{e:resolv1} and~\eqref{e:resolv2} that $\El_{1,\eps,\mu}$ is invertible and we find a $\mu$- and $\eps$-independent constant $C_0 > 0$ such that
\begin{align*}
\left\|\El_{2,\eps,\mu} \El_{1,\eps,\mu}^{-1}\right\|_{L^2 \to L^2} \leq C_0\mu
\end{align*}
for $\mu \in (0,1)$. Hence, there exists an $\eps$-independent constant $\mu_0 > 0$ such that, if $\mu \in (0,\mu_0)$, then the operator $I +\El_{2,\eps,\mu} \El_{1,\eps,\mu}^{-1}$ may be inverted using Neumann series. Therefore, $(I +\El_{2,\eps,\mu} \El_{1,\eps,\mu}^{-1})\El_{1,\eps,\mu} = \El_{1,\eps,\mu} + \El_{2,\eps,\mu}$ is invertible for $\mu \in (0,\mu_0)$. Udoing the rescaling, we infer that there exists an $\eps$-independent constant $\varrho_2 > 0$ such that $\El_\eps - b - \ri \kappa \Omega$ is invertible for $b \in [-\frac{3}{4}\eps\gamma,\varrho_1]$, $\kappa \in \{\pm 1\}$, and $\Omega > \varrho_2$, which concludes the proof.
\end{proof}

\section{Exponential di- and trichotomies} \label{appexpdi}

Exponential di- and trichotomies play a central role in the spectral analysis of linear differential operators and can be used to characterize invertibility and Fredholm properties~\cite{COP,SandstedeReview,PAL,Palmer2}. 

A linear nonautonomous ordinary differential equation admits an exponential dichotomy when it has a fundamental set of solutions that exhibit exponential decay in forward or backward time.

\begin{definition} {\upshape
Let $n \in \N_{0}$, $J \subset \R$ an interval, and $A \in C(J,\C^{n \times n})$. Denote by $T(x,y)$ the evolution operator of
\begin{align} \phi_x = A(x)\phi, \qquad \phi \in \C^n. \label{linsys}\end{align}
Equation~\eqref{linsys} has \emph{an exponential dichotomy on $J$ with constants $K,\alpha > 0$ and projections $P(x) \in \C^{n \times n}$} if for all $x,y \in J$ it holds
\begin{itemize}
 \item $P(x)T(x,y) = T(x,y) P(y)$;
 \item $\|T(x,y)P(y)\| \leq K\re^{-\alpha(x-y)}$ for $x \geq y$;
 \item $\|T(x,y)(I_n-P(y))\| \leq K\re^{-\alpha(y-x)}$ for $y \geq x$.
\end{itemize}}
\end{definition}

Exponential trichotomies describe linear systems that, in addition to exhibiting exponential decay in forward and backward time, possess a central subspace corresponding to bounded or neutral dynamics.

\begin{definition}{\upshape
Let $n \in \mathbb{N}_{> 0}$, $J \subset \R$ an interval, and $A \in C(J,\C^{n \times n})$. Denote by $T(x,y)$ the evolution operator of~\eqref{linsys}. Equation~\eqref{linsys} has an \emph{exponential trichotomy on $J$ with constants $K,\alpha > 0$ and projections $P^{\uu}(x),P^{\su}(x),P^{\cc}(x)\in \C^{n \times n}$} if for all $x,y \in J$ it holds
\begin{itemize}
 \item $P^{\uu}(x) + P^{\su}(x) + P^{\cc}(x) = I$;
 \item $P^{\su,\uu,\cc}(x)T(x,y) = T(x,y) P^{\su,\uu,\cc}(y)$;
 \item $\|T(x,y)P^{\su}(y)\|, \|T(y,x)P^{\uu}(x)\|  \leq K\re^{-\alpha(x-y)}$ for $x \geq y$;
 \item $\|T(x,y)P^{\cc}(y)\| \leq K$.
\end{itemize}
}\end{definition}

\section{Properties of the Airy function \texorpdfstring{$\mathrm{Ai}(s)$}{ }}\label{app:airy}
In this section, we collect several facts concerning the Airy function $\mathrm{Ai}(s)$. We begin with the following result from~\cite[\S9]{handbook}.
\begin{lemma}\label{lem:airyfunction}
The Airy function $\mathrm{Ai} \colon \R \to \R$ satisfies the properties:
\begin{enumerate}
    \item $\mathrm{Ai}''(s)=s\mathrm{Ai}(s)$ for $s\in\mathbb{R}$
    \item $\mathrm{Ai}(s)$ has an infinite number of zeros, all of which are negative, the largest of which is simple and denoted by $-\Omega_0<0$.
    \item \label{lem:airybound} There exists a constant $C>0$ such that 
    \begin{align*}
        0\leq \mathrm{Ai}(s) \leq C\re^{-\frac{2}{3}|s|^{3/2}}, \qquad |\mathrm{Ai}'(s)|\leq C \left(1 +|s|^{\frac14}\right)\re^{-\frac{2}{3}|s|^{3/2}}
    \end{align*}
    for $s\in [-\Omega_0,\infty)$.
\end{enumerate}
\end{lemma}
Using Lemma~\ref{lem:airyfunction}, we arrive at the following result.
\begin{proposition}\label{prop:I0properties}
Consider the entire function $I_0 \colon \C \to \C$ given by
\begin{align} \label{eq:defI0}
I_0(z)=\frac{z}{\mathrm{Ai}'(-\Omega_0)^2}\int^\infty_{-\Omega_0}\re^{-z\left( s+\Omega_0\right)   } \left(\mathrm{Ai}'(s)^2-s\mathrm{Ai}(s)^2\right)\mathrm{d}s,
\end{align}
where $-\Omega_0 < 0$ is largest zero of the Airy function $\mathrm{Ai}(s)$. There exists a constant $C> 0$ such that the following hold:
\begin{enumerate}
\item $I_0(z)= \tfrac{2\Omega_0}{3}z+\mathcal{O}(|z|^2)$ for $z \in \C$ with $|z| \ll 1$.
\item $I_0(0)=0, I_0'(z)>0$ for $z\in \R$, and $I_0(z)=1+\mathcal{O}\left(\tfrac{1}{z^2}\right)$ for $z \in \R$ with $z\gg 1$.
\item $\displaystyle I_0(z)=1 -\frac{1}{\mathrm{Ai}'(-\Omega_0)^2}\int^\infty_{-\Omega_0}\re^{-z\left( s+\Omega_0\right)}\mathrm{Ai}(s)^2\mathrm{d}s$ for each $z \in \C$.
\item $|I_0'(z)| \leq C$ for all $z \in \C$ with $\Re(z) \geq -1$.
\end{enumerate}
\end{proposition}
\begin{proof}
Using integration by parts and Lemma~\ref{lem:airyfunction}, we establish
\begin{align*}
\int^\infty_{-\Omega_0} \mathrm{Ai}'(s)^2-s\mathrm{Ai}(s)^2\mathrm{d}s &= 2\int^\infty_{-\Omega_0} \mathrm{Ai}'(s)^2\mathrm{d}s =2s\mathrm{Ai}'(s)^2 \Big|^\infty_{-\Omega_0}-4\int^\infty_{-\Omega_0} s^2\mathrm{Ai}'(s)\mathrm{Ai}(s)\mathrm{d}s\\
&=2\Omega_0\mathrm{Ai}'(-\Omega_0)^2 +4\int^\infty_{-\Omega_0} s\mathrm{Ai}(s)^2\mathrm{d}s =2\Omega_0\mathrm{Ai}'(-\Omega_0)^2 -4\int^\infty_{-\Omega_0} \mathrm{Ai}'(s)^2\mathrm{d}s
\end{align*}
so that
\begin{align*}
    \frac{1}{\mathrm{Ai}'(-\Omega_0)^2}\int^\infty_{-\Omega_0} \mathrm{Ai}'(s)^2-s\mathrm{Ai}(s)^2\mathrm{d}s = \frac{2\Omega_0}{3}.
\end{align*}
Together with Lemma~\ref{lem:airyfunction}, this implies $I_0(0) = 0$ and $I_0'(0) = 2\Omega_0/3$, 
which completes the proof of (i) and the first statement in (ii). For the remaining claims, we use integration by parts to obtain
\begin{align}    \label{eq:I0_IBP}
    \begin{split}
    I_0(z)&= -\frac{1}{\mathrm{Ai}'(-\Omega_0)^2}\int^\infty_{-\Omega_0}\frac{\mathrm{d}}{\mathrm{d}s}\left[\re^{-z\left( s+\Omega_0\right)}\right] \left(\mathrm{Ai}'(s)^2-s\mathrm{Ai}(s)^2\right)\mathrm{d}s \\
    &=-\frac{1}{\mathrm{Ai}'(-\Omega_0)^2}\re^{-z\left( s+\Omega_0\right)}\left(\mathrm{Ai}'(s)^2-s\mathrm{Ai}(s)^2\right)\Big|^\infty_{-\Omega_0} -\frac{1}{\mathrm{Ai}'(-\Omega_0)^2}\int^\infty_{-\Omega_0}\re^{-z\left( s+\Omega_0\right)}\mathrm{Ai}(s)^2\mathrm{d}s \\
    &=1 -\frac{1}{\mathrm{Ai}'(-\Omega_0)^2}\int^\infty_{-\Omega_0}\re^{-z\left( s+\Omega_0\right)}\mathrm{Ai}(s)^2\mathrm{d}s
    \end{split}
\end{align}
for $z\in \C$, which establishes (iii). Differentiating~\eqref{eq:I0_IBP} and using Lemma~\ref{lem:airyfunction}, we have
\begin{align*}
I_0'(z)&= \frac{1}{\mathrm{Ai}'(-\Omega_0)^2}\int^\infty_{-\Omega_0}\re^{-z\left( s+\Omega_0\right)}\left( s+\Omega_0\right)\mathrm{Ai}(s)^2\mathrm{d}s
\end{align*}
for $z \in \C$, implying $I_0'(z) > 0$ for $z \in \R$. Taking absolute values, we find
\begin{align*}
|I_0'(z)| &\leq \frac{1}{\mathrm{Ai}'(-\Omega_0)^2}\int^\infty_{-\Omega_0}\re^{s+\Omega_0}\left( s+\Omega_0\right)\mathrm{Ai}(s)^2\mathrm{d}s
\end{align*}
for all $z \in \C$ with $\Re(z) \geq -1$, which yields (iv) by Lemma~\ref{lem:airyfunction}. Finally, for $z \in \R$ with $z\gg1$, integrating by parts and using Lemma~\ref{lem:airyfunction}, from~\eqref{eq:I0_IBP} we obtain
\begin{align*}
    I_0(z)&= 1-\frac{2}{\mathrm{Ai}'(-\Omega_0)^2 z}\int^\infty_{-\Omega_0}\re^{-z\left( s+\Omega_0\right)}\mathrm{Ai}(s)\mathrm{Ai}'(s)\mathrm{d}s\\
    &= 1-\frac{2}{\mathrm{Ai}'(-\Omega_0)^2 z^2}\int^\infty_{-\Omega_0}\re^{-z\left( s+\Omega_0\right)}\left(\mathrm{Ai}'(s)^2+s\mathrm{Ai}(s)^2\right)\mathrm{d}s = 1+\mathcal{O}\left( \frac{1}{z^2}\right),
\end{align*}
as claimed.
\end{proof}
\section{Numerical continuation of spectra}\label{s:num}
In this section, we present some details on the numerical computation of spectra used to obtain the plots in Figure~\ref{fig:lambda_coeff_cont}. To verify the (negative) quadratic tangency of the critical spectral curve at the origin and its asymptotic scaling, we compute the leading quadratic coefficient $\lambda_\eps''(0)$ numerically and compare this with the predicted leading-order behavior~\eqref{eq:r1_critical_curve_estimates}. We follow~\cite[\S4]{DSSS} to derive an expression for the coefficient $\lambda_\eps''(0)$. We recall that the wave-train solution $\phi_\eps(\xi;c) = (u_\eps,w_\eps)(\xi;c)$ is an $L_\eps(c)$-periodic solution of the traveling wave ODE~\eqref{eq:FHN_twode}. As $L_\eps$ is a monotonically increasing function of $c$, we can equivalently parameterize the solution $\phi_\eps$ and the wave speed $c$ as functions of the period $L_\eps$. Defining the spatial wavenumber $\ell = \frac{2\pi}{L_\eps}$, and setting $\omega(\ell)\coloneqq  \ell c(\ell)$, we have that $(u_*,w_*)(\theta;\ell)=(u_\eps,w_\eps)(\theta/\ell;c(\ell))$ is a $2\pi$-periodic solution of the equation
\begin{align}
\begin{split}\label{eq:twode_rescaled}
0 &= \ell^2 u_{\theta\theta} + f(u) - w + \omega u_\theta,\\
0 &=\epsilon(u-\gamma w - a) + \omega w_\theta.
\end{split}
\end{align}
As the translation eigenvalue is simple by Theorem~\ref{thm:spectral_stability}, the derivative $(\partial_\theta u_*, \partial_\theta w_*)$ thus spans the kernel of the operator
\begin{align*}
\El_* = \begin{pmatrix} \ell^2 \partial_{\theta\theta} + f'(u_*(\theta;\ell))u +\omega\partial_\theta & -1 \\ \epsilon & -\eps \gamma w + \omega \partial_\theta\end{pmatrix}.
\end{align*}
The adjoint operator
\begin{align*}
\El_*^\mathrm{ad} = \begin{pmatrix} \ell^2 \partial_{\theta\theta} + f'(u_*(\theta;\ell))u -\omega\partial_\theta & \eps \\ -1 & -\eps \gamma w - \omega \partial_\theta\end{pmatrix}
\end{align*}
has a nontrivial solution, which we denote by $(u_\mathrm{ad},w_\mathrm{ad})(\theta;\ell)$. Differentiating~\eqref{eq:twode_rescaled} with respect to $\ell$, we find that $(\partial_\ell u_*, \partial_\ell w_*)$ satisfies the equation
\begin{align}\label{eq:u_star_ell_equation}
\El_* \begin{pmatrix} \partial_\ell u_*\\ \partial_\ell u_* \end{pmatrix} = \begin{pmatrix} -2\ell\partial_{\theta\theta}u_*  -\omega'(\ell)\partial_\theta u_* \\ -\omega'(\ell)\partial_\theta w_*\end{pmatrix},
\end{align}
where, upon taking the $L^2$ inner product of both sides with $(u_\mathrm{ad},w_\mathrm{ad})(\theta;\ell)$, we find that 
\begin{align*}
    \omega'(\ell) = -\frac{\left\langle\begin{pmatrix} -2\ell\partial_{\theta\theta}u_*\\ 0 \end{pmatrix}, \begin{pmatrix} u_\mathrm{ad}\\ w_\mathrm{ad} \end{pmatrix}\right\rangle}{\left\langle\begin{pmatrix} \partial_\theta u_*\\ \partial_\theta w_* \end{pmatrix}, \begin{pmatrix} u_\mathrm{ad}\\ w_\mathrm{ad} \end{pmatrix}\right\rangle}.
\end{align*}
Following the discussion in~\cite[\S4.2]{DSSS}, by shifting in $\theta$, we can arrange for the solution $(\partial_\ell u_*, \partial_\ell w_*)$ of~\eqref{eq:u_star_ell_equation} to satisfy
\begin{align}\label{eq:adjoint_normalization_condition}
\left\langle\begin{pmatrix} \partial_\ell u_*\\ \partial_\ell w_* \end{pmatrix}, \begin{pmatrix} u_\mathrm{ad}\\ w_\mathrm{ad} \end{pmatrix}\right\rangle =0.
\end{align}
Turning to the eigenvalue problem~\eqref{eq:Floquet_eigenvalue_problem}, the critical spectral curve $\lambda_\eps(\rho)$ satisfies the reformulated eigenvalue problem
\begin{align}\label{eq:eigenvalue_problem_rescaled}
   \El_{\rho,*} \begin{pmatrix} u_\rho\\ w_\rho\end{pmatrix} = \lambda_\eps(\rho) \begin{pmatrix} u_\rho\\ w_\rho\end{pmatrix}
\end{align} for $\rho\in \left[-\tfrac{\pi}{L_\eps}, \tfrac{\pi}{L_\eps} \right)$, 
where the operator
\begin{align*}
    \El_{\rho,*} =  \begin{pmatrix} \ell^2\big( \partial_\theta+\ri \frac{\rho}{\ell}\big)^2+ \omega\big(\partial_\theta+\ri\frac{\rho}{\ell }\big) + f'(u_*(\theta;\ell)) & - 1 \\ \epsilon & \omega\big(\partial_\theta+\ri\frac{\rho}{\ell }\big) - \eps \gamma \end{pmatrix}, 
\end{align*}
and $(u_\rho, w_\rho) = (\partial_\theta u_*, \partial_\theta w_*)$ at $\rho=0$. By differentiating~\eqref{eq:eigenvalue_problem_rescaled} with respect to $\rho$, and following~\cite[\S4.2]{DSSS}, we obtain
\begin{align}\label{eq:lambda_coeff_numerical}
    \lambda_\eps''(0) = \frac{\left\langle\begin{pmatrix} 4\ell\partial_{\theta}\partial_\ell u_*+2\partial_\theta u_*\\ 0 \end{pmatrix}, \begin{pmatrix} u_\mathrm{ad}\\ w_\mathrm{ad} \end{pmatrix}\right\rangle}{\left\langle\begin{pmatrix} \partial_\theta u_*\\ \partial_\theta w_* \end{pmatrix}, \begin{pmatrix} u_\mathrm{ad}\\ w_\mathrm{ad} \end{pmatrix}\right\rangle}.
\end{align}
We compute~\eqref{eq:lambda_coeff_numerical} numerically in AUTO by solving for $(\partial_\theta u_*, \partial_\theta w_*)$ and $(u_\mathrm{ad},w_\mathrm{ad})$ as the solutions of 
\begin{align*}
    \El_* \begin{pmatrix} u\\ w\end{pmatrix} =0,
\end{align*}
and 
\begin{align*}
    \El_*^\mathrm{ad} \begin{pmatrix} u\\ w\end{pmatrix} =0,
\end{align*}
respectively, up to normalization, and for $(\partial_\ell u_*, \partial_\ell w_*)$ as the solution of~\eqref{eq:u_star_ell_equation}
subject to~\eqref{eq:adjoint_normalization_condition}. The inner products in~\eqref{eq:lambda_coeff_numerical} can then be evaluated numerically to determine $\lambda_\eps''(0)$. Figure~\ref{fig:lambda_coeff_cont} depicts the results of numerical continuation of $\lambda_\eps''(0)$ for decreasing $\eps$ for a wave-train solution of~\eqref{eq:FHN_pde} for $a=0.2, \gamma=1$, and wave speed $c=2$. We see good agreement between the numerically computed expression~\eqref{eq:lambda_coeff_numerical} and the leading-order coefficient $-k\eps^{2/3}$ obtained analytically in Proposition~\ref{prop:region_r1}. A log-log plot of the difference between these two expressions suggests that the error is indeed higher order: While we have not carried out the higher order analysis in the proof of Proposition~\ref{prop:region_r1} necessary to capture the error, we conjecture that this error scales approximately as $\sim \eps \log\eps$. This would be in line with the fact that the next order error term in the $\eps^{2/3}$ bifurcation delay associated with slow passage through the fold is of order $\eps \log \eps$~\cite{krupaszmolyan2001}. This scaling appears to be corroborated by Figure~\ref{fig:lambda_coeff_cont}. 

\section{Direct numerical simulations}\label{s:dns}

Direct simulations for Figures~\ref{f:random} and Figure~\ref{f:deff} were implemented using a fourth order  exponential time differencing scheme~\cite{coxmatthews} after spectral discretization, evaluating the linear part via fast Fourier transform. Typical discretization parameters were $dx=0.1$ and $dt=0.1$.  Domain sizes were chosen so that the domain fit 20 repeat phase waves and 60 repeat trigger waves. The scheme was implemented in Matlab and computations were carried out on an Nvidia Quadro GV100 GPU.  We found wave trains by first simulating in a fundamental period with a spatial relaxation type initial profile until changes in time were small. The result was used to initialize a Newton method with the same spectral discretization and a comoving frame derivative with speed approximately from direct simulations. The resulting wave train was then repeated 20 times for phase waves or 60 times for trigger waves to yield an exact equilibrium in the large domain. We simulated the system in the large domain in both steady and comoving frames with initial perturbations random or localized in space (Figures~\ref{f:random} and~\ref{f:deff}, respectively).  We found the size of the perturbation $u_\mathrm{pert}$ in the first component of the solution $u_*$ by finding the closest  perfect wave-train solution. For this, we constructed the family of translates of wave trains $u_\mathrm{wt}(\cdot+\xi_0)$ using spectral interpolation and then finding the minimum location $\xi_0^*=\mathrm{argmin}_{\xi_0}\int_\xi |u_*(\xi,t)-u_\mathrm{wt}(\xi+\xi_0)|^{0.1}d\xi$. The exponent $0.1$ used for finding the minimum location effectively penalizes a distribution of the mismatch across the domain, so that the error is localized in a region where it is actually large. The perturbation shown in Figure~\ref{f:deff} and used to compute the width of the region with amplitude larger than $10^{-5}$  is $u_*(t,\xi)-u_\mathrm{wt}(\xi+\xi_0)$. 

\bibliographystyle{abbrv}
\bibliography{references}

@article {AEKV,
    AUTHOR = {Awal, Naziru M. and Epstein, Irving R. and Kaper, Tasso J. and
              Vo, Theodore},
     TITLE = {Symmetry-breaking rhythms in coupled, identical fast-slow
              oscillators},
   JOURNAL = {Chaos},
  FJOURNAL = {Chaos. An Interdisciplinary Journal of Nonlinear Science},
    VOLUME = {33},
      YEAR = {2023},
    NUMBER = {1},
     PAGES = {Paper No. 011102, 14},
      ISSN = {1054-1500,1089-7682},
   MRCLASS = {34C15 (92C45)},
  MRNUMBER = {4528920},
       DOI = {10.1063/5.0131305},
       URL = {https://doi-org.ezp2.lib.umn.edu/10.1063/5.0131305},
}

@article{homoscACRS,
    AUTHOR={Avery, Montie and Carter, Paul and de Rijk, Bj\"orn and Scheel, Arnd},
TITLE={Stability and instability of coupled relaxation oscillations in reaction diffusion systems},
YEAR={2025},
JOURNAL={In preparation.}
}

@misc{li2025nonlinearstabilitylargeperiodtraveling,
      title={Nonlinear Stability of Large-Period Traveling Waves Bifurcating from the Heteroclinic Loop in the {F}itz{H}ugh-{N}agumo Equation}, 
      author={Ji Li and Ke Wang and Qiliang Wu and Qing Yu},
      year={2025}, 
      journal={arXiv preprint arXiv:2503.21509},
}

@article {vdPloeg_2005,
    AUTHOR = {van der Ploeg, Harmen and Doelman, Arjen},
     TITLE = {Stability of spatially periodic pulse patterns in a class of
              singularly perturbed reaction-diffusion equations},
   JOURNAL = {Indiana Univ. Math. J.},
  FJOURNAL = {Indiana University Mathematics Journal},
    VOLUME = {54},
      YEAR = {2005},
    NUMBER = {5},
     PAGES = {1219--1301},
      ISSN = {0022-2518,1943-5258},
   MRCLASS = {35K57 (35B10 35B25 35B35)},
  MRNUMBER = {2177102},
MRREVIEWER = {Daniel\ \v{S}ev\v{c}ovi\v{c}},
       DOI = {10.1512/iumj.2005.54.2792},
       URL = {https://doi.org/10.1512/iumj.2005.54.2792},
}

@article {sslongwavelength,
    AUTHOR = {Sandstede, Bj\"orn and Scheel, Arnd},
     TITLE = {On the stability of periodic travelling waves with large
              spatial period},
   JOURNAL = {J. Differential Equations},
  FJOURNAL = {Journal of Differential Equations},
    VOLUME = {172},
      YEAR = {2001},
    NUMBER = {1},
     PAGES = {134--188},
      ISSN = {0022-0396,1090-2732},
   MRCLASS = {35B35 (35B10 35B32 35K57)},
  MRNUMBER = {1824088},
MRREVIEWER = {V.\ N.\ Razzhevaikin},
       DOI = {10.1006/jdeq.2000.3855},
       URL = {https://doi.org/10.1006/jdeq.2000.3855},
}

@article{deelanger,
	title = {Propagating Pattern Selection},
	author = {Dee, G. and Langer, J. S.},
	journal = {Phys. Rev. Lett.},
	volume = {50},
	issue = {6},
	pages = {383--386},
	year = {1983},
}

@article{FHNPushed,
    title = {Selection of pushed pattern-forming fronts in the {F}itz{H}ugh--{N}agumo system},
    author = {M. Avery and P. Carter and B. de Rijk},
YEAR={2025},
    journal = {In preparation},
}

@article{bordiougov2006trigger,
  title={From trigger to phase waves and back again},
  author={Bordiougov, Grigori and Engel, Harald},
  journal={Physica D: Nonlinear Phenomena},
  volume={215},
  number={1},
  pages={25--37},
  year={2006},
  publisher={Elsevier}
}

@article{AronsonWeinberger,
  title={Multidimensional nonlinear diffusion arising in population genetics},
  author={Aronson, Donald G and Weinberger, Hans F},
  journal={Advances in Mathematics},
  volume={30},
  number={1},
  pages={33--76},
  year={1978},
  publisher={Elsevier}
}

@article{HS,
  title={Stability of pulse solutions for the discrete {F}itz{H}ugh--{N}agumo system},
  author={Hupkes, H and Sandstede, Bj{\"o}rn},
  journal={Transactions of the American Mathematical Society},
  volume={365},
  number={1},
  pages={251--301},
  year={2013}
}

@book{eszter1999evans,
  title={Evans function analysis of the stability of periodic travelling wave solutions associated with the {F}itz{H}ugh-{N}agumo system},
  author={Eszter, Edgardo Gabriel},
  year={1999},
  publisher={University of Massachusetts Amherst}
}

@article {Gardner1993,
    AUTHOR = {Gardner, R. A.},
     TITLE = {On the structure of the spectra of periodic travelling waves},
   JOURNAL = {J. Math. Pures Appl. (9)},
  FJOURNAL = {Journal de Math\'{e}matiques Pures et Appliqu\'{e}es. Neuvi\`eme S\'{e}rie},
    VOLUME = {72},
      YEAR = {1993},
    NUMBER = {5},
     PAGES = {415--439},
      ISSN = {0021-7824},
   MRCLASS = {35K55 (34L99 35B10 47H15 47N20)},
  MRNUMBER = {1239098},
MRREVIEWER = {James F. Reineck},
}

@article{doelman1998stability,
  title={Stability analysis of singular patterns in the 1D Gray-Scott model: a matched asymptotics approach},
  author={Doelman, Arjen and Gardner, Robert A and Kaper, Tasso J},
  journal={Physica D: Nonlinear Phenomena},
  volume={122},
  number={1-4},
  pages={1--36},
  year={1998},
  publisher={Elsevier}
}

@article{veerman2013pulses,
  title={Pulses in a Gierer--Meinhardt equation with a slow nonlinearity},
  author={Veerman, Frits and Doelman, Arjen},
  journal={SIAM Journal on Applied Dynamical Systems},
  volume={12},
  number={1},
  pages={28--60},
  year={2013},
  publisher={SIAM}
}

@incollection{soto2001geometric,
  title={A geometric method for periodic orbits in singularly-perturbed systems},
  author={Soto-Trevi{\~n}o, Cristina},
  booktitle={Multiple-time-scale dynamical systems},
  pages={141--202},
  year={2001},
  publisher={Springer}
}

@inproceedings {vH95,
    AUTHOR = {van Harten, Aart},
     TITLE = {Modulated modulation equations},
 BOOKTITLE = {Proceedings of the {IUTAM}/{ISIMM} {S}ymposium on {S}tructure
              and {D}ynamics of {N}onlinear {W}aves in {F}luids ({H}annover,
              1994)},
    SERIES = {Adv. Ser. Nonlinear Dynam.},
    VOLUME = {7},
     PAGES = {117--130},
 PUBLISHER = {World Sci. Publ., River Edge, NJ},
      YEAR = {1995},
   MRCLASS = {76E30 (35Q55)},
  MRNUMBER = {1685855},
}

@book {Kato1995Perturbation,
    AUTHOR = {Kato, Tosio},
     TITLE = {Perturbation theory for linear operators},
    SERIES = {Classics in Mathematics},
      NOTE = {Reprint of the 1980 edition},
 PUBLISHER = {Springer-Verlag, Berlin},
      YEAR = {1995},
     PAGES = {xxii+619},
      ISBN = {3-540-58661-X},
   MRCLASS = {47A55 (46-00 47-00)},
  MRNUMBER = {1335452},
}

@article{FHNpulled,
      title={Stability of coherent pattern formation through invasion in the {F}itz{H}ugh-{N}agumo system}, 
      author={Montie Avery and Paul Carter and Bj\"orn de Rijk and Arnd Scheel},
      journal={J. Eur. Math. Soc.},
      year={2025},
}

@book{KapitulaPromislow,
    title = {Spectral and dynamical stability of nonlinear waves},
    author = {T. Kapitula and K. Promislow},
    publisher = {Springer New York, NY},
    series = {Applied Mathematical Sciences},
    year = {2013},
}

@incollection{SandstedeReview,
title = {Chapter 18 - Stability of Travelling Waves},
editor = {Bernold Fiedler},
series = {Handbook of Dynamical Systems},
publisher = {Elsevier Science},
volume = {2},
pages = {983-1055},
year = {2002},
booktitle = {Handbook of Dynamical Systems},
issn = {1874-575X},
doi = {https://doi.org/10.1016/S1874-575X(02)80039-X},
url = {https://www.sciencedirect.com/science/article/pii/S1874575X0280039X},
author = {Björn Sandstede}
}

@book{ReedSimon,
    title = {Methods of Modern Mathematical Physics IV: Analysis of Operators},
    author = {M. Reed and B. Simon},
    publisher = {Academic Press, Inc.},
    year = {1978},
}

@article{Palmer2,
	author = {Palmer, K.J.},
	title = {Exponential Dichotomies and {F}redholm Operators},
	journal = {Proc. Amer. Math. Soc.},
	volume = {104},
	year = {1988},
	pages = {149-156},
}

@article {Palmer1987,
    AUTHOR = {Palmer, Kenneth J.},
     TITLE = {Exponential dichotomies for almost periodic equations},
   JOURNAL = {Proc. Amer. Math. Soc.},
  FJOURNAL = {Proceedings of the American Mathematical Society},
    VOLUME = {101},
      YEAR = {1987},
    NUMBER = {2},
     PAGES = {293--298},
      ISSN = {0002-9939},
   MRCLASS = {34C27 (34C11)},
  MRNUMBER = {902544},
MRREVIEWER = {Bianca Manfredi},
       DOI = {10.2307/2045998},
       URL = {https://doi.org/10.2307/2045998},
}

@article{Sattinger,
    title = {On the stability of waves of nonlinear parabolic systems},
    author = {D. H. Sattinger},
    journal = {Adv. Math.},
    volume = {22},
    year = {1976},
    pages = {312-355},
}

@article{SandstedeScheelDefects,
    author = {B. Sandstede and A. Scheel},
    title = {Defects in oscillatory media: toward a classification},
    journal = {SIAM J. Appl. Dyn. Sys.},
    volume = {3},
    pages = {1-68},
    year = {2004},
}

@book {EngelNagel,
    AUTHOR = {Engel, Klaus-Jochen and Nagel, Rainer},
     TITLE = {One-parameter semigroups for linear evolution equations},
    SERIES = {Graduate Texts in Mathematics},
    VOLUME = {194},
      NOTE = {With contributions by S. Brendle, M. Campiti, T. Hahn, G.
              Metafune, G. Nickel, D. Pallara, C. Perazzoli, A. Rhandi, S.
              Romanelli and R. Schnaubelt},
 PUBLISHER = {Springer-Verlag, New York},
      YEAR = {2000},
     PAGES = {xxii+586},
      ISBN = {0-387-98463-1},
   MRCLASS = {47D06 (34G10 35K90 47N20)},
  MRNUMBER = {1721989},
MRREVIEWER = {Charles Batty},
}

@article {dodsonlewis,
    AUTHOR = {Dodson, Stephanie and Lewis, Timothy J.},
     TITLE = {Wave reflections in excitable media linked to existence and
              stability of one-dimensional spiral waves},
   JOURNAL = {SIAM J. Appl. Dyn. Syst.},
  FJOURNAL = {SIAM Journal on Applied Dynamical Systems},
    VOLUME = {21},
      YEAR = {2022},
    NUMBER = {2},
     PAGES = {1631--1659},
      ISSN = {1536-0040},
   MRCLASS = {35B36 (35K57 35P30)},
  MRNUMBER = {4444567},
       DOI = {10.1137/21M1425025},
       URL = {https://doi.org/10.1137/21M1425025},
}

@misc{handbook,
         key = "{\relax DLMF}",
       title = "{\it NIST Digital Library of Mathematical Functions}",
howpublished = "http://dlmf.nist.gov/, Release 1.1.1 of 2021-03-15",
         url = "http://dlmf.nist.gov/",
        note = "F.~W.~J. Olver, A.~B. {Olde Daalhuis}, D.~W. Lozier, B.~I. Schneider,
                R.~F. Boisvert, C.~W. Clark, B.~R. Miller, B.~V. Saunders,
                H.~S. Cohl, and M.~A. McClain, eds."}

@article{fenichel1979geometric,
  title={Geometric singular perturbation theory for ordinary differential equations},
  author={Fenichel, Neil},
  journal={Journal of differential equations},
  volume={31},
  number={1},
  pages={53--98},
  year={1979},
  publisher={Academic Press}
}

@article{vanSaarloos,
    author = {W. van Saarloos},
    title = {Front propagation into unstable states},
    journal = {Phys. Rep.},
    volume = {386},
    pages = {29-222},
    year = {2003},
}

@article {CASCH,
    AUTHOR = {Carter, Paul and Scheel, Arnd},
     TITLE = {Wave train selection by invasion fronts in the
              {F}itz{H}ugh-{N}agumo equation},
   JOURNAL = {Nonlinearity},
  FJOURNAL = {Nonlinearity},
    VOLUME = {31},
      YEAR = {2018},
    NUMBER = {12},
     PAGES = {5536--5572},
      ISSN = {0951-7715},
   MRCLASS = {35K57 (35A24 35B25 35C07 35M31)},
  MRNUMBER = {3876553},
MRREVIEWER = {Anna R. Ghazaryan},
       DOI = {10.1088/1361-6544/aae1db},
       URL = {https://doi.org/10.1088/1361-6544/aae1db},
}

@article {PBR,
    AUTHOR = {Carter, P. and de Rijk, B. and Sandstede, B.},
     TITLE = {Stability of traveling pulses with oscillatory tails in the
              {F}itz{H}ugh--{N}agumo system},
   JOURNAL = {J. Nonlinear Sci.},
  FJOURNAL = {Journal of Nonlinear Science},
    VOLUME = {26},
      YEAR = {2016},
    NUMBER = {5},
     PAGES = {1369--1444},
      ISSN = {0938-8974},
   MRCLASS = {35K57 (35B25 35B35 35C07 35M30 35P15)},
  MRNUMBER = {3551276},
MRREVIEWER = {Daniel \v{S}ev\v{c}ovi\v{c}},
       DOI = {10.1007/s00332-016-9308-7},
       URL = {https://doi.org/10.1007/s00332-016-9308-7},
}

@article{tyson1987spiral,
  title={Spiral waves in a model of myocardium},
  author={Tyson, John J and Keener, James P},
  journal={Physica D: Nonlinear Phenomena},
  volume={29},
  number={1-2},
  pages={215--222},
  year={1987},
  publisher={Elsevier}
}

@book{COP,
	Author = {Coppel, W. A.},
	Isbn = {3-540-08536-X},
	Mrclass = {34A30 (34DXX)},
	Mrnumber = {0481196 (58 \#1332)},
	Mrreviewer = {G. R. Sell},
	Pages = {ii+98},
	Publisher = {Springer-Verlag, Berlin-New York},
	Series = {Lecture Notes in Mathematics},
	Title = {Dichotomies in stability theory},
	Volume = {629},
	Year = {1978},
}

@article{chow1994center,
  title={Center manifold and stability for skew-product flows},
  author={Chow, Shui-Nee and Yi, Yingfei},
  journal={Journal of Dynamics and Differential Equations},
  volume={6},
  pages={543--582},
  year={1994},
  publisher={Springer}
}

@article {BDR2,
    AUTHOR = {de Rijk, Bj\"{o}rn},
     TITLE = {Spectra and stability of spatially periodic pulse patterns
              {II}: the critical spectral curve},
   JOURNAL = {SIAM J. Math. Anal.},
  FJOURNAL = {SIAM Journal on Mathematical Analysis},
    VOLUME = {50},
      YEAR = {2018},
    NUMBER = {2},
     PAGES = {1958--2019},
      ISSN = {0036-1410},
   MRCLASS = {35K40 (34L10 35B10 35B25 35B35 35K57)},
  MRNUMBER = {3784109},
       DOI = {10.1137/17M1127594},
       URL = {https://doi.org/10.1137/17M1127594},
}

@article{jencks2025stable,
  title={Stable and unstable spatially-periodic canards created in singular subcritical Turing bifurcations in the Brusselator system},
  author={Jencks, Robert and Doelman, Arjen and Kaper, Tasso J and Vo, Theodore},
  journal={arXiv preprint arXiv:2509.04835},
  year={2025}
}

@article{carter2022wiggly,
  title={Wiggly canards: Growth of traveling wave trains through a family of fast-subsystem foci},
  author={Carter, Paul and Champneys, Alan R},
  journal={Discrete \& Continuous Dynamical Systems-S},
  year={2022}
}

@article{vo2025canards,
  title={Les Canards de Turing},
  author={Vo, Theodore and Doelman, Arjen and Kaper, Tasso J},
  journal={SIAM Journal on Applied Dynamical Systems},
  volume={24},
  number={4},
  pages={2618--2684},
  year={2025},
  publisher={SIAM}
}

@article {BDR,
    AUTHOR = {de Rijk, Bj\"{o}rn and Doelman, Arjen and Rademacher, Jens},
     TITLE = {Spectra and stability of spatially periodic pulse patterns:
              {E}vans function factorization via {R}iccati transformation},
   JOURNAL = {SIAM J. Math. Anal.},
  FJOURNAL = {SIAM Journal on Mathematical Analysis},
    VOLUME = {48},
      YEAR = {2016},
    NUMBER = {1},
     PAGES = {61--121},
      ISSN = {0036-1410},
   MRCLASS = {35K40 (35B10 35B25 35B35 35K57)},
  MRNUMBER = {3439761},
MRREVIEWER = {Rodica Luca},
       DOI = {10.1137/15M1007264},
       URL = {https://doi.org/10.1137/15M1007264},
}

@article {ssspiral,
    AUTHOR = {Sandstede, Bj\"orn and Scheel, Arnd},
     TITLE = {Spiral waves: linear and nonlinear theory},
   JOURNAL = {Mem. Amer. Math. Soc.},
  FJOURNAL = {Memoirs of the American Mathematical Society},
    VOLUME = {285},
      YEAR = {2023},
    NUMBER = {1413},
     PAGES = {v+126},
      ISSN = {0065-9266,1947-6221},
      ISBN = {978-1-4704-6309-0; 978-1-4704-7483-6},
   MRCLASS = {35B36 (35B40 37L15)},
  MRNUMBER = {4580296},
       DOI = {10.1090/memo/1413},
       URL = {https://doi.org/10.1090/memo/1413},
}

@article{PAL,
	Author = {Palmer, K. J.},
	Coden = {JDEQAK},
	Doi = {10.1016/0022-0396(84)90082-2},
	Fjournal = {Journal of Differential Equations},
	Issn = {0022-0396},
	Journal = {J. Differential Equations},
	Mrclass = {58F15 (34C25)},
	Mrnumber = {764125 (86d:58088)},
	Mrreviewer = {S. R. Bernfeld},
	Number = {2},
	Pages = {225--256},
	Title = {Exponential dichotomies and transversal homoclinic points},
	Url = {http://dx.doi.org/10.1016/0022-0396(84)90082-2},
	Volume = {55},
	Year = {1984}}

@phdthesis{SAN1993,
	Author = {Sandstede, B.},
	School = {University of Stuttgart},
	Title = {Verzweigungstheorie homokliner Verdopplungen},
	Year = {1993}}

@article{AveryScheelSelection,
	author = {Avery, M. and Scheel, A.},
	title = {Universal selection of pulled fronts},
	journal = {Comm. Amer. Math. Soc.},
	volume = {2},
	pages = {172-231},
	year = {2022},
}

@article{holzer2013existence,
  title={Existence and stability of traveling pulses in a reaction--diffusion-mechanics system},
  author={Holzer, Matt and Doelman, Arjen and Kaper, Tasso J},
  journal={Journal of nonlinear science},
  volume={23},
  number={1},
  pages={129--177},
  year={2013},
  publisher={Springer}
}

@Article{Hammond2007,
author={Hammond, Constance
and Bergman, Hagai
and Brown, Peter},
title={Pathological synchronization in Parkinson's disease: networks, models and treatments},
journal={Trends in Neurosciences},
year={2007},
month={Jul},
day={01},
publisher={Elsevier},
volume={30},
number={7},
pages={357-364},
issn={0166-2236},
doi={10.1016/j.tins.2007.05.004},
url={https://doi.org/10.1016/j.tins.2007.05.004}
}

@article {gohscheel,
    AUTHOR = {Goh, Ryan and Scheel, Arnd},
     TITLE = {Growing patterns},
   JOURNAL = {Nonlinearity},
  FJOURNAL = {Nonlinearity},
    VOLUME = {36},
      YEAR = {2023},
    NUMBER = {10},
     PAGES = {R1--R51},
      ISSN = {0951-7715,1361-6544},
   MRCLASS = {35B36 (35A18 35B32 92C15)},
  MRNUMBER = {4646017},
MRREVIEWER = {Irina\ Victorovna\ Konopleva},
       DOI = {10.1088/1361-6544/acf265},
       URL = {https://doi.org/10.1088/1361-6544/acf265},
}

@article{epilepsy,
author = {Jiruska, Premysl and de Curtis, Marco and Jefferys, John G. R. and Schevon, Catherine A. and Schiff, Steven J. and Schindler, Kaspar},
title = {Synchronization and desynchronization in epilepsy: controversies and hypotheses},
journal = {The Journal of Physiology},
volume = {591},
number = {4},
pages = {787-797},
doi = {https://doi.org/10.1113/jphysiol.2012.239590},
url = {https://physoc.onlinelibrary.wiley.com/doi/abs/10.1113/jphysiol.2012.239590},
eprint = {https://physoc.onlinelibrary.wiley.com/doi/pdf/10.1113/jphysiol.2012.239590},
year = {2013}
}

@article {kollarscheel,
    AUTHOR = {Koll\'ar, Richard and Scheel, Arnd},
     TITLE = {Coherent structures generated by inhomogeneities in
              oscillatory media},
   JOURNAL = {SIAM J. Appl. Dyn. Syst.},
  FJOURNAL = {SIAM Journal on Applied Dynamical Systems},
    VOLUME = {6},
      YEAR = {2007},
    NUMBER = {1},
     PAGES = {236--262},
      ISSN = {1536-0040},
   MRCLASS = {37L10 (34C15 35F20 35K57)},
  MRNUMBER = {2299640},
MRREVIEWER = {Hannes\ Uecker},
       DOI = {10.1137/060666950},
       URL = {https://doi.org/10.1137/060666950},
}

@article {Chang_Riccati,
    AUTHOR = {Chang, K. W.},
     TITLE = {Remarks on a certain hypothesis in singular perturbations},
   JOURNAL = {Proc. Amer. Math. Soc.},
  FJOURNAL = {Proceedings of the American Mathematical Society},
    VOLUME = {23},
      YEAR = {1969},
     PAGES = {41--45},
      ISSN = {0002-9939},
   MRCLASS = {34.54},
  MRNUMBER = {254373},
MRREVIEWER = {A. C. Garibaldi},
       DOI = {10.2307/2037483},
       URL = {https://doi.org/10.2307/2037483},
}

@article {stichmikhailov,
    AUTHOR = {Stich, Michael and Mikhailov, Alexander S.},
     TITLE = {Target patterns in two-dimensional heterogeneous oscillatory
              reaction-diffusion systems},
   JOURNAL = {Phys. D},
  FJOURNAL = {Physica D. Nonlinear Phenomena},
    VOLUME = {215},
      YEAR = {2006},
    NUMBER = {1},
     PAGES = {38--45},
      ISSN = {0167-2789,1872-8022},
   MRCLASS = {35K57 (92C15)},
  MRNUMBER = {2232443},
       DOI = {10.1016/j.physd.2006.01.011},
       URL = {https://doi.org/10.1016/j.physd.2006.01.011},
}

@article{klausmeier,
author = {Christopher A. Klausmeier },
title = {Regular and Irregular Patterns in Semiarid Vegetation},
journal = {Science},
volume = {284},
number = {5421},
pages = {1826-1828},
year = {1999},
doi = {10.1126/science.284.5421.1826},
URL = {https://www.science.org/doi/abs/10.1126/science.284.5421.1826},
eprint = {https://www.science.org/doi/pdf/10.1126/science.284.5421.1826}}

@article{fieldnoyes,
    author = {Field, Richard J. and Noyes, Richard M.},
    title = {Oscillations in chemical systems. IV. Limit cycle behavior in a model of a real chemical reaction},
    journal = {The Journal of Chemical Physics},
    volume = {60},
    number = {5},
    pages = {1877-1884},
    year = {1974},
    month = {03},
    issn = {0021-9606},
    doi = {10.1063/1.1681288},
    url = {https://doi.org/10.1063/1.1681288},
    eprint = {https://pubs.aip.org/aip/jcp/article-pdf/60/5/1877/18889964/1877\_1\_online.pdf},
}

@article{barkleypipe,
  title = {Simplifying the complexity of pipe flow},
  author = {Barkley, Dwight},
  journal = {Phys. Rev. E},
  volume = {84},
  issue = {1},
  pages = {016309},
  numpages = {8},
  year = {2011},
  month = {Jul},
  publisher = {American Physical Society},
  doi = {10.1103/PhysRevE.84.016309},
  url = {https://link.aps.org/doi/10.1103/PhysRevE.84.016309}
}

@article{guckenheimer2010homoclinic,
  title={Homoclinic orbits of the FitzHugh--Nagumo equation: Bifurcations in the full system},
  author={Guckenheimer, John and Kuehn, Christian},
  journal={SIAM Journal on Applied Dynamical Systems},
  volume={9},
  number={1},
  pages={138--153},
  year={2010},
  publisher={SIAM}
}

@article{RevModPhys.78.1213,
  title = {Dynamical principles in neuroscience},
  author = {Rabinovich, Mikhail I. and Varona, Pablo and Selverston, Allen I. and Abarbanel, Henry D. I.},
  journal = {Rev. Mod. Phys.},
  volume = {78},
  issue = {4},
  pages = {1213--1265},
  numpages = {0},
  year = {2006},
  month = {Nov},
  publisher = {American Physical Society},
  doi = {10.1103/RevModPhys.78.1213},
  url = {https://link.aps.org/doi/10.1103/RevModPhys.78.1213}
}

@article{pikovsky2001synchronization,
  title={Synchronization},
  author={Pikovsky, Arkady and Rosenblum, Michael and Kurths, J{\"u}rgen},
  journal={Cambridge university press},
  volume={12},
  year={2001},
  publisher={Cambridge university press}
}

@incollection {mielkecgl,
    AUTHOR = {Mielke, Alexander},
     TITLE = {The {G}inzburg-{L}andau equation in its role as a modulation
              equation},
 BOOKTITLE = {Handbook of dynamical systems, {V}ol. 2},
     PAGES = {759--834},
 PUBLISHER = {North-Holland, Amsterdam},
      YEAR = {2002},
      ISBN = {0-444-50168-1},
   MRCLASS = {37L15 (35Q55 37N10 76E30)},
  MRNUMBER = {1901066},
MRREVIEWER = {Arnd\ Scheel},
       DOI = {10.1016/S1874-575X(02)80036-4},
       URL = {https://doi.org/10.1016/S1874-575X(02)80036-4},
}

@article {Alexander_1990,
    AUTHOR = {Alexander, J. and Gardner, R. and Jones, C.},
     TITLE = {A topological invariant arising in the stability analysis of
              travelling waves},
   JOURNAL = {J. Reine Angew. Math.},
  FJOURNAL = {Journal f\"{u}r die Reine und Angewandte Mathematik. [Crelle's
              Journal]},
    VOLUME = {410},
      YEAR = {1990},
     PAGES = {167--212},
      ISSN = {0075-4102,1435-5345},
   MRCLASS = {58D25 (35B35 35K57 58C40)},
  MRNUMBER = {1068805},
MRREVIEWER = {Norman\ Dancer},
       DOI = {10.1515/crll.1990.410.167},
       URL = {https://doi.org/10.1515/crll.1990.410.167},
}

@article{worldofcgl,
  title = {The world of the complex {G}inzburg-{L}andau equation},
  author = {Aranson, Igor S. and Kramer, Lorenz},
  journal = {Rev. Mod. Phys.},
  volume = {74},
  issue = {1},
  pages = {99--143},
  numpages = {0},
  year = {2002},
  month = {Feb},
  publisher = {American Physical Society},
  doi = {10.1103/RevModPhys.74.99},
  url = {https://link.aps.org/doi/10.1103/RevModPhys.74.99}
}

@article {strogatzhuygens,
    AUTHOR = {Goldsztein, Guillermo H. and Nadeau, Alice N. and Strogatz,
              Steven H.},
     TITLE = {Synchronization of clocks and metronomes: {A} perturbation
              analysis based on multiple timescales},
   JOURNAL = {Chaos},
  FJOURNAL = {Chaos. An Interdisciplinary Journal of Nonlinear Science},
    VOLUME = {31},
      YEAR = {2021},
    NUMBER = {2},
     PAGES = {Paper No. 023109, 16},
      ISSN = {1054-1500,1089-7682},
   MRCLASS = {70E20 (34E05 34E10)},
  MRNUMBER = {4208894},
       DOI = {10.1063/5.0026335},
       URL = {https://doi.org/10.1063/5.0026335},
}

@article{electricgrid,
author = {Florian Dörfler  and Michael Chertkov  and Francesco Bullo },
title = {Synchronization in complex oscillator networks and smart grids},
journal = {Proceedings of the National Academy of Sciences},
volume = {110},
number = {6},
pages = {2005-2010},
year = {2013},
doi = {10.1073/pnas.1212134110},
URL = {https://www.pnas.org/doi/abs/10.1073/pnas.1212134110},
eprint = {https://www.pnas.org/doi/pdf/10.1073/pnas.1212134110}}

@Article{bzdroplets,
author ="Delgado, Jorge and Li, Ning and Leda, Marcin and González-Ochoa, Hector O. and Fraden, Seth and Epstein, Irving R.",
title  ="Coupled oscillations in a 1D emulsion of {B}elousov–{Z}habotinsky droplets",
journal  ="Soft Matter",
year  ="2011",
volume  ="7",
issue  ="7",
pages  ="3155-3167",
publisher  ="The Royal Society of Chemistry",
doi  ="10.1039/C0SM01240H",
url  ="http://dx.doi.org/10.1039/C0SM01240H"}

@article {millenium,
    AUTHOR = {Eckhardt, Bruno and Ott, Edward and Strogatz, Steven H. and
              Abrams, Daniel M. and McRobie, Allan},
     TITLE = {Modeling walker synchronization on the {M}illennium bridge},
   JOURNAL = {Phys. Rev. E (3)},
  FJOURNAL = {Physical Review E. Statistical, Nonlinear, and Soft Matter
              Physics},
    VOLUME = {75},
      YEAR = {2007},
    NUMBER = {2},
     PAGES = {021110, 10},
      ISSN = {1539-3755,1550-2376},
   MRCLASS = {74H45 (37N20 82C31)},
  MRNUMBER = {2354011},
       DOI = {10.1103/PhysRevE.75.021110},
       URL = {https://doi.org/10.1103/PhysRevE.75.021110},
}

@article {kuramotoinertia,
    AUTHOR = {Choi, Young-Pil and Ha, Seung-Yeal and Yun, Seok-Bae},
     TITLE = {Complete synchronization of {K}uramoto oscillators with finite
              inertia},
   JOURNAL = {Phys. D},
  FJOURNAL = {Physica D. Nonlinear Phenomena},
    VOLUME = {240},
      YEAR = {2011},
    NUMBER = {1},
     PAGES = {32--44},
      ISSN = {0167-2789,1872-8022},
   MRCLASS = {34C15 (34D06 37C99)},
  MRNUMBER = {2740100},
MRREVIEWER = {Ju\ Rang\ Yan},
       DOI = {10.1016/j.physd.2010.08.004},
       URL = {https://doi.org/10.1016/j.physd.2010.08.004},
}

@article {chiba,
    AUTHOR = {Chiba, Hayato},
     TITLE = {A proof of the {K}uramoto conjecture for a bifurcation
              structure of the infinite-dimensional {K}uramoto model},
   JOURNAL = {Ergodic Theory Dynam. Systems},
  FJOURNAL = {Ergodic Theory and Dynamical Systems},
    VOLUME = {35},
      YEAR = {2015},
    NUMBER = {3},
     PAGES = {762--834},
      ISSN = {0143-3857,1469-4417},
   MRCLASS = {82C23 (34C15 37L10)},
  MRNUMBER = {3334903},
       DOI = {10.1017/etds.2013.68},
       URL = {https://doi.org/10.1017/etds.2013.68},
}

@article{kuramoto1975international,
  title={International symposium on mathematical problems in theoretical physics},
  author={Kuramoto, Yoshiki},
  journal={Lecture notes in Physics},
  volume={30},
  pages={420},
  year={1975}
}

@article{RevModPhys.77.137,
  title = {The {K}uramoto model: A simple paradigm for synchronization phenomena},
  author = {Acebr\'on, Juan A. and Bonilla, L. L. and P\'erez Vicente, Conrad J. and Ritort, F\'elix and Spigler, Renato},
  journal = {Rev. Mod. Phys.},
  volume = {77},
  issue = {1},
  pages = {137--185},
  numpages = {0},
  year = {2005},
  month = {Apr},
  publisher = {American Physical Society},
  doi = {10.1103/RevModPhys.77.137},
  url = {https://link.aps.org/doi/10.1103/RevModPhys.77.137}
}

@article {beckjonesschaefferwechselberger,
    AUTHOR = {Beck, Margaret and Jones, Christopher K. R. T. and Schaeffer,
              David and Wechselberger, Martin},
     TITLE = {Electrical waves in a one-dimensional model of cardiac tissue},
   JOURNAL = {SIAM J. Appl. Dyn. Syst.},
  FJOURNAL = {SIAM Journal on Applied Dynamical Systems},
    VOLUME = {7},
      YEAR = {2008},
    NUMBER = {4},
     PAGES = {1558--1581},
      ISSN = {1536-0040},
   MRCLASS = {35K57 (34E15 35P05 92C30)},
  MRNUMBER = {2470977},
MRREVIEWER = {Peter\ L.\ Simon},
       DOI = {10.1137/070709980},
       URL = {https://doi.org/10.1137/070709980},
}

@article{jones1984stability,
  title={Stability of the travelling wave solution of the FitzHugh-Nagumo system},
  author={Jones, Christopher KRT},
  journal={Transactions of the American Mathematical Society},
  volume={286},
  number={2},
  pages={431--469},
  year={1984}
}

@article {coxmatthews,
    AUTHOR = {Cox, S. M. and Matthews, P. C.},
     TITLE = {Exponential time differencing for stiff systems},
   JOURNAL = {J. Comput. Phys.},
  FJOURNAL = {Journal of Computational Physics},
    VOLUME = {176},
      YEAR = {2002},
    NUMBER = {2},
     PAGES = {430--455},
      ISSN = {0021-9991,1090-2716},
   MRCLASS = {65L05 (65M20)},
  MRNUMBER = {1894772},
       DOI = {10.1006/jcph.2002.6995},
       URL = {https://doi.org/10.1006/jcph.2002.6995},
}

@article {schn96,
    AUTHOR = {Schneider, Guido},
     TITLE = {Diffusive stability of spatial periodic solutions of the
              {S}wift-{H}ohenberg equation},
   JOURNAL = {Comm. Math. Phys.},
  FJOURNAL = {Communications in Mathematical Physics},
    VOLUME = {178},
      YEAR = {1996},
    NUMBER = {3},
     PAGES = {679--702},
      ISSN = {0010-3616,1432-0916},
   MRCLASS = {35Q35 (76E30)},
  MRNUMBER = {1395210},
MRREVIEWER = {C.\ Eugene\ Wayne},
       URL = {http://projecteuclid.org/euclid.cmp/1104286771},
}

@article {GS1,
    AUTHOR = {Gallay, Thierry and Scheel, Arnd},
     TITLE = {Diffusive stability of oscillations in reaction-diffusion
              systems},
   JOURNAL = {Trans. Amer. Math. Soc.},
  FJOURNAL = {Transactions of the American Mathematical Society},
    VOLUME = {363},
      YEAR = {2011},
    NUMBER = {5},
     PAGES = {2571--2598},
      ISSN = {0002-9947,1088-6850},
   MRCLASS = {35K45 (35B10 35B35 35C20 35K57)},
  MRNUMBER = {2763727},
       DOI = {10.1090/S0002-9947-2010-05148-7},
       URL = {https://doi.org/10.1090/S0002-9947-2010-05148-7},
}

@article {JZ,
    AUTHOR = {Johnson, Mathew A. and Zumbrun, Kevin},
     TITLE = {Nonlinear stability of spatially-periodic traveling-wave
              solutions of systems of reaction-diffusion equations},
   JOURNAL = {Ann. Inst. H. Poincar\'e{} C Anal. Non Lin\'eaire},
  FJOURNAL = {Annales de l'Institut Henri Poincar\'e{} C. Analyse Non
              Lin\'eaire},
    VOLUME = {28},
      YEAR = {2011},
    NUMBER = {4},
     PAGES = {471--483},
      ISSN = {0294-1449,1873-1430},
   MRCLASS = {35K40 (35B35 35C07 35K57)},
  MRNUMBER = {2823880},
MRREVIEWER = {Chun-Hua\ Ou},
       DOI = {10.1016/j.anihpc.2011.05.003},
       URL = {https://doi.org/10.1016/j.anihpc.2011.05.003},
}

@article{CrossHohenberg,
    author = {M. Cross and P. Hohenberg},
    title = {Pattern formation outside of equilibrium},
    journal = {Rev. Modern Phys.},
    volume = {65},
    pages = {851},
    year = {1993},
}

@article {JNRZ,
    AUTHOR = {Johnson, Mathew A. and Noble, Pascal and Rodrigues, L. Miguel
              and Zumbrun, Kevin},
     TITLE = {Nonlocalized modulation of periodic reaction diffusion waves:
              nonlinear stability},
   JOURNAL = {Arch. Ration. Mech. Anal.},
  FJOURNAL = {Archive for Rational Mechanics and Analysis},
    VOLUME = {207},
      YEAR = {2013},
    NUMBER = {2},
     PAGES = {693--715},
      ISSN = {0003-9527,1432-0673},
   MRCLASS = {35K57 (35B35 35C07)},
  MRNUMBER = {3005327},
MRREVIEWER = {Chun-Hua\ Ou},
       DOI = {10.1007/s00205-012-0573-9},
       URL = {https://doi.org/10.1007/s00205-012-0573-9},
}

@article {DSSS,
    AUTHOR = {Doelman, Arjen and Sandstede, Bj\"{o}rn and Scheel, Arnd and
              Schneider, Guido},
     TITLE = {The dynamics of modulated wave trains},
   JOURNAL = {Mem. Amer. Math. Soc.},
  FJOURNAL = {Memoirs of the American Mathematical Society},
    VOLUME = {199},
      YEAR = {2009},
    NUMBER = {934},
     PAGES = {viii+105},
      ISSN = {0065-9266},
      ISBN = {978-0-8218-4293-5},
   MRCLASS = {35K57 (35A35 35Q53)},
  MRNUMBER = {2507940},
MRREVIEWER = {Temur Jangveladze},
       DOI = {10.1090/memo/0934},
       URL = {https://doi.org/10.1090/memo/0934},
}

@article {SSSU,
    AUTHOR = {Sandstede, Bj\"{o}rn and Scheel, Arnd and Schneider, Guido and
              Uecker, Hannes},
     TITLE = {Diffusive mixing of periodic wave trains in reaction-diffusion
              systems},
   JOURNAL = {J. Differential Equations},
  FJOURNAL = {Journal of Differential Equations},
    VOLUME = {252},
      YEAR = {2012},
    NUMBER = {5},
     PAGES = {3541--3574},
      ISSN = {0022-0396},
   MRCLASS = {35K40 (35B10 35C07 35K57)},
  MRNUMBER = {2876664},
       DOI = {10.1016/j.jde.2011.10.014},
       URL = {https://doi.org/10.1016/j.jde.2011.10.014},
}

@article{buchholtz1995diffusion,
  title={Diffusion-induced instabilities near a canard},
  author={Buchholtz, Frank and Dolnik, Milos and Epstein, Irving R},
  journal={The Journal of Physical Chemistry},
  volume={99},
  number={41},
  pages={15093--15101},
  year={1995},
  publisher={ACS Publications}
}

@article{kaper2021new,
  title={A new class of chimeras in locally coupled oscillators with small-amplitude, high-frequency asynchrony and large-amplitude, low-frequency synchrony},
  author={Kaper, Tasso J and Vo, Theodore},
  journal={Chaos: An Interdisciplinary Journal of Nonlinear Science},
  volume={31},
  number={12},
  year={2021},
  publisher={AIP Publishing}
}

@article{avitabile2020local,
  title={Local theory for spatio-temporal canards and delayed bifurcations},
  author={Avitabile, Daniele and Desroches, Mathieu and Veltz, Romain and Wechselberger, Martin},
  journal={SIAM Journal on Mathematical Analysis},
  volume={52},
  number={6},
  pages={5703--5747},
  year={2020},
  publisher={SIAM}
}

@article{zhu2024existence,
  title={Existence of traveling wave solutions in a singularly perturbed predator-prey equation with spatial diffusion},
  author={Zhu, Zirui and Liu, Xingbo},
  journal={Discrete and Continuous Dynamical Systems},
  volume={44},
  number={1},
  pages={78--115},
  year={2024},
  publisher={Discrete and Continuous Dynamical Systems}
}

@article{van2008pulse,
  title={Pulse dynamics in a three-component system: stability and bifurcations},
  author={van Heijster, Peter and Doelman, Arjen and Kaper, Tasso J},
  journal={Physica D: Nonlinear Phenomena},
  volume={237},
  number={24},
  pages={3335--3368},
  year={2008},
  publisher={Elsevier}
}

@article{CSosc,
  title={Fast pulses with oscillatory tails in the {F}itz{H}ugh--{N}agumo system},
  author={Carter, Paul and Sandstede, Björn},
  journal={SIAM Journal on Mathematical Analysis},
  volume={47},
  number={5},
  pages={3393--3441},
  year={2015},
  publisher={SIAM}
}

@article{krupaszmolyan2001,
  title={Extending geometric singular perturbation theory to nonhyperbolic points---fold and canard points in two dimensions},
  author={Krupa, Martin and Szmolyan, Peter},
  journal={SIAM journal on mathematical analysis},
  volume={33},
  number={2},
  pages={286--314},
  year={2001},
  publisher={SIAM}
}

@article{lin1990,
  title={Using {M}elnikov's method to solve {S}ilnikov's problems},
  author={Lin, Xiao-Biao},
  journal={Proceedings of the Royal Society of Edinburgh Section A: Mathematics},
  volume={116},
  number={3-4},
  pages={295--325},
  year={1990},
  publisher={Royal Society of Edinburgh Scotland Foundation}
}

@article{deng1991,
  title={The existence of infinitely many traveling front and back waves in the {F}itzHugh--{N}agumo equations},
  author={Deng, Bo},
  journal={SIAM journal on mathematical analysis},
  volume={22},
  number={6},
  pages={1631--1650},
  year={1991},
  publisher={SIAM}
}

@article{hastings1974,
  title={The existence of periodic solutions to {N}agumo's equation},
  author={Hastings, Stuart},
  journal={The Quarterly Journal of Mathematics},
  volume={25},
  number={1},
  pages={369--378},
  year={1974},
  publisher={Oxford University Press}
}

@article{hastings1976,
  title={On the existence of homoclinic and periodic orbits for the {F}itz{H}ugh-{N}agumo equations},
  author={Hastings, SP},
  journal={The Quarterly Journal of Mathematics},
  volume={27},
  number={1},
  pages={123--134},
  year={1976},
  publisher={Oxford University Press}
}

@article{hastings1982,
  title={Single and multiple pulse waves for the {F}itz{H}ugh--{N}agumo},
  author={Hastings, Stuart P},
  journal={SIAM Journal on Applied Mathematics},
  volume={42},
  number={2},
  pages={247--260},
  year={1982},
  publisher={SIAM}
}

@article{liu2006,
  title={Turning points and traveling waves in {F}itz{H}ugh--{N}agumo type equations},
  author={Liu, Weishi and Van Vleck, Erik},
  journal={Journal of Differential Equations},
  volume={225},
  number={2},
  pages={381--410},
  year={2006},
  publisher={Elsevier}
}

@article{CSbanana,
  title={Unpeeling a homoclinic banana in the {F}itz{H}ugh-{N}agumo system},
  author={Carter, Paul and Sandstede, Björn},
  journal={SIAM Journal on Applied Dynamical Systems},
  volume={17},
  number={1},
  pages={236--349},
  year={2018},
  publisher={SIAM}
}

@incollection{jones1991construction,
  title={Construction of the {F}itz{H}ugh-{N}agumo pulse using differential forms},
  author={Jones, CKRT and Kopell, N and Langer, R},
  booktitle={Patterns and dynamics in reactive media},
  pages={101--115},
  year={1991},
  publisher={Springer}
}

@article{evans1972nerve,
  title={Nerve axon equations: III Stability of the nerve impulse},
  author={Evans, John W},
  journal={Indiana University Mathematics Journal},
  volume={22},
  number={6},
  pages={577--593},
  year={1972},
  publisher={JSTOR}
}

@article{carpenter1977geometric,
  title={A geometric approach to singular perturbation problems with applications to nerve impulse equations},
  author={Carpenter, Gail A},
  journal={Journal of Differential Equations},
  volume={23},
  number={3},
  pages={335--367},
  year={1977},
  publisher={Elsevier}
}

@article{flores1991stability,
  title={Stability analysis for the slow travelling pulse of the {F}itz{H}ugh-{N}agumo system},
  author={Flores, Gilberto},
  journal={SIAM journal on mathematical analysis},
  volume={22},
  number={2},
  pages={392--399},
  year={1991},
  publisher={SIAM}
}

@article{alexopoulos2024nonlinear,
    AUTHOR = {Alexopoulos, Joannis and de Rijk, Bj\"{o}rn},
     TITLE = {Nonlinear stability of periodic wave trains in the
              {F}itz{H}ugh-{N}agumo system against fully nonlocalized
              perturbations},
   JOURNAL = {J. Differential Equations},
  FJOURNAL = {Journal of Differential Equations},
    VOLUME = {457},
      YEAR = {2026},
     PAGES = {114013},
      ISSN = {0022-0396,1090-2732},
   MRCLASS = {35B10 (35B35 35K57 44A10)},
  MRNUMBER = {4998758},
       DOI = {10.1016/j.jde.2025.114013},
       URL = {https://doi.org/10.1016/j.jde.2025.114013},
}

@article{BengeldeRijk2025,
  title={Multiple front and pulse solutions in spatially periodic systems},
  author={Bengel, Lukas and de Rijk, Bj{\"o}rn},
  journal={arXiv preprint arXiv:2502.02467},
  year={2025}
}

\end{document}